\documentclass[twoside,12pt]{book}
\usepackage{setspace,cite} %spacing
\usepackage[a4paper, margin=2.5cm]{geometry}
\usepackage{indentfirst} %?

\usepackage[small]{caption}
\usepackage{subcaption} 
\usepackage{enumitem}
\usepackage{indentfirst}
\usepackage{makeidx,verbatim}
\usepackage{tocbibind}
\usepackage{latexsym,amssymb,amsmath,amsfonts,amsthm}
\usepackage{graphicx,epsf}
\usepackage[T1]{fontenc}

\newtheorem{definition}{Definition}[chapter]
\newtheorem{lemma}{Lemma}[chapter]
\newtheorem{proposition}{Proposition}[chapter]
\newtheorem{theorem}{Theorem}[chapter]
\newtheorem*{nntheorem}{Theorem}
\newtheorem{corollary}{Corollary}[chapter]

\newtheorem*{hypothesis}{Conjecture}

\theoremstyle{definition}
\newtheorem{remark}{Remark}[chapter]
\newtheorem{example}{Example}[chapter]

\numberwithin{equation}{chapter}

\def\sideremark#1{\ifvmode\leavevmode\fi\vadjust{\vbox to0pt{\vss % the remark
      \hbox to 0pt{\hskip\hsize\hskip1em           %                will appear only
 \vbox{\hsize2cm\tiny\raggedright\pretolerance10000%                on the side
 \noindent #1\hfill}\hss}\vbox to8pt{\vfil}\vss}}}%
\newcommand{\rdif}{\text{Diff\ }^r}                                             %

\allowdisplaybreaks[1]
\pdfoutput=1
\begin{document}

\onehalfspacing

\pagestyle{plain}
\frontmatter
\begin{center}
\center{University of Zagreb}
\center{Faculty of Science}
\center{Department of Mathematics}
\bigskip

\bigskip

\center{MAJA RESMAN}
\medskip

\center{\Large \bf{Fixed points of diffeomorphisms,}}
\center{\Large \bf{singularities of vector fields}}
\center{\Large \bf{and epsilon-neighborhoods of their orbits}}
\bigskip

\center{Doctoral thesis}

\bigskip

\bigskip

Supervisors:

Pavao Marde\v si\' c, Universit\' e de Bourgogne, Dijon, France\\
Vesna \v Zupanovi\' c, University of Zagreb, Zagreb, Croatia

\vfill

\center{Zagreb, 2013}
\end{center}
\chapter*{Acknowledgements}

In the first place, I would like to thank my two supervisors, Professor Pavao Marde\v si\' c and Professor Vesna \v Zupanovi\' c, for making this work possible.

I am enormously grateful to Professor Pavao Marde\v si\' c,  for proposing the subject of my thesis, for inviting me to spend my final year at the University of Burgundy to work with him intensively, for always having time and patience to discuss with me, without exception. I am thankful for all the ideas and motivating talks, which kept me interested and moving forward all the time, for the optimism and curiosity he passed along to me, as well as for all the encouragement and support I received.

I am extremely thankful to Professor Vesna \v Zupanovi\' c for many useful ideas, valuable advices and great help and support, for taking care of my progress from the very beginning, and always providing fresh ideas and motivation for research in the field of fractal geometry. I am grateful for our numerous mathematical discussions which kept my work in progress during the last six years. 
\medskip

I would also like to thank all the members of the Seminar for differential equations and nonlinear analysis in Zagreb for motivating talks in fractal analysis. Special thanks goes to Professor Darko \v Zubrini\' c, for useful discussions, and to Professor Mervan Pa\v si\' c and Professor Luka Korkut. I owe my thanks also to many professors I have met during my stay abroad, for their willingness to hear about my research and discuss with me: Professors Jean-Philippe Rolin, David Sauzin, Frank Loray, Guy Casale, Lo\" ic Teyssier, Reynhard Schaefke and Johann Genzmer.
\medskip

Finally, I thank the commitee members for my thesis defense, Professor Jean-Philippe Rolin, Professor Sini\v sa Slijep\v cevi\' c, Professor Sonja \v Stimac and Professor Darko \v Zubrini\' c.
\bigskip

In the end, I would like to mention the financial support that enabled me to spend my final year in Dijon and to finish my thesis: the French government scholarship for the academic year 2012/13 and the 2012 Elsevier/AFFDU grant. 

\chapter*{Summary}

In the thesis, we consider discrete dynamical systems generated by local diffeomorphisms in a neighborhood of an isolated fixed point. Such discrete dynamical systems associate to each point its \emph{orbit}. We investigate to what extent we can \emph{recognize a diffeomorphism and read its intrinsic properties from fractal properties of one of its orbits}. By fractal properties of a set, we mean its \emph{box dimension} and \emph{Minkowski content}. The definitions of the box dimension and the Minkowski content of a set are closely related to the notion of \emph{$\varepsilon$-neighborhoods} of the set. More precisely, one considers the asymptotic behavior of the Lebesgue measure of the $\varepsilon$-neighborhoods, as $\varepsilon$ goes to zero. Thus, in a broader sense, by fractal property of a set we mean the measure of the $\varepsilon$-neighborhood of the set, as a function of small parameter $\varepsilon>0$. When necessary, in the thesis we introduce natural generalizations of fractal properties, which are better adapted to our problems. The relevance of this method lies in the fact that fractal properties of one orbit can be computed numerically.

Local diffeomorphisms appear naturally in many problems in dynamical systems. For example, they appear in planar polynomial systems as the first return maps or Poincar\' e maps of spiral trajectories, defined on transversals to monodromic limit periodic sets (singular elliptic points, periodic orbits and hyperbolic polycycles). With appropriate parametrizations of the transversal, Poincar\' e maps are local diffeomorphisms on $(0,\delta)$, except possibly at zero, having zero as an isolated fixed point. Recognizing the multiplicity of zero as a fixed point of the first return map is important for determining an upper bound on the number of limit cycles that are born from limit periodic sets in unfoldings. This number of limit cycles is called the \emph{cyclicity} of the limit periodic set. The problem is closely related to the \emph{$16^{th}$ Hilbert problem}. 

In the first part of this thesis, we show that there exists a bijective correspondence between the multiplicity of zero as a fixed point of a diffeomorphism $f:(0,\delta)\to (0,\delta)$ and the appropriate generalization of the box dimension of its orbit, which we call the \emph{critical Minkowski order}.
 
\smallskip
Another occurrence of diffeomorphisms is when considering the holonomy maps in $\mathbb{C}^2$. In particular, we consider germs of complex saddle vector fields in $\mathbb{C}^2$. Their \emph{holonomy maps} on cross-sections transversal to the saddle are germs of diffeomorphisms $f:(\mathbb{C},0)\to (\mathbb{C},0)$, with isolated fixed point at the origin. It is known that the formal and the analytic normal form of such fields can be deduced from the formal and the analytic classes of their holonomy maps.

In the second part of the thesis, we therefore study germs of diffeomorphisms $f:(\mathbb{C},0)\to(\mathbb{C},0)$, from the viewpoint of fractal geometry. We restrict ourselves to germs whose linear part is a contraction, a dilatation (the so-called \emph{hyperbolic fixed point cases}) or a rational rotation (the so-called \emph{parabolic fixed point cases}). The hyperbolic cases are easy to treat and most of the time we deal with the non-hyperbolic case of rational rotations, that is, with \emph{parabolic germs}. In this work we omit the very complicated case when the linear part is an irrational rotation.

We give the complete formal classification result: there exists a bijective correspondence between the formal invariants of a germ of a diffeomorphism and the fractal properties of any of its orbits near the origin. We use generalizations of fractal properties that we call \emph{directed fractal properties}. 

On the other hand, we have not solved the analytic classification problem for parabolic germs, using fractal properties of orbits. We state negative results in this direction and problems that are encountered. For germs inside the simplest formal class, we investigate analytic properties of $\varepsilon$-neighborhoods of orbits and compare them with the well-known results about analytic classification.

\vfill

{\bf Keywords}: $\varepsilon$-neighborhoods, box dimension, Minkowski content, fixed points, germs of diffeomorphisms, multiplicity, Poincar\' e map, cyclicity, parabolic germs, complex saddle vector fields, holonomy maps, saddle loop, formal classification, analytic classification, Abel equation, Stokes phenomenon 

\onehalfspacing

\tableofcontents

\onehalfspacing

\mainmatter

\chapter{Introduction}\label{one}
\section{Motivation}\label{oneone}

In this thesis, we use the methods of \emph{fractal analysis}. Our main tool is computing the box dimension and the Minkowski content of sets. The box dimension of trajectories of a dynamical system is an appropriate tool for measuring complexity of the system, and it reveals important properties of the system itself. By their definition, box dimension and Minkowski content of sets are related to the first term in the asymptotic development of the Lebesgue measure of the $\varepsilon$-neighborhood of the set, as $\varepsilon$ tends to zero, see precise definitions in Section \ref{onethree}. In our considerations, we sometimes mean $\varepsilon$-neighborhoods of sets, for small parameters $\varepsilon$, as their fractal property in the broader sense. We mainly consider discrete dynamical systems generated by diffeomorphisms in a neighborhood of a fixed point and tending to the fixed point. We conclude intrinsic properties of the generating diffeomorphism by fractal analysis of one of its \emph{orbits}. The properties we read are important in light of the $16^{th}$ Hilbert problem for planar polynomial vector fields. The fractal method for obtaining them used in this work is new. We investigate where are its limits in recognizing the diffeomorphism. The applicability of our method lies in the fact that fractal properties of one orbit are of purely geometric nature and can be determined numerically.

After a few words about fractal analysis and its historical development, we describe how it was exploited so far in the field of dynamical systems. Fractal analysis has been rapidly evolving since the end of the $20^{th}$ century. It was noted that many sets in the nature (for example, coastline or a snow-flake) have fractal structure, meaning that at every point they are of infinite length and have self-similarity property: any part is similar to the whole. The need to measure their complexity led to the introduction of new notion of fractal dimension, different from topological dimension, which takes noninteger values, depending on the complexity of the set. The motivation can be found in e.g. Mandelbrot's book \cite{mand} about fractals in nature. The famous fractal sets are Koch curve, Cantor set, Sierpinski carpet or Julia set. Their fractal dimensions can be found in e.g. \cite{falconer}. 

The oldest and most widely used fractal dimension is \emph{Hausdorff dimension}, introduced by Carath\' eodory. There are many works where Hausdorff dimension was exploited in dynamical systems. For example, in papers of Douady, Sentenac, Zinsmeister \cite{dsz} and Zinsmeister \cite{zins}, the discontinuity in Hausdorff dimension at a parabolic fixed point of Julia set indicates the moment of change of local dynamics (the so-called \emph{parabolic implosion phenomenon}). 
Many dynamical systems posess the strange, chaotic attractors with fractal structure, for example Lorenz attractor, Smale horseshoe, H\' enon attractor. They are difficult to describe and Hausdorff dimension shows the complexity of such attractors. For their fractal dimension and its application in dynamical systems, see for example the overview article of \ \v Zupanovi\' c, \v Zubrini\' c \cite{overview} and references therein.  

The notion of \emph{box dimension} was introduced later, at the beginning of the $20^{th}$ century, by Minkowski and Bouligand. In literature, it is also called the \emph{limit capacity} or the \emph{box counting dimension}. For the overview and more information on fractal dimensions, see for example the book of Falconer \cite{falconer} or the book of Tricot \cite{tricot}.

In this work we use the box dimension and the Minkowski content as relevant fractal properties of sets. The extensive use of box dimension in the study of dynamical systems and differential equations started around the year 2000 by a group of authors in Zagreb: Pa\v si\' c, \v Zubrini\' c, \v Zupanovi\' c, Korkut et al., in papers e.g. \cite{pas1}, \cite{buletin}, \cite{zubr}, \cite{neveda}. The use of box dimension in their work was motivated by the book of Tricot \cite{tricot}, which provides box dimension of two special sets: of the graph of \emph{$(\alpha,\beta)$-chirp function: $f(x)=x^\alpha\sin x^{-\beta}$, $x\in(0,1]$, $0<\alpha\leq \beta$}, and of spiral accumulating at the origin: $r=\varphi^{-\alpha}$, $\alpha\in(0,1)$. Also, in \cite{lappom}, Lapidus and Pomerance computed the box dimension of one-dimensional discrete sequences accumulating at zero with well-defined asymptotics and connected it in their \emph{modified Weyl-Berry conjecture} to the asymptotics of eigenvalue counting function for fractal strings.

Box dimension of sets in $\mathbb{R}^N$ takes values in the interval $[0,N]$. It takes also noninteger values, depending on \emph{how much of the ambient space the set occupies} (the density of the accumulation of the set). The precise definitions are given in Section~\ref{onethree}. For most sets, it coincides with Hausdorff dimension. Nevertheless, the Hausdorff dimension, unlike the box dimension, posesses the countable stability property, see e.g. \cite{falconer}. This results in the fact that any countable set of poins accumulating at some point, regardless of the density of the accumulation, has Hausdorff dimension equal to $0$.  Similarly, any countable set of smooth curves accumulating at some set has Hausdorff dimension equal to $1$. On the other hand, their box dimension takes values between $0$ and $1$, or $1$ and $2$ respectively, depending on the density of the accumulation. To conclude, Hausdorff dimension for the trajectories of continuous and discrete dynamical systems is trivial and provides no information. The box dimension is an appropriate tool.

This thesis is a natural continuation of the work of the above mentioned group of professors in Zagreb. The considerations are extended to systems where standard box dimension is not enough and we need natural generalizations. Here we present shortly an overview of relevant former results and explain what is done in the thesis.

In a broader sense, the results of \v Zubrini\' c and \v Zupanovi\' c are mostly concerned with the \emph{$16^{th}$ Hilbert problem}, from the viewpoint of fractal geometry. Planar polynomial vector fields are considered. The $16^{th}$ Hilbert problem asks for an upper bound $H(n)$, depending only on the degree $n$ of the polynomial field, on the number of limit cycles (isolated periodic orbits). The problem is so far completely open. To approach this question, it is important to detect invariant sets from which limit cycles are born in generic analytic unfoldings of vector fields. They are called \emph{limit periodic sets}. The question then reduces to a simpler question, on the maximal number of limit cycles that can be born from each limit periodic set in a generic analytic unfolding. This is called the \emph{cyclicity} of the limit periodic set. A nice overview of the $16^{th}$ Hilbert problem and related problems can be found in the book of Roussarie \cite{roussarie}.   

We restrict ourselves to the simplest cases of monodromic limit periodic sets: isolated singular elliptic points (singularities with eigenvalues with nonzero imaginary part, i.e., strong or weak foci), limit cycles, and homoclinic loops (with a hyperbolic saddle point at the origin). Monodromic means that the set is accumulated on one side by spiral trajectories. In these cases, the cyclicity is finite (Dulac problem proven by Il'yashenko \cite{ilja} and Ecalle  \cite{ecalle1} independently) and known, as is also in some special cases of polycycles. See works of Mourtada, El Morsalani, Il'yashenko, Yakovenko and others, as referenced in \cite[Chapter 5]{roussarie}. 

The question arose if it was possible to read these bounds using this new method: fractal analysis of spiral trajectories tending to the limit periodic set. The idea was triggered by article \cite{buletin} of \v Zubrini\' c and \v Zupanovi\' c. In the article, the example of Hopf bifurcation was treated from the viewpoint of fractal analysis. In Hopf bifurcation, a weak focus of the first order bifurcates to strong focus and one limit cycle is born at the moment of bifurcation. It was noted that the box dimension of a spiral trajectory around the focus point jumps from the value $4/3$ for weak focus to the smaller value $1$ at the very moment of bifurcation, obviously signaling the birth of a limit cycle from the weak focus point. To generalize the result to all focus points, in the next article \cite{belgproc} by the same authors, the Takens normal form for a generic unfolding of a vector field in a neighborhood of a focus point from Takens \cite{takens} was related to the box dimension of spiral trajectories tending to the focus point. Thus the box dimension of a spiral trajectory accumulating at a focus point recognizes between strong and weak foci and, further, between weak foci of different orders. It thus signals the number of limit cycles that are born from them in perturbations. In \cite{belgproc}, box dimension of spiral trajectories was related to cyclicity also for limit cycles. In computing the box dimension of a spiral trajectory, the box dimension of an orbit of the Poincar\' e map on a transversal was computed, and related to the box dimension of a spiral trajectory by the flow-sector (for focus points) or the flow-box theorem from \cite{DLA}(for limit cycles). The theorems state a product structure of a trajectory locally around a transversal. This suggested that the box dimension of a spiral trajectory around limit periodic sets in fact carries the same information as the box dimension of a discrete, one-dimensional orbit of its Poincar\' e map. The box dimension of (relevant) one-dimensional discrete systems was computed in Elezovi\' c, \v Zupanovi\' c, \v Zubrini\' c \cite{neveda}. Later, in her thesis \cite{lana} and in paper \cite{lanacl}, Horvat-Dmitrovi\' c considered bifurcations of one-dimensional discrete dynamical systems, noting that a jump in the box dimension of the system indicates the moment of bifurcation, while its size reveals the complexity of bifurcation. The connection with Poincar\' e maps for continuous planar systems was stressed as an application, the birth of limit cycles in the unfolding corresponding to the bifurcation of a fixed point of the Poincar\' e map into new fixed points. Indeed, the cyclicity in generic unfoldings of weak foci and limit cycles equals the multiplicity of zero as a fixed point of the corresponding Poincar\' e map. Instead of considering the spiral trajectories, one can equivalently perform fractal analysis of one-dimensional orbits of the Poincare map on a transversal to get information on cyclicity. In \cite{neveda}, one-dimensional discrete systems generated by functions sufficiently differentiable at a fixed point were considered. The bijective correspondence was found between the box dimension of such systems and the multiplicity of fixed points of the generating functions.  

In the first part of this thesis, we express an explicit bijective correspondence between the cyclicity of elliptic points and limit cycles in generic unfoldings and the behavior of the $\varepsilon$-neighborhood of any orbit of their Poincar\' e maps, as $\varepsilon\to 0$. The behavior is expressed by the box dimension. Then we generalize the results to homoclinic loops and simple saddle polycycles. The results were published in 2012 in the paper \cite{mrz} by Marde\v si\' c, Resman, \v Zupanovi\' c. Unlike in focus or limit cycle case, where Poincar\' e map was differentiable at fixed point and could be expanded in power series, the first return map for homoclinic loop is \emph{no more differentiable} at fixed point.  The theorem from \cite{neveda} connecting the multiplicity of differentiable generators with the box dimension of their orbits cannot be applied. However, it is known that the Poincar\' e maps for generic analytic unfoldings of a homoclinic loop have nice structure: they decompose in a well-ordered (by flatness at zero) scale of logarithmic monomials, which mimics in some way the power scale. This result is given in book of Roussarie \cite[Chapter 5]{roussarie}. Such scale is an easy example of the so-called \emph{Chebyshev scale}. The Chebyshev scales are discussed in detail in the book of Marde\v si\' c \cite{mardesic}. We encounter two problems in generalizing the previous result to non-differentiable generators belonging to a Chebyshev scale. First, the multiplicity, to be well defined, should be exchanged with \emph{multiplicity in a Chebyshev scale}. A definition was introduced by Joyal, \cite{J}. Secondly, we show that for generators not belonging to the power scale, box dimension of their discrete orbits is not precise enough to reveal the exact multiplicity of the generator. By its very definition, box dimension is adapted to the power scale: it compares the area of the $\varepsilon$-neighborhood of sets with powers of $\varepsilon$. Even Tricot in his book \cite[p. 121]{tricot} warns about lack of precision of box dimension for sets with logarithmic dependence on $\varepsilon$ of the area of their $\varepsilon$-neighborhood and emphasizes the need for appropriate, finer scale to which this area should be compared. We introduce a new notion of \emph{critical Minkowski order}, which presents a generalization of box dimension which is adapted to a Chebyshev scale. It compares the behavior of the $\varepsilon$-neighborhoods with appropriate, finer scale for a given problem. With this new notion, we manage to recover a bijective correspondence as before. In cases of homoclinic loops, knowing the critical Minkowski order of only one orbit of Poincar\' e map, together with understanding the scale for a generic unfolding, are sufficient to determine the cyclicity of the homoclinic loop. We stress that the problem of our method for more complicated saddle polycycles lies in the fact that the depth of the logarithmic scale for generic unfoldings is not known in general. The scales have been investigated only in very special cases of polycycles, by El Morsalani, Gavrilov, Mourtada and many others. 
\smallskip

Due to the deficiency of fractal analysis applied directly to planar vector fields for more complicated limit periodic sets, in the second part of the thesis (Chapters~\ref{three} and \ref{four}), we consider complexified systems from the viewpoint of fractal analysis of orbits. By complexifying germs of planar vector fields at both elliptic (weak focus) and hyperbolic (saddle) singular points, we obtain germs of complex saddle vector fields in $\mathbb{C}^2$, see \cite[Chapters 4 and 22]{ilyajak}. It was noticed in \cite{buletin} or \cite{belgproc} that the box dimension of an orbit of the Poincar\' e map or, equivalently, of a spiral trajectory around the elliptic point, distinguishes between \emph{weak} and \emph{strong} foci, which are classified by the order of the first non-zero term in their normal forms. However, planar fractal analysis fails in distinguishing between weak and strong resonant saddle points, since they are not monodromic points: there is no recurring spiral trajectory accumulating at them and the Poincar\' e map is not defined. In this case, we complexify the resonant saddles. The leaves of resonant complex saddles are monodromic and an analogue of the first return map, called the \emph{holonomy map}, is well defined. In this case, we expect the box dimension of leaves, or of orbits of their holonomy maps, to distinguish between weak and strong saddles, which difer by the order of the first non-zero resonant term in their formal normal forms. Already Il'yashenko, in his proof of the Dulac problem about nonaccumulation of limit periodic sets on elementary planar polycycles, considered complexified systems in $\mathbb{C}^2$. Therefore, we hope that the analysis of complexified dynamics may give some insight into unsolved planar cases in the future. 

An important way of classifying and recognizing germs of complex saddle vector fields are their orbital formal and analytic normal forms. We are here concerned only with \emph{orbital formal classification of complex saddles}, which can be found for example in the book of Il'yashenko and Yakovenko \cite[Section 22]{ilyajak} or in the book of Loray \cite[Chapter 5]{loray}. The germs of vector fields with a complex saddle are either formally orbitally linearizable or their formal normal form is described by two parameters called \emph{formal invariants}.  Our first goal is to see if we can read formal invariants of a complex saddle by fractal analysis of one leaf of a foliation near the origin, or, equivalently, by fractal analysis of one orbit of its holonomy map defined on a cross-section to the saddle. In the complex case, one leaf of a foliation can be understood as one trajectory of the system (in complex time), while complex holonomy map is the complex equivalent of the Poincar\' e map. 

As was the case with Poincar\' e maps of planar vector fields, in complex saddle cases the analysis of holonomy maps is sufficient for classifying complex saddles. It was stated by Mattei, Moussu \cite{mm} that formal (analytic) orbital normal forms of germs of complex saddles can be read from formal (analytic) classes of their holonomy maps, see \cite[Th\' eor\` eme 5.2.1]{loray}. Furthermore, by Lemma 22.2 in \cite{ilyajak}, holonomy maps of complex saddles are germs of complex diffeomorphisms fixing the origin, $f:(\mathbb{C},0)\to (\mathbb{C},0)$. Their formal classification was given by Birkhoff, 
Kimura, Szekeres and Ecalle in the mid $20^{th}$ century and can be found in \cite[Section 22B]{ilyajak} or \cite[Section 1.3]{loray}. We consider all germs of diffeomorphisms except the most complicated cases of irrational rotations in the linear part. Sections~\ref{threeone} and \ref{threetwo} of this thesis are thus dedicated to establishing a bijective correspondence between the formal classification of germs of diffeomorphisms of the complex plane and the fractal properties of only one discrete orbit.  The results from this chapter are mostly published in 2013 in paper \cite{formal} by Resman. Since the formal invariants are complex numbers, we had to generalize fractal properties to become complex numbers, revealing not only the density, but also the direction of the orbit. We call them the \emph{directed fractal properties}. By definition, they are related to the \emph{directed area} or the \emph{complex measure} of the $\varepsilon$-neighborhood of the orbit defined here. It incorporates not only the area, but also the center of the mass of the $\varepsilon$-neighborhood. 

The results are then directly applied to germs of complex saddle fields in Chapter~\ref{threethree}. We read the orbital formal normal form of a saddle field, using fractal properties of its holonomy map. Furthermore, we compute the box dimension of a trajectory around a planar saddle loop, and thus give the preliminary steps for computing the box dimension of a leaf of a foliation for germs of resonant complex saddle fields. We state the conjecture connecting the dimension of one leaf of a foliation and the first formal invariant of a resonant nonlinearizable complex saddle. For linearizable resonant saddles, box dimension of a leaf should be \emph{trivial}, that is, equal to 2. The conjecture has yet to be proven.
\smallskip

The formal classification problem for complex germs of diffeomorphisms being fully solved, in Chapter~\ref{four} we investigate how far methods of fractal analyis can bring us in the problem of \emph{analytic classification}. We consider germs of parabolic diffeomorphisms. The analytic classification problem for parabolic diffeomorphism was solved by Ecalle \cite{ecalle} and Voronin \cite{voronin} around the year 1980. From then on, many authors have been working on understanding ideas and tools from \cite{ecalle} and treating the problem from different viewpoints, for example Loray, Sauzin, Dudko etc. For a good overview, we recommend the preprint of Sauzin \cite{sauzin}, the book of Loray \cite{loray} or recent thesis of Dudko \cite{dudko}. Most of the authors restrict to the simplest, model formal class of diffeomorphisms. We also follow this fashion. 
The complexity of the problem lies in the fact that the analytic class of a parabolic diffeomorphism is given by finitely many pairs of germs of diffeomorphisms in $\mathbb{C}$, the so-called \emph{Ecalle-Voronin moduli of analytic classification}. The same can be expressed in terms of infinite sequences of numbers. The complexity of the space of analytic invariants is not unexpected. Indeed, it was shown by Ecalle that, for determining the analytic class, we need information on the whole diffeomorphism. No finite jet of a diffeomorphism is sufficient. More precisely, each parabolic diffeomorphism is formally equivalent to its formal normal form, but the formal change of variables converges only sectorially to analytic functions. The neighboring sectors overlap, resembling the petals of a flower. The analytic conjugacies on sectors are obtained as sectorial solutions of the Abel (trivialisation) difference equation for a diffeomorphism. The difference between them on the intersections of sectors is exponentially small. This is an ocurrence of the famous \emph{Stokes phenomenon}, which can be overviewed in book \cite{stokes}. The Ecalle-Voronin moduli are obtained comparing the analytic solutions on intersections of sectors, and incorporate information on exponentially small differences. The moduli are not computable nor operable even in the simplest cases. Therefore, some authors restrict themselves to considering only the computable tangential derivative to the moduli, for example Elizarov in \cite{elizarov}. 

The approach to the problem of analytic classification using fractal properties of orbits in this thesis is new. It is still not clear whether it is possible to read the analytic moduli using $\varepsilon$-neighborhoods of orbits. The problem seems to be very difficult. In Chapter~\ref{four} of the thesis, we investigate the analyticity properties of the complex measure of $\varepsilon$-neighborhoods of orbits, as function of parameter $\varepsilon>0$ and of initial point $z\in\mathbb{C}$. We show the lack of analyticity in each variable. Nevertheless, we note that the first coefficient dependent on the initial point $z$ in the asymptotic development in $\varepsilon$ of the complex measure of the $\varepsilon$-neighborhood, regarded as function of $z$, has sectorial analyticity property. We call this coefficient \emph{the principal part} of the complex measure. It satisfies the difference equation similar to the Abel (trivialisation) equation, which we call the \emph{$1$-Abel equation}. We generalize both equations introducing the \emph{generalized Abel equations}, and give the necessary and sufficient conditions on a diffeomorphisms for the global analyticity of solutions of their generalized Abel equations. We apply the results to obtain examples which show that the global analyticity of principal parts is not in correlation with the analytic conjugacy of the diffeomorphism to the model. To support this statement theoretically, in a similar way as analytic classes were defined comparing sectorial solutions of Abel equation, we define a new classification of diffeomorphisms comparing the sectorial solutions of $1$-Abel equation. Thus we obtain equivalence classes which we call the \emph{$1$-conjugacy classes}. We show that these new classes are 'far' from the analytic classes. In fact, they are in transversal position with respect to analytic classes. This means, inside each analytic class we can find diffeomorphisms belonging to any $1$-class. We also define higher conjugacy classes, with respect to generalized Abel equations with right-hand sides of higher orders.

\section{The thesis overview}\label{onetwo}

Here we repeat shortly by chapters the main results presented in the thesis. 
\medskip

\textbf{Chapter~\ref{two}} is dedicated to two main results published in Marde\v si\' c, Resman, \v Zupanovi\' c \cite{mrz}. They concern fractal analysis of discrete systems generated by local diffeomorphisms of the real line at a fixed point. In case of generators sufficiently differentiable at fixed point, the bijective correspondence between the multiplicity of the fixed point and the box dimension of any orbit is given in Theorem~\ref{neveda}. In case of generators differentiable except at fixed point and belonging to a Chebyshev scale, we show that the box dimension of orbits cannot recognize the multiplicity precisely. Therefore we introduce the critical Minkowski order, as a generalization of box dimension, which is adapted to a given scale. In Theorem~\ref{chebsaus}, the bijective correspondence is given between the multiplicity of a generator in a given scale and the critical Minkowski order of one orbit. At the end of Chapter~\ref{two}, in Section~\ref{twothree}, the results are applied to Poincar\' e maps for elliptic points, limit cycles and homoclinic loops. The application is in reading the cyclicity of these sets in generic unfoldings from the box dimension or the critical Minkowski order of only one orbit of their Poincar\'e maps.
\medskip

\textbf{Chapter~\ref{three}} treats complex germs of diffeomorphisms $f:(\mathbb{C},0)\to (\mathbb{C},0)$, whose linear part is not an irrational rotation, and germs of resonant complex saddles in $\mathbb{C}^2$. 

In Sections~\ref{threeone} and \ref{threetwo}, fractal analysis of orbits is brought to relation with existing formal classification results. In Section~\ref{threeone}, box dimension of orbits of hyperbolic germs is computed in Proposition~\ref{hype} to be equal to $0$. Its triviality is consistent with analytic linearizability of such germs. In Subsection~\ref{threetwo}, formal classification of parabolic diffeomorphisms is treated. The results were published in \cite{formal}. The area of $\varepsilon$-neighborhood of orbit is generalized as the \emph{directed area} or the \emph{complex measure} of the $\varepsilon$-neighborhood, incorporating the area and the center of the mass of the $\varepsilon$-neighborhood. Three coefficients in its asymptotic development: box dimension, directed Minkowski content and directed residual content are introduced in a natural way. Main Theorems~\ref{fnf} and \ref{fnfe} state the bijective correspondence between the three fractal properties of any discrete orbit near the origin and the elements of the formal normal form of the generating diffeomorphism.
\smallskip

In Section~\ref{threethree}, the results are applied to the formal orbital classification of resonant complex saddles in $\mathbb{C}^2$. In Subsection~\ref{threethreetwo}, the direct application of the previous results to vector fields, using their holonomy maps, in given in Proposition~\ref{holoholo}. In Subsection~\ref{threethreethree}, in Theorem~\ref{codim}, we compute the box dimension of the spiral trajectory around the planar homoclinic loop. It is a preliminary result containing expected techniques for computing the box dimension of leaves of foliations given by vector fields in $\mathbb{C}^2$ with complex saddle. In Subsection~\ref{threethreefour}, we finally state a conjecture about the box dimension of a leaf of a foliation for a resonant formally nonlinearizable saddle: it is in a bijective correspondence with the first element of the formal normal form. For formally linearizable resonant saddles, we conjecture that the box dimension is 2. 
\medskip

In \textbf{Chapter~\ref{four}}, we consider analytic classification of parabolic diffeomorphisms, from the viewpoint of $\varepsilon$-neighborhoods of their orbits. 
In Section~\ref{fourone}, we make a rather long introduction about analytic classification problem from the literature, with definitions and techniques we will need. In Section~\ref{fourthree}, we show that the function of complex measure of $\varepsilon$-neighborhoods of orbits does not posses the analyticity property in any variable. We define \emph{the principal part} of the complex measure as the first coefficient $H(z)$ dependent on the initial point $z$ in the development in $\varepsilon$ of the complex measure of the $\varepsilon$-neighborhoods of orbits. It is regarded as a function of $z$. Theorem~\ref{ppart} states sectorial analyticity properties of this principal part and a difference equation it satisfies. We call such equations the \emph{generalized Abel equations}. In Section~\ref{fourtwo}, we consider analyticity properties of solutions of generalized Abel equations and state in Theorem~\ref{glo} the necessary and sufficient conditions for the existence of a globally analytic solution. Finally, in Theorem~\ref{pringlo}, we characterize the diffeomorphisms whose principal parts of orbits are globally analytic. In Section~\ref{fourfour}, we compare analyticity results concerning principal parts with analytic classification results. We first show some examples that suggest that the analytic classes cannot be read from the principal parts of orbits. Then, to confirm the anticipated, we introduce a new classification of diffeomorphisms using the equation for the principal parts, called the $1$-classification. We show finally, in Theorem~\ref{surjecti} and Proposition~\ref{formclas}, that the newly defined classes are in transversal position with respect to the analytic classes, meaning that inside each analytic class there exist diffeomorphisms belonging to any $1$-class.

\section{Main definitions and notations}\label{onethree}
Here we state main definitions and notations used throughout the thesis. The definitions and notations specific for each chapter, on the other hand, are introduced at the beginning of each chapter. 
\medskip

First we define two \emph{fractal properties} of sets, the box dimension and the Minkowski content. For more details, see for example the book of Falconer \cite{falconer} or Tricot \cite{tricot}.

Let $U\subset\mathbb{R}^N$ be a bounded set. By $U_\varepsilon$, $\varepsilon>0$, we denote its $\varepsilon$-neighborhood:
$$
U_\varepsilon=\{x\in\mathbb{R}^N|\ d(x,U)\leq \varepsilon \}.
$$
Let $U_\varepsilon$, $\varepsilon>0$, be Lebesgue measurable, and let $|U_\varepsilon|$ denote its Lebesgue measure. In the thesis, the Lebesgue measure in $\mathbb{R}$, the length, will be denoted by $|.|$, while the Lebesgue measure in $\mathbb{R}^2$ or $\mathbb{C}$, the area, will be denoted by $A(.)$. The fractal properties of set $U$ are related to the asymptotic behavior of the Lebesgue measure of its $\varepsilon$-neighborhood $|U_\varepsilon|$, as $\varepsilon\to 0$. The rate of decrease of $|U_\varepsilon|$, as $\varepsilon\to 0$, reveals the density of accumulation of the set in the ambient space. It is measured by the box dimension and the Minkowski content of $U$. 

By {\it lower and upper $s$-dimensional  Minkowski content of $U$}, $0\leq s\leq N$, we mean
$$
{\mathcal M}_{*}^{s} (U)=\liminf_{{\varepsilon}\to0}\frac{A(U_\varepsilon)}{{\varepsilon}^{N-s}}\text{\ \  and\ \ }{\mathcal M}^{*s} (U)=\limsup_{{\varepsilon}\to 0}\frac{A(U_\varepsilon)}{{\varepsilon}^{N-s}}
$$
respectively. Furthermore, \emph{lower and upper box dimension of $U$} are defined by
$$
\underline{\dim}_\text{\it B} U=\inf\{s\geq0\  |\ {\mathcal M}_{*}^s(U)=0\},\ \overline{\dim}_\text{\it B} U=\inf\{s\geq0\  |\ {\mathcal M}^{*s}(U)=0\}.
$$
As functions of $s\in[0,N]$, $\mathcal{M}^{*s}(U)$ and $\mathcal{M}_*^s(U)$ are step functions that jump only once from $+\infty$ to zero as $s$ grows, and upper or lower box dimension are equal to the value of $s$ when jump in upper or lower content appears, see Figure~\ref{hohio}.

\begin{figure}[ht]
\begin{center}
\vspace{-23cm}
\includegraphics{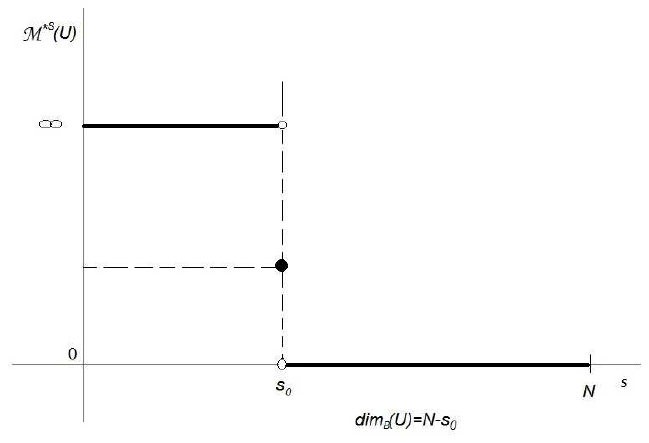}
\vspace{-1cm}
\caption{\small The upper $s$-Minkowski content $\mathcal{M}^{*s}(U)$ of a set $U\subset\mathbb{R}^N$, as a function of $s\in[0,N]$. The moment of jump $s_0$ is indicated in the upper box dimension of $U$.}\label{hohio}
\end{center}
\end{figure}

If $\underline{\dim}_B U=\overline{\dim}_\text{\it B} U$, then we put $\dim_\text{\it B}(U)=\underline{\dim}_B U=\overline{\dim}_\text{\it B} U$ and call it the \emph{box dimension of $U$}. In literature, the upper box dimension of $U$ is also referred to as the \emph{limit capacity} of $U$, see for example \cite{PT}.

If $d=\dim_B(U)$ and $0<\mathcal{M}_{*}^d(U),\ \mathcal{M}^{*d}(U)<\infty$, we say that $U$ is \emph{Minkowski nondegenerate}. If moreover $\mathcal{M}_{*}^d(U)= \mathcal{M}^{*d}(U)\in(0,\infty)$, we say that $U$ is \emph{Minkowski measurable}. The notion was introduced by Stach\' o \cite{stacho} in 1976. In that case, we denote the common value of the Minkowski contents simply by $\mathcal{M}(U)$, and call it the \emph{Minkowski content of $U$}.

In this thesis we deal only with \emph{nice} sets, for which the upper and the lower box dimension and also the upper and the lower Minkowski contents coincide. Therefore, from now on, we speak only about the box dimension $\dim_B(U)$ and the Minkowski content $\mathcal M(U)$.
\smallskip

In the next example, we show how box dimension and, additionally, Minkowski content distinguish between the rates of growth of $\varepsilon$-neighborhoods and thus between densities of sets in the ambient space.
\begin{example}[Box dimension and asymptotic behavior of $\varepsilon$-neighborhoods of sets]\label{fterm}\
\begin{itemize}
\item If $|U_\varepsilon|\sim C \varepsilon^s$, as $\varepsilon\to 0$, $s\in[0,N]$, $C>0$, in the sense that $\lim_{\varepsilon\to 0}\frac{|U_\varepsilon|}{\varepsilon^{s}}=C$, then $\dim_{\text{\it B}}(U)=N-s$ and $\mathcal{M}(U)=C$.
\item If $|U_\varepsilon|\sim C \varepsilon^{s}(-\log\varepsilon)$, as $\varepsilon\to 0$, $C>0$, then $\dim_B(U)=N-s$, but $\mathcal{M}(U)=\infty$, signaling that the set fills more space than in the above example.
\item Similarly, if  $|U_\varepsilon|\sim C \frac{\varepsilon^{s}}{-\log\varepsilon}$, as $\varepsilon\to 0$, $C>0$, then $\dim_B(U)=N-s$, but $\mathcal{M}(U)=0$, signaling lower density of accumulation.
\item If $|U_\varepsilon|\simeq \varepsilon^s$, as $\varepsilon\to 0$, in the sense that there exist $A,\ B>0$ such that $A\leq\frac{|U_\varepsilon|}{\varepsilon^s}\leq B,\ \varepsilon<\varepsilon_0$, then $\dim_B(U)=N-s$, but the upper and the lower Minkowski content do not necessarily coincide.
\end{itemize}
\end{example}

The sets that we study have an accumulation set. For example, the sets consist of points accumulating at the origin, of spiral trajectories accumulating at singular points or polycycles, of hyperbolas accumulating at saddles etc. Box dimension and Minkowski contents of such sets measure the density of the accumulation. In these cases, for computing the behavior of the Lebesgue measure of their $\varepsilon$-neighborhoods, we always use the direct, simple procedure that was described in book of Tricot \cite{tricot}. We divide the $\varepsilon$-neighborhood into \emph{tail} $T_\varepsilon$ and \emph{nucleus} $N_\varepsilon$, different in geometry, and compute their behavior separately. The tail denotes the disjoint finitely many first parts of the $\varepsilon$-neighborhood, while the remaining connected part is called the \emph{nucleus}.
\smallskip

We state some important properties of box dimension that we will use in this work, from Falconer \cite{falconer}. Let $U,\ V\subset\mathbb{R}^N$, such that $\dim_B(U), \dim_B(V)$ exist.
\begin{itemize}
\item (\emph{box dimension under lipschitz mappings}) Let $F:U\to \mathbb{R}^M$ be a \emph{lipschitz mapping}\footnote{There exists a constant $A>0$ such that $\frac{|F(x)-F(y)|}{|x-y|}\leq A,\ \ x,\ y\in U.$}. Then
$$
\dim_B(F(U))\leq \dim_B(U).
$$ 
In particular, if $F$ is bilipschitz\footnote{There exists constants $A,\ B>0$ such that $B\leq\frac{|F(x)-F(y)|}{|x-y|}\leq A,\ \ x,\ y\in U.$}, then
$$
\dim_B(F(U))=\dim_B(U).
$$ In particular, any diffeomorphism is a bilipschitz mapping.

\item (\emph{the finite stability property}) $\dim_B(U\cup V)=\max\{\dim_B(U),\ \dim_B(V)\}$. On the contrary, the \emph{countable stability property} does not hold.
\item (\emph{monotonicity}) Let $U\subset V$. Then, $\dim_B(U)\leq \dim_B(V)$.
\item (\emph{box dimension of the closure}) $\dim_B(U)=\dim_B(\overline U)$.
\item (\emph{Cartesian product}) $\dim_B(U\times V)=\dim_B(U)+\dim_B(V)$.
\end{itemize}
Furthermore, the lipschitz property and the monotonicity property hold for the lower and the upper box dimension. The finite stability property holds only for the upper box dimension.

In the end, let us mention that the box dimension and the Minkowski content, as shown in Example~\ref{fterm}, address only the first term in the asymptotic development of the Lebesque measure of the $\varepsilon$-neighborhood of the set. Further development does not matter. In this thesis, we sometimes need finer information. We sometimes refer to the whole function of the Lebesgue measure of the $\varepsilon$-neighborhoods, $\varepsilon\in(0,\varepsilon_0)$, as a fractal property of the set, or to its asymptotic development in $\varepsilon$ up to a finite number of terms.

\bigskip

To avoid any confusion, we state here the definition of \emph{formal series} and \emph{(formal) asymptotic development} that we use many times throughout Chapters~\ref{two} and \ref{three}. 

Let $\mathcal{I}=\{u_i(x)|\ i\in\mathbb{N}_0\}$ be a sequence of functions $u_i:(0,\delta)\to\mathbb{R}$, ordered by increasing flatness at $x=0$:
\begin{equation}\label{uvj}
\lim_{x\to 0}\frac{u_{i+1}(x)}{u_i(x)}=0,\ i\in\mathbb{N}_0.
\end{equation}
For example, the scale $\mathcal I$ can be the power scale, $\mathcal{I}=\{1,x,x^2,\ldots\}$, the logarithmic scale $\mathcal{I}=\{x(-\log x),x,x^2(-\log x, x^2,\ldots)\}$, the exponential scale $\mathcal{I}=\{e^{-1/x},e^{-2/x},e^{-3/x},\ldots\}$, or any other scale satisfying \eqref{uvj}.

\noindent The series of functions $u_k(x)$, \begin{equation}\label{sumi}\sum_{k=0}^{\infty}\alpha_k u_k(x),\ \alpha_k\in\mathbb{R},\end{equation} without addressing the question of convergence of the series, is called \emph{the formal series}.
Furthermore, we say that a function $f:(0,\delta)\to \mathbb{R}$ has formal asymptotic development \eqref{sumi} or \emph{a formal asymptotic development in the scale $\mathcal{I}$}, as $x\to 0$, if, for every $k\in\mathbb{N}$, it holds that
$$
f(x)-\sum_{i=0}^k \alpha_i u_i(x)=o(u_k(x)),\ x\to 0.
$$

Note that the series \eqref{sumi} may or may not converge in a neighborhood of $x=0$. We do not raise the question of convergence. The function $f$ develops in a given scale, but it may not be true that the series actually converges on any small neighborhood of $0$. 

In the same way, we define the asymptotic development at $x=\infty$.

We use formal asymptotic developments and formal series many times in the thesis: asymptotic developments in Chebyshev scales in Chapter~\ref{two}, formal asymptotic development of $\varepsilon$-neighborhoods, as $\varepsilon\to 0$, in Chapter~\ref{three}, formal changes of variables for parabolic diffeomorphisms in Chapter~\ref{three}, etc.

\smallskip

In complex plane $\mathbb{C}$, we sometimes consider a \emph{formal Taylor series} at $z=0$:
\begin{equation}\label{kjet}
\widehat f(z)=\sum_{k=0}^{\infty} a_k z^k,\ a_k\in\mathbb{C},
\end{equation}
without addressing the question of its convergence.
Usually, it is used in the context of \emph{germs\footnote{The notion of the \emph{germ} refers to a function defined on some small neighborhood of the origin, not addressing the size of its domain.} of formal diffeomorphisms} fixing 0 or \emph{formal changes of variables}, with $\alpha_0=0$ and $\alpha_1\neq 0$ in \eqref{kjet}. They represent a composition of countably many changes of variables of the type $\varphi_1(z)=\alpha_k z$ or $\varphi_k(z)=z+b_k z^k$, $a_k,\ b_k\in\mathbb{C}^*.$ The composition may not converge.

The set of all formal series at $z=0$ will be denoted by $\mathbb{C}[[z]]$. The set $z^k\mathbb{C}[[z]]$ denotes all formal series with initial term of order $k$ or higher, $k\in\mathbb{N}$. 

On the other hand, if the series \eqref{kjet} converges around the origin, we call it an \emph{analytic germ}. The set of all analytic germs is denoted by $\mathbb{C}\{z\}$. 

We adapt the usual convention and denote formal series by hat sign, $\widehat f(z)$, while convergent series are denoted simply by $f(z)$.
\medskip

By $J_n \widehat{f}=\sum_{k=0}^{n} a_k z^k$, we denote the \emph{$n$-jet}, $n\in\mathbb{N}$, of a formal series $\widehat f$ from \eqref{kjet}. 
\medskip

Similarly as in real case, we say that a germ $f:\mathbb{C}\to\mathbb{C}$ has \emph{formal development $\widehat{f}(z)$, as $z\to 0$}, on some open sector $V$ centered at the origin if, for every $n\in\mathbb{N}$ and every closed subsector $W\subset V\cup \{0\}$, there exists a constant $C_{W,n}>0$, such that it holds 
$$
|f(z)-J_n \widehat f(z)|\leq C_{W,n}|z|^{n+1},\ z\in W.
$$

\bigskip

\medskip

Finally, for two real, positive germs of real variable $f,\ g:(0,\delta)\to\mathbb{R}$, we write $$f(x)\sim g(x),\ x\to 0,$$ if $\lim_{x\to 0}\frac{f(x)}{g(x)}=a$, for some $0<a<\infty$. We write $$f(x)\simeq g(x),\ x\to 0,$$ if there exist $A,\ B>0$, and $x_0<\delta$, such that $A\leq f(x)/g(x)\leq B$, for all $x\in(0,x_0)$. The same notation is used for germs at infinity.
\smallskip

For real or complex germs $f(z)$ and $g(z)$, we write $$f(z)=o(g(z)),\ z\to 0,$$ if it holds that $\lim_{z\to 0}\frac{|f(z)|}{|g(z)|}=0$. We write $$f(z)=O(g(z)),\ z\to 0,$$ if there exists a constant $C>0$ and a punctured neighborhood $U$ of $z=0$, such that it holds $|f(z)|\leq C|g(z)|,\ z\in U$. 

\chapter{Application of fractal analyis in reading multiplicity of fixed points for diffeomorphisms on the real line}\label{two}

\section{Introduction and definitions}\label{twozero}
In this chapter, we consider one-dimensional discrete systems on the real line, generated by diffeomorphisms around their fixed points.

Let $g:(0,\delta)\to (0,\delta)$, $\delta>0$, be a function with fixed point $0$, which is a diffeomorphism on $(0,\delta)$, but not necessarily at the fixed point. This function is called \emph{a generator} of a dynamical system.  If $g$ is (sufficiently) differentiable at the fixed point $x=0$, we refer to it as case of \emph{differentiable generator}. If not, we call it \emph{non-differentiable generator} case.
\medskip

Let $x_0\in(0,\delta)$. Suppose that the sequence of iterates $g^{\circ n}(x_0)$, $n\in\mathbb{N}$, remains in $(0,\delta)$. This sequence is called an \emph{orbit generated by diffeomorphism $g$}, with \emph{initial point} $x_0$, and denoted
$$
S^g(x_0)=\{x_n|\ x_{n+1}=g(x_n),\ n\in\mathbb{N}_0\}.
$$
Changing the initial point $x_0\in(0,\delta)$ we get a one-dimensional discrete dynamical system generated by $g$. In this work, we consider generators whose orbits around the fixed point accumulate at the fixed point.
\medskip 

Fractal properties of orbit $S^g(x_0)$, namely its box dimension and Minkowski content, are, by definition in Chapter~\ref{one}, closely related to the asymptotic behavior of the length of the $\varepsilon$-neighborhood of the orbit, denoted $|S^g(x_0)_\varepsilon|$, as $\varepsilon\to 0$. In this chapter, we study the relationship between the multiplicity of a fixed point of a function $g$, and the dependence on $\varepsilon$ of the length of $\varepsilon$-neighborhoods of any orbit of $g$ near the fixed point.
\medskip

In Section~\ref{twoone}, we consider the case of a differentiable generator. The results were mostly given by Elezovi\' c,\ \v Zupanovi\' c,\ \v Zubrini\' c  in \cite{neveda}. In Section~\ref{twotwo}, we generalize these results to non-differentiable cases. Finally, in Section~\ref{twothree}, we apply the results to Poincar\' e maps around monodromic limit periodic sets and to Abelian integrals. The fractal method that considers fractal properties of only one orbit of the Poincar\' e map is a new method in estimating cyclicity of limit periodic sets. All results of this chapter are published in Marde\v si\' c, Resman, \v Zupanovi\' c \cite{mrz}.
\bigskip

We recall here the basic definitions we will use in this chapter. They are mainly taken from the book of P. Marde\v si\' c about Chebyshev systems, \cite{mardesic}.

Recall the standard definition of \emph{multiplicity of a fixed point} of a function \emph{differentiable at the fixed point}.
Let $\rdif[0,\delta)$, $\delta>0$, denote the family of $C^r$-diffeomorphisms on $[0,\delta)$ ($0$ included), $r\in\mathbb{N}\cup\{\infty\}$. Let $g\in\rdif[0,\delta)$ have a fixed point at $x=0$. 

Let $f=id-g$ on $[0,\delta)$. Any fixed point of $g$ becomes a zero point of $f$. 

\begin{definition}[Multiplicity of a fixed point of a differentiable function]\label{smult}
We say that $x=0$ is a fixed point of $g$ of multiplicity $k$, $k\in\mathbb{N}$, $k<r$, and denote $\mu_0^{fix}(g)=k$, if it holds that
\begin{equation}\label{mults}
f(0)=f'(0)=\ldots=f^{(k-1)}(0)=0,\ f^{(k)}(0)\neq 0.
\end{equation}
That is, if $x=0$ is a zero point of $f$ of multiplicity $k$ in the standard sense. 
\end{definition}
Equivalently, since $f\in \rdif[0,\delta)$, $r>k$, condition \eqref{mults} can be expressed in terms of Taylor series for $f$. It holds that $\mu_0^{fix}(g)=k$ if and only if $x^k$ is the first monomial with non-zero coefficient in Taylor expansion of $f=id-g$ at $x=0$. 

Note that this definition strongly depends on sufficient differentiability of $g$ at $x=0$. However, we can put the definition of multiplicity of fixed point for differentiable functions in more general context of multiplicity of fixed point within some family of functions. This definition does not depend on differentiability, and can therefore be generalized to functions non-differentiable at the fixed point.
In fact, multiplicity of a fixed point of $g\in\rdif(0,\delta]$ denotes the number of fixed points that bifurcate from the fixed point in small bifurcations of $g$ within the differentiable family $\rdif(0,\delta]$. This motivates the following definition:

\begin{definition}[Multiplicity of a fixed point within a family of functions, Definition 1.1.1 in \cite{mardesic}]\label{mult}
Let $g:[0,\delta)\rightarrow \mathbb{R}$. Let $\Lambda$ be a topological space of parameters and let $\mathcal{G}=\{g_\lambda|\ \lambda\in \Lambda\}$, $g_\lambda: [0,\delta)\rightarrow \mathbb{R}$, be a family of functions, such that $g=g_{\lambda_0}$, for some $\lambda_0\in\Lambda$. Let $g$ have an isolated fixed point at $x=0$. 
We say that $x=0$ is \emph {a fixed point of multiplicity $m\in\mathbb{N}_0$ of function $g$ in the family of functions ${\mathcal G}$}
 if $m$ is the largest possible integer, such that there exists a sequence of parameters $\lambda_n\to\lambda_0$, as $n\to\infty$, such that, for every $n\in\mathbb{N}_0$, $g_{\lambda_n}$ has $m$ distinct fixed points $y_1^n,\ldots,y_m^n\in[0,\delta)$ different from $x=0$ and $y_j^n\to 0$, as $n\to\infty$, $j=1,\ldots,m$. We write
$$\mu_0^{fix} (g,{\mathcal G})=m.$$ 
If such $m$ does not exist, we say that $\mu_0^{fix} (g,{\mathcal G})=\infty$.
\end{definition}
If we denote $f_\lambda=id-g_\lambda$, $\lambda\in\Lambda$, the above definition can also be expressed as multipicity of zero point $x=0$ of function $f$ in the family of functions $\mathcal{F}=\{f_\lambda|\ \lambda\in\Lambda\}$.

\smallskip
Note that the multiplicity from Definition~\ref{mult} depends on the family within which we consider function $g$. If $g\in\mathcal{G}_1\subset \mathcal{G}$, then obviously 
$$
\mu_0^{fix} (g,{{\mathcal G}_1})\leq \mu_0^{fix} (g,{\mathcal G}).
$$

\begin{example}[\cite{mardesic}]\label{m1}\

\begin{enumerate}
\item (differentiable case) Let $g\in\rdif[0,\delta)$, with fixed point $x=0$. It holds that $$\mu_0^{fix}(g)=\mu_0^{fix}\Big(g,\rdif[0,\delta)\Big).$$ Here, the metric space of parameters is $\Lambda=\rdif[0,\delta)$, with the distance function
$d(g_1,g_2)=\sup_{k=0,\ldots,r}|g_1^{(k)}(0)-g_2^{(k)}(0)|$.

\item (non-differentiable case) Let $\mathcal{F}$ denote the family of all functions $f:[0,\delta)\to\mathbb{R}$ with asymptotic development\footnote{See the definition of asymptotic development in Chapter~\ref{one}.}, as $x\to 0$, in the scale
$$
\mathcal{I}=\{v_0,\ v_1,\ v_2,\ \ldots\},
$$
where $v_{2j}(x)=x^j$ and $v_{2j-1}(x)=x^j(-\log x)$, $j\in\mathbb{N}_0$, extended to zero continuously by $v_i(0)=0$, $i\in\mathbb{N}$.
Let the family $\mathcal{G}$ be derived from family the $\mathcal{F}$ in the usual manner, i.e. $\mathcal{G}=id-\mathcal{F}$. 

Let $f\in\mathcal{F}$, $f(0)=0$, be such that $f(x)\sim v_i(x)$, as $x\to 0$, for some $i\in\mathbb{N}$ $($the first monomial with nonzero coefficient in the asymptotic development of $f(x)$ is $v_i(x))$. Then,
$$
\mu_0^{fix} (g,{{\mathcal G}})=i.
$$
For example, if $f(x)\sim x^3$, then $\mu_0^{fix} (g,{{\mathcal G}})=6$. On the other hand, if we consider $f$ in the subfamily unfolding in the subscale $\mathcal{I}_1=\{v_0,\ v_2,\ \}\subset \mathcal{I}$, we get smaller multiplicity
$
\mu_0^{fix}(g,{{\mathcal G}_1})=3.
$
\end{enumerate}
\end{example}

We dedicate a paragraph to the proof of the differentiable case $1.$ The proof is important and illustrative, since it shows how fixed points bifurcate from a fixed point of multiplicity greater than zero. 
\smallskip

\noindent{Proof of case 1.}(\cite[Example 1.1.1]{mardesic})\

First, let function $f=id-g$ have a zero point of multiplicity bigger than or equal to $m$ in the family $\rdif[0,\delta)$. By Definition~\ref{mult} and by Rolle's theorem applied $m$ times, passing to limit we conclude that $f^{(k)}(0)=0$, $0\leq k\leq m-1$. 

Conversely, suppose that $f^{(k)}(0)=0$, $0\leq k\leq m-1$. The first monomial in expansion of $f$ is then $x^k$ or of higher order. Let $\varepsilon_i>0$, $\varepsilon_i\to 0$, as $i\to\infty$. We construct a sequence $$f_i(x)=\sum_{k=0}^m \alpha_k^i x^k+f(x),\ i\in\mathbb{N},$$ where $\alpha_k^i$ are chosen small enough that $d(f_i,f)<\varepsilon_i$, and, moreover, that each $f_i(x)$ has $m$ different zeros in $(0,\varepsilon_i)$.  

First, we construct $f_1$ in $m+1$ steps from $f$, adding $x^m$ and $m$ missing monomials one by one, with appropriately chosen coefficients. Take $f_{1,m}(x)=\alpha_m^1 x^m+f(x)$, where $\alpha_m^1$ is small enough such that $d(f_{1,m},f)<\varepsilon_1/m$ and $f_{1,m}^{(m)}(0)\neq 0$. Then, take $f_{1,m-1}(x)=\alpha_{m-1}^1 x^{m-1}+f_{1,m}$, with $\alpha_{m-1}^1$ small enough, such that $d(f_{1,m-1},f_{1,m})<\varepsilon_1/m$, $f_{1,m-1}^{(m-1)}(0)\neq 0$ and that $f_{1,m-1}$ has one zero point in $(0,\varepsilon_1)$ different from zero (possible by inverse function theorem applied to $f_{1,m}/x^{m-1}$). We continue in this fashion up to $f_1=f_{1,0}$, adding the last monomial $x^0$. Obviously, $d(f_1,f)<\varepsilon_1$ and we constructed $m$ different zeros in $(0,\varepsilon_1)$. The same can be repeated for $\varepsilon_i$ and $f_i(x)$, $i=2,\ldots,\infty.$ \hfill $\Box$
\smallskip

We note in this construction that, for constructing $m$ zero points bifurcating from zero point $x=0$ of $f$, we need $m$ degrees of freedom ($m$ powers up to the first monomial $x^m$ in $f$, whose coefficients we then choose freely). That is, we need to consider $m$-parameter bifurcations of $f$ (of codimension $m$). 

Proof of $2.$ is done following the same idea, but we have to introduce generalized derivatives that act on nondifferentiable scale in the same manner as standard derivatives act on power scale. We will introduce generalized derivatives below.
\bigskip

Differentiable generators that belong to the class $\rdif[0,\delta)$, $r\in\mathbb{N}\cup\{\infty\},$ unfold by Taylor formula in the scale of powers, $\mathcal{I}=\{x,x^2,x^3,\ldots,x^r\}$. In our study of non-differentiable generators, we restrict ourselves to special classes, which have the asymptotic development in \emph{Chebyshev scales}. The definition of the Chebyshev scale is based on the notion of extended complete Chebyshev (or Tchebycheff) systems (ECT-s), see \cite{J} and \cite{mardesic}. The notion of asymptotic Chebyshev scale was introduced by Dumortier, Roussarie in \cite{dumortier}.

\begin{definition}[Chebyshev scale]\label{cheb}
A finite or infinite sequence $\mathcal{I}=\{u_0,u_1,u_2,\ldots\}$ of functions of the class $C[0,\delta)\cap \rdif(0,\delta)$, $r\in\mathbb{N}\cup\{\infty\}$, is called \emph{a Chebyshev scale} if the following holds:
\begin{itemize} 
\item[i)] A system of differential operators $D_i$, $i=0,\ldots,r$, is well defined on $(0,\delta)$ inductively by the following division and differentiation algorithm:
\begin{eqnarray*}\label{diff}
D_0(u_k)&=&\frac{u_k}{u_0},\\
D_{i+1}(u_k)&=&\frac{(D_i(u_k))'}{(D_i(u_{i+1}))'},\ i=0,\ldots,r,
\end{eqnarray*}
for every $k\in\mathbb{N}_0$, except possibly at $x=0$, to which they are extended by continuity.
\item[ii)] The functions $D_i(u_{i+1})$ are strictly increasing on $[0,\delta)$, $i\in\mathbb{N}_0$.
\item[iii)] $\lim_{x\to 0} D_j u_i (x)=0$, for $j<i$, $i\in\mathbb{N}_0$.
\end{itemize} 

\noindent We call $D_i(f)$ \emph{the $i$-th generalized derivative of $f$ in the scale $\mathcal{I}$}.
\end{definition}

\medskip
\begin{example}[Examples of Chebyshev scales]\

\begin{enumerate}
\item[$i)$] differentiable case: $\mathcal{I}=\{1,\ x,\ x^2,\ x^3,\ x^4,\ldots\}$,
\item[$ii)$] non-differentiable cases:
\begin{itemize}
\item[-]$\mathcal{I}=\{x^{\alpha_0},\ x^{\alpha_1},\ x^{\alpha_2},\ldots\}$,\ $\alpha_i\in\mathbb{R}$,\ $0<\alpha_0<\alpha_1<\alpha_2<\ldots$
\item[-] $\mathcal{I}=\{e^{-\frac{\alpha_1}{x}},\ e^{-\frac{\alpha_2}{x}},\ e^{-\frac{\alpha_2}{x}},\ldots\}$, $\alpha_i\in\mathbb{R}$,\ $0<\alpha_0<\alpha_1<\alpha_2<\ldots$
\item[-]$\mathcal{I}=\{1,\ x(-\log x),\ x,\ x^2(-\log x),\ x^2,\ x^3(-\log x),\ x^3,\ldots\}$
\item[-] More generally, any scale of monomials of the type $x^k(-\log x)^l$, ordered by increasing flatness:
$$
x^i(-\log x)^j<x^k(-\log x)^l \text{ if and only if } (i<k) \text{ or } (i=j \text{ and } j>l).
$$

\end{itemize}
\end{enumerate}
\end{example}
\bigskip

We say that a function $f$ has a \emph{development of order $k$ in Chebyshev scale $\mathcal{I}=\{u_0,\ldots,u_k\}$ } if there exist coefficients $\alpha_i\in\mathbb{R},\ i=0,\ldots,k$, such that
\begin{equation}\label{asymp}
f(x)=\sum_{i=0}^{k} \alpha_i u_i(x)+\psi_k(x),
\end{equation}
and the generalized derivatives $D_i(\psi_k(x))$, $i=0,\ldots,k$, verify $D_i(\psi_k(0))=0$ (in the limit sense). Similarly, we say that $f$ has \emph{an asymptotic development in Chebyshev scale} $\mathcal{I}=\{u_0,u_1,\ldots\}$ if there exists a sequence $\alpha_i,\ i\in\mathbb{N},$ such that for every $k$ there exists $\psi_k$ such that \eqref{asymp} holds. Note that this is just a reformulation of definition of asymptotic development in a scale from Section~\ref{onethree} for Chebyshev scales.

\medskip
The generalized derivatives $D_i$ act on Chebyshev scales in the same way as standard derivatives act on the power scale: the $k$-th generalized derivative annulates the first $k-1$ monomials of the Chebyshev scale. Therefore, in the asymptotic development \eqref{asymp} above,
$$D_i(f)(0)=\alpha_i,\ i\in\mathbb{N}_0$$
Furthermore, it is equivalent
$$
f(x)\sim u_k(x),\ x\to 0, \text{\quad and \quad} D_i(f)(0)=0,\ i=0,\ldots,k-1,\ D_k(f)(0)\neq 0.
$$

We say that a parametrized family $\mathcal{F}=\{f_\lambda|\lambda\in \Lambda\}$ has a \emph{uniform development of order $k$ in a family of Chebyshev scales}
$\mathcal{I_\lambda}=(u_0(x,\lambda),\ldots,u_k(x,\lambda))$, if there exist coefficients $\alpha_i(\lambda)\in\mathbb{R}$, $i=1,\ldots,k$, such that it holds
\begin{equation}\label{unasymp}
f_\lambda(x)=\sum_{i=0}^{k} \alpha_i(\lambda)\cdot u_i(x,\lambda)+\psi_k(x,\lambda),\quad \lambda\in \Lambda,
\end{equation}
and the generalized derivatives $D_i(\psi_k(x,\lambda))$, $i=0,\ldots,k$, verify $D_i(\psi_k(0,\lambda))=0$, in the limit sense uniformly with respect to $\lambda\in\Lambda$.
\bigskip

The following lemma is a generalization of the statement from Example~\ref{m1},1., where differentiable case was considered. It is a combination of results from Lemma 1.2.2 in \cite{mardesic} and Joyal's Theorem 21. in \cite{roussarie}.
\begin{lemma}[Generalized derivatives and multiplicity]\label{multgen}
Let $\mathcal{F}=(f_ \lambda)$ be a family of functions having a uniform development of order $k$ in a family of Chebyshev scales $\mathcal{I}_\lambda=(u_0(x,\lambda),\ldots,u_k(x,\lambda))$, $\lambda\in \Lambda$. Let $f=f_{\lambda_0}\in\mathcal{F}$. Let $g,\ g_\lambda,\ \mathcal{G}$ be derived from $f,\ f_\lambda,\ \mathcal{F}$ in the usual way, $f_\lambda=id-g_\lambda$. If the  generalized derivatives of $f$ at $x=0$ satisfy 
\begin{equation}\label{genderr}
D_i(f)(0)=0,\ i=0,\ldots,m-1, \text{ and } D_{m}(f)(0)\neq 0,\ \ \text{for some\ } m\leq k,
\end{equation} (that is, if $\alpha_{m}(\lambda_0)$ is the first nonzero coefficient in the development of $f$), then 
$$
\mu_0^{fix}(g,\mathcal{G}) \leq m.
$$
Moreover, if $\Lambda\subset \mathbb{R}^N$, $m\leq N$, and the matrix 
\begin{equation}\label{ra}\Big[\frac{\partial \alpha_i(\lambda_0)}{\partial \lambda_j}\Big]_{i=0\ldots m-1,\ j=1\ldots m}\end{equation} is of maximal rank $m$, then \eqref{genderr} is equivalent to $\mu_0^{fix}(g,\mathcal{G}) = m$. 
\end{lemma}

\emph{Idea of the proof.} 
Proof is similar as in Example~\ref{m1}, using generalized derivatives instead of standard. The distance function on $\mathcal{F}\subset C[0,\delta)\cap \rdif(0,\delta)$ is given analogously by $d(f,g)=sup_{i=0,\ldots,r}|D^i(f)(0)-D^i(g)(0)|$, $f,\ g\in\mathcal{F}$.

One direction follows as before by Rolle's theorem. The contrary does not hold without regularity assumption \eqref{ra}. The difficulty here is that we are restricted by the family in the choice of small deformations $\alpha_i(\lambda),\ i=0,\ldots,m-1,$ that, added to $f$, need to generate $m$ small zeros.   
Nevertheless, if condition \eqref{ra} is satisfied, by the implicit function theorem we get freedom in choice of small coefficients $\alpha_i$, by expressing parameters $\lambda_j$, $j=1,\ldots, m,$ as functions of independent variables $\alpha_0,\ldots,\alpha_{m-1},\lambda_{m+1},\ldots,\lambda_N$.\hfill$\Box$

\section{Generators differentiable at a fixed point}\label{twoone}

In this section, we consider generators sufficiently differentiable at a fixed point and state a bijective correspondence between the multiplicity of the fixed point and the box dimension of any orbit tending to the fixed point. The results are just a reformulation of results from Elezovi\' c, \v Zupanovi\' c, \v Zubrini\' c  (see Theorems 1, 5 and Lemma 1 in \cite{neveda}).

\begin{theorem}[Multiplicity of fixed points and box dimension of orbits, differentiable case, Theorem 1 from \cite{neveda} reformulated]\label{neveda}
Let $f$ be sufficiently differentiable on $[0,\delta)$, $f(0)=0$ and positive on $[0,\delta)$. Let $g=id-f$ and let $S^g(x_0)$ be any orbit with initial point $x_0$ sufficiently close to 0.

If $1<\mu_0^{fix}(g)<\infty$, then it holds 
\begin{equation}\label{sausage}
|S^g(x_0)_\varepsilon| \simeq \varepsilon^{1/\ \mu_0^{fix}(g)},\text{ as $\varepsilon\to 0$}.
\end{equation} 

If $\mu_0^{fix}(g)=1$ and additionally $f(x)<x$ on $(0,\delta)$, then it holds that
\begin{equation}\label{mu1}
|S^g(x_0)_\varepsilon| \simeq \left\{ \begin{array}{ll} \varepsilon(-\log \varepsilon),&\text{ if }f'(0)<1\\
\varepsilon \log(-\log \varepsilon),&\text{ if }f'(0)=1
\end{array}\right.,\text{ as $\varepsilon\to 0$}.
\end{equation} 

Moreover, for $1\leq \mu_0^{fix}(g)<\infty$, a bijective correspondence holds
\begin{equation}\label{boxd}
\mu^{fix}_0 (g)=\frac{1}{1-\dim_B(S^g(x_0))}.
\end{equation}
\end{theorem}

\begin{proof}
By Taylor formula applied to a sufficiently differentiable function $f$, we get $f(x)\simeq x^{\mu_0^{fix}(g)},\ x\to 0$. Therefore we are under assumptions of Theorems~1 and 5 from \cite{neveda} and the dimension result \eqref{boxd} follows from these theorems. However, the asymptotic development of $|S^g(x_0)_\varepsilon|$ was not explicitely computed there, therefore we do it here. 

We estimate the length $|S^g(x_0)_\varepsilon|$ directly, dividing the $\varepsilon$-neighborhood of $S^g(x_0)$ in two parts: the nucleus, $N_\varepsilon$, and the tail, $T_\varepsilon$. This way of computing was suggested by Tricot \cite{tricot}. The tail is the union of all disjoint $(2\varepsilon)$-intervals of the $\varepsilon$-neighborhood, before they start to overlap. It holds that
\begin{equation}\label{ukku}
|S^g(x_0)_\varepsilon|=|N_\varepsilon|+|T_\varepsilon|.
\end{equation}

Let $n_\varepsilon$ denote the critical index separating the tail and the nucleus. It describes the moment when $(2\varepsilon)$-intervals around the points start to overlap. The critical index is well-defined since the points of orbit $S^g(x_0)$ tend to zero with strictly decreasing distances between consecutive points.
We have that
\begin{equation}\label{up1}
|N_\varepsilon|=x_{n_\varepsilon}+\varepsilon,\quad |T_\varepsilon|\simeq n_\varepsilon\cdot\varepsilon,\ \varepsilon\to 0.
\end{equation}
Denote by $d_n=|x_{n+1}-x_n|$ the distances between consecutive points.
To compute asymptotic behavior of $n_\varepsilon$, as $\varepsilon\to 0$,
we have to solve (to first term only)
\begin{equation}\label{nepsilo} 
d_{n_\varepsilon}\simeq 2\varepsilon.
\end{equation} 
By Theorem 1 in \cite{neveda}, in case $1\leq \mu_0^{fix}(g)<\infty$ the points $x_n$ of orbit $S^g(x_0)$ and their distances $d_n$ have the following asymptotic behavior:
$$
x_n\simeq n^{-\frac{1}{\mu_0^{fix}(g)-1}},\ d_n=f(x_n)\simeq n^{-\frac{\mu_0^{fix}(g)}{\mu_0^{fix}(g)-1}},\ n\to\infty.
$$
In case $\mu_0^{fix}(g)=1$ and $f(x)<x$, it either holds
\begin{enumerate}
\item[(i)] $g(x)=x^{\beta}+o(x^{\beta}),\ \beta\in\mathbb{N},\ \beta>1$,\quad if $f'(0)=1$, or 
\item[(ii)] $g(x)=kx+o(x)$, $k\in(0,1)$,\quad if $f'(0)<1$. 
\end{enumerate}
Directly iterating $x_{n+1}=g(x_n)$, and since $d_n=f(x_n)\simeq x_n,$ $n\to\infty$, we get the following estimates
\begin{align}\label{ll}
&\text{case $(i)$:}\quad C_1\cdot (Ax_0)^{\beta^n}\leq x_n\leq C_2\cdot (Bx_0)^{\beta^n},\ \ D_1\cdot (Ax_0)^{\beta^n}\leq d_n\leq D_2\cdot (Bx_0)^{\beta^n},\nonumber\\
&\text{case $(ii)$:\quad }C_1\cdot k_1^n x_0\leq x_n\leq C_2 \cdot k_2^n x_0,\ \ D_1\cdot k_1^n x_0\leq d_n\leq D_2\cdot k_2^n x_0,
\end{align}
for $n\geq n_0$ and some positive constants $0<k_1<k_2<1$ and $A,\ B,\ C_1,\ C_2,\ D_1,\ D_2>0$.

We illustrate further computations only in case $g(x)=kx+o(x)$, $k\in(0,1)$. Other cases can be treated similary. Using \eqref{nepsilo} and \eqref{ll}, we conclude
\begin{equation}\label{laz}
n_\varepsilon\simeq -\log\varepsilon,\ \varepsilon\to 0.
\end{equation}
Since $x_{n_\varepsilon}\simeq d_{n_\varepsilon}\simeq \varepsilon$, $\varepsilon\to 0$, from \eqref{laz} and \eqref{up1} we get
$$
|N_\varepsilon|\simeq \varepsilon,\qquad |T_\varepsilon|\simeq \varepsilon(-\log\varepsilon),\quad \varepsilon\to 0. 
$$
By \eqref{ukku}, $|S^g(x_0)_\varepsilon|\simeq \varepsilon(-\log\varepsilon)$, as $\varepsilon\to 0$.
\end{proof}
\medskip

In the case $\mu_0^{fix}(g)=1$ (equivalently, $|g'(0)|<1$), the fixed point zero of $g$ is called a \emph{hyperbolic fixed point}. The definition of a hyperbolic fixed point of a diffeomorphism can be found in e.g. \cite[Definition 1]{perko}. We distinguish between two hyperbolic cases, $|f'(0)|<1$ or $|f'(0)|=1$. We call the latter case \emph{degenerate hyperbolic}. If $\mu_0^{fix}(g)>1$, the fixed point zero is called a \emph{non-hyperbolic fixed point}. 

At hyperbolic fixed points, the convergence of orbits to the fixed point is \emph{exponentially fast}.  Furthermore, at degenerate hyperbolic fixed points the convergence is faster than at standard hyperbolic points. To illustrate, Figure~\ref{onedimsystems} below shows orbits accumulating at fixed point zero of differentiable generators in degenerate hyperbolic, hyperbolic and nonhyperbolic cases.

\begin{figure}[ht]
\centering
\vspace{-1cm}
\begin{subfigure}{.3\textwidth}
  \centering
  \includegraphics[trim={-1cm 25cm 0cm 15cm},clip,width=2.6\linewidth]{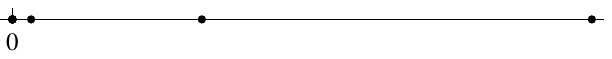}
  \vspace{3cm}
  \caption{$g(x)=x^2+x^4$}
\end{subfigure}%
\begin{subfigure}{.3\textwidth}
  \centering
  \includegraphics[trim={-1cm 25cm 0cm 15cm},clip,width=2.6\linewidth]{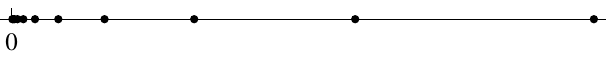}
  \vspace{3cm}
  \caption{$g(x)=1/2 x+x^3$}
\end{subfigure}
\begin{subfigure}{.3\textwidth}
  \centering
  \includegraphics[trim={-1cm 25cm 0cm 15cm},clip,width=2.6\linewidth]{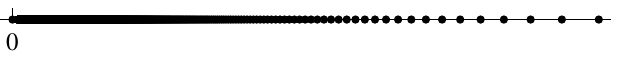}
  \vspace{3cm}
  \caption{$g(x)=x-x^3+x^4$}
\end{subfigure}
\caption{Orbits generated by diffeomorphism $g$ with  $(a)$ degenerate hyperbolic, $(b)$ hyperbolic or $(c)$ non-hyperbolic fixed point, with the same initial point.}
\label{onedimsystems}
\end{figure}

We see in Theorem~\ref{neveda} that trivial box dimension of orbits in hyperbolic cases recognizes exponentially fast convergence. However, box dimension of orbits cannot distinguish between hyperbolic and faster, degenerate hyperbolic cases. On the other hand, we see in \eqref{mu1} in Theorem~\ref{neveda} that the first term in the asymptotic development in $\varepsilon$ of the length of the $\varepsilon$-neighborhood of orbit shows the difference.

\smallskip
Already on this hyperbolic fixed point example we have noticed that more precise information is carried in the first asymptotic term of $\varepsilon$-neighborhood than in box dimension. The idea of considering the behavior of the length of the $\varepsilon$-neighborhoods of orbits instead of only the box dimension of orbits will be important in non-differentiable generator cases. By its very definition, box dimension compares the length of the $\varepsilon$-neighborhood to power scale in $\varepsilon$. Therefore, box dimension carries complete information on the asymptotic behavior of the length only in special cases, when behavior is of power type. This was the case for differentiable generators at non-hyperbolic fixed points treated in Teorem~\ref{neveda}, see formula \eqref{sausage}. In these cases, box dimension turns out to be a sufficient tool for recognizing multiplicity.

\section{Non-differentiable generators at a fixed point}\label{twotwo}

Note that in Theorem~\ref{neveda} in Section~\ref{twoone} we assumed the generator to be differentiable at fixed point $x=0$. In this section, we generalize Theorem~\ref{neveda} to non-differentiable generators at $x=0$, but with asymptotic developments in Chebyshev scales. All the notation and results from this section are published in Marde\v si\' c, Resman, \v Zupanovi\' c \cite{mrz}.

We first state definitions we introduced in \cite{mrz} that compare the asymptotic behavior of functions at $x=0$ with power functions.
\begin{definition}[Weak comparability to powers and sublinearity]\label{comparable}
A positive function $f:(0,d)\to\mathbb{R}$, $d>0$, is \emph{weakly comparable to powers} if there exist $\delta>0$ and constants $m>0$ and $M>0$ such that
\begin{equation}\label{lowup}
m\leq x\cdot (\log f)'(x) \leq M,\ x\in(0,\delta).
\end{equation} 
We call the left-hand side of \eqref{lowup} the \emph{lower power condition} and the right-hand side the \emph{upper power condition}. 
A function $f$ is \emph{sublinear} if it satisfies the lower power condition and $m>1$.
\end{definition}

A similar notion of comparability with power functions in Hardy fields appears in literature, see Fliess, Rudolph \cite{fliess} and Rosenlicht \cite{hardy}. A Hardy field $H$ is a field of real-valued functions of the real variable defined on $(0,\delta)$, $\delta>0$, closed under differentiation and with valuation $\nu$ defined in an ordered Abelian group. Let $f,\ g\in H$ be positive on $(0,\delta)$ and let $\lim_{x\to 0}f(x)=0$, $\lim_{x\to 0}g(x)=0$. If there exist integers $M, N\in\mathbb{N}$ and positive constants $\alpha,\beta>0$ such that 
\begin{equation}\label{class}
f(x)\leq \alpha g(x)^M \text{ and } g(x)\leq \beta f(x)^N,\ x\in(0,\delta),
\end{equation} 
it is said that $f$ and $g$ belong to the same comparability class/are comparable in $H$. \\ 
%\edz{Remark 2 utopljen u tekst, modificiran}
Let us state a sufficient condition for comparability from Rosenlicht \cite{hardy}, Proposition 4:
\begin{proposition}[Proposition 4 in \cite{hardy}]
Let $H$ be a Hardy field, $f(x),\ g(x)$ nonzero positive elements of $H$ such that  $\lim_{x\to 0}f(x)=0$, 
and $\lim_{x\to 0}g(x)=0$. If 
\begin{equation}\label{rosen}
\nu((\log f)')=\nu((\log g)'),
\end{equation} then $f$ and $g$ are comparable.
\end{proposition}

The condition \eqref{rosen} is equivalent to (see Theorem 0 in \cite{hardy})
\begin{equation}\label{strcomp}
\lim_{x\to 0}\frac{(\log f)'(x)}{(\log g)'(x)}=L,\ 0<L<\infty.
\end{equation}
Rosenlicht's condition \eqref{rosen}, i.e. \eqref{strcomp} is stronger than our condition \eqref{lowup} of weak comparability to powers. If $\lim_{x\to 0}\frac{(\log f)'(x)}{1/x}=L,\ 0<L<\infty$ (\eqref{lowup} obviously follows), then $f$ is comparable to power functions in the sense \eqref{class}.
\bigskip

\noindent We explain the conditions from Definition~\ref{comparable} on following Examples~\ref{cps} and \ref{constr}.

\begin{example}[Weak comparability to powers and sublinearity]\label{cps}\ 
\begin{enumerate}
\item Functions of the form
$$
f(x)=x^\alpha (-\log x)^\beta,\ \alpha>0,\ \beta\in\mathbb{R},
$$
are weakly comparable to powers.\\This class obviously includes functions of the form $x^\alpha,\ x^\alpha(-\log x)^{\beta}$ and $\frac{x^\alpha}{(-\log x)^\beta}$, for $\alpha>0$ and $\beta>0$. If additionally $\alpha>1$, they are also sublinear.

\item Functions of the form $$f(x)=\frac{1}{(-\log x)^\beta},\ \beta>0,$$
do not satisfy the lower power condition in $(\ref{lowup})$.

\item Infinitely flat (at $x=0$) functions of the form $$f(x)= e^{-\frac{1}{x^\alpha}},\ \alpha>0,$$
do not satisfy the upper power condition in $(\ref{lowup})$, but they are sublinear.

\item More generally, functions infinitely flat at zero\footnote{$f:(0,\delta)\to\mathbb{R}$ is infinitely flat at zero if all the derivatives vanish at $x=0$: $f^{(k)}(0)=0$,\ for all $k\in\mathbb{N}_0$ (in the limit sense).} do not satisfy the upper power condition.
\smallskip

\noindent \emph{Proof of 4.}
Suppose the contrary. The upper power condition can easily be reformulated as $\Big(\log\frac{f(x)}{x^M}\Big)'\leq 0,\ x\in(0,\delta)$. This implies that $f(x)/x^M$ is a nonincreasing positive function on $(0,\delta)$. On the other hand, if $f$ is infinitely flat at $x=0$, then, by L'Hospital rule, $\lim_{x\to 0}\frac{f(x)}{x^a}=0$, for every $a>0$. In particular, it is true for $\alpha=M$. This is a contradiction.  \qed
\end{enumerate}
\end{example}

The converse of \emph{4.} in Example~\ref{cps} does not hold -- the upper power condition is not equivalent to not being infinitely flat at zero. There exist functions that are not infinitely flat, but nevertheless do not satisfy the upper power condition on any small interval. We construct an example in Example~\ref{constr}.

\begin{example}[Function not infinitely flat at zero, upper power condition not satisfied]\label{constr}\ 

In the construction, the main idea is to bound the function by power functions $x^{\alpha+1}$ and $x^\alpha$, $\alpha>0$, therefore it cannot be infinitely flat. Next we construct intervals, tending to zero, on which its logarithmic growth is faster than the logarithmic growth of $x^\alpha$. This violates the upper power condition. We construct function $f$ in logarithmic chart. That is, we construct function $h(x)=\log f(x)$ on some interval $(0,\delta)$.

Let $h_1(x)=\log (x^{\alpha})=\alpha\log x$ and let $h_2(x)=\log (x^{\alpha+1})=(\alpha+1)\log x$. Let us take $x_1$ close to $x=0$. The segment $I_1$ connects the points $(x_1,h_1(x_1))$ and $(x_1/2, h_2(x_1/2))$. Now we choose point $x_2$ such that $h_1(x_2)<h_2(x_1/2)$ (to ensure that $f$ is increasing). We get segment $I_2$ by connecting $(x_1/2,h_2(x_1/2))$ and $(x_2,h_1(x_2))$. We repeat the procedure with $x_2$ instead of $x_1$ to get segment $I_3$, etc. Inductively, we get the sequence $(x_n)$ tending to $0$, as $n\to\infty$, and the sequence of segments $(I_n)$ which are becoming steeper and steeper very quickly, see Figure \ref{fig}. 

The graph of our function $h$ is the union of the segments $\bigcup_{n=1}^{\infty}I_n$, smoothened on edges. Obviously $f(x)=e^{h(x)}$ is bounded by $x^{\alpha+1}$ and $x^\alpha$. Furthermore, if we take the sequence $(y_n)$ such that $x_n/2<y_n<x_n$, we compute
$$
h'(y_n)\cdot y_n=\frac{\alpha\log x_n-(\alpha+1)\log \frac{x_n}{2}}{x_n/2}\cdot y_n\ \simeq\ -\log x_n, \text{\ \ as $n\to\infty$}.
$$ 
For the sequence $(y_n)$, tending to $0$, it holds that $h'(y_n)y_n\to \infty$, as $n\to\infty$. Therefore, $f$ does not satisfy the upper power condition.
\end{example}

\begin{figure}[ht]
\centering
\vspace{-32cm}
\includegraphics[scale=1.3, trim={-2.5cm 0cm 0cm 0cm}]{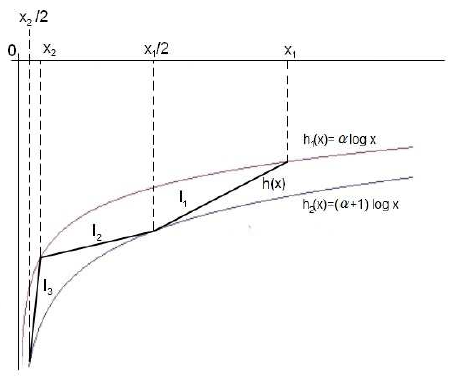}
\vspace{-1.2cm}
\caption{{\small Function $h(x)=\log f(x)$ from Example~\ref{constr}}.}
\label{fig}

\end{figure}

\subsection{Main results}\label{twotwoone}
We now state our first main theorem about the asymptotic behavior in $\varepsilon$ of the length of the $\varepsilon$-neighborhoods of orbits, which is valid also for non-differentiable generators at a fixed point. Its proof is in Subsection~\ref{twotwotwo}.
\begin{theorem}[Asymptotic behavior of lengths of $\varepsilon$-neighborhoods of orbits, (non) differentiable case]\label{gensaus}
Let $f\in C[0,\delta)\cap \rdif(0,\delta)$ be positive on $(0,\delta)$ and let $f(0)=f'(0)=0$\footnote{$f'(0)=0$ is meant in the limit sense, as $x\to 0$.}. Assume that $f$ is a sublinear function. 
Let $g=id-f$. For any initial point $x_0$ sufficiently close to the origin, it holds that
\begin{equation}\label{gensausage}
|S^g(x_0)_\varepsilon| \simeq f^{-1}(\varepsilon),\ \varepsilon\to 0.
\end{equation} 
\end{theorem}
\noindent The sublinearity condition $m>1$ in the lower power condition cannot be omitted from Theorem~\ref{gensaus}:

\begin{remark}[Sublinearity in Theorem~\ref{gensaus}]\label{superlin}
The condition $m>1$ in the lower power condition in Theorem~\ref{gensaus} cannot be weakened. If we take, for example, the function $$f(x)=\frac{x}{-\log x},$$ it obviously satisfies all assumptions of Theorem \ref{gensaus}, except sublinearity: the lower power condition holds only for $m\leq 1$. If we compute $|S^g(x_0)_\varepsilon|$ for this  function directly (as it is computed in the proof of Theorem \ref{gensaus} in Subsection~\ref{twotwotwo}), we get that $\frac{|S^g(x_0)_\varepsilon|}{f^{-1}(\varepsilon)}$ tends to infinity, as $\varepsilon\to 0$, and therefore the conclusion \eqref{gensausage} is not true.

On the other hand, for functions of the form $$f(x)=\frac{x^{1+\alpha}}{-\log x},\ \ \alpha>0,$$ which are obviously sublinear with $m=1+\alpha>1$, the same computation shows that $|A_{\varepsilon}(S^g(x_0))|\simeq f^{-1}(\varepsilon)$, as $\varepsilon\to 0$. 
\end{remark}
\bigskip

In the sequel, we will consider generators which unfold in Chebyshev scales. These classes of generators are important in applications, see Section \ref{twothree}. Our goal is to recover an analogue of Theorem~\ref{neveda} in non-differentiable cases: to read multiplicity of fixed points of such generators in appropriate families from asymptotic behavior of the lengths of $\varepsilon$-neighborhoods of orbits. Theorem~\ref{gensaus} shows that, for non-differentiable generators, the behavior of length of $\varepsilon$-neighborhoods of orbits in $\varepsilon$ is not of power-type. On the other hand, by its definition, box dimension compares lengths with power scale and \emph{hides} precise information on behavior. We illustrate it on the following example. It is an example of generators distinct in growth, that have equal box dimensions of orbits. The difference in density of orbits is too small for box dimension to see it. It is not visible in power scale.

\begin{example}[Deficiency of box dimension for non-differentiable generators]\label{defffi} 
Let $$f_1(x)=x^k\text{ and }f_2(x)=x^k (-\log x), k>1,$$ and let $g_1=id-f_1$ and $g_2=id-f_2$. For any two orbits generated by $g_1$ and $g_2$, with initial point close to the origin, it holds that 
$$\dim_B(S^{g_1}(x_0))=\dim_B(S^{g_2}(y_0))=1-\frac{1}{k}.$$
\end{example}
\begin{proof}
Box dimension for $g_1$ follows from Theorem~\ref{neveda}. For $g_2$, we have that
\begin{align*}
\lim_{x\to 0}\frac{f_2(x)}{x^{k+\delta_1}}=+\infty,\ \delta_1\geq 0,&\quad \lim_{x\to 0}\frac{f_2(x)}{x^{k-\delta_2}}=0,\ \delta_2>0,\\
\lim_{\varepsilon\to 0}\frac{f_2^{-1}(\varepsilon)}{\varepsilon^{1/(k+\delta_1)}}=0,\ \delta_1\geq 0,&\quad \lim_{\varepsilon\to 0}\frac{f_2^{-1}(\varepsilon)}{\varepsilon^{1/(k-\delta_2)}}=+\infty,\ \delta_2>0.
\end{align*}
By Theorem~\ref{gensaus}, $|S^{g_2}(y_0)_\varepsilon|\simeq f_2^{-1}(\varepsilon)$, $\varepsilon\to 0$, and the result follows by definition of box dimension.
\end{proof}

To solve this problem, for a given class of Chebyshev generators, using Theorem~\ref{gensaus}, we find an appropriate scale to compare lengths of $\varepsilon$-neighborhoods with. We were motivated by the notion of \emph{generalized Minkowski content} that exists in literature, see He, Lapidus \cite{helapidus}, and \v Zubrini\' c, \v Zupanovi\' c  \cite{buletin}. It was suitable (instead of standard Minkowski content) in situations where the leading term of the Lebesgue measure of the $\varepsilon$-neighborhood of a set is not a power function. Then, the Lebesgue measure was compared to powers of $\varepsilon$ multiplied by an appropriate \emph{gauge function}. Here, we define generalized Minkowski content with respect to a \emph{family of gauge functions}.

\medskip

Let $\mathcal{I}=\{u_0,u_1,\ldots\}$ be a Chebyshev scale, such that monomials $u_i$ are positive and strictly increasing on $(0,\delta)$, for $i\geq 1$. 
Suppose that $f$ has a development in scale $\mathcal{I}$ of order $\ell$ and, moreover, that $f$ satisfies assumptions from Theorem~\ref{gensaus} and the upper power condition. Let $g=id-f$. By Theorem~\ref{gensaus}, the $\varepsilon$-neighborhood $|S^g(x_0)_\varepsilon|$ should be compared to the inverted scale $\mathcal{I}$:
$$\{u_1^{-1}(\varepsilon),\ u_2^{-1}(\varepsilon),\ u_3^{-1}(\varepsilon)\ldots\}.$$ 

\begin{definition}\label{generaldim}
By \emph{lower (upper) generalized  Minkowski content of orbit $S^g(x_0)$ with respect to $u_i$},  $i=1,\ldots,\ell$, we denote the limits
\begin{eqnarray*}
{\mathcal M_{*}}(S^g(x_0),u_i)&=&\liminf_{{\varepsilon}\to0}\frac{|S^g(x_0)_\varepsilon|}{u_i^{-1}({\varepsilon})},\\ {\mathcal M^{*}}(S^g(x_0),u_i)&=&\limsup_{{\varepsilon}\to0}\frac{|S^g(x_0)_\varepsilon|}{u_i^{-1}({\varepsilon})},
\end{eqnarray*}
respectively. Furthermore, the moments of jump in generalized lower (upper) Minkowski contents,
\begin{eqnarray*}
\underline{m}(S^g(x_0),\mathcal{I})&=&\max\{i\ge 1\  |\ {\mathcal M_*}   (S^g(x_0),u_i) >0\},\\ 
\overline{m}(S^g(x_0),\mathcal{I})&=&\max\{i\ge 1\  |\ {\mathcal M^*}   (S^g(x_0),u_i) >0\}, 
\end{eqnarray*}
are called \emph{lower (upper) critical Minkowski order of $S^g(x_0)$ with respect to the scale $\mathcal{I}$}.
If $\underline{m}(S^g(x_0),\mathcal{I})=\overline{m}(S^g(x_0),\mathcal I)$, we call it \emph{critical Minkowski order with respect to the scale $\mathcal{I}$} and denote it simply by $m(S^g(x_0),\mathcal I)$.
\end{definition}
It is easy to see from Theorem~\ref{gensaus} and Lemma~\ref{doubling}.$(i)b)$ that the upper and lower generalized Minkowski contents ${\mathcal M}(S^g(x_0),u_i)$, viewed as functions of discrete parameter $i$, $i=1,\ldots,\ell$, pass from the value $+\infty$, through a finite value and drop to $0$ as $i$ grows. Moreover, the critical index $i_0$ is the same for upper and lower content and therefore $m(g,\mathcal I)=i_0$. This is a behavior analogous to the behavior of the standard upper (lower) Minkowski contents as function of continuous parameter $s\in[0,1]$, where box dimension denoted the moment of jump from $+\infty$ to $0$.
Generalized Minkowski content as function of $i\in\mathbb{N}$ is shown in Figure~\ref{diskr}. Compare it to Figure~\ref{hohio} in the definition of Minkowski content and box dimension.
\begin{figure}[ht]
\centering
\vspace{-25.5cm}
\includegraphics[scale=1.1,trim={-1cm 0cm 0cm 0cm}]{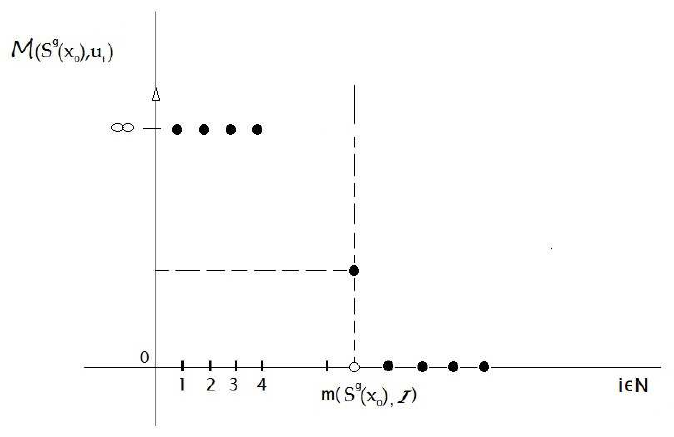}
\vspace{-1cm}
\caption{\small Generalized Minkowski content, as function of $i\in\mathbb{N}$. Critical Minkowski order is the moment $i_0$ of jump from $+\infty$ to $0$.} \label{diskr}

\end{figure}

Furthermore, by Theorem~\ref{gensaus}, the behavior of the length of the $\varepsilon$-neighborhood of orbits is independent of the choice of the initial point $x_0<\delta$ from the attracting basin of $0$. Therefore \emph{all orbits of $g$ have the same critical Minkowski order}.
\smallskip

\begin{remark}[Box dimension and critical Minkowski order for differentiable generators]\label{gbd}
Let $f\in \rdif[0,d)$ be a differentiable function. The box dimension of orbits of $g=id-f$ is bijectively related to their critical Minkowski order with respect to the Taylor scale, by the formula
\begin{equation*}
\dim_B(S^g(x_0))=1-\frac{1}{m(S^g(x_0),\mathcal I)}.
\end{equation*}
\end{remark}
\begin{proof}
If $f\in \rdif[0,d)$, then it has an asymptotic development of order $r$ in the differentiable, Taylor scale, $\mathcal I=\{1,x,x^2,\ldots\,x^r\}$.
Assume $f(x)\simeq x^k$, $1<k\leq r$. By Theorem~\ref{neveda}, $|S^g(x_0)_\varepsilon|\simeq\varepsilon^{1/k}$, $\varepsilon\to 0$. This gives $m(S^g(x_0),\mathcal I)=k$ and $\dim_B(S^g(x_0))=1-1/k$. In case $k=1$, by Theorem~\ref{neveda}, the asymptotic development of $|S^g(x_0)_\varepsilon|$ is given in \eqref{sausage}, and we can conclude directly by definitions that $\dim_B(S^g(x_0))=1$, $m(S^g(x_0),\mathcal I)=1$.  
\end{proof}

We saw in Remark~\ref{gbd} that box dimension and critical Minkowski order carry the same information for differentiable generators. However, in non-differentiable cases, critical Minkowski order with respect to an appropriately chosen scale is a more precise measure for density of the orbit than is the box dimension. An example follows.
\begin{example}[Example~\ref{defffi} revisited]
Let $f_1$ and $f_2$ be as in Example~\ref{defffi}. Their orbits share the same box dimension. To get more precise information, we need to define a scale in which $f_1$ and $f_2$ both have developments, for example $$\mathcal{I}=\{1,x(-\log x),x,x^2(-\log x),x^2,\ldots\}.$$
The critical Minkowski orders of their orbits $S^{g_1}(x_0)$ and $S^{g_2}(y_0)$ with respect to scale $\mathcal{I}$ are distinct numbers:
 $$m(S^{g_1}(x_0),\mathcal{I})=2k\ \ >\  \ m(S^{g_2}(y_0),\mathcal{I})=2k-1,$$
 showing a difference in density of orbits.
\end{example}

\bigskip
Finally, we state the main theorem of Section~\ref{twotwo}. Theorem~\ref{chebsaus} is a generalization of Theorem~\ref{neveda} to non-differentiable cases. Derivatives are replaced by generalized derivatives in an appropriate Chebyshev scale, and box dimension by critical Minkowski order with respect to the given scale. With this new notions, we recover a bijective correspondence between the multiplicity of a fixed point and critical Minkowski order of any orbit tending to the fixed point. The proof is in Subsection~\ref{twotwotwo}.
\smallskip

Let $\mathcal{F}=\{f_\lambda|\lambda\in\Lambda\}$ be a family of $C^r$-functions on $[0,\delta)$, admitting a uniform asymptotic development \eqref{unasymp} of order $r$ in a family of Chebyshev scales $\mathcal I_\lambda=\{u_0(x,\lambda),u_1(x,\lambda),\ldots\}$:
\begin{equation*}
f_\lambda(x)=\sum_{i=0}^{r} \alpha_i(\lambda)\cdot u_i(x,\lambda)+\psi_r(x,\lambda),\quad \lambda\in \Lambda.
\end{equation*}
Let $f=f_{\lambda_0}$. Let $\mathcal I=\mathcal I_{\lambda_0}=\{u_0,u_1,\ldots\}$, where $u_i,\ i\geq 1,$ are positive and strictly increasing on $(0,\delta)$.  

\begin{theorem}[Multiplicity of fixed points and critical Minkowski order of orbits, (non) differentiable case] \label{chebsaus} Let $f=f_{\lambda_0}$ be a function from the family $\mathcal{F}$ above, satisfying all assumptions of Theorem~\ref{gensaus} and the upper power condition. Let $g=id-f$. Let $k<r$. Then the following claims are equivalent:
\begin{enumerate}
\item[$(i)$] $D_i(f)(0)=0$, for $i=0,\ldots,k-1$, and $D_k(f)(0)>0$,\ \ for some $k\geq 1$, \\$($that is, $f\simeq u_k$ for some $\ k\geq 1$$),$ 
\item[$(ii)$] $|S^g(x_0)_\varepsilon|\simeq u_k^{-1}(\varepsilon)$,\ $\varepsilon\to 0$,
\item[$(iii)$] $m(S^g(x_0),\mathcal I)=k$.
\end{enumerate}
If, moreover, $\Lambda\subset\mathbb{R}^N$, $k\leq N$, and \begin{equation}\label{regu}\text{the matrix }\left[\frac{\partial \alpha_i(\lambda_0)}{\partial \lambda_j}\right]_{i=0\ldots k-1,\ j=1 \ldots k} \text{ is of maximal rank $k$, }\end{equation}then  $(1),\ (2)$ or $(3)$ is also equivalent to
\begin{enumerate}
\item[$(iv)$] $
\mu_0^{fix}(g,\mathcal{G})=k.$
\end{enumerate}
Without this regularity assumption, $(i)$, $(ii)$ or $(iii)$ implies
$$
\mu_0^{fix}(g,\mathcal{G})\leq k.
$$
The statement does not depend on the choice of the initial point $x_0$ in the attracting basin of $0$.
\end{theorem}
\noindent The upper power condition in Theorem~\ref{chebsaus} cannot be omitted:
\begin{remark}[The upper power condition in Theorem~\ref{chebsaus}]\label{upperpower}
The upper power condition assumption on $f$ is needed in Theorem~\ref{chebsaus}. As a counterexample, we take the following Chebyshev scale
$$
\mathcal{I}=\{e^{-\frac{1}{x}},e^{-\frac{2}{x}},e^{-\frac{3}{x}},\ldots\}.
$$
Let e.g. $f(x)=e^{-\frac{3}{x}}$, $g$=id$-f$. Obviously, $f$ does not satisfy the upper power condition.
Since $D_0(f)(0)=D_1(f)(0)=0$ and $D_2(f)(0)>0$, by Lemma~\ref{multgen}, it holds that $\mu_0(g,\mathcal{G})\leq 2$. On the other hand, $u_1^{-1}(\varepsilon)\simeq u_2^{-1}(\varepsilon)\simeq u_3^{-1}(\varepsilon)\simeq\ldots\simeq \frac{1}{-\log\varepsilon}$, therefore critical Minkowski order $m(S^g(x_0),\mathcal{G})$ is infinite. In this case, we are not able to read the multiplicity neither from the critical Minkowski order of orbits nor from behavior of the length of the $\varepsilon$-neighborhoods of orbits. 
\end{remark}

\bigskip

In the differentiable case, we see that differentiation diminishes critical Minkowski order by 1. Let $f\in \rdif[0,\delta)$ and suppose $\mu_0(f')>1$. Put $g=id-f$ and $h=id-f'$. By Theorem~\ref{neveda} and Remark~\ref{gbd}, we have that $$m(S^h(x_0),\mathcal I)=m(S^g(y_0),\mathcal I)-1,$$ where $\mathcal{I}=\{1,x,x^2,\ldots,x^r\}$ is a differentiable Chebyshev scale.

The same property is valid in non-differentiable cases when $f$ has asymptotic development in a Chebyshev scale, if the derivatives are substituted by generalized derivatives in the scale. The following corollary is a direct consequence of Theorem~\ref{chebsaus} and definition of generalized derivatives:
\begin{corollary}[Behavior of the critical Minkowski order of orbits under differentiation]
Let $\mathcal{I}=\{u_0,u_1,\ldots,u_k\}$ be a Chebyshev scale and let $D_1(\mathcal I)$ denote the Chebyshev scale of the first generalized derivatives of $\mathcal I$, that is, $D_1(\mathcal I)=\{D_1(u_1), D_1(u_2),\ldots,D_1(u_k)\}$. Let $f$ have an asymptotic development of order $k$ in scale $\mathcal I$ and let $f$ and $D_1(f)$ satisfy assumptions of Theorem~\ref{gensaus} and the upper power condition. Let $g=$id$-f$, $h=$id$-D_1(f)$. It holds that
$$
m(S^h(x_0),D_1(\mathcal I))=m(S^g(y_0),\mathcal I)-1.
$$ 
Here, $x_0$ and $y_0$ are arbitrary initial points sufficiently close to 0.
\end{corollary}

\subsection{Proofs of the main results}\label{twotwotwo}

\noindent In the proof of Theorem~\ref{gensaus} and Theorem~\ref{chebsaus} we need the following lemma:
\begin{lemma}[Inverse property]\label{doubling} 
Let $d>0$ and let $f, g\in C^1(0,d)$ be positive, strictly increasing functions on $(0,d)$.
\begin{enumerate}
\item[$i)$] If there exists a positive constant $M>0$ such that the upper power condition holds, \begin{equation}\label{up}x\cdot(\log f)'(x)\leq M,\ x\in(0,d),\end{equation} then
\begin{eqnarray*}
&(a)&f^{-1}(y)\simeq g^{-1}(y),\text{ as }y\to 0, \text{ implies } f(x)\simeq g(x),\text{ as }x\to 0;\\
&(b)&\lim_{x\to 0}\frac{f(x)}{g(x)}=0\ (+\infty) \text{ implies } \lim_{y\to 0}\frac{f^{-1}(y)}{g^{-1}(y)}=+\infty\ (0).
\end{eqnarray*}
\item[$ii)$] If there exists a positive constant $m>0$ such that the lower power condition holds, \begin{equation}\label{low}m\leq x\cdot(\log f)'(x),\ x\in(0,d),\end{equation} then
\begin{equation*}
f(x)\simeq g(x),\text{ as }x\to 0, \text{ implies } f^{-1}(y)\simeq g^{-1}(y),\text{ as }y\to 0.
\end{equation*}
\end{enumerate}
\end{lemma}

\begin{proof}\ \

$i) a)$ From $f^{-1}\simeq g^{-1}$ we have that there exist constants $A<1$, $B>1$ and $\delta>0$ such that
$$
Ag^{-1}(y)\leq  f^{-1}(y)\leq Bg^{-1}(y) ,\ y\in(0,\delta).
$$
Putting $x=g^{-1}(y)$ and applying $f$ (strictly increasing) on the above inequality we get that there exists $\delta_1>0$ such that
\begin{equation}\label{eqq}
f(Ax)\leq g(x)\leq f(Bx), \ x\in(0,\delta_1).
\end{equation}
For each constant $C>1$ we have, for small enough $x$, 
\begin{eqnarray}\label{doup}
\log f(Cx)-\log f(x)&=&(\log f)'(\xi)(C-1)x\nonumber \\
&<&(\log f)'(\xi)(C-1)\xi,\quad \xi\in(x,Cx).
\end{eqnarray}
Combining $(\ref{up})$ and $(\ref{doup})$, we get that there exist constants $m_C>1$ and $d_C>0$ such that 
\begin{equation}\label{doubl}
\frac{f(Cx)}{f(x)}\leq m_C,\ x\in(0,d_C).
\end{equation}
Now, using property $(\ref{doubl})$ and inequality $(\ref{eqq})$, for small enough $x$ we obtain
$$
\frac{1}{m_{1/A}}f(x)\leq g(x)\leq m_{B}f(x),
$$
i.e. $f(x)\simeq g(x)$, as $x\to 0$.

$i)\ b)$
Suppose $\lim_{x\to 0}\frac{f(x)}{g(x)}=+\infty$. We prove that $\lim_{y\to 0}\frac{f^{-1}(y)}{g^{-1}(y)}=0$ by proving that limit superior and limit inferior are equal to zero. Suppose the contrary, that is,
$$
\liminf_{y\to 0}\frac{f^{-1}(y)}{g^{-1}(y)}=M, \text{ for some $M>0$, or $M=\infty$}.
$$
By definition of limit inferior, there exists a sequence $y_n\to 0$, as $n\to\infty$, such that \begin{equation}\label{tend}\frac{f^{-1}(y_n)}{g^{-1}(y_n)}\to M, \text{ as }n\to\infty. \end{equation}
From \eqref{tend} it follows that there exist $n_0\in\mathbb{N}$ and $C>0$ such that
\begin{equation}\label{invineq}
g^{-1}(y_n)<C f^{-1}(y_n),\ n\geq n_0.
\end{equation}
Now, as in $i)(a)$, by a change of variables $x_n=g^{-1}(y_n)$, $x_n\to 0$, and applying $f$ (strictly increasing) on \eqref{invineq}, we get 
$$
m_C\cdot g(x_n)\geq f(x_n),\ n\geq n_0,\ x_n\to 0,\quad \text{ for }m_C>0,
$$
which is obviously a contradiction with $\lim_{x\to 0}\frac{f(x)}{g(x)}=+\infty$. Therefore 
$$
\liminf_{y\to 0}\frac{f^{-1}(y)}{g^{-1}(y)}=0.
$$
It can be proven in the same way that limit superior is equal to zero.

Now suppose $\lim_{x\to 0}\frac{f(x)}{g(x)}=0$. Same as above, we prove that $\lim_{y\to 0}\frac{g^{-1}(y)}{f^{-1}(y)}=0$.

\medskip

$ii)$ It is easy to see by the change of variables $x=f^{-1}(y)$ that property $(\ref{low})$ of $f$ is equivalent to property $(\ref{up})$ of $f^{-1}$. The statement then follows from $i)$.
\end{proof}

\begin{remark}[Counterexamples in Lemma~\ref{doubling}]
In Lemma \ref{doubling}.$i)$, upper power condition $(\ref{up})$ is important. We take, for example, functions $f(x)=e^{-\frac{1}{2x}}$ and $g(x)=e^{-\frac{1}{x}}$. They do not satisfy $(\ref{up})$ and, obviously,
$$
\lim_{x\to 0}\frac{f(x)}{g(x)}=\infty,\ \text{\ but\ } f^{-1}(y)=-\frac{1}{2\log y}\ \simeq\ g^{-1}(y)=-\frac{1}{\log y}.
$$
We can do the same for lower power condition $(\ref{low})$ in Lemma \ref{doubling}.$ii)$, by considering, for example, $f(x)=-\frac{1}{\log x}$ and $g(x)=-\frac{1}{2\log x}$.
\end{remark}

\medskip

\noindent \emph{Proof of Theorem \ref{gensaus}}.\

From lower power condition together with $f'(0)=0$, we get that $f(x)=o(x)$ and that $f(x)$ is strictly increasing on $(0,d)$. It can easily be checked that $x_n\to 0$ and $d(x_n,x_{n+1})\to 0$, as $n\to\infty$. Denote by $N_\varepsilon$ and $T_\varepsilon$ the nucleus and the tail of the $\varepsilon$-neighborhood of the sequence. That are $\varepsilon$-neighborhoods of two subsets of the orbit satisfying the inequality $ d(x_n,x_{n+1})\leq 2\varepsilon $ for the nucleus, and  $d(x_n,x_{n+1})>2\varepsilon$ for the tail.
 Therefore,
\begin{equation}\label{area}
|S^g(x_1)_\varepsilon|=|N_\varepsilon|+|T_\varepsilon|.
\end{equation}
Here, $|N_\varepsilon|$ is the length of the nucleus, and $|T_\varepsilon|$ the length of the tail of the $\varepsilon$-neighborhood. The idea of division in the tail and the nucleus stems from Tricot \cite{tricot}. To compute the lengths, we have to find the critical index $n_\varepsilon\in\mathbb{N}$, such that
\begin{equation}\label{eps} 
f(x_{n_\varepsilon})< 2\varepsilon,\  f(x_{n_\varepsilon-1})\geq 2 \varepsilon.
\end{equation}
That is, the smallest index $n_\varepsilon$ such that $\varepsilon$-neighborhoods of the points $x_{n_\varepsilon},\ x_{n_\varepsilon+1},\ $etc start to overlap.
Then we have
\begin{equation}
|N_\varepsilon|=x_{n_\varepsilon}+\varepsilon,\ 
|T_\varepsilon|\simeq n_\varepsilon\cdot\varepsilon,\ \varepsilon\to 0.\label{tailo}
\end{equation}

First we estimate $|N_{\varepsilon}|$.
From $f(x)=o(x)$ we get
\begin{equation}\label{behinv}
\lim_{y\to 0}\frac{y}{f^{-1}(y)}=0.
\end{equation}
Since $f^{-1}$ is strictly increasing, from $(\ref{eps})$ we easily get $x_{n_\varepsilon}\simeq f^{-1}(2\varepsilon)$. Since $f$ satisfies the lower power condition, by Lemma \ref{doubling}.$ii)$ it follows $x_{n_{\varepsilon}}\simeq f^{-1}(\varepsilon)$. This, together with $(\ref{tailo})$ and $(\ref{behinv})$, implies $|N_\varepsilon|\simeq f^{-1}(\varepsilon)$.

Now let us estimate the length of the tail, $|T_\varepsilon|$, by estimating $n_\varepsilon$.\\ 
Putting $\Delta x_n:=x_{n}-x_{n+1}$, from $x_{n+1}-x_n=-f(x_n)$ we get
\begin{equation}\label{nepsil}
\frac{\Delta x_n}{f(x_n)}=1 \text{ and }\sum_{n=n_0}^{n_\varepsilon}\frac{\Delta x_n}{f(x_n)}=\sum_{n=n_0}^{n_\varepsilon}1=n_{\varepsilon}-n_0\simeq n_\varepsilon, \text{ as }\varepsilon\to 0,
\end{equation}
for some fixed $n_0\in\mathbb{N}$. 

As in \eqref{relat} below, we get that $\frac{f(x_{n+1})}{f(x_n)}$ tends to $1$, as $n$ tends to infinity, and thus we can choose the integer $n_0$ so that 
\begin{equation}\label{nzero}
Af(x_{n+1})<f(x_n)<Bf(x_{n+1}),\ \ n\geq n_0,
\end{equation}
for some constants $A,\ B>0$.

Since the function $\frac{1}{f(x)}$ is strictly decreasing on $(0,d)$ and $\lim_{x\to 0}\frac{1}{f(x)}=+\infty$, the sum $\sum_{n=n_0}^{n_\varepsilon}\frac{\Delta x_n}{f(x_n)}$ is equal to the sum of the areas of the rectangles in Figure~\ref{figproof}.$1$ and, analogously, the sum $\sum_{n=n_0}^{n_\varepsilon}\frac{\Delta x_n}{f(x_{n+1})}$ is equal to the sum of the areas of the rectangles in Figure~\ref{figproof}.$2$. Therefore, we have the following inequalities:
\begin{equation}\label{intap}
\sum_{n=n_0}^{n_\varepsilon}\frac{\Delta x_n}{f(x_n)}\ \leq \int_{x_{n_{\varepsilon}+1}}^{x_{n_0}}\frac{dx}{f(x)} \leq \sum_{n=n_0}^{n_\varepsilon}\frac{\Delta x_n}{f(x_{n+1})}.
\end{equation}

\begin{figure}[t]
\centering
\vspace{-39cm}
\includegraphics[scale=1.5,trim={-1cm 0cm 0cm 0cm}]{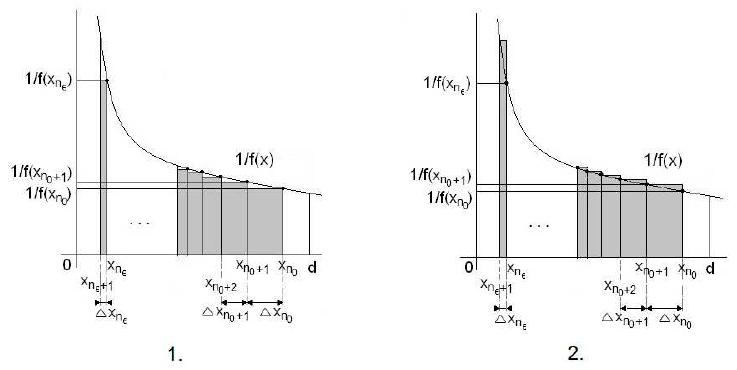}
\caption{\small The upper and the lower bound, by sums of rectangles, on the integral from \eqref{intap}.}
\label{figproof}
\end{figure}
From \eqref{nzero}, we get
\begin{equation}\label{upp}
\sum_{n=n_0}^{n_\varepsilon} \frac{\Delta x_n}{f(x_{n+1})}<B\sum_{n=n_0}^{n_\varepsilon}\frac{\Delta x_n}{f(x_n)},
\end{equation}
so finally, putting \eqref{upp} in \eqref{intap} and using \eqref{nepsil}, we get the following estimate for $n_\varepsilon$:
\begin{equation}\label{innnt}
n_\varepsilon\simeq \int_{x_{n_\varepsilon+1}}^{x_{n_0}}\frac{dx}{f(x)},\text { as }\varepsilon\to 0.
\end{equation}

Substituting $x=f^{-1}(y)$, from the lower power condition we get $$\frac{f^{-1}(y)}{y^2}\geq m\frac{(f^{-1})'(y)}{y}$$ and, consequently, for $y\in(0,f(d))$,
\begin{equation}\label{estim}
-\left(\frac{f^{-1}(y)}{y}\right)'=-\frac{(f^{-1})'(y)}{y}+\frac{f^{-1}(y)}{y^2}\geq (m-1)\cdot\frac{(f^{-1})'(y)}{y}.
\end{equation}

Now substitution $x=f^{-1}(s)$ in the integral $(\ref{innnt})$ together with $(\ref{estim})$ gives
\begin{equation}\label{nepssss}
n_\varepsilon\simeq \int_{f(x_{n_\varepsilon+1})}^{f(x_{n_0})}\frac{(f^{-1})'(s) ds}{s}\leq \frac{1}{m-1}\left(-\frac{f^{-1}(s)}{s}\right) \Big{|}_{f(x_{n_\varepsilon+1})}^{f(x_{n_0})}.
\end{equation}
It holds
\begin{eqnarray*}
\frac{f(x_{n_\varepsilon})}{f(x_{n_\varepsilon-1})}&=&\frac{f(x_{n_\varepsilon-1}-f(x_{n_\varepsilon-1}))}{f(x_{n_\varepsilon-1})}=\\
&=&\frac{f(x_{n_\varepsilon-1})+f'(\xi_\varepsilon)(-f(x_{n_\varepsilon-1}))}{f(x_{n_\varepsilon-1})}=1-f'(\xi_\varepsilon),
\end{eqnarray*}
for some $\xi_\varepsilon\in(x_{n_\varepsilon},x_{n_{\varepsilon}-1})$, so $f'(0)=0$ implies
\begin{equation}\label{relat}
\lim_{\varepsilon\to 0}\frac{f(x_{n_\varepsilon})}{f(x_{n_\varepsilon-1})}=1.
\end{equation} 
From $(\ref{eps})$ and $(\ref{relat})$, we now conclude that $f(x_{n_{\varepsilon}+1})\simeq \varepsilon$.
 Therefore $(\ref{nepssss})$ becomes
$$n_\varepsilon\leq C  \frac{f^{-1}(\varepsilon)}{\varepsilon},$$
for some $C>0$.
From \eqref{tailo}, we have that $$|T_\varepsilon|\simeq n_\varepsilon\cdot\varepsilon\leq C_1\cdot f^{-1}(\varepsilon),$$ for some $C_1>0$ and $\varepsilon$ small enough. This, together with $|N_\varepsilon|\simeq f^{-1}(\varepsilon)$ obtained above, by \eqref{area} implies 
$$
|A_\varepsilon(S^g(x_1))|\simeq f^{-1}(\varepsilon), \text{ as }\varepsilon\to 0.
$$\qed

\bigskip

\noindent \emph{Proof of Theorem \ref{chebsaus}}.

We first prove $(i)\Rightarrow (ii) \Rightarrow (iii)$. Suppose that $D_i(f)(0)=0,\ i=0,\ldots,k-1,\ D_k(f)(0)>0$. That is, $f\simeq u_k$, as $x\to 0$. Theorem~\ref{gensaus} applied to $f$ gives $|S^g(x_1)_\varepsilon|\simeq f^{-1}(\varepsilon)$. Since $f\simeq u_k$, by Lemma~\ref{doubling}.$ii)$ we get that $f^{-1}\simeq u_k^{-1}$. Therefore, $|S^g(x_1)_\varepsilon|\simeq u_k^{-1}(\varepsilon)$. Since $u_k$ satisfies the upper power condition, by Lemma~\ref{doubling} and by definition of the critical Minkowski order, we get $m(S^g(x_0),\mathcal{I})=k$.

Now we prove $(iii)\Rightarrow (ii) \Rightarrow (i)$. Suppose $m(S^g(x_0),\mathcal I)=k$ and $f\simeq u_l$, for some $l\neq k$. As above, we conclude that $m(S^g(x_0),\mathcal{I})=l\neq k$, which is a contradiction. Therefore $f\simeq u_k$ and, again as above, $|S^g(x_1)_\varepsilon|\simeq u_k^{-1}(\varepsilon)$, $\varepsilon\to 0$.  

By Lemma~\ref{multgen}, we conclude that $(i)$ implies $\mu_0^{fix}(g,\mathcal{G})\leq k$.
If, moreover, regularity condition from theorem holds, by Lemma \ref{multgen}, $(i)$ is equivalent to $(iv)$.
\qed

\section{Application to cyclicity for planar vector fields using fractal analysis of Poincar\' e maps}\label{twothree}

A problem closely related to the open $16^{th}$ Hilbert problem\footnote{$16^{th}$ Hilbert problem asks about the existence of an upper bound $H(n)$, depending only on the degree $n$ of the field, on the number of limit cycles in planar polynomial fields.} is determining the cyclicity of limit periodic sets of analytic planar vector fields. A good overview of the problem and precise definitions are given in the book of Roussarie \cite[Chapter 2]{roussarie}. 

A \emph{limit periodic set} $\Gamma$ of a vector field $X$ for the unfolding $(X_\lambda)$, $\lambda\in\Lambda$, topological space, is an invariant set for $X$, from which limit periodic sets (isolated periodic orbits) bifurcate in the unfolding $(X_\lambda)$. The maximal number of limit cycles that bifurcate from $\Gamma$ in unfolding $(X_\lambda)$ is called cyclicity of $\Gamma$ in the unfolding $(X_\lambda)$, and denoted $Cycl(\Gamma,(X_\lambda))$.
We are interested in the \emph{cyclicity} of $\Gamma$ in the \emph{universal unfolding}\footnote{an unfolding topologically equivalent to any other unfolding, thus incorporating all possible phase portraits in all possible unfoldings of $\Gamma$}. We refer to it only as \emph{cyclicity of $\Gamma$}. Similarly, we can estimate cyclicity in a \emph{generic unfolding} of $\Gamma$, that means, in a sufficiently general unfolding.

We consider monodromic\footnote{accumulated on at least one side by spiral trajectories} limit periodic sets of finite codimension\footnote{i.e. not of centre type}. Elementary monodromic limit periodic sets are elliptic singular points (strong and weak foci), limit cycles and saddle or saddle-node polycycles. 
\smallskip

Let $(g_\lambda),\ \lambda\in\Lambda$, denote the family of \emph{first return maps} or \emph{Poincar\' e maps} for the unfolding $(X_\lambda)$ of $\Gamma$, defined on a transversal to $\Gamma$. Let $f_\lambda=id-g_\lambda$ denote the \emph{displacement functions}. Let $X=X_{\lambda_0}$ and let $g=g_{\lambda_0}$, $f=f_{\lambda}$ denote the Poincar\' e map and the displacement function around $\Gamma$.

It is known that $g_\lambda(s)$ are diffeomorphisms $g_\lambda:(0,\delta)\to(0,\delta)$, and that $g=g_{\lambda_0}$ has an isolated fixed point at $s=0$ corresponding to the  intersection of the transversal with $\Gamma$. Moreover, $f_\lambda$ have uniform asymptotic developments in family of Chebyshev scales, as $s\to 0$. The family $(g_\lambda)$ is differentiable at $s=0$ in case of elliptic points and limit cycle cases. At saddle polycycles, the family is not differentiable at fixed point zero. This can be found in e.g. \cite[Chapters 4,5]{roussarie}. On the other hand, bifurcated limit cycles correspond to fixed points of Poincar\' e maps $(g_\lambda)$. The number of limit cycles that bifurcate from a monodromic limit periodic set in an unfolding is, directly by definition, equal to the multiplicity of the fixed point zero of the Poincar\' e map in the family of Poincar\' e maps for the given unfolding, see e.g. Proposition 2 in \cite{dumortier}.

\smallskip
Our approach to cyclicity using fractal analysis of orbits is the following. After establishing in which scale $(f_\lambda)$ unfolds in a generic unfolding $(X_\lambda)$, we apply results from Section~\ref{twoone} to Poincar\' e maps. The behavior of the $\varepsilon$-neighborhood of any (only one) orbit of the Poincar\' e map $g(s)$ around limit periodic set $\Gamma$, if compared to an appropriate scale for generic unfolding, reveals cyclicity. The behavior is measured by critical Minkowski order of the orbit with respect to the appropriate scale. Thus, fractal properties of orbits contain information on cyclicity. 

The gain of this fractal method is that critical Minkowski order of only one orbit can be determined numerically (after the scale is known). On the other hand, the limits of the method lie in the fact that for saddle polycycles more complicated than saddle loop, the Chebyshev scale for $(g_\lambda)$, or even a reasonable superset of the scale, is in general not known without some very strong assumptions on the saddle. In these cases, we do not know with which scale we should compare the behavior of $\varepsilon$-neighborhoods of orbits of $g$, and our method cannot be applied.

\subsection{Limit cycle}\label{twothreeone}

Let the field $X=X_{\lambda_0}$ have a stable or semistable limit cycle $\Gamma$ and let $X_\lambda$ be an arbitrary analytic unfolding of $X$.
The asymptotic development of displacement functions $f_\lambda(s)$, as $s\to 0$, can be found in e.g. \cite[4.1.1]{roussarie}. The functions $f_\lambda(s)$ are analytic on $[0,\delta)$, for $\lambda$ close to $\lambda_0$. Expanding in Taylor series, we get
\begin{equation*}
f_\lambda(s)=\alpha_0(\lambda)+\alpha_1(\lambda)s+\alpha_2(\lambda)s^2+\alpha_3(\lambda)s^3+\ldots.
\end{equation*} 

\noindent Moreover, the family $f_\lambda$ has a uniform asymptotic development of any order $\ell\in\mathbb{N}$ in the Chebyshev scale $$\mathcal I=\{1,s,s^2,\ldots,s^\ell\}.$$ %Let $\mathcal{G}=\{g_\lambda=id-f_\lambda\}$ be the family of Poincar\' e maps. By Theorem~\ref{chebsaus}, if $f_{\lambda_0}\simeq x^{k}$, for some %$2\leq k\leq\ell$, then the critical Minkowski order of $g_{\lambda_0}$ with respect to $\I_{\lambda_0}$ is equal to $k$, %$m(g_{\lambda_0},\I_{\lambda_0})=k$. The cyclicity of the limit cycle in the unfolding $X_\lambda$ is equal to %$\mu_0^{fix}(g_{\lambda_0},\mathcal{G})\leq k$. If moreover the unfolding $(X_\lambda)$ is general enough so that the regularity condition from %Theorem~\ref{chebsaus} is satisfied, then  the cyclicity $\mu_0^{fix}(g_{\lambda_0},\mathcal{G})=k$.

Let $S^g(s_0)$ be any orbit of $g$ at a transversal to the limit cycle $\Gamma$. By Theorem~\ref{gensaus}, to obtain an upper bound on cyclicity of $\Gamma$, $|S^g(s_0)_\varepsilon|$ should be compared to the inverted scale, $\{\varepsilon,\varepsilon^{1/2},\varepsilon^{1/3},\ldots\}$.
By Theorem~\ref{chebsaus}, it holds that \begin{equation}\label{cicla}Cycl(\Gamma,X_\lambda)=\mu_0^{fix}(g,(g_\lambda))\leq m(S^g(x_0),\mathcal I).\end{equation} Note that $f(s)\simeq s^k,\ k\geq 1$, as $s\to 0$, is equivalent to $m(S^g(s_0),\mathcal I)=k$. Moreover, under regularity assumption \eqref{regu} on the unfolding $(X_\lambda)$, we get the equality in \eqref{cicla}. The unfoldings satisfying \eqref{regu} are generic enough, and we get an upper bound on cyclicity of $\Gamma$ in generic unfoldings.

\subsection{Weak focus}\label{twothreetwo}

Let $x_0$ be a stable weak focus point of the field $X=X_{\lambda_0}$ ($DX(x_0)$ has two strictly imaginary, conjugated complex eigenvalues). The asymptotic development of displacement functions $f_\lambda(s)$ for an arbitrary analytic unfolding $(X_\lambda)$ of $X$ can be found in \cite[4.1.2]{roussarie}.
The displacement functions are again analytic on $[0,\delta)$, for $\lambda$ close to $\lambda_0$, but by symmetry argument for spiral trajectories around $x_0$, the leading monomials can only be the ones with odd exponents:
\begin{equation*}
f_{\lambda}(s)=\beta_1(\lambda)(s+g_1(\lambda,s))+\beta_3(\lambda)(s^3+g_3(\lambda,s))+\beta_5(\lambda)(s^5+g_5(\lambda,s))+\ldots,
\end{equation*}
where $g_i(\lambda,s)$ denotes some linear combination of monomials from Taylor expansion of order strictly greater than $s^i$ and with coefficients depending on $\lambda$. Moreover, the family $f_\lambda$ has a uniform asymptotic development in a family of Chebyshev scales $\mathcal I_{\lambda}$ of any order $2\ell+1$:
$$
\mathcal{I}_\lambda=\{s+g_1(\lambda,s),\ s^3+g_3(\lambda,s),\ s^5+g_5(\lambda,s),\ldots,\ s^{2\ell+1}+g_{2\ell+1}(\lambda,s)\}.
$$

To obtain an upper bound on the cyclicity of the focus, by Theorem~\ref{gensaus}, $|S^g(s_0)_\varepsilon|$ should be compared to the inverted scale of $\mathcal{I}_{\lambda_0}$, $\{\varepsilon,\ \varepsilon^{1/3},\ \varepsilon^{1/5},\ldots\}$. We proceed as in the example above.
%For the value $\lambda_0$ it holds 
%$\beta_1(\lambda_0)=0$ and therefore $f_{\lambda_0}(x)\simeq x^{2k+1}, \text{ as } \linebreak  x\to 0$, for some $k\geq 1$. 

%Let $\mathcal{G}=\{g_\lambda=\text{id}-f_\lambda\}$ be a family of Poincar\' e maps. 
%By Theorem~\ref{chebsaus}, the critical Minkowski order of $g_{\lambda_0}$ with respect to $\mathcal{I}_{\lambda_0}$ is equal to $k$, %$m(g_{\lambda_0},\I_{\lambda_0})=k$, and cyclicity of the focus is equal to $\mu_0^{fix}(g_{\lambda_0},\mathcal{G})\leq k$. If moreover the unfolding %$(X_\lambda)$ is general enough so that the regularity condition from Theorem~\ref{chebsaus} is satisfied, then the cyclicity %$\mu_0^{fix}(g_{\lambda_0},\mathcal{G})=k$.

%It is not difficult to check that $\mathcal{I}(\lambda)$, $\lambda\in W$, and $f_{\lambda_0}$  meet the conditions of Corollary \ref{chebsaus} and %thus knowing $|A_{\varepsilon}(S^{f_{\lambda_0}}(x_1))|\simeq u_k(\lambda_0)^{-1}$, we can conclude that cyclicity of such focus is less than or equal %to $k$, $k\geq 1$.

\subsection{Saddle loop}\label{twothreethree}

Let $X$ have a stable saddle loop $\Gamma$ at $x_0$, with ratio of hyperbolicity of the saddle $r=1$ (i.e. both eigenvalues of $DX(x_0)$ are real, with ratio $-1$). Suppose $(X_\lambda)$ is an analytic unfolding of $X$, such that each $X_\lambda$ has a saddle point of ratio $r(\lambda)$ at $x_0$, with the same stable and unstable manifolds (the loop, on the other hand, is broken in the unfolding).

The asymptotic development as $s\to 0$ of the family $f_\lambda(s)$ of displacement functions on a transversal to the loop is by \cite[Chapter 5]{roussarie} given by:
\begin{eqnarray}\label{exploop}
f_\lambda(s)&=&\beta_0(\lambda)+\alpha_1(\lambda)[s\omega(s,\alpha_1(\lambda))+g_1(s,\lambda)]+ \\
&+&\beta_1(\lambda)s+\alpha_2(\lambda)[s^2\omega(s,\alpha_1(\lambda))+g_2(s,\lambda)]+\beta_2(\lambda)s^2+\ldots+\nonumber \\
&+&\beta_{n-1}(\lambda)s^{n-1}+\alpha_n(\lambda)[s^n\omega(s,\alpha_1(\lambda))+g_n(s,\lambda)]+\beta_n(\lambda)s^n+ o(s^n),\ \ n\in\mathbb{N}.\nonumber
\end{eqnarray}
Here, $\alpha_1(\lambda)=1-r(\lambda)$, and $g_i(s,\lambda),\ i\in\mathbb{N},$ denote linear combinations of monomials of the type $s^k\omega^l$ of strictly greater order\footnote{order on monomials $s^k \omega^l$ is defined by increasing flatness, $s^i \omega^j<s^k \omega^l$ if ($i<k$) or ($i=k$ and $j>l$)} than $s^i\omega$, and $\omega$ is the \emph{Roussarie-Ecalle compensator} given by
$$
\omega (s,\alpha)=\left\{
\begin{array}{ll} \frac{s^{-\alpha}-1}{\alpha} &\text{ if } \alpha\neq 0,\\ -\log s&\text { if }\alpha=0.
\end{array}
\right.$$

The family $f_\lambda(s)$ cannot be extended analyticaly to $s=0$, but has a uniform asymptotic development in the non-differentiable family of Chebyshev scales $\mathcal I_\lambda$ of any order:
$$
\mathcal{I}_\lambda=\{1,s\omega(s,\alpha_1(\lambda))+g_1(s,\lambda),s,s^2\omega(s,\alpha_1(\lambda))+g_2(s,\lambda),s^2,\ldots\}.
$$
Putting $\lambda=\lambda_0$ in \eqref{exploop}, we get the following expansion for the displacement function $f(s)$ around the loop ($\alpha_1=0,\ f(0)=0$):
\begin{eqnarray*}
f(s)&=&\beta_1 s+\alpha_2 s^2(-\log s)+\beta_2s^2+\alpha_3 s^3(-\log s)+\beta_3 s^3+\ldots.
\end{eqnarray*}

To obtain information on cyclicity of $\Gamma$ in the unfolding $(X_\lambda)$, $|S^g(S_0)_\varepsilon|$ should be compared to the inverted scale of the scale $$\mathcal{I}=\mathcal I_{\lambda_0}=\{1,\ s(-\log s),\ s,\ s^2(-\log s),\ s^2,\ldots\}.$$ The critical Minkowski order signals the moment when the comparability occurs.
By Theorem~\ref{chebsaus}, if $f(s)\simeq s^k$, as $s\to 0$, $k\geq 1$, then $m(S^g(s_0),\mathcal I)= 2k$. If $f(s)\simeq s^k(-\log s)$, $k\geq 2$, then $m(S^g(s_0),\mathcal{I})= 2k-1$. Consequently, the cyclicity of the loop is less than or equal to $2k$, $2k-1$ respectively. Equality can be obtained if the unfolding $(X_\lambda)$ is regular enough so that the regularity condition \eqref{regu} is satisfied. We can think of it again as a generic unfolding.

\subsection{Hamiltonian 2-saddle cycle with constant hyperbolicity ratios}\label{twothreefour}

Suppose $(X_\lambda)$ is an analytic unfolding of a Hamiltonian 2-saddle cycle $\Gamma$ of the field $X=X_{\lambda_0}$, in which saddle points are preserved and at least one separatrix remains unbroken. Such a situation appears for polycycles having part of the line at infinity as the unbroken separatrix. Suppose that the ratios of hyperbolicity of both saddles $S_1$ and $S_2$ of $\Gamma$ are $r_1=r_2=1$. This example is taken from \cite{caubergh}.

The breaking parameter of the broken separatrix is denoted by $\beta_1(\lambda)$, then $\beta_1(\lambda_0)=0$. By $s\in(0,\delta)$, we parametrize the (inner side) of the transversal to the stable manifold of one of the saddles, and we choose the saddle whose stable manifold is on the unbroken separatrix, say $S_1$. In search of the scale for the asymptotic development of displacement functions $f_\lambda$, for simplicity we can consider the family of maps $(\Delta_\lambda)$ obtained from $(f_\lambda)$ by composition with analytic family:
$$
\Delta_\lambda(s)=D_2^\lambda\circ R_2^\lambda(s)-R_1^\lambda\circ D_1^\lambda(s).
$$
Here, $D_1^\lambda$ and $D_2^\lambda$ represent Dulac (transition) maps of the saddles $S_1$ and $S_2$, $R_1^\lambda$ is the regular map along the broken separatrix and $R_2^\lambda$ the regular map along the unbroken separatrix. Obviously, $R_1^\lambda(0)$ equals the breaking parameter of the separatrix, $\beta_1(\lambda)$, and $R_2^\lambda(0)=0$ for the unbroken separatrix. Using the developments of Dulac maps from \cite{roussarie} and subtracting the developments $D_2^\lambda\circ R_2^\lambda(s)$ and $R_1^\lambda\circ D_2^\lambda(s)$, similarly as in the example of saddle loop, $\Delta_\lambda$ has a uniform development in the monomials from two Chebyshev scales $\mathcal{I}_\lambda^1$ and $\mathcal{I}_\lambda^2$: 
\begin{align*}
\mathcal{I}_{\lambda}^1&=\{1,s\omega_1(s,\alpha_1(\lambda)),s,s^2\omega_1^2(s,\alpha_1(\lambda)),s^2\omega_1(s,\alpha_1(\lambda))),s^2,\\
&\qquad s^3\omega_1^3(s,\alpha_1(\lambda)),s^3\omega_1^2(s,\alpha_1(\lambda)),s^3\omega_1(s,\alpha_1(\lambda)),s^3,\ldots\},\\
\mathcal{I}_{\lambda}^2&=\{1,s\omega_2(s,\alpha_2(\lambda)),s,s^2\omega_2^2(s,\alpha_2(\lambda)),s^2\omega_2(s,\alpha_2(\lambda))),s^2,\\&\qquad s^3\omega_2^3(s,\alpha_2(\lambda)),s^3\omega_2^2(s,\alpha_2(\lambda)),s^3\omega_2(s,\alpha_2(\lambda)),s^3,\ldots\}.
\end{align*}
For the development, see \cite{caubergh}.
For each monomial $s^k\omega_{i}^l$, $k\geq 1,\ l\geq 0$, it necessarily holds that $k\geq l$, $\alpha_1(\lambda)=1-r_1(\lambda)$, $\alpha_2(\lambda)=1-r_2(\lambda)$, and  $\omega_1$ and $\omega_2$ are as defined in the section above. They are known as independent compensators, since they are not comparable by flatness, and thus disable the concatenation of $\mathcal{I}_{\lambda}^1$ and $\mathcal{I}_{\lambda}^2$ in one Chebyshev scale. 

Therefore we additionally suppose that the ratios of hyperbolicity $r_1=1$ and $r_2=1$ are preserved throughout the unfolding. Then we have $$\omega_1(s,\alpha_1(\lambda))=\omega_2(s,\alpha_2(\lambda))=-\log s,\text{ for all $\lambda$}.$$
In this case the Chebyshev scale in which all of $\Delta_\lambda$ (and then also displacements $f_\lambda$) for the unfolding $(X_\lambda)$ have uniform development is
\begin{align*}
\mathcal{I}=&\{1,\ s,\ s^2(-\log s)^2,\ s^2(-\log s),\ s^2,\\
&\qquad \ \ \ s^3(-\log s)^3,\ s^3(-\log s)^2,\ s^3(-\log s),\ s^3 \ldots\}.
\end{align*}
Note that this scale is obtained as superset of the actual scale for the unfolding $(X_\lambda)$. We do not have precise information on the actual scale.
To obtain an upper bound on cyclicity of 2-cycle $\Gamma$ in the unfolding $(X_{\lambda})$, by Theorem~\ref{gensaus}, $|S^g(s_0)_\varepsilon|$ for any orbit of the Poincar\' e map should be computed numerically and compared to the inverted scale of $\mathcal{I}$. It holds that $Cycl(\Gamma,(X_\lambda))\leq m(S^g(s_0),\mathcal{I})$. 

Let us note here that this upper bound is not optimal, since the family of scales $\mathcal{I_\lambda}$ is taken to be the largest possible for a given problem. It is too optimistic to hope that regularity condition \eqref{regu} is satisfied with this family of scales: there may be terms in them that do not actually appear in the unfolding. Better results on upper bound are obtained in \cite{dumortier}, using asymptotic developments of Abelian integrals, and in \cite{gavrilov}. In \cite{gavrilov}, the upper bound is given in terms of characteristic numbers of holonomy maps, not using asymptotic development of the Poincar\' e map.

\section{Application to number of zeros of Abelian integrals}\label{twofour}

\emph{Abelian integrals} are integrals of polynomial $1$-form $\omega$ along the continuous family of $1$-cycles of the polynomial Hamiltonian field, lying in the level sets of the Hamiltonian $H$, $\delta_t\subset \{H=t\}$,
\begin{equation}\label{abbel}
I_{\omega}(t)=\int_{\delta_t}\omega.
\end{equation}

In $\mathbb{R}^2$, determining zero points of Abelian integrals has been used as a tool for determining cyclicity of limit periodic sets of Hamiltonian vector fields (for details and examples see e.g. Zoladek \cite[Chapter 6]{zoladek}). 

\noindent Suppose that we have the following $\lambda$-perturbed family $(X_\lambda)$ of a Hamiltonian field $X=X_0$,
\begin{eqnarray}\label{perturb}
\dot{x}&=&\frac{\partial H}{\partial y}+\lambda P(x,y,\lambda),\nonumber\\
\dot{y}&=-&\frac{\partial H}{\partial x}+\lambda Q(x,y,\lambda),
\end{eqnarray} 
where $P,\ Q,\ H$ are polynomials and $\lambda>0$. Let $\omega_\lambda=Q dx- P dy$ be the polynomial $1$-form defined by $P,\ Q$. Let $t=0$ be a critical value of the Hamiltonian (the level set $\{H=0\}$ corresponding to limit periodic set $\Gamma$ in whose cyclicity we are interested), such that there exists $d>0$ and a continuous family of 1-cycles $(\delta_t)$ belonging to the level sets $\delta_t\subset \{H=t\}$,\ $t\in (0,d)$. Let $\tau$ be a transversal to the family of cycles $(\delta_t)$ on a small neighborhood of $t=0$, parametrized by $t\in[0,d)$. Then (see e.g. Zoladek \cite[Chapter 6]{zoladek}), the family of displacement functions $(f_\lambda)$ on $\tau$ is given by
\begin{equation}\label{app}
f_{\lambda}(t)=\lambda I_{\omega_\lambda}(t)+o(\lambda).
\end{equation}
Here, $I_{\omega_{\lambda}}(t)$ denotes the Abelian integrals for Hamiltonian $H$, along its cycles $\delta_t$, of polynomial forms $\omega_\lambda$. 
\smallskip

From \eqref{app}, for $\lambda$ small enough, the Abelian integral $I_{\omega_\lambda}(t)$ is the first approximation of the displacement function $f_\lambda$. Here we suppose that $I_{\omega_\lambda}(t)$ is not identically equal to zero, i.e. that $\omega_\lambda$ is not relatively exact. Therefore, it is natural that zeros of Abelian integrals  $(I_{\omega_{\lambda}})$ give information on multiplicity of zero points of displacement functions $(f_{\lambda})$, that is, on $Cycl(\Gamma,(X_\lambda))$.

Indeed, on some segment $[\alpha,\beta] \subset (0,d)$ away from critical value $t=0$, it is known that the number of zeros of Abelian integral gives an upper bound on the number of zeros of the displacement function $f_{\lambda}(t)$ on $[\alpha,\beta]$ of the perturbed system \eqref{perturb}, for $\lambda$ small enough (both counted with multiplicities), i.e. on the number of limit cycles born in perturbed system \eqref{perturb} in the area $\bigcup_{t\in[\alpha,\beta]}\delta_t$, for $\lambda<\lambda_0$ small enough (for this result, see e.g. \cite[Theorem 2.1.4]{spain}).

However, a problem arises if we approach the critical value $t=0$ and the result cannot be applied to the whole interval $[0,d)$. In some systems, some limit cycles (called \emph{alien cycles} in \cite{cau}) visible as zeros of displacement function are not visible as zeros of corresponding Abelian integral, because sometimes the approximation \eqref{app} is not good enough. One of the examples is the perturbation of the Hamiltonian field in the neighborhood of the saddle polycycle with 2 or more vertices, see Dumortier, Roussarie \cite{dumortier}. Abelian integrals near saddle polycycles have an expansion linear in $\log t$, see expansion \eqref{e} or \cite[Proposition 1]{dumortier}. On the other hand, see Roussarie \cite{roussarie}, the asymptotic expansion of the displacement function near the saddle polycyle with more than one vertex involves also powers of $\log t$ greater than $1$. 

In a neighborhood of the center singular point and of the saddle loop ($1$-saddle polycycle) of the Hamiltonian field, however, the multiplicity of corresponding Abelian integral gives correct information about cyclicity, see e.g. Dumortier, Roussarie \cite[Theorem 4]{dumortier}. 

\medskip
On the other hand, we have the following asymptotic expansion of Abelian integral \eqref{abbel} at critical point $t=0$ (see Arnold \cite[Chapter 10, Theorem 3.12]{arnold} and Zoladek \cite[Chapter 5]{zoladek}):
\begin{equation}\label{e}
I_{\omega}(t)=\sum_{\alpha}\sum_{k=0}^{1}a_{k,\alpha}(\omega) t^{\alpha}(-\log t)^k,
\end{equation}
where $\alpha$ runs over an increasing sequence of nonnegative rational numbers depending only on Hamiltonian $H(x,y)$ (such that $e^{2\pi i \alpha}$ are eigenvalues of monodromy operator of the singular value) and $a_{k,\alpha}\in\mathbb{R}$ depend on $\omega$. The Abelian integrals have thus an asymptotic development in Chebyshev scale:
\begin{equation*}
\mathcal{I}=\{t^{\alpha_1}(-\log t),t^{\alpha_1},
t^{\alpha_2}(-\log t),t^{\alpha_2},\ldots,
t^{\alpha_m}(-\log t),t^{\alpha_m},\ldots\}.
\end{equation*}

It makes sense, in the above example, to compute critical Minkowski order of the orbit $S^g(t_0)$, $g(t)=t-I_{\omega_0}(t)$ ($\lambda=0$), with respect to family $\mathcal I$, and obtain the multiplicity of a zero point $t=0$ of Abelian integral $I_{\omega_0}$ in the family of integrals $(I_{\omega_\lambda})$. From above comments, it is related to cyclicity of $\Gamma$, though not necessarily equal.

\chapter{Application of fractal analysis in formal classification of complex diffeomorphisms and saddles}\label{three}

\section{Introduction}\label{threezero}
We consider germs of complex diffeomorphisms, $f:(\mathbb{C},0)\to (\mathbb{C},0)$, with fixed point at the origin. Locally around the origin, they are of the form
\begin{equation}\label{opc}
f(z)=a_1 z+\sum_{k=2}^\infty a_k z^k,\ a_k\in\mathbb{C},\ a_1\neq 0.
\end{equation}
Depending on the multiplier $a_1$ of the linear part, we distinguish between three main types of local dynamics at the origin. The names are not consistent in the literature, therefore we precise them here. If $|a_1|\neq 1$, we will say that the origin is a \emph{hyperbolic fixed point} or that $f(z)$ is a \emph{hyperbolic germ}. If $|a_1|=1$, the germ will be called \emph{nonhyperbolic}. Furthermore, in nonhyperbolic case, the multiplier can be written as
$$a_1=e^{2\pi i\alpha}, \ \alpha\in\mathbb{R}.$$
We again distinguish between two cases:
\begin{itemize}
\item{(NH1)} Linear part is an \emph{irrational rotation}, $\alpha\in\mathbb{R}\setminus \mathbb{Q}$. 
\item{(NH2)} Linear part is a \emph{rational rotation}, $\alpha\in\mathbb{Q}$, $\alpha=\frac{p}{q},\ p,\ q \in\mathbb{N},\ p\leq q,\ (p,q)=1$. 

In case (NH2), without  loss of generality, we can suppose that the germ is \emph{tangent to the identity}, i.e. $a_1=1$. Otherwise, instead of $f(z)$, we consider its $n$-th iterate $f^{\circ q}(z)$, which is tangent to the identity. The germs tangent to the identity are called \emph{parabolic}, and the origin is called the \emph{parabolic fixed point}. 
\end{itemize}

\noindent We will describe the local dynamics for each case separately in the following sections. Let $$S^f(z_0)=\{f^{\circ n}(z_0)|\ n\in\mathbb{N}_0\}$$  denote the orbit of a diffeomorphism $f$, with initial point $z_0$. The problem that we deal with in this and the next chapter is: 

\emph{Can we recognize formal or even analytic normal form of a diffeomorphism}

\emph{using fractal properties of only one orbit?}
\medskip

We explain here shortly the well-known notions on formal and analytic normal forms. A \emph{formal normal form} $f_0$ of a germ $f$ is the \emph{simplest} germ that can be obtained from $f$ by formal changes of variables. More precisely, $f$ is formally conjugated to $f_0$ if and only if there exists a formal diffeomorphism $\widehat\varphi(z)\in z \mathbb{C}[[z]]$, $\widehat\varphi(z)=\sum_{i=1}^{\infty}\lambda_i z^i,\ \lambda_i\in\mathbb{C},\ \lambda_1\neq 0$, such that
$$
f_0=\widehat\varphi^{-1}\circ f \circ \widehat\varphi.
$$
Here, $\widehat\varphi^{-1}$ is meant in the sense of formal inverse. In other words, $f_0(z)$ is obtained from $f(z)$ by applying an infinite composition of changes of variables of the form $h(z)=c z$ or $h(z)=z+c z^k$, $c\in\mathbb{C}^*$. The name \emph{formal} suggests that we do not address the question of convergence of $\widehat\varphi(z)$ locally around the origin.

On the other hand, if $\widehat\varphi$ converges, that is, if $\varphi\in z\mathbb{C}\{z\}$, we say that $f$ and $f_0$ are \emph{analytically conjugated}. Then, $f_0$ is the \emph{analytic normal form} for $f$.

Finally, we say that the germ $f(z)$ in \eqref{opc} is \emph{formally (analytically) linearizable} if its formal (analytic) normal form is linear, that is, if by formal (analytic) changes of variables it can be reduced to its linear part $f_0(z)=a_1 z$.
\medskip

At the end of this introductory chapter, we comment on the case (NH1) of irrational rotation in the linear part. It is the only type of complex germs that we do not treat in this thesis. The germs are formally linearizable, see \cite[Proposition 1.3.1]{loray}. When analytic linearizability is concerned, every holomorphic germ with a fixed point of multiplier $\lambda=e^{2\pi i \alpha}$ is locally analytically linearizable if and only if $\alpha$ is not \emph{almost rational}. This rather complicated necessary and sufficient condition for linearizability was proven by Brjuno and Yoccoz. There is a dichotomy in telling whether for a generic $\alpha\in\mathbb{R}\setminus\mathbb{Q}$ the germ is linearizable. Indeed, the set of almost rational $\alpha$-s is of Lebesgue measure zero, but dense in $\mathbb{R}\setminus \mathbb{Q}$ and uncountably infinite, see \cite[Corollaries 11.3, 11.5]{milnor}. See for example \cite[Section 11]{milnor} for a short overview. Furthermore, analytic classification of nonlinearizable germs of type (NH1) is not known. Fractal analysis of orbits and connection with linearizability is a subject for future research. The following proposition follows immediately. However, the converse is not clear. 
\begin{proposition}\label{irat}
Suppose that the complex germ $f:(\mathbb{C},0)\to (\mathbb{C},0)$ of the type \emph{(NH1)} is analytically linearizable. For any orbit $S^f(z_0)$ of $f$, with initial point $z_0$ sufficiently close to the origin, it holds that
$$
\dim_B(S^f(z_0))=1.
$$ 
\end{proposition}

\begin{proof}
The analytic conjugacy is a bilipschitz mapping. It sends the orbit $S^f(z_0)$ of $f$ to a corresponding orbit of irrational rotation on some small circle. The orbit of irrational rotation is dense on the circle. Since box dimension of the set equals the box dimension of its closure, the orbit of irrational rotation has box dimension equal to $1$. Using bilipschitz property, the box dimension of $S^f(z_0)$ is also equal to 1.
\end{proof}

\section{Hyperbolic germs}\label{threeone}

Let $f(z)$ be a hyperbolic germ \eqref{opc}, with multiplier $|a_1|\neq 1$. By \cite[Section 8]{milnor}, the origin is an \emph{attracting} point for the local dynamics if $|a_1|<1$ and \emph{repelling} if $|a_1|>1$. If $|a_1|<1$, the orbit $S^f(z_0)$ with initial point $z_0$ sufficiently close to the origin accumulates at the origin. If $|a_1|>1$, the inverse diffeomorphism $f^{-1}(z)$ has an attracting fixed point at the origin, and its orbit accumulates at the origin.  Figure~\ref{pichyp} shows orbits of some hyperbolic germs with an attracting fixed point at the origin. 

\begin{figure}[h]
\vspace{-9cm}
\centering
\begin{subfigure}{.43\textwidth}
  \centering
  \includegraphics[width=1.5\linewidth]{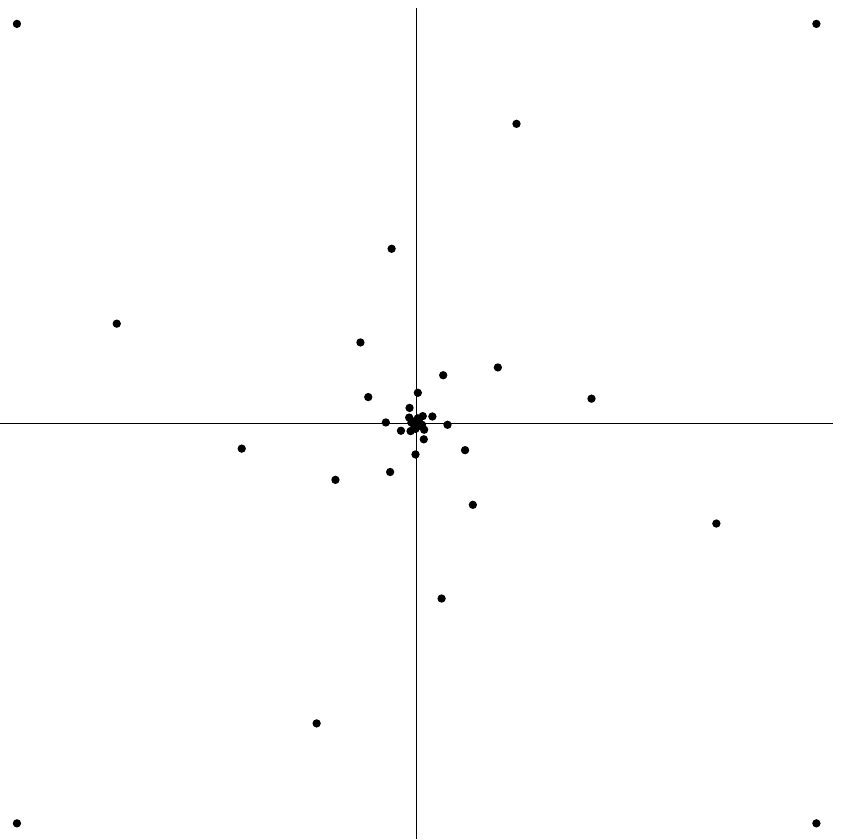}
  \caption{$f(z)=(1/2+1/4\cdot i)z+z^4$}
\end{subfigure}%
\begin{subfigure}{.43\textwidth}
  \centering
  \includegraphics[width=1.5\linewidth]{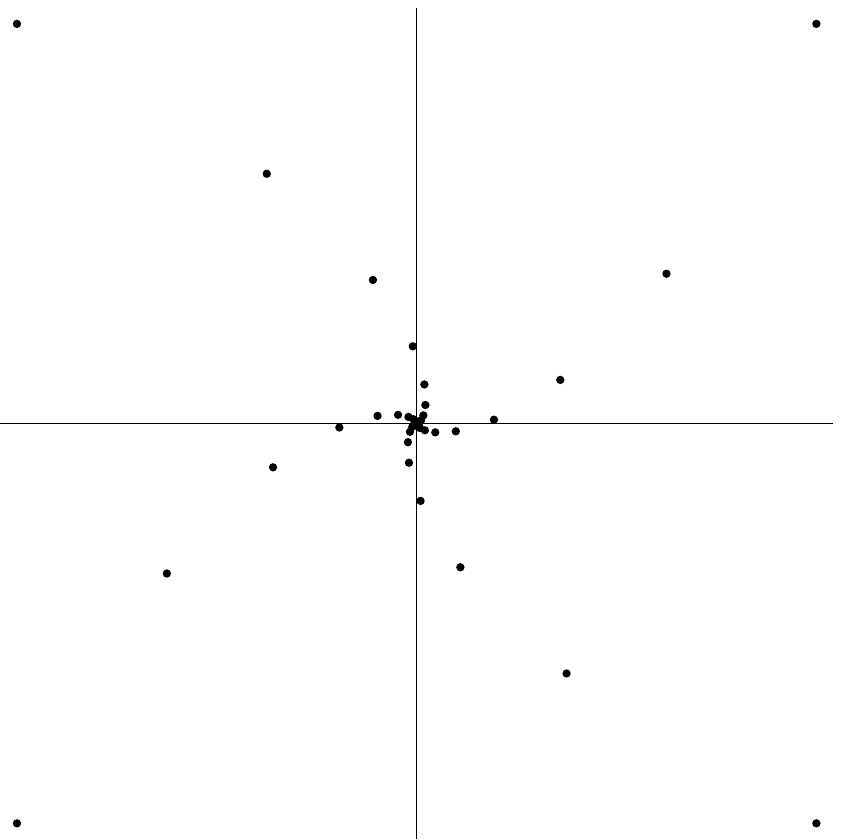}
  \caption{$f(z)=(-1/2+1/8\cdot i)z+z^6$}
\end{subfigure}
\caption{Some discrete orbits of hyperbolic germs at the origin.}
\label{pichyp}
\end{figure}

We cite Koenigs theorem from 1884 about analytic linearizability of hyperbolic germs:
\begin{nntheorem}[Koenigs linearization, Theorem 8.2\cite{milnor}]
Let $f(z)$ be a hyperbolic germ of a complex diffeomorphism. There exists a local analytic change of coordinates conjugating $f(z)$ with its linear part $f_0(z)=a_1 z$.
\end{nntheorem}

We show how fractal analysis of only one orbit of a hyperbolic germ \emph{recognizes} the analytic linearizability stated in Koenigs linearization theorem. The convergence of the orbit to the hyperbolic fixed point is very fast (actually, exponentially fast) and we expect its box dimension to be trivial. We prove in Proposition~\ref{hype} that this is indeed the case.

\begin{proposition}[Box dimension of orbits of hyperbolic germs]\label{hype}
Let the germ $f(z)$ have an attracting hyperbolic fixed point at the origin. Let $S^f(z_0)$ be any orbit with initial point sufficiently close to 0. It holds that
$$
\dim_B(S^f(z_0))=0.
$$
\end{proposition}

If the origin is repelling for $f(z)$, we consider the germ $f^{-1}(z)$ and its orbits instead.
\begin{proof}
Since $f$ is analytically linearizable and since box dimension is invariant under bilipschitz mappings, the box dimension is equal to the box dimension of one orbit of the linear part $f_0(z)=a_1 z$. This orbit is given explicitely by $$S^f(w_0)=\{a_1^k\cdot w_0\ |\ k\in\mathbb{N}_0\}, |a_1|<1,\ w_0\approx 0.$$ The distances of the points from the origin decrease exponentially. We can now estimate the asymptotic behavior of the area of the $\varepsilon$-neighborhood directly, as was done many times in the previous chapter. 

The quicker proof is to consider the complex diffeomorphism $f(z)$ as a planar diffeomorphism (in $\mathbb{R}^2$), with hyperbolic fixed point at the origin. We apply directly Theorem 3.17 from the thesis of Horvat-Dmitrovi\' c \cite{lana}, which states the triviality of the box dimension of any orbit of a hyperbolic germ in $\mathbb{R}^2$. 
\end{proof}

\section{Formal classification of parabolic germs}\label{threetwo}

Let $f(z)$ be a germ of the type (NH2):
$$
f(z)=e^{2\pi i p/q}z+a_2z^2+a_3 z^3+o(z^3),\ p,\ q\in\mathbb{N},\ p\leq q,\ (p,q)=1.
$$
We suppose in the sequel that the germ is \emph{parabolic}, i.e., tangent to the identity. Otherwise, the classification of $f$ is given by the rotation angle $p/q$, together with classification of its $q$-th iterate $f^{\circ q}$, which is a parabolic germ. Note that the rotation angle $p/q$ is visible from one orbit $S^f(z_0)$ of $f$. The orbit consists of $q$ disjoint orbits of the $q$-th iterate $f^{\circ q}(z)$, where the next one is approximately the rotation of the former by angle $2\pi p/q$. Therefore, $q$ represents the number of disjoint orbits of $f^{\circ q}$ in $S^f(z_0)$, and $p$ the number of orbits (counted anticlockwise) inbetween two consecutive points of the orbit. As an example, Figure \ref{niter} shows one orbit of a germ of the type (NH2), not tangent to the identity.

\medskip

\begin{figure}[htp]
\begin{center}
  % replace aims_logo.eps by your figure file name
  \vspace{-37cm}
  \includegraphics[scale=1.5, trim={-1cm 0cm 0cm 0cm}]{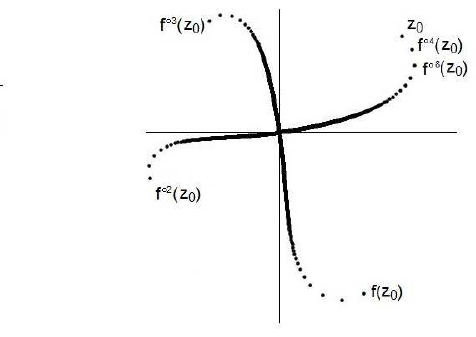}
  \vspace{-2cm}
  \caption{\small{One orbit $S^f(z_0)$ of germ $f(z)=e^{2\pi i\cdot 3/4}z-z^2+z^4+z^5$, consisting of four orbits of the parabolic germ $f^{(\circ 4)}(z)$.}}\label{niter}
  \end{center}
\end{figure}

Therefore, from now on, we suppose
\begin{equation}\label{tp}
f(z)=z+\alpha_1 z^{k+1}+\alpha_2 z^{k+2}+o(z^{k+2}),\ \alpha_1\neq 0,\ \alpha_i\in\mathbb{C}.
\end{equation}
Here, $k+1$ is the \emph{multiplicity of the fixed point zero} of $f$ in the sense of Definition~\ref{smult}.

We cite below the formal classification theorem for parabolic germs, as well as the idea of the proof which shows the way of reducing the diffeomorphism to its formal normal form. It is due to Birkhoff, Ecalle, Kimura, Szekeres around the year 1950.

Let $Exp(X_{k,\lambda})$ denote the time-one map of a vector field
$$
X_{k,\lambda}=\frac{z^{k+1}}{1+\frac{\lambda}{2\pi i}z^k}\frac{d}{dz},\quad k\in\mathbb{N},\ \lambda\in\mathbb{C}.
$$
By formula in e.g. \cite[Proposition~1.2.3]{loray} for computing the time-one map of a germ of holomorphic vector field $X$, 
$
Exp(X)=\sum_{n\geq 0}\frac{X^n.id}{n!},
$
we have the following development:
$$
Exp(X_{k,\lambda})=z+z^{k+1}+\left(\frac{k+1}{2}-\frac{\lambda}{2\pi i}\right)z^{2k+1}+o(z^{2k+1}).
$$ 

\begin{proposition}[Formal normal form for parabolic germs, Proposition 1.3.1 in \cite{loray}]\label{snff}
Let $f(z)$ be a parabolic germ \eqref{tp}, different from the identity map. By formal changes of variables, it can be reduced to
\begin{equation}\label{fnforma}
f_0(z)=Exp(X_{k,\lambda})=z+z^{k+1}+\left(\frac{k+1}{2}-\frac{\lambda}{2\pi i}\right)z^{2k+1}+o(z^{2k+1}).
\end{equation}
Here, $k+1$ is the multiplicity of $f$ as in \eqref{tp} and $\lambda\in\mathbb{C}$ is a complex number.
\end{proposition}

\begin{proof} First we apply the change of variables $\varphi_1(z)=c_1 z$, where $c_1=a_1^{-1/k}$ is chosen such that the coefficient in front of $z^{k+1}$ after the change becomes equal to $1$. Then we eliminate each term in \eqref{tp} successively, by applying a sequence of changes of the form $\varphi_l(z)=z+c_l z^l$, $l\geq 2$, $c_l\in\mathbb{C}$: 
\begin{align*}
&f\circ \varphi_l(z)-\varphi_l\circ f(z)=(k+1-l)c_l z^{k+l}+o(z^{k+l}),\\
&\varphi_l^{-1}\circ f\circ \varphi_l(z)=f(z)+(k+1-l)c_l z^{k+l}+o(z^{k+l}).
\end{align*}
The coefficient $c_l$ is chosen so that the change $\varphi_l$ eliminates the term $z^{k+l}$, and at the same time leaves the previous terms intact.
In such way it is possible to eliminate all terms except $z^{k+1}$ and the residual term $z^{2k+1}$. They therefore remain in the formal normal form. 
\end{proof}

The simpler germ $f_0$ in \eqref{fnforma} is called the \emph{standard formal normal form}.
The elements $(k,\lambda)$ of the formal normal form are called the \emph{formal invariants} of a parabolic germ. The coefficient $\left(\frac{k+1}{2}-\frac{\lambda}{2\pi i}\right)$ in front of $z^{2k+1}$ in \eqref{fnforma} equals to the \emph{residual fixed point index} of the diffeomorphism $f$,
$$
\iota(f,0)=Res\left(\frac{1}{f(z)-z},0\right),
$$ see \cite{milnor} or \cite{ilyajak} for definition. The residual fixed point index of a diffeomorphism is invariant under the formal changes of variables. We may conclude that the formal invariants of a difeomorphism $f$ consist of multiplicity of zero as a fixed point of $f$ and of the residual fixed point index of $f$ at zero.
\smallskip

In the proof we see that, instead of $f_0$ in \eqref{fnforma}, as the standard formal normal form of $f$ we can also assume the finite germ
$$
f_0(z)=z+z^{k+1}+\left(\frac{k+1}{2}-\frac{\lambda}{2\pi i}\right)z^{2k+1}.
$$
The terms after $z^{2k+1}$ in \eqref{fnforma} can be eliminated one by one by formal composition of infinitely many changes of variables.
\smallskip

For convenience in our considerations, we introduce a slightly different formal normal form. We admit only formal changes of variables tangent to the identity, so that the first coefficient $a_1$ remains unchanged. Thus we get a slightly restricted formal classes. We will see later in Subsection~\ref{threetwotwo} the reason for this: we want all the diffeomorphisms inside one formal class to share the same fractal properties of orbits. 
\begin{proposition}[Extended formal normal form for parabolic germs]\label{neuf}
Let $f(z)$ be as above. By formal changes of variables \emph{tangent to the identity} $f(z)$ can be reduced to 
$$
g_0(z)=z+a_1 z^{k+1}+a_1^2\left(\frac{k+1}{2}-\frac{\lambda}{2\pi i}\right)z^{2k+1}.
$$
\end{proposition}
To avoid confusion with the standard normal form $f_0(z)$, we call the normal form $g_0(z)$ the \emph{extended formal normal form} of $f(z)$. In this case, the formal invariants are given by the triple $(k,a_1,\lambda)$ instead of the pair $(k,\lambda)$.

\begin{proof}
In reducing $f$ to $g_0$, we repeat the same procedure as in the proof of the standard formal normal form, except that we omit the first change $\varphi_1(z)=c_1 z,\ c_1^k=a_1$ which is not tangent to the identity. Thus, the standard formal normal form $f_0$ is obtained from $g_0$ simply by the additional change $\varphi_1(z)$.
\end{proof}

\smallskip
Let us just mention beforehand that the case $f=id$ mentioned in the theorem is very simple. In this case, any trajectory $S^f(z_0)$ for $z_0$ near the origin is periodic and consists of only one point $z_0$. Its box dimension is thus equal to 0. We neglect this trivial case.

\medskip
We now describe the local dynamics of discrete orbits generated by parabolic germs. The description was given in the well-known \emph{Leau-Fatou flower theorem}, stated in Leau's these at the end of $19^{th}$ century. The theorem can be found in e.g. \cite[Theorem 10.5]{milnor} or \cite[Theorem 2.3.1]{loray}. In short, for a diffeomorphism of multiplicity $k+1$, there exist $k$ attracting and $k$ repelling equidistant directions, given by complex roots of the first coefficient $a_1$: $$(-a_1)^{-1/k} \text{ (attracting)},\quad a_1^{-1/k}\text{ (repelling)}.$$ Around them, invariant attracting and repelling open sets are formed in the form of overlapping petals. The repelling petals for $f(z)$ are in fact attracting petals for the inverse diffeomorphism $f^{-1}(z)$ and the other way round. The orbits are tangent to attracting and repelling directions at the origin. The orbits (both positive and negative iterations considered) with initial points in the intersection of attracting and repelling petals are closed. For better insight, see Figure~\ref{petal}. 

\begin{figure}[htp]
\begin{center}
  \vspace{-34cm}
  % replace aims_logo.eps by your figure file name
  \includegraphics[scale=1.4, trim={-2.3cm 0cm 0cm 0cm}]{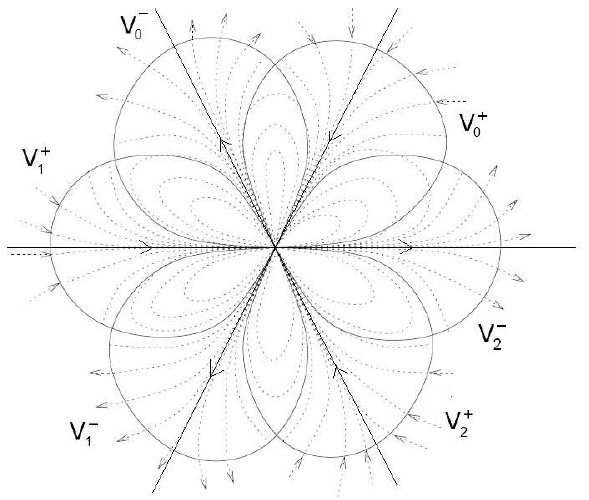}\\
  \caption{\small{Attracting and repelling petals for e.g. $f(z)=z+z^4$, taken from \cite[Figure 2]{loray}}.}\label{petal}
  \end{center}
\end{figure}

\bigskip
In this chapter we address the problem of bijective correspondence between formal type of a parabolic diffeomorphism and fractal properties of only one orbit. The problem is based on two questions:
\begin{enumerate}
\item Can we recognize the type of a diffeomorphism from fractal properties of one orbit? 
\item If we know the type of a diffeomorphism, can we tell fractal properties of its orbits?
\end{enumerate}

By fractal properties, we usually assumed the box dimension and the Minkowski content of the orbit, which are by definition computed from the rate of growth of the area of the $\varepsilon$-neighborhood of the orbit, as $\varepsilon\to 0$. See Section~\ref{onethree} for precise definitions. 

However, there is a deficiency of the standard fractal properties in our problem. From the asymptotic development of the area of the $\varepsilon$-neighborhood, only real information on complex formal invariants can be obtained. In the next definition we therefore generalize the notion of the area of the $\varepsilon$-neighborhood of a set. It becomes a complex number whose modulus is the area and whose argument refers to the direction of the set in the plane. We call it the \emph{directed area}. 

\begin{definition}[The directed area of a measurable set]\label{dir}
Let $U\subset\mathbb{C}$ be a measurable set, whose center of mass is not the origin. We define the \emph{directed area of the set $U$}, denoted by $A^\mathbb{C}(U)$, as the complex number
$$
A^{\mathbb{C}}(U)=A(U)\cdot \nu_{t(U)},
$$
where $A(U)$ denotes the area of $U$, $t(U)\in\mathbb{C}$ the center of mass of $U$ and $\nu_{t(U)}=\frac{t(U)}{|t(U)|}\in\mathbb{C}$, $|\nu_{t(U)}|=1$, the normalized center of mass of $U$.
\end{definition}

Note that the directed area is not a (vector) measure, as defined in e.g. \cite{klu}. It does not verify the countable stability property, that is, it is not true in general that $A^\mathbb{C}(\bigcup_{i=1}^{\infty} V_i)=\sum_{i=1}^{\infty}A^\mathbb{C}(V_i)$, for pairwise disjoint sets $V_i\subset\mathbb{C}$, $i\in\mathbb{N}$. Furthermore, this notion should not be confused with the directional $\varepsilon$-neighborhood, also called the directional Minkowski sausage, defined in \cite{tricot}.
\medskip

Sometimes, mostly in Chapter~\ref{four}, we will use a slightly different notion of \emph{complex measure of a set}. It can be easily verified that it is indeed a vector measure in the standard sense.
\begin{definition}[The complex measure of a measurable set]\label{dir1}
Let $U\subset\mathbb{C}$ be a measurable set. We define the \emph{complex measure of the set $U$}, denoted by $\widetilde{A^\mathbb{C}}(U)$, as the complex number
$$
\widetilde{A^{\mathbb{C}}}(U)=A(U)\cdot t(U).
$$
\end{definition}
This definition differs from the directed area only by the fact that the center of mass is not normalized. 

\medskip
In Subsection~\ref{threetwotwo} we compute the asymptotic development in $\varepsilon$ of the directed areas of the $\varepsilon$-neighborhoods of orbits. Then, in Subsection~\ref{threetwofour}, we connect the coefficients in the development with fractal properies of orbits. Furthermore, we state our main results about the bijective correspondence between formal invariants and fractal properties of orbits of diffeomorphisms. The results were published in 2013 in Resman \cite{formal}.

\subsection{Asymptotic development for $\varepsilon$-neighborhoods of orbits}\label{threetwotwo}

This section is dedicated to computing the asymptotic developments of the directed areas of the $\varepsilon$-neighborhoods of orbits of parabolic diffeomorphisms, as $\varepsilon\to 0$. To be able to read all formal invariants, we need not only the first term, but the first $(k+1)$ terms in the development. This is not surprising, since the formal invariants are determined by the $(k+1)$-jet of the diffeomorphism $f$. On the other hand, we show that the $j$-th coefficient in the development of the $\varepsilon$-neighborhood is determined by the $j$-jet of $f$. 

Let $f(z)$ be a parabolic diffeomorphism,
$$
f(z)=z+a_1 z^{k+1}+a_2 z^{k+2}+o(z^{k+2}),\ a_i\in\mathbb{C},\ a_1\neq 0.
$$
Let $S^f(z_0)$ denote an attracting orbit of $f(z)$, with initial point $z_0$ in an attracting petal. We can analogously take a repelling orbit and consider the inverse diffeomorphism $f^{-1}(z)$ instead.  Let
$$
A=(-ka_1)^{-\frac{1}{k}}
$$
be one of the $k$ attracting directions in whose attracting sectors the initial condition $z_0$ lies. In other words, we chose the $k$-th complex root of $-1/a_1$  whose argument is closest to $z_0$. By $\nu_A$, we denote the normalized complex number $A$.
\begin{theorem}[Asymptotic development of the directed area of $\varepsilon$-neighborhoods of orbits]\label{asy}
Let $k>1$. The directed area of the $\varepsilon$-neighborhood of orbit $S^f(z_0)$ has the following asymptotic development, as $\varepsilon\to 0$:
\begin{equation}\label{comarea}
\begin{split}
A^{\mathbb{C}}(S^f(z_0)_\varepsilon)=&K_1 \varepsilon^{1+\frac{1}{k+1}}+K_2 \varepsilon^{1+\frac{2}{k+1}}+\ldots+K_{k-1}\varepsilon^{1+\frac{k-1}{k+1}}+K_{k}\varepsilon^{1+\frac{k}{k+1}}\log\varepsilon+  \\
&\ \ \ +H^f(z_0)\varepsilon^{1+\frac{k}{k+1}}+K_{k+1} \varepsilon^2\log\varepsilon+R(z_0,\varepsilon),\ R(z_0,\varepsilon)=o(\varepsilon^2\log\varepsilon).
\end{split}
\end{equation}

All coefficients $K_i$, $i=1,\ldots,k+1,$ are complex-valued functions of $k$, $A$ and the first $i$ coefficients $a_2,\ \ldots,\ a_i$ of the diffeomorphism. For $2\leq i\leq k$, if $a_2,\ldots,a_i=0$, it holds that $K_i=0$. The coefficient $H^f(z_0)$ is a complex-valued function of the initial condition $z_0$, which depends on the whole diffeomorphism $f$.

\noindent Furthermore, \emph{important} coefficients $K_1$ and $K_{k+1}$ are of the form:
\begin{equation}\label{impcoef}
\begin{split}
&K_1=\frac{k+1}{k}\cdot\sqrt\pi\cdot\frac{\Gamma(1 + \frac{1}{2k+2})}{\Gamma(\frac{3}{2}+\frac{1}{2k+2})}\bigg(\frac{2}{|a_1|}\bigg)^{1/(k+1)}\cdot \nu_A,\\
&K_{k+1}=\nu_A\cdot\Bigg[-\frac{\pi}{k+1}Re\big(\frac{a_{k+1}}{a_1^2}-\frac{k+1}{2}\big)+\Bigg.\\
&\qquad \Bigg. \bigg(\frac{2(k-1)}{k+1}\Big(\frac{|a_1|}{2}\Big)^{1/(k+1)}\frac{\frac{\Gamma(\frac{1}{2}+\frac{1}{2k+2})}{\Gamma(2+\frac{1}{2k+2})}-\sqrt{\pi}}{\frac{\Gamma(\frac{1}{k+1})}{\Gamma(\frac{3}{2}+\frac{1}{k+1})}+\sqrt{\pi}}\bigg)\cdot i\cdot Im\big(\frac{a_{k+1}}{a_1^2}-\frac{k+1}{2}\big) \Bigg]+\\
&\hspace{9cm} +g(k,A,a_2,\ldots,a_{k}).
\end{split}
\end{equation}
Here, $g(k,A,a_2,\ldots,a_k)$ is a complex-valued function such that $g(k,A,0,\ldots,0)=0$.
\end{theorem}
Note that the coefficients $K_i$, $i=1,\ldots,k+1$ do not depend on the initial point $z_0$, but only on the attracting sector of the initial point (via $A$). We will see in the proof that the dependence of $H^f(z_0)$ on the initial point comes from the directed area of the tail of the $\varepsilon$-neighborhood of the orbit. For obtaining formal invariants, we are not interested in the properties of the remainder term $R(z_0,\varepsilon)$. They will be discussed in more detail in Section~\ref{threethree}, concerning the analytic classification. 
\medskip

The proof of Theorem~\ref{asy} is rather technical and given in Subsection~\ref{threetwothree}. We will see in the proof that in the special, boundary case $k=1$, we obtain a slightly different development:
\begin{proposition}[Asymptotic developments in the boundary case $k=1$]\label{asy1}
Let $$f(z)=z+a_1 z^2+a_2 z^3+o(z^2),\ a_1\neq 0,\ a_2\in\mathbb{C},$$ 
be a parabolic diffeomorphism of multiplicity $2$. With the same notations as above, the following development for the area and the center of mass of the $\varepsilon$-neighborhood of an attracting orbit $S^f(z_0)$ holds:
\begin{align}\label{deff}
&A(S^f(z_0)_\varepsilon)=\sqrt{\frac{\pi}{2}}\cdot\frac{\Gamma(1/4)}{\Gamma(7/4)}\cdot|a_1|^{-1/2} \ \varepsilon^{3/2}+\frac{\pi}{2}Re\left(1-\frac{a_2}{a_1^2}\right)\ \varepsilon^2\log\varepsilon+o(\varepsilon^2\log\varepsilon),\\
&\widetilde{A^\mathbb{C}}(S^f(z_0)_\varepsilon)=\frac{\pi}{2a_1}\ \varepsilon^2\log\varepsilon+H^f(z_0)\varepsilon^2+\nonumber\\
&\ +\left(-\frac{5\pi}{4\sqrt 2}+\frac{\sqrt\pi}{4\sqrt 2}\cdot\frac{\Gamma(3/4)}{\Gamma(5/4)}\right)|a_1|^{1/2}\frac{1}{a_1}\cdot i\cdot Im\left(1-\frac{a_2}{a_1^2}\right)\ \varepsilon^{5/2}\log\varepsilon+o(\varepsilon^{5/2}\log\varepsilon),\ \varepsilon\to 0.\nonumber
\end{align}
\end{proposition} 

\noindent We comment on this case here only for the sake of completeness. The development of $A^\mathbb{C}(S^f(z_0)_\varepsilon)$ can be computed using \eqref{deff}:  $$A^\mathbb{C}(S^f(z_0)_\varepsilon)=K_1\varepsilon^{3/2}+H^f(z_0,\varepsilon)+K_2\varepsilon^2\log\varepsilon+o(\varepsilon^2\log\varepsilon),\ \varepsilon\to 0.$$ However, due to a somewhat different structure of the developments in the case $k=1$, the coefficients $K_1$ and $K_2$ in the development of the directed area above are deficient in reading the complete coefficients $a_1$ and $a_2$, as was case for $k>1$. Note that the extended formal normal form of $f$ is given by $(1,a_1,a_2)$. To read $a_1$ and $a_2$, we need to consider the coefficients in the developments of both the area and the complex measure. In the sequel, we consider only the case $k>1$. In the case $k=1$, similar conclusions can be drawn considering these two developments.  
\smallskip

\begin{example}[The development for formal normal forms]
Let $k>1$. Let $$g_0(z)=z+a_1 z^{k+1}+a z^{2k+1},\ a_1\neq 0 \text{\ \ $($already the formal normal form$)$}.$$ Since the only coefficients different from zero are $a_1$ and $a$, from Theorem~\ref{asy} we get that 
$$
A^{\mathbb{C}}(S^{g_0}(z_0)_\varepsilon)=K_1 \varepsilon^{1+\frac{1}{k+1}}+H^f(z_0)\varepsilon^{1+\frac{k}{k+1}}+K_{k+1} \varepsilon^2\log\varepsilon+o(\varepsilon^2\log\varepsilon),\ \varepsilon\to 0.
$$
Here, $K_1=K_1(k,A)$ and $K_{k+1}=K_{k+1}(k,A,a)$ are functions of $k,\ A$ and $a$ only.
\end{example}

\subsection{Proof of the asymptotic development}\label{threetwothree}

Here we prove the asymptotic developments from Theorem~\ref{asy} and Proposition~\ref{asy1} stated in Subsection~\ref{threetwotwo}.
\medskip

The proof of Theorem~\ref{asy} is rather long and each step is contained in a separate lemma below. Some auxiliary, purely technical propositions needed in proofs of the lemmas are in Subsection~\ref{threetwofive}.

\noindent \emph{Overview of the proof of Theorem~\ref{asy}}. Suppose $\varepsilon>0$. By Definition~\ref{dir},
$$
A^{\mathbb{C}}(S^f(z_0)_\varepsilon)=A(S^f(z_0)_\varepsilon)\cdot \nu_{t}(S^f(z_0)_\varepsilon).
$$
Therefore, we need to compute the first $k+1$ terms in the development of the area of the $\varepsilon$-neighborhood and the first $k+1$ terms in the development of its normalized center of mass. Following the ideas from \cite{tricot}, $\varepsilon$-neighborhoods of the orbit, $S^f(z_0)_\varepsilon$, can be regarded as disjoint unions of the nucleus $N_{\varepsilon}$ and the tail $T_{\varepsilon}$. The tail $T_{\varepsilon}$ is the union of disjoint discs $K(z_i,\varepsilon),\ i=0,\ldots,n_\varepsilon$. The nucleus $N_{\varepsilon}$ is the union of overlapping discs $K(z_i,\varepsilon)$, $i=n_\varepsilon+1,\ldots,\infty$. Here, $n_\varepsilon$ denotes the index when discs around the points start to overlap, see Figure~\ref{division}. In our case, this `critical' index $n_\varepsilon$ is unique and well-defined, since the distances between two consecutive points are strictly decreasing, see Proposition~\ref{nhood}.$i)$ in Subsection~\ref{threetwofive}.

\newpage

\begin{figure}[htp]
\begin{center}
  % replace aims_logo.eps by your figure file name
  \vspace{-45cm}
  \includegraphics[scale=1.6,trim={-2cm 0cm 0cm 0cm}]{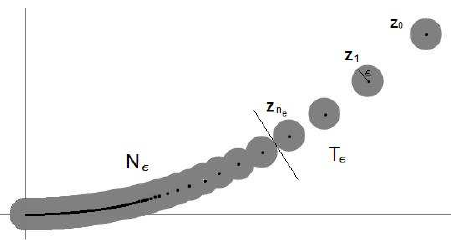}
  \caption{\small{The $\varepsilon$-neighborhood of an orbit $S^f(z_0)$ of a parabolic germ, for small $\varepsilon$, divided into tail $T_\varepsilon$ and nucleus $N_\varepsilon$.}}\label{division}
  \end{center}
\end{figure}

\emph{Step 1.} In Lemma~\ref{asyneps}, we compute the first $k+1$ terms in the asymptotic development of the index $n_\varepsilon$, as $\varepsilon\to 0$.

\emph{Step 2.} Using the development for $n_\varepsilon$, we compute the first $k+1$ terms in the development of the area of the $\varepsilon$-neighborhood of the orbit, $A(S^f(z_0)_\varepsilon)$, as $\varepsilon\to 0$. This consists of two parts: first, in Lemma~\ref{asynucl}, we compute the development of the area of the nucleus, $A(N_\varepsilon)$. Second, in Lemma~\ref{asytail}, we compute the development of the area of the tail, $A(T_\varepsilon)$. Finally,
\begin{equation}\label{sumarea}
A(S^f(z_0)_\varepsilon)=A(N_\varepsilon)+A(T_\varepsilon).
\end{equation}

\emph{Step 3.} We need to find first $k+1$ terms in the development of the normalized center of mass of the $\varepsilon$-neighborhood of the orbit, $\nu_{t(S^f(z_0)_\varepsilon)}$, as $\varepsilon\to 0$.
Obviously,
\begin{equation}\label{summass}
\nu_{t(S^f(z_0)_\varepsilon)}=\frac{t(N_\varepsilon)\cdot A(N_\varepsilon)+t(T_\varepsilon)\cdot A(T_\varepsilon)}{|t(N_\varepsilon)\cdot A(N_\varepsilon)+t(T_\varepsilon)\cdot A(T_\varepsilon)|}.
\end{equation}
Again, in Lemma~\ref{massnucl}, we compute the first $k+1$ terms for the nucleus, $t(N_\varepsilon)\cdot A(N_\varepsilon)$. In Lemma~\ref{masstail}, we do the same for the tail, $t(T_\varepsilon)\cdot A(T_\varepsilon)$.

Now, combining the obtained developments \eqref{sumarea} and \eqref{summass}, the development for $A^\mathbb{C}(S^f(z_0)_\varepsilon)$ follows.

\medskip
We now prove the lemmas used in the proof. They provide asymptotic developments up to the first $k+1$ terms of the expressions that are neccessary for computing the first $k+1$ terms of asymptotic development of the directed area. In all these developments, we provide precise information only on the first and on the $(k+1)$-st coefficient, since they are the only ones that affect the first and the $(k+1)$-st coefficient in the development of the final directed area. The proofs are rather direct and very technical. They rely on deducing the asymptotic development of the points of the orbit $z_n=f^{\circ n}(z_0)$, as $n\to\infty$, simply using the difference equation $z_{n+1}=f(z_n)$ and deducing the development term by term. It is given in Proposition~\ref{orbit}.

As before, let $f(z)=z+a_1z^{k+1}+a_2z^{k+2}+a_3z^{k+3}+\ldots$, $a_i\in\mathbb{C}$, $a_1\neq 0,\ k\geq 1$, be a parabolic diffeomorphism. Let the initial point $z_0$ belong to an attracting sector. We denote by $A$ the attracting direction
$$
A=(-ka_1)^{-\frac{1}{k}},
$$
where we chose the one of $k$ complex roots for which $z_0$ is closest to the direction $A$. That is, $z_0$ lies in the attracting petal around the attracting direction $A$.

\begin{proposition}[Asymptotic development of $z_n$]\label{orbit}
Let $z_n=f^{(\circ n)}(z_0)$, $n\in\mathbb{N}_0$, denote the points of the orbit $S^f(z_0)$. Let $k\geq 1$. Then
\begin{equation}\label{znasy}
\begin{split}
z_n&=g_1n^{-\frac{1}{k}}+g_2 n^{-\frac{2}{k}}+g_3n^{-\frac{3}{k}}+g_4 n^{-\frac{4}{k}}+\ldots+g_k n^{-1}
+\\
&\hspace{5cm}+g_{k+1} n^{-\frac{k+1}{k}}\log n +o(n^{-\frac{k+1}{k}}\log n).
\end{split}
\end{equation}
Here, coefficients $g_i=g_i(k,A,a_2,\ldots,a_i)$, $i=2,\ldots,k$, are complex-valued functions of $k$ and first $i$ coefficients of $f(z)$, with the property $g_i(k,A,0,\ldots,0)=0$.
Furthermore,
$$
g_1=A,\ \ g_{k+1}=-\frac{1}{k}A^{k+1}\left(\frac{a_{k+1}}{a_1}-\frac{a_1(k+1)}{2}+h(k,A,a_2,\ldots,a_{k})\right),
$$
where $h=h(k,A,a_2,\ldots,a_{k})$ is a complex-valued function satisfying $h(k,A,0,\ldots,0)=0$.
\end{proposition}

\begin{proof} The following proof mimics the standard technique for obtaining the asymptotic development of a real iterative sequence from e.g. \cite[Chapter 8.4]{bruijn}. In the complex case, we apply the whole technique sectorially.  Suppose as above that $z_0$ lies in an attracting sector around attracting vector $A$. By already explained dynamics, the whole orbit $\{z_n\}$ lies in that attracting sector and is tangent to $A$ at the origin. On this sector, the change of variables
\begin{equation}\label{change}
z=A w^{-\frac{1}{k}}
\end{equation}
is well-defined, the complex root of $w$ being uniquely determined. The trajectory $\{z_n\}$ is transformed to $\{w_n\}$ and obviously
\begin{equation}\label{conv}
Arg(w_n^{-\frac{1}{k}})\to 0,\text{ as } n\to\infty.
\end{equation}
The recurrence relation for $z_n$
\begin{equation*}
z_{n+1}=z_n+a_1z_n^{k+1}+a_2z_n^{k+2}+a_3z_n^{k+3}+\ldots
\end{equation*}
transforms to the following recurrence relation for $w_n$:
\begin{equation}\label{chdiff}
\begin{split}
w_{n+1}=&w_n+1+\frac{a_2}{a_1}Aw_n^{-\frac{1}{k}}+\frac{a_3}{a_1}A^2w_n^{-\frac{2}{k}}+\ldots+\\
&+\frac{a_{k}}{a_1}A^{k-1}w_n^{-\frac{k-1}{k}}+\left(\frac{a_{k+1}}{a_1}-\frac{(k+1)a_1}{2}\right)A^{k}w_n^{-1}+o(w_n^{-1}).
\end{split}
\end{equation}
Obviously, $\frac{w_n-w_0}{n}=\frac{1}{n}\sum_{l=1}^{n}(w_l-w_{l-1})$. By \eqref{chdiff}, it holds that $(w_l-w_{l-1})\to 1$, as $l\to\infty$, therefore $\lim_{n\to\infty}\frac{w_n}{n}=1$. From \eqref{chdiff} we then have $$w_{n+1}-w_{n}=1+O(n^{-\frac{1}{k}}).$$ By recursion and using integral approximation of the sum, we get
$$
w_{n}=n+O(n^\frac{k-1}{k}).
$$
For the standard technique of integral approximation of the sum, see Proposition~\ref{inte} in Subsection~\ref{threetwofive}.
To compute the exact constant of the second term, with this development, we return to \eqref{chdiff} and get
$$
w_{n+1}-w_n=1+\frac{a_2}{a_1}A n^{-\frac{1}{k}}+O(n^{-\frac{2}{k}}).
$$
By recursion and using integral approximation of the sum,
$$
w_n=n+\frac{a_2}{a_1}A\frac{k}{k-1}n^{\frac{k-1}{k}}+O(n^\frac{k-2}{k}).
$$
Repeating this procedure $k$ times, we get the first $k+1$ terms in the development of $w_n$:
\begin{align*}
w_n=&n+\frac{k}{k-1}A\frac{a_2}{a_1}n^{\frac{k-1}{k}}+\left[\frac{k}{k-2}A^2\frac{a_3}{a_2}+h_2(k,A,a_2)\right]n^{\frac{k-2}{k}}+\\
&+\left[\frac{k}{k-3}A^3\frac{a_4}{a_1}+h_3(k,A,a_2,a_3)\right]n^{\frac{k-3}{k}}+\ldots+\nonumber\\
&+\left[kA^{k-1}\frac{a_{k}}{a_1}+h_{k-1}(k,A,a_2,\ldots,a_{k-1})\right]n^{\frac{1}{k}}+\nonumber\\
&+\left[A^{k}\frac{a_{k+1}}{a_1}+\frac{k+1}{2k}+h_{k-1}(k,A,a_2,\ldots,a_{k})\right]\log n+o(\log n),\ n\to\infty.\nonumber
\end{align*}
Here, $h_i$ are complex-valued functions and $h_i(k,A,0,\ldots,0)=0$, $i=2,\ldots,k$. Note the form of the coefficients and their dependence on coefficients of the diffeomorphism. The development for $z_n$ now follows from the development for $w_n$, \eqref{change} and \eqref{conv}. 
\end{proof}

\begin{lemma}[Asymptotic development of $n_\varepsilon$]\label{asyneps}
Let $k\geq 1$. Suppose $n_\varepsilon$ is the critical index separating the nucleus and the tail. Then it has the following asymptotic development:
\begin{equation}\label{devneps}
n_\varepsilon=p_1\varepsilon^{-1+\frac{1}{k+1}}+p_2\varepsilon^{-1+\frac{2}{k+1}}+
p_3\varepsilon^{-1+\frac{3}{k+1}}+\ldots+p_k\varepsilon^{-1+\frac{k}{k+1}}+p_{k+1}\log\varepsilon+o(\log\varepsilon),\ \varepsilon\to 0,
\end{equation}
where coefficients $p_i=p_i(k,A,a_2,\ldots,a_i)$, $i=2,\ldots,k,$ are real-valued functions of $k$ and first $i$ coefficients of $f(z)$ with the property $p_i(k,A,0,\ldots,0)=0$. Furthermore,
\begin{align*}
p_1=&\Big(\frac{2}{|a_1A^{k+1}|}\Big)^{-1+\frac{1}{k+1}},\\ p_{k+1}=&\frac{k}{k+1}Re\Big[\Big(\frac{a_{k+1}}{a_1}-\frac{(k+1)a_1}{2}\Big)A^{k}\Big]+g(k,A,a_2,\ldots,a_k),
\end{align*}
where $g=g(k,A,a_2,\ldots,a_k)$ is a real-valued function which satisfies $g(k,A,0,\ldots,0)=0$.
\end{lemma}

\begin{proof}
By $d_n=|z_{n+1}-z_n|$,\ $n\in\mathbb{N}_0$, we denote the distances between two consecutive points of the orbit. The critical index $n_\varepsilon$ is then determined by the inequalities
\begin{equation}\label{neps}
d_{n_\varepsilon}<2\varepsilon,\ d_{n_\varepsilon-1}\geq 2\varepsilon.
\end{equation}

To obtain the asymptotic development of $n_\varepsilon$, we first compute asymptotic development for $d_n$, as $n\to\infty$. Using development \eqref{znasy} for $z_n$ from Proposition~\ref{orbit}, we get
\begin{equation}\label{znrasy}
\begin{split}
&z_{n+1}-z_n=a_1A^{k+1}n^{-1-\frac{1}{k}}+h_2n^{-1-\frac{2}{k}}+\ldots+h_k n^{-2}-\\
&\quad -\left[\Big(a_{k+1}-\frac{(k+1)a_1^2}{2}\Big)A^{2k+1}\frac{k+1}{k}+g(k,A,a_2,\ldots,a_{k})\right]n^{-2-\frac{1}{k}}\log n+\\
&\hspace{7.5cm}+o(n^{-2-\frac{1}{k}}\log n), \ n\to\infty,
\end{split}
\end{equation}
where $h_i=h_i(k,A,a_2,\ldots,a_i)$, $i=2,\ldots,k,$ and $g=g(k,A,a_2,\ldots,a_k)$ are complex-valued functions and $g(k,A,0,\ldots,0)=0$. Furthermore,
\begin{equation}\label{dn}
\begin{split}
&d_n=|a_1A^{k+1}|n^{-1-\frac{1}{k}}+q_2n^{-1-\frac{2}{k}}+\ldots+q_kn^{-2}-\\
&-\left[\frac{k+1}{k}|a_1A^{k+1}|Re\Big(\Big(\frac{a_{k+1}}{a_1}-\frac{(k+1)a_1}{2}\Big)A^{k}\Big)+r(k,A,a_2,..,a_{k})\right]n^{-2-\frac{1}{k}}\log n+\\
&\hspace{8cm}+o(n^{-2-\frac{1}{k}}\log n),\ n\to\infty,
\end{split}
\end{equation}
where $q_i=q_i(k,A,a_2,\ldots,a_i)$, $i=2,\ldots,k,$ and $r=r(k,A,a_2,\ldots,a_{k})$ are real-valued functions and $r(k,A,0,\ldots,0)=0$.

From \eqref{neps} and \eqref{dn} we deduce the asymptotic development of $n_{\varepsilon}$ as $\varepsilon\to 0$, iteratively, term by term.
\end{proof}
Note that the above proof provides developments \eqref{znrasy} and \eqref{dn} for $z_n-z_{n+1}$ and for the distances $d_n$ between two consecutive points, which we also need later.

\medskip

\begin{lemma}[Asymptotic development of the area of the nucleus]\label{asynucl}
Let $k\geq 1$. The following asymptotic development for the area of the nucleus of the $\varepsilon$-neighborhood of the orbit holds:
\begin{equation}\label{nucleus}
\begin{split}
A(N_\varepsilon)=\frac{2^{-\frac{k}{k+1}}\sqrt\pi}{k}&\left(\frac{\Gamma(\frac{1}{2k+2})}{\Gamma(\frac{3}{2}+\frac{1}{2k+2})}-\sqrt\pi\right)|a_1|^{-\frac{1}{k+1}}\cdot \varepsilon^{1+\frac{1}{k+1}}+h_2\varepsilon^{1+\frac{2}{k+1}}+\\
&+\ldots+h_{k}^{(1)} \varepsilon^{1+\frac{k}{k+1}}+h_{k}^{(2)} \varepsilon^2\log\varepsilon+o(\varepsilon^2\log\varepsilon),\ \varepsilon\to 0.
\end{split}
\end{equation}
Here, $h_i=h_i(k,A,a_2,\ldots,a_i)$, $i=2,\ldots,k$, are real-valued functions of $k$ and first $i$ coefficients of $f(z)$, such that $h_i(k,A,0,\ldots,0)=0$.
\end{lemma}

\begin{proof}
By Proposition~\ref{nhood}.$ii)$ in Subsection~\ref{threetwofive}, the area of the nucleus can be computed by adding areas of infinitely many crescent-shaped contributions. Furthermore, Proposition~\ref{crescent} provides the formula for computing such areas. We have
\begin{equation*}
A(N_\varepsilon)
=\varepsilon^2\pi+2\varepsilon^2 \sum_{n=n_{\varepsilon}}^{\infty}\left(\frac{d_n}{2\varepsilon}\sqrt{1-\frac{d_n^2}{4\varepsilon^2}}+\arcsin{\frac{d_n}{2\varepsilon}}\right).
\end{equation*}
By Proposition~\ref{auxsum} in Subsection~\ref{threetwofive}, this sum can be replaced by the following integral:
\begin{equation}\label{int}
A(N_\varepsilon)= 2\varepsilon^2\int_{x=n_\varepsilon}^{\infty}\left(\frac{d(x)}{2\varepsilon}\sqrt{1-\frac{d(x)^2}{4\varepsilon^2}}+\arcsin{\frac{d(x)}{2\varepsilon}}\right) dx +\text{O}(\varepsilon^2),\ \varepsilon\to 0,
\end{equation}
where $d(x)$ is the strictly decreasing function from Proposition~\ref{auxsum}:\\ $$d(x)=q_1x^{-1-\frac{1}{k}}+q_2x^{-1-\frac{2}{k}}+\ldots+q_kx^{-2}+q_{k+1}x^{-2-\frac{1}{k}}\log x+Dx^{-2-\frac{1}{k}}.$$
\smallskip

We now compute the first $k+1$ terms in the asymptotic development of the integral from \eqref{int}, as $\varepsilon\to 0$. Applying the change of variables $t=\frac{d(x)}{2\varepsilon}$, we get
\begin{equation}\label{intch}
I=-2\varepsilon \int_{0}^{\frac{d(n_\varepsilon)}{2\varepsilon}}\left(t \sqrt{1-t^2}+\arcsin{t}\right) \frac{1}{d'(x(t))} dt.
\end{equation}

Here, $x(t)=d^{-1}(2\varepsilon t)$. Note that, for a given $\varepsilon$, $t$ is bounded in $[0,1)$. Therefore it holds that:
\begin{equation}\label{unif}(\varepsilon t)\to 0, \text{ as }\varepsilon\to 0,\text{ uniformly in }t.\end{equation}
The development of $x(t)=d^{-1}(2\varepsilon t)$, as $\varepsilon \to 0$, can be deduced using the already computed development for $n_\varepsilon=d^{-1}(2\varepsilon)$ in Lemma~\ref{asyneps}. We have that
\begin{equation}\label{ratio}
\begin{split}
&\frac{1}{d'(x(t))}=-\frac{k}{k+1}2^{-2+\frac{1}{k+1}}|a_1A^{k+1}|^{1-\frac{1}{k+1}}(\varepsilon t)^{-2+\frac{1}{k+1}}+p_2(\varepsilon t)^{-2+\frac{2}{k+1}}+\ldots+\\
&\ +p_k^{(1)} (\varepsilon t)^{-2+\frac{k}{k+1}}+p_k^{(2)}(\varepsilon t)^{-1}\log(\varepsilon t)+O\big((\varepsilon t)^{-1+\frac{1}{k+1}}\big),\quad t\in\left[0,\frac{d(n_\varepsilon)}{2\varepsilon}\right),\ \varepsilon \to 0,
\end{split}\end{equation}
where $p_i=p_i(k,A,a_2,\ldots,a_i)$, $i=2,\ldots,k$, are real-valued functions.

Using Proposition~\ref{auxi2} in Subsection~\ref{threetwofive}, we remove $\varepsilon$ from the boundary of $I$. The integral in \eqref{intch} is equal to
\begin{equation}\label{inte2}
I=-2\varepsilon\int_{0}^{1}(t\sqrt{1-t^2}+\arcsin{t})\frac{1}{d'(x(t))} dt +o(\log\varepsilon), \ \varepsilon\to 0.
\end{equation}

Now, substituting the development \eqref{ratio} in \eqref{inte2} and using \eqref{unif} to evaluate the last term, we get
\begin{equation}\label{enddev}
\begin{split}
I&\ =\left(\frac{2}{|a_1A^{k+1}|}\right)^{-2+\frac{1}{k+1}}\frac{k}{k+1}\cdot T_1\cdot \varepsilon^{-1+\frac{1}{k+1}}-2p_2\cdot T_2\cdot \varepsilon^{-1+\frac{2}{k+1}}-\\[0.2cm]
&\quad -\ldots-2p_k^{(1)}\cdot T_k\cdot\varepsilon^{-1+\frac{k}{k+1}}-2p_k^{(2)}\cdot S_{k+1}\cdot \log\varepsilon+o(\log\varepsilon),\ \varepsilon\to 0.
\end{split}
\end{equation}
Here, functions $p_i$ are real-valued functions from the development \eqref{ratio} and $S_{k+1}$ and $T_i$, $i=1,\ldots,k,$ are the following finite integrals:
\begin{align*}
T_i=&\int_{0}^{1}(t\sqrt{1-t^2}+\arcsin t)t^{-2+\frac{i}{k+1}}\ dt, \ i=1,\ldots,k,\\
S_{k+1}=&\int_{0}^{1}(t\sqrt{1-t^2}+\arcsin t)t^{-1}\log t\ dt.
\end{align*}
Since $$T_1=\frac{(k+1)\sqrt \pi}{2k}\left(\frac{\Gamma(\frac{1}{2k+2})}{\Gamma(\frac{3}{2}+\frac{1}{2k+2})}-\sqrt\pi\right),$$ combining \eqref{int}, \eqref{inte2} and \eqref{enddev}, we get the development \eqref{nucleus} for $A(N_\varepsilon)$.
\end{proof}

\begin{lemma}[Asymptotic development of the area of the tail]\label{asytail}
Let $k\geq 1$. The area of the tail of the $\varepsilon$-neighborhood of the orbit has the following asymptotic development:
\begin{equation*}
\begin{split}
&A(T_{\varepsilon})=\pi\Big(\frac{2}{|a_1A^{k+1}|}\Big)^{-1+\frac{1}{k+1}}\varepsilon^{1+\frac{1}{k+1}}+f_2\varepsilon^{1+\frac{2}{k+1}}+\ldots+f_k\varepsilon^{1+\frac{k}{k+1}}+\\
&\ \ +\Big[\pi\frac{k}{k+1}Re\Big(\Big(\frac{a_{k+1}}{a_1}-\frac{(k+1)a_1}{2}\Big)A^{k}\Big)+g(k,A,a_2,\ldots,a_k)\Big]\varepsilon^2\log\varepsilon+\\
&\hspace{9cm}+\text{o}(\varepsilon^2 \log\varepsilon),\ \varepsilon\to 0.
\end{split}
\end{equation*}
Here, $f_i=f_i(k,A,a_2,\ldots,a_i)$, $i=2,\ldots,k$, $g=g(k,A,a_2,\ldots,a_k)$ are real-valued functions which depend only on $k$ and the first $i$ coefficients of $f(z)$, with the property $g(k,A,0,\ldots,0)=0,\ f_i(k,A,0,\ldots,0)=0.$
\end{lemma}

\begin{proof}
Since the tail, by definition, consists of $n_\varepsilon-1$ disjoint $\varepsilon$-discs, we have that $|T_\varepsilon|=(n_\varepsilon-1)\cdot \varepsilon^2\pi$. The statement follows from \eqref{devneps}.
\end{proof}

\begin{lemma}[Asymptotic development of the center of mass of the nucleus]\label{massnucl}
Let $k>1$. Let $t(N_\varepsilon)$ denote the center of mass of the nucleus of the $\varepsilon$-neighborhood. The following asymptotic development holds:
\begin{equation}\label{nuclcenter}
\begin{split}
t(N_\varepsilon)\cdot A(N_\varepsilon)&=q_1\varepsilon^{1+\frac{2}{k+1}}
+q_2\varepsilon^{1+\frac{3}{k+1}}+q_3\varepsilon^{1+\frac{4}{k+1}}+\ldots+q_{k}\varepsilon^{2}+\\
&\qquad\ \  +q_{k+1}\varepsilon ^{2+\frac{1}{k+1}}\log\varepsilon+o(\varepsilon^{2+\frac{1}{k+1}}\log\varepsilon),\ \varepsilon\to 0.
\end{split}
\end{equation}
Here, $q_i=q_i(k,A,a_2,\ldots,a_i)$, $i=2,\ldots,k,$ are complex-valued functions which depend on $k$ and on the first $i$ coefficients of $f(z)$ with the property $q_i(k,A,0,\ldots,0)=0$. Furthermore,
\begin{align*}
q_1&=\frac{k\sqrt\pi}{2(k-1)}\left(\frac{\Gamma(\frac{1}{k+1})}{\Gamma(\frac{3}{2}+\frac{1}{k+1})}-\sqrt\pi\right)\left(\frac{2}{|a_1A^{k+1}|}\right)^{-1+\frac{2}{k+1}}\cdot A,\\
q_{k+1}&=-\frac{k\sqrt\pi}{2(k+1)^2}\left(\sqrt\pi-\frac{\Gamma(\frac{1}{2}+\frac{1}{2k+2})}{\Gamma(2+\frac{1}{2k+2})}\right)\left(\frac{2}{|a_1A^{k+1}|}\right)^{\frac{1}{k+1}}\cdot\nonumber\\
&\hspace{4cm}\cdot Im\big(\frac{a_{k+1}}{a_1^2}-\frac{k+1}{2}\big)\cdot A\cdot i+g(k,A,a_2,\ldots,a_k),
\end{align*}
where $g=g(k,A,a_2,\ldots,a_k)$ is a complex-valued function satisfying $g(k,A,0,\ldots,0)=0$.
\end{lemma}

\begin{proof}
By definition of the centre of mass and by Propositions~\ref{nhood}.$ii)$ and \ref{crescent} in Subsection~\ref{threetwofive}, we have that
\begin{equation*}
\begin{split}
t&(N_\varepsilon)\cdot A(N_{\varepsilon})=z_{n_{\varepsilon}}\cdot\varepsilon^2\pi+\sum_{n=n_{\varepsilon}+1}^{\infty}A(D_n)t(D_n)=\\
&=z_{n_{\varepsilon}}\cdot\varepsilon^2\pi+2\varepsilon^2\sum_{n=n_{\varepsilon}+1}^{\infty}\left(\frac{d_n}{2\varepsilon}\sqrt{1-\frac{d_n^2}{4\varepsilon^2}}+\arcsin{\frac{d_n}{2\varepsilon}}\right)z_{n}+\\
&\quad \qquad +\varepsilon^2\sum_{n=n_{\varepsilon}+1}^{\infty}\Big(\frac{d_n}{2\varepsilon}\sqrt{1-\frac{d_n^2}{4\varepsilon^2}}-\arcsin\sqrt{1-\frac{d_n^2}{4\varepsilon^2}}\Big)(z_n-z_{n+1}).
\end{split}
\end{equation*}
Here, $D_n$, $n\geq n_\varepsilon$, denote the contributions to the nucleus from $\varepsilon$-discs of the points $z_n$.

We first show that
\begin{equation}\label{cmass1}
t(N_\varepsilon)\cdot A(N_{\varepsilon})=2\varepsilon^2\sum_{n=n_{\varepsilon}+1}^{\infty}\left(\frac{d_n}{2\varepsilon}\sqrt{1-\frac{d_n^2}{4\varepsilon^2}}+\arcsin{\frac{d_n}{2\varepsilon}}\right)z_{n}+O(\varepsilon^{2+\frac{1}{k+1}}),
\end{equation}
as $\varepsilon\to 0$.
From \eqref{devneps} and \eqref{znasy}, $z_{n_{\varepsilon}}\cdot\varepsilon^2\pi=O(\varepsilon^{2+\frac{1}{k+1}})$, as $\varepsilon\to 0$. On the other hand,
by \eqref{znrasy}, we have that $z_n-z_{n+1}=O(n^{-\frac{k+1}{k}})$, as $n\to\infty$. Therefore, using boundedness of the term in parenthesis, integral approximation of the sum and then \eqref{devneps}, we get
\begin{equation*}
\begin{split}
&\Big|\varepsilon^2\sum_{n=n_{\varepsilon}+1}^{\infty}\Big(\frac{d_n}{2\varepsilon}\sqrt{1-\frac{d_n^2}{4\varepsilon^2}}-\arcsin\sqrt{1-\frac{d_n^2}{4\varepsilon^2}}\Big)(z_n-z_{n+1})\Big|\leq \\
&\hspace{4cm}\leq C_1\varepsilon^2\sum_{n=n_\varepsilon+1}^{\infty}n^{-\frac{k+1}{k}}\leq C_2\varepsilon^2 n_\varepsilon^{-\frac{1}{k}}\leq C \varepsilon^{2+\frac{1}{k+1}},
\end{split}
\end{equation*} for some constant $C>0$. This proves \eqref{cmass1}.

To compute the first $k+1$ terms in the asymptotic development of the sum in \eqref{cmass1},
\begin{equation*}
S=\sum_{n=n_{\varepsilon}+1}^{\infty}\left(\frac{d_n}{2\varepsilon}\sqrt{1-\frac{d_n^2}{4\varepsilon^2}}+\arcsin{\frac{d_n}{2\varepsilon}}\right)z_{n},
\end{equation*} as $\varepsilon\to 0$, we use the same idea as in Lemma~\ref{asynucl}. Therefore we omit the details. To make the integral approximation of the sum $S$, we have to cut off the formal developments $d_n$ and $z_n$ to finitely many terms. Let $d_n^*$ be as in Proposition~\ref{auxsum}, $d_n^*=J_{k+1} d_n+Dn^{-2-\frac{1}{k}}$. By $J_{k+1} z_n$, we denote the first $k+1$ terms in the asymptotic development of $z_n$. It can be shown similarly as before that
$$
S=\sum_{n=n_\varepsilon+1}^\infty \Big[\Big(\frac{d_n^*}{2\varepsilon}\sqrt{1-\frac{(d_n^*)^2}{4\varepsilon^2}}+\arcsin{\frac{ d_n^*}{2\varepsilon}}\Big) J_k z_n\Big]+o(\varepsilon^{\frac{1}{k+1}}\log\varepsilon).
$$
Since the real and the imaginary part of the function under the summation sign are strictly decreasing, as $n\to\infty$, we can make the integral approximation of the sum:
\begin{equation*}
S=\int_{n_\varepsilon}^{\infty}\Big(\frac{d(x)}{2\varepsilon}\sqrt{1-\frac{d(x)^2}{4\varepsilon^2}}+\arcsin{\frac{d(x)}{2\varepsilon}}\Big)z(x) dx+o(\varepsilon^{\frac{1}{k+1}}\log\varepsilon).
\end{equation*}
The function $d(x)$ is as defined in \eqref{dx} in proof of Proposition~\ref{orbit}, and $z(x)$ is equal to
\begin{equation}\label{xz}
z(x)=g_1x^{-\frac{1}{k}}+g_2 x^{-\frac{2}{k}}+g_3x^{-\frac{3}{k}}+g_4 x^{-\frac{4}{k}}+\ldots+g_k x^{-1}
+g_{k+1} x^{-\frac{k+1}{k}}\log x,
\end{equation}
with coefficients $g_i\in\mathbb{C}$ from the development \eqref{znasy} of $z_n$ in Proposition~\ref{orbit}.

By making the change of variables $t=\frac{d(x)}{2\varepsilon}$ in the integral, we get
\begin{equation}\label{int3}
S=-2\varepsilon \int_{0}^{1+O(\varepsilon^{1-\frac{1}{k+1}})}\Big(t\sqrt{1-t^2}+\arcsin t\Big)\frac{z(x(t))}{d'(x(t))} dt+o(\varepsilon^{\frac{1}{k+1}}\log\varepsilon),
\end{equation}
as $\varepsilon\to 0$.

Using  \eqref{ratio}, \eqref{xz} and the development for $x(t)$ from the proof of Lemma~\ref{asynucl}, after some computation we get the development for $\frac{z(x(t))}{d'(x(t)}$, as $\varepsilon t\to 0$.
Again, let us note that $\varepsilon t \to 0$ uniformly in $t$, as $\varepsilon\to 0$, see \eqref{unif} before. The coefficients of the development are again obtained evaluating the integrals
\begin{align}\label{kjedan}
I_s&=\int_{0}^{1}(t\sqrt{1-t^2}+\arcsin t)t^{-2+\frac{s}{k+1}}\ dt,\ \ \ s=2,\ldots,k+1,\nonumber\\
I_{k+2}&=\int_{0}^{1}(t\sqrt{1-t^2}+\arcsin t)t^{-1+\frac{1}{k+1}}\log t\ dt,\nonumber\\
I_{k+3}&=\int_{0}^{1}(t\sqrt{1-t^2}+\arcsin t)t^{-1+\frac{1}{k+1}}\ dt,
\end{align}
which are finite.

Substituting the development in \eqref{int3} and proceeding in a similar way as in Lemma~\ref{asynucl}, we get the development \eqref{nuclcenter}.
\end{proof}

\begin{lemma}[Development of the center of mass of the tail]\label{masstail}
Let $k>1$. The following development for the center of the mass of the tail of the $\varepsilon$-neighborhood holds:
\begin{equation}\label{tailcenter}
\begin{split}
&t(T_\varepsilon)A(T_\varepsilon)=\frac{k}{k-1}\pi\left(\frac{2}{|a_1A^{k+1}|}\right)^{-1+\frac{2}{k+1}}\cdot A\cdot \varepsilon^{1+\frac{2}{k+1}}
+g_2\varepsilon^{1+\frac{3}{k+1}}+\\
&\hspace{4cm}+\ldots
+g_{k-1}\varepsilon^{1+\frac{k}{k+1}}+g_k \varepsilon^2\log\varepsilon+S^f(z_0)\varepsilon^2-\\
&-\left[\frac{\pi}{k+1}\left(\frac{2}{|a_1A^{k+1}|}\right)^{\frac{1}{k+1}}Im\big(\frac{a_{k+1}}{a_1^2}-\frac{k+1}{2}\big)\cdot i\cdot A+h(k,A,a_2,\ldots,a_{k})\right]\cdot\\
&\hspace{7cm}\cdot\varepsilon^{2+\frac{1}{k+1}}\log\varepsilon+o(\varepsilon^{2+\frac{1}{k+1}}\log\varepsilon),\ \varepsilon\to 0.
\end{split}
\end{equation}
Here, $g_i=g_i(k,A,a_2,\ldots,a_i)$, $i=2,\ldots, k$, are complex-valued functions of $k,\ A$ and first $i$ coefficients of $f(z)$, such that $g_i=g_i(k,A,0,\ldots,0)$. The function $S^f(z_0)$ is a complex-valued function of the initial point $z_0$ which depends on the whole diffeomorphism $f$. The function $h=h(k,A,a_2,\ldots,a_k)$ is a complex-valued function which satisfies $h(k,A,0,\ldots,0)=0$.\end{lemma}
\begin{proof}
\begin{equation*}
\begin{split}
t(T_\varepsilon)\cdot A(T_{\varepsilon}&)=\frac{\sum_{n=1}^{n_\varepsilon-1}z_n\cdot \varepsilon^2\pi}{A(T_\varepsilon)}\cdot A(T_\varepsilon)=\varepsilon^2\pi\sum_{n=1}^{n_\varepsilon-1}z_n=\\
&=\varepsilon^2\pi(z_0+\ldots z_{n(f,z_0)})+\varepsilon^2\pi\cdot\sum_{n=n(f,z_0)}^{n_\varepsilon-1}z_n.
\end{split}
\end{equation*}
Here, $n(f,z_0)$ is chosen to be the first index, obviously depending on the diffeomorphism $f$ and on the initial condition $z_0$, such that  $$z_n=J_{k+1}z_n+R(n),\text{ where }|R(n)|\leq Cn^{-1-\frac{1}{k}},\text{ for }n\geq n(f,z_0),$$ for some constant $C>0$. Then
\begin{equation}\label{sumsum}
\begin{split}
\sum_{n=n(f,z_0)}^{n_\varepsilon-1}z_n=g_1\sum_{n=n(f,z_0)}^{n_\varepsilon-1}n^{-\frac{1}{k}}  +&g_2\sum_{n=n(f,z_0)}^{n_\varepsilon-1}n^{-\frac{2}{k}}+\ldots   +g_k\sum_{n=n(f,z_0)}^{n_\varepsilon-1}n^{-1}+\\
&\quad +g_{k+1}\sum_{n=n(f,z_0)}^{n_\varepsilon-1}n^{-1-\frac{1}{k}}\log n+ \sum_{n=n(f,z_0)}^{n_\varepsilon-1}R(n),
\end{split}
\end{equation}
where complex numbers $g_i$ are as in the development of $z_n$, see Proposition~\ref{orbit}.

We now compute the first $k+1$ terms in the asymptotic developments of \eqref{sumsum}, as $n_\varepsilon\to\infty$.

Firstly, we concentrate on the last sum in \eqref{sumsum}. We show that
\begin{equation}\label{residu}
\sum_{l=n(z_0,f)}^n R(l)=C(z_0,f)+O(n^{-\frac{1}{k}}), \ n\to\infty,
\end{equation}
where $R(l)=O(l^{-1-\frac{1}{k}})$, as $l\to\infty$, and $C(z_0,f)$ is a complex constant depending on the diffeomorphism and on the initial condition. From the asymptotics of $R(l)$, the sum $\sum_{l=n(z_0,f)}^\infty R(l)$ is obviously convergent and equal to some constant $C(z_0,f)\in\mathbb{C}$. We write
\begin{equation*}
\sum_{l=n(z_0,f)}^n R(l)=\sum_{l=n(z_0,f)}^\infty R(l)-\sum_{l=n}^\infty R(l)=C(z_0,f)+O(n^{-\frac{1}{k}}),\ n\to\infty,
\end{equation*} where the second sum is evaluated as $O(n^{-1/k})$ by integral approximation of the sum.

Secondly, we estimate first three terms in the asymptotic developments of the first $k+1$ sums in \eqref{sumsum}, as $n_\varepsilon\to\infty$.
We show the procedure on the first sum. Let $$F(n)=\sum_{l=n(f,z_0)}^{n}l^{-\frac{1}{k}}.$$
Obviously, it satisfies the recurrence relation $$F(n+1)-F(n)=(n+1)^{-\frac{1}{k}}, \ n\in\mathbb{N},$$ with initial condition $F(n(f,z_0))=n(f,z_0)^{-\frac{1}{k}}$. We determine the first term in its development by integral approximation: $$F(n)=\frac{k}{k-1}n^{\frac{k-1}{k}}+R(n),$$ where $R(n)=o(n^\frac{k-1}{k})$ as $n\to\infty$. Using this development, from recurrence relation for $F(n)$ we get the recurrence relation for $R(n)$ and the initial condition $R(n(z_0,f))$. By recursion, we get $$R(n)=R(n(z_0,f))+\sum_{l=n(z_0,f)}^n O(l^{-1-\frac{1}{k}}).$$
Using \eqref{residu}, we conclude that $\sum_{l=n(f,z_0)}^{n}l^{-\frac{1}{k}}=C(n(z_0,f))+O(n^{-1/k})$, $n\to\infty$. The same procedure can be repeated for other sums.

Thus we obtain the development of the sum $\sum_{n=n(z_0,f)}^{n_\varepsilon-1}z_n$, as $n_\varepsilon\to\infty$. Substituting $n_\varepsilon$ with the development \eqref{devneps}, we get the development \eqref{tailcenter}, as $\varepsilon\to 0$.
\end{proof}

\noindent\emph{Proof of Theorem~\ref{asy}.} The proof follows from Lemmas \ref{asyneps} to \ref{masstail}, as described at the beginning of the section.\qed
\medskip

\noindent\emph{Proof of Proposition~\ref{asy1} $($the case $k=1$$)$.} Note that Lemmas~\ref{asynucl} and \ref{asytail} are true also when $k=1$ and can be directly applied to compute the coefficients. We get the development for the area $A(S^f(z_0))$.
In Lemma~\ref{massnucl}, the formula for the first coefficient $q_1$ is different. It is obtained from the integral $I_2=\int_{0}^{1}(t\sqrt{1-t^2}+\arcsin t)t^{-1}\ dt$, given by a different formula than the other integrals in \eqref{kjedan}. We get:
\begin{align*}
&A(N_\varepsilon)t(N_\varepsilon)=-\frac{1}{2a_1}\cdot \frac{\pi}{4}(1 + \log 4)\cdot \varepsilon^{2}-\\
&\qquad -\frac{\sqrt \pi}{8}(2|a_1|)^{\frac{1}{2}}\left(\sqrt\pi-\frac{\Gamma(3/4)}{\Gamma(5/4)}\right)\cdot \frac{1}{a_1}\cdot i \cdot Im\big(1-\frac{a}{a_1^2}\big)\varepsilon^{\frac{5}{2}}\cdot\log\varepsilon+o(\varepsilon^{\frac{5}{2}}\log\varepsilon),\ \varepsilon\to 0.
\end{align*}
In Lemma~\ref{masstail}, repeating the procedure in the case $k=1$, we see that the development begins with a logarithmic term instead of a power term:
\begin{align*}
A(T_\varepsilon)t(T_\varepsilon)&=\frac{\pi}{2a_1}\varepsilon^2\log\varepsilon+\\
&\quad \ +S^f(z_0)\varepsilon^2-\pi\left(\frac{|a_1|}{2}\right)^{1/2}\frac{1}{a_1}\cdot i\cdot Im\left(1-\frac{a}{a_1^2}\right)\varepsilon^{\frac{5}{2}}\log\varepsilon+o(\varepsilon^{\frac{5}{2}}\log\varepsilon),\ \varepsilon\to 0.
\end{align*}\qed

\subsection{Bijective correspondence between formal invariants and fractal properties of orbits}\label{threetwofour}
We now relate the relevant coefficients in the development of the directed area of the $\varepsilon$-neighborhoods of orbits with fractal properties of orbits (the box dimension, the Minkowski content) and their generalizations defined here (the \emph{directed Minkowski content} and the \emph{residual content}). In this section we suppose that $k>1$, so that the developments from Theorem~\ref{asy} hold.

The directed Minkowski content of a set is a complex generalization of the standard Minkowski content. After introducing the directed area of the $\varepsilon$-neighborhoods of sets in Definition~\ref{dir}, the definition of the directed Minkowski content follows in the natural manner. The Minkowski content of a bounded set $U\subset\mathbb{C}$ is by definition equal to the the first coefficient in the asymptotic development of the area of the $\varepsilon$-neighborhood of U, if such development exists. We define the \emph{directed Minkowski content} analogously, but using the directed area of the $\varepsilon$-neighborhood instead.

\begin{definition}[Directed Minkowski content of a measurable set] \label{ccontent}
Let $U\subset\mathbb{C}$ be a bounded set with measurable $\varepsilon$-neighborhoods $U_\varepsilon$. Let the centers of mass of $\varepsilon$-neighborhoods be different from the origin. Let $d=\dim_{\text{\it B}}U$. If the limit exists, we call the complex number
\begin{equation*}
\mathcal{M}^{\mathbb{C}}(U)=\lim_{\varepsilon\to 0}\frac{A^{\mathbb{C}}(U_\varepsilon)}{\varepsilon^{2-d}}
\end{equation*}
the \emph{directed Minkowski content of $U$.} 
\end{definition}

\noindent By definition, $|\mathcal{M}^{\mathbb{C}}(U)|=\mathcal{M}(U)$, where $\mathcal{M}(U)$ is the Minkowski content of $U$. Therefore, $\mathcal{M}^{\mathbb{C}}(U)$ is a natural generalization of the Minkowski content $\mathcal{M}(U)$.

\medskip
Let $S^f(z_0)$ be an attracting orbit of a parabolic diffeomorphism. From development \eqref{comarea}, it holds that $$A(S^f(z_0)_\varepsilon)=|A^{\mathbb{C}}(S^f(z_0)_\varepsilon)|=|K_1|\varepsilon^{1+\frac{1}{k+1}}+o(\varepsilon^{1+\frac{1}{k+1}}),\ \varepsilon\to 0.$$ Therefore, we have that any orbit $S^f(z_0)$ is Minkowski measurable, with:
\begin{equation}\label{fra1}\dim_\text{\it B}(S^f(z_0))=1-\frac{1}{k+1},\quad \mathcal{M}(S^f(z_0))=|K_1|, \quad \mathcal{M}^{\mathbb{C}}(S^f(z_0))=K_1.\end{equation}

Motivated by the fact that the first coefficient of \eqref{comarea} incorporates the directed Minkowski content of the orbit, we define the \emph{directed residual content} of the orbit as the coefficient in front of the special logarithmic term, $\varepsilon^2\log\varepsilon$.
\begin{definition}[Directed residual content]
The \emph{directed residual content} $\mathcal{R}^{\mathbb{C}}(S^f(z_0))$ of the orbit $S^f(z_0)$ is the complex number
\begin{equation}\label{fra2}
\mathcal{R}^{\mathbb{C}}(S^f(z_0))=K_{k+1},
\end{equation}
where $K_{k+1}$ is the coefficient in front of the logarithmic term $\varepsilon^2\log\varepsilon$ in the development \eqref{comarea}.
\end{definition}

\medskip
In Theorems~\ref{fnf} and \ref{fnfe} below, we state our two main results. First, the standard formal invariants $(k,\lambda)$ (see Proposition~\ref{snff}) of a given parabolic diffeomorphism can be deduced from fractal properties of only one of its orbits near the origin.
\begin{theorem}[Standard formal normal form and fractal properties of an orbit]\label{fnf}\
The standard formal type $(k,\ \lambda)$ of a parabolic diffeomorphism $f(z)$ is uniquely determined by $dim_B(S^f(z_0))$, $\mathcal{M}^\mathbb{C}(S^f(z_0))$ and $\mathcal{R}^\mathbb{C}(S^f(z_0))$ of any attracting orbit $S^f(z_0)$ near the origin.

\noindent Moreover, the following explicit formulas hold:
\begin{equation}\label{form}
\begin{split}
k\ =\ &\frac{\dim_B(S^f(z_0))}{1-\dim_B(S^f(z_0))},\\
\lambda\ =\ &2(k+1)\cdot i\cdot Re\left(\frac{\mathcal{R}^\mathbb{C}(S^f(z_0))}{\nu_{\mathcal{M}^\mathbb{C}(S^f(z_0))}}\right)+2\pi\cdot \phi(k)\cdot\mathcal{M}(S^f(z_0))\cdot Im\left(\frac{\mathcal{R}^\mathbb{C}(S^f(z_0))}{\nu_{\mathcal{M}^\mathbb{C}(S^f(z_0))}}\right).
\end{split}
\end{equation}
Here, $\nu_{\mathcal{M}^\mathbb{C}(S^f(z_0))}$ denotes the normalized directed Minkowski content and $\phi(k)$ is a function of $k$, explicitly given by
$$
\phi(k)=\frac{k(k+1)}{k-1}\cdot \frac{1}{\sqrt\pi}\cdot\frac{\frac{\Gamma(\frac{1}{k+1})}{\Gamma(\frac{3}{2}+\frac{1}{k+1})}+\sqrt{\pi}}{\frac{\Gamma(\frac{1}{2}+\frac{1}{2k+2})}{\Gamma(2+\frac{1}{2k+2})}-\sqrt{\pi}}\cdot\frac{\Gamma(1 + \frac{1}{2k+2})}{\Gamma(\frac{3}{2}+\frac{1}{2k+2})}.
$$
Here, $\Gamma$ denotes the gamma function.
\end{theorem}
The converse of Theorem~\ref{fnf} is not true. Formal changes of variables reducing a diffeomorphism to its standard formal normal form may affect Minkowski and residual content of the orbit. Aside from the box dimension, the \emph{diffeomorphisms from the same standard formal class do not necessarily share the same fractal properties}. By \eqref{fra1} and \eqref{fra2}, the (directed) Minkowski content and the directed residual content of orbits depend on the first coefficient $a_1$ of a diffeomorphism. This coefficient changes in the changes of variables which are not tangent to the identity. \medskip

Nevertheless, if we consider the extended formal normal form from Proposition~\ref{neuf} instead of standard formal normal form, then Theorem~\ref{fnf} takes a form of the stronger equivalence statement.
\begin{theorem}[Extended formal normal form and fractal properties of an orbit]\label{fnfe}
There exists a bijective correspondence between the following triples:
\begin{itemize}
\item[(i)] the extended formal type of a diffeomorphism, $(k,\ a_1,\ \lambda),$
\item[(ii)] $\big(\dim_\text{\it B}(S^f(z_0)),\ \mathcal{M}^\mathbb{C}(S^f(z_0)),\ \mathcal{R}^\mathbb{C}(S^f(z_0))\big),$
\end{itemize}
where $S^f(z_0)$ is any attracting orbit of a diffeomorphism.
The bijective correspondence is given by formulas \eqref{form} and the following formula for $a_1$:
\begin{equation}\label{a1}
a_1=\mathcal{M}^\mathbb{C}(S^f(z_0))^{-k}\cdot \frac{(-2)^{-k}}{\mathcal{M}(S^f(z_0))}\cdot \Big(\frac{k}{\sqrt \pi}\frac{\Gamma(\frac{3}{2}+\frac{1}{2k+2})}{\Gamma(\frac{1}{2k+2})}\Big)^{-(k+1)}.
\end{equation}
\end{theorem}

The converse in fact states that \emph{all the attracting orbits of all the diffeomorphisms of the same extended formal type share the same fractal properties}.
Actually, for the precise converse statement, we have to make the following remark about the sectorial dependence of fractal properties on the initial point of the orbit. Suppose that we know only the extended formal type of a diffeomorphism and we want to compute the directed Minkowski content and the directed residual content of any attracting orbit $S^f(z_0)$ of the diffeomorphism. The directed Minkowski content is given by reformulation of the formula \eqref{a1}:
$$
\mathcal{M}^\mathbb{C}(S^f(z_0))=\frac{k+1}{k}\cdot\sqrt\pi\cdot\frac{\Gamma(1 + \frac{1}{2k+2})}{\Gamma(\frac{3}{2}+\frac{1}{2k+2})}\bigg(\frac{2}{|a_1|}\bigg)^{1/(k+1)}\cdot \nu_A,
$$
where $\nu_A$ is the attracting direction in whose attracting sector $z_0$ lies. Therefore, the fractal properties do differ slightly in argument for the orbits in $k$ different attracting sectors, but they do not differ for the orbits inside one attracting sector. Their modules, in particular the Minkowski content, are the same in all sectors.

\begin{remark}
By \eqref{fra1} and \eqref{fra2}, we can as well express the correspondence in Theorem~\ref{fnf} and Theorem~\ref{fnfe} in terms of coefficients $K_1$, $K_{k+1}$ and the exponent $k$ in the asymptotic development of $A^\mathbb{C}(S^f(z_0)_\varepsilon)$, instead in terms of fractal properties of the orbit $S^f(z_0)$.
\end{remark}
\medskip

\subsubsection{Proofs of Theorems~\ref{fnf} and \ref{fnfe}}

In the proofs of Theorems~\ref{fnf} and \ref{fnfe}, we need the following lemma. It shows that the leading exponent and the relevant first and $(k+1)$-st coefficient in the development of the directed area remain unchanged by a change of variables tangent to the identity, transforming the diffeomorphism to its extended formal normal form.

\begin{lemma}[Invariance of fractal properties in the extended formal class]\label{inv}\ Let $f_1(z)$ and $f_2(z)$ be two germs of parabolic diffeomorphisms which belong to the same extended formal class $(k,a_1,\lambda)$. Then it holds:
\begin{align*}
\dim_{\text{\it B}}(S^{f_1}(w_0))&=\dim_\text{\it B}(S^{f_2}(z_0)),\\
\mathcal{M}^\mathbb{C}(S^{f_1}(w_0))&=\mathcal{M}^\mathbb{C}(S^{f_2}(z_0)),\\ \mathcal{R^\mathbb{C}}(S^{f_1}(w_0))&=\mathcal{R^\mathbb{C}}(S^{f_2}(z_0)).
\end{align*}
Here, $z_0$ and $w_0$ are any two initial points chosen from the attracting sectors of $f_1$ and $f_2$ with the same attracting direction.
\end{lemma}

In the proof of Lemma~\ref{inv} we need the following auxiliary lemma:
\begin{lemma}\label{auxinv}
Let $f(z)$ be a parabolic diffeomorphism and let $g(z)=\phi_l^{-1}\circ f\circ \phi_l(z)$, where $\phi_l(z)=z+cz^l$, $l\geq 2$. Let $S^f(z_0)=\{z_n\}$ be an attracting orbit of $f(z)$ and let $S^g(v_0)=\{w_n=\phi_l(z_n)\}$ be the corresponding attracting orbit of $g(z)$. Then it holds that
\begin{equation}\label{orb}
K_1^{S^f(z_0)}=K_1^{S^g(v_0)},\ K_{k+1}^{S^f(z_0)}=K_{k+1}^{S^g(v_0)},
\end{equation}
where $K_1$ and $K_{k+1}$ denote the first and the $(k+1)$-st coefficients in the asymptotic developments \eqref{comarea} of the directed areas of the $\varepsilon$-neighborhoods of the corresponding orbits.
Furthermore, the equalities \eqref{orb} hold also if $S^f(z_0)$ and $S^g(v_0)$ are any two orbits of $f(z)$ and $g(z)$ respectively which converge to the same attracting direction.
\end{lemma}

\begin{proof}
Let $\{z_n\}$ be an attracting orbit of $f(z)$. We first take $S^g(v_0)=\{w_n\}$ to be the image of $\{z_n\}$ under $\phi_l,\ l>1$.
Using development \eqref{znasy} for $z_n$, we compute the development of $w_n=\phi_l(z_n)$. It is easy to see that, since $l>1$, the first coefficient and the $(k+1)$-st coefficient remain the same as in $z_n$, while the other coefficients can change. In particular, the attracting direction $A$ for $S^g(v_0)$ remains the same as for $S^f(z_0)$.

On the other hand, it can be seen in proofs in Section~\ref{threetwothree} that only the first and the $(k+1)$-st coefficient of the development of $z_n$ participate in the first and the $(k+1)$-st coefficient of the developments of $z_n-z_{n+1}$, $d_n$ and $n_\varepsilon$. Finally, the first and the $(k+1)$-st coefficient in the development of $A^{\mathbb{C}}(S^f(z_0)_\varepsilon)$, $K_1$ and $K_{k+1}$, depend only on the first and the $(k+1)$-st coefficient in the development of $z_n$, not on other coefficients. Therefore, the two coefficients remain unchanged in the change of variables $\phi_l(z)$.	 Finally, since $K_1$ and $K_{k+1}$ do not depend on the choice of initial point $z_0$ and $v_0$ inside one sector, we can choose any two orbits of the initial and of the transformed diffeomorphism converging to the same attracting direction $A$.
\end{proof}

\noindent\emph{Proof of Lemma~\ref{inv}.}
For diffeomorphisms $f_1(z)$ and $f_2(z)$, let $\phi^{1,2}=\phi_{k}^{1,2}\circ \phi_{k-1}^{1,2}\circ\ldots\circ \phi_2^{1,2}$ denote the changes of variables obtained by composition of $k-1$ transformations of the above type, which present the first $k-1$ steps in transforming $f_1$ and $f_2$ to their extended formal normal forms. Let
\begin{align}\label{jets}
g_1=&(\phi^1)^{-1}\circ f_1 \circ \phi^1=z+a_1 z^{k+1}+a_1^2 \left(\frac{k+1}{2}-\frac{\lambda}{2\pi i}\right) z^{2k+1}+\ldots,\\
g_2=&(\phi^2)^{-1}\circ f_2 \circ \phi^2=z+a_1 z^{k+1}+a_1^2 \left(\frac{k+1}{2}-\frac{\lambda}{2\pi i}\right) z^{2k+1}+\ldots.\nonumber
\end{align}
Obviously, by Lemma~\ref{auxinv}, it holds that:
\begin{equation}\label{e1}
K_1^{g_1}=K_1^{f_1},\ K_1^{g_2}=K_1^{f_2} \text{ and } K_{k+1}^{g_1}=K_{k+1}^{f_1},\ K_{k+1}^{g_2}=K_{k+1}^{f_2},
\end{equation}
for the orbits corresponding to the same attracting direction. The notation $K_1^{f}$ is a bit imprecise, since the value differs for orbits in $k$ sectors, but we use it for simplicity and keep in mind that we always consider orbits converging to the same attracting direction.

\noindent Let $g_0$ be the extended formal normal form, $g_0(z)=z+a_1 z^{k+1}+a_1^2 a z^{2k+1}$. By further changes of variables, transforming $g_1$ and $g_2$ to the extended formal normal form $g_0$, the $(2k+1)$-jets from \eqref{jets} remain the same. Therefore we have, by the development \eqref{impcoef} in Theorem~\ref{asy}, that
\begin{equation}\label{e2}
K_1^{g_1}=K_1^{g_0},\ K_1^{g_2}=K_1^{g_0} \text{ and }  K_{k+1}^{g_1}=K_{k+1}^{g_0},\ K_{k+1}^{g_2}=K_{k+1}^{g_0},
\end{equation}
for the orbits corresponding to the same attracting direction. By \eqref{e1} and \eqref{e2}, it follows that $K_1^{f_1}=K_1^{f_2}$ and $K_{k+1}^{f_1}=K_{k+1}^{f_2}$, for the orbits of $f_1$ and $f_2$ converging to the same attracting direction.

Finally,  changes of variables do not change the multiplicity $k+1$ of the diffeomorphism. Therefore the leading exponent of the directed areas for all the orbits equals $1-\frac{1}{k+1}$.

Relating the coefficients $K_1$, $K_{k+1}$ and exponent $k$ with fractal properties of orbits, by \eqref{fra1} and \eqref{fra2}, the statement follows.
\qed

\smallskip
Note that the statement of the above Lemma is no longer true if we admit changes of variables which are not tangent to the identity. Only the box dimension is then preserved.
\bigskip

\noindent \emph{Proof of Theorem~\ref{fnfe}.}
Let $f(z)=z+a_1z^{k+1}+o(z^{k+1})$ be a parabolic germ and let $g_0(z)=z+a_1z^{k+1}+a_1^2\cdot \left(\frac{k+1}{2}-\frac{\lambda}{2\pi i}\right)\cdot z^{2k+1}$ be its extended formal normal form. Let $S^f(z_0)$ be an attracting orbit of $f(z)$ and let $S^{g_0}(w_0)$ be an attracting orbit of $g_0(z)$ with the same attracting direction.

The bijective correspondence between $k$ and $\dim_B(S^f(z_0))$ is obvious by \eqref{fra1}. Let $k$ then be fixed. Applying formulas \eqref{impcoef} from Theorem~\ref{asy} to the orbit of the formal normal form $g_0(z)$, we get the following formulas:
\begin{align}\label{impcoeffnf}
&K_1^{g_0}=\frac{k+1}{k}\cdot\sqrt\pi\cdot\frac{\Gamma(1 + \frac{1}{2k+2})}{\Gamma(\frac{3}{2}+\frac{1}{2k+2})}\bigg(\frac{2}{|a_1|}\bigg)^{1/(k+1)}\cdot \nu_A,\\
&K_{k+1}^{g_0}=\nu_A\cdot\Bigg[\frac{\pi}{k+1}Re\big(\frac{\lambda}{2\pi i}\big)-\Bigg.\nonumber\\
&\qquad \Bigg.- \bigg(\frac{2(k-1)}{k+1}\Big(\frac{|a_1|}{2}\Big)^{1/(k+1)}\frac{\frac{\Gamma(\frac{1}{2}+\frac{1}{2k+2})}{\Gamma(2+\frac{1}{2k+2})}-\sqrt{\pi}}{\frac{\Gamma(\frac{1}{k+1})}{\Gamma(\frac{3}{2}+\frac{1}{k+1})}+\sqrt{\pi}}\bigg)\cdot i\cdot Im\big(\frac{\lambda}{2\pi i}\big) \Bigg].\nonumber
\end{align}
By Lemma~\ref{inv},
\begin{equation}\label{co1}
K_1^f=K_1^{g_0},\ K_{k+1}^f=K_{k+1}^{g_0}.
\end{equation}
On the other hand, by \eqref{fra1} and \eqref{fra2},
\begin{equation}\label{co2}
K_1^{f}=\mathcal{M}^{\mathbb{C}}(S^f(z_0)),\ K_{k+1}^{f}=\mathcal{R}^{\mathbb{C}}(S^f(z_0)).
\end{equation}
Using \eqref{co1} and \eqref{co2}, we see that formulas \eqref{form} and \eqref{a1} in Theorem~\ref{fnfe} are just reformulations of \eqref{impcoeffnf}. They give, for a fixed $k$, the bijective correspondence between the pairs $(a_1,\lambda)$ and $(\mathcal{M}^{\mathbb{C}}(S^f(z_0)),\ \mathcal{R}^{\mathbb{C}}(S^f(z_0)))$.
\qed
\bigskip

\noindent \emph{Proof of Theorem~\ref{fnf}.} Let $f(z)$ and $g_0(z)$ be as in the above proof. The standard formal normal form $f_0(z)$ is given by $f_0(z)=z+z^k+\left(\frac{k+1}{2}-\frac{\lambda}{2\pi i}\right)z^{2k+1}$, where $\lambda$ is the same as in the extended form $g_0(z)$. The normal form $f_0(z)$ is obtained from $g_0(z)$ by making one extra change of variables of the type
$$
\phi(z)=a_1^{-1/k}z,
$$
in order to make coefficient $a_1$ equal to $1$. Since $\lambda$ and $k$ are the same as in $g_0(z)$, formulas \eqref{form} expressing $k$ and $\lambda$ from the fractal properties of the orbit $S^f(z_0)$ have already been obtained in the proof of Theorem~\ref{fnfe}. Therefore, the standard formal normal form of a diffeomorphism, described by the pair $(k,\lambda)$, can be deduced from fractal properties $\big(\dim_B(S^f(z_0)),\ \mathcal{M}^{\mathbb{C}}(S^f(z_0),\ \mathcal{R}^{\mathbb{C}}(S^f(z_0))\big)$ of just one orbit of the diffeomorphism.
\qed

\subsection{Proofs of auxiliary statements}\label{threetwofive}
Here we state auxiliary propositions that we need in the proof of Theorem~\ref{asy}. First let us recall (without proof) a standard, well-known result about integral approximation of the sum that we use a few times. The proof consists in bounding the integral by left Riemann sum from below and by right Riemann sum from above.
\begin{proposition}[Integral approximation of the sum]\label{inte}
Suppose $f(x)$ is monotonically decreasing, continuous function on the interval $[m-1,n]$, $m, n\in\{\mathbb{N}\cup\infty\},\ m<n$. Then the following inequality holds:
$$
\int_{m}^{n}f(x)dx\leq \sum_{k=m}^{n}f(k)\leq \int_{m-1}^{n}f(x) dx.
$$
\end{proposition}

\medskip
The next two propositions give the tool for computing areas and centers of mass of $\varepsilon$-neighborhoods of orbits. Let $f(z)$ be a parabolic diffeomorphism and $S^f(z_0)=\{z_n,\ n\in\mathbb{N}_0\}$ its attracting orbit. Let $K(z_i,\varepsilon)$ denote the $\varepsilon$-disc centered at $z_i$.  We represent the $\varepsilon$-neighborhood of $S^f(z_0)$ as
\begin{equation*}
S^f(z_0)_\varepsilon=\bigcup_{i=0}^\infty D_i.
\end{equation*}
Here, $D_0=K(z_0,\varepsilon)$ and $D_i=K(z_i,\varepsilon)\backslash \bigcup_{j=0}^{i-1} K(z_j,\varepsilon)$, $i\in\mathbb{N},$ are contributions from $\varepsilon$-discs of points $z_i$. 

\begin{proposition}[Geometry of $\varepsilon$-neighborhoods of orbits]\ \label{nhood}
\begin{itemize}
\item[(i)] Distances between two consecutive points of the orbit, $|z_{n+1}-z_n|$, are, starting from some $n_0$, strictly decreasing as $n\to \infty$ .
\item[(ii)] For small enough $\varepsilon>0$,
$$
K(z_i,\varepsilon)\backslash \bigcup_{j=0}^{i-1}K(z_j,\varepsilon)=K(z_i,\varepsilon) \backslash K(z_{i-1},\varepsilon),\ i\in\mathbb{N}.
$$
\end{itemize}
\end{proposition}
Proposition~\ref{nhood}.$(ii)$ means that all contributions $D_i$ at point $z_i$ are in crescent or full-disc form, determined only by the distance to the previous point $z_{i-1}$. The positions of the points before the previous, $z_0,\ldots,z_{i-2}$, do not affect the shape of $D_i$, see Figure~\ref{admiss}.
\begin{figure}[htp]
\begin{center}
\vspace{-44cm}
  % replace aims_logo.eps by your figure file name
  \includegraphics[scale=1.6,trim={-2cm 0cm 0cm 0cm}]{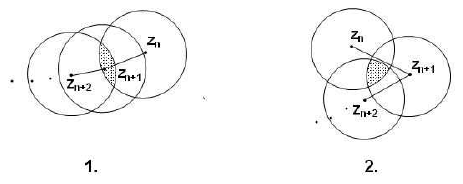}\\
  \caption{{\small 1. Admissible position of discs,\ 2. Nonadmissible position of discs.}}\label{admiss}
  \end{center}
\end{figure}

\begin{proof}
\emph{(i)}
Let us denote by $w_n=z_n-z_{n+1}-(z_{n+1}-z_{n+2})$. Using development \eqref{znrasy}, we compute: $$w_n=A\frac{k+1}{k^2}n^{-\frac{2k+1}{k}}+o(n^{-\frac{2k+1}{k}}),\ \ z_{n+1}-z_{n+2}=\frac{A}{k}n^{-\frac{k+1}{k}}+o(n^{-\frac{k+1}{k}}).$$  Obviously, in the limit as $n\to\infty$, the arguments of $w_n$ and $z_{n+1}-z_{n+2}$ are both equal to $Arg(A)$. For $n$ big enough, the value of the nonordered angle between $z_{n+1}-z_{n+2}$ and $w_n$ is therefore less than $\frac{\pi}{2}$. Since $z_n-z_{n+1}=(z_{n+1}-z_{n+2})+w_n$, it follows that $|z_n-z_{n+1}|>|z_{n+1}-z_{n+2}|$, for $n$ big enough.
\smallskip

\emph{(ii)} Let $T_n$ denote the midpoint and $s_n$ the bisector of the segment $[z_{n+1},z_n]$, $n\in\mathbb{N}$. It will suffice to show that there exists $\varepsilon>0$ such that for every $n\geq n_0$ and for every $k\in\mathbb{N}$, the distance from the intersection of $s_n$ and $s_{n+k}$, denoted $S_{n,k}$, to the midpoint $T_n$ is greater than $\varepsilon$. In this way we ensure that the union of intersections of $\varepsilon$-disc of each new point of the orbit with the $\varepsilon$-discs of all the previous points is a subset of the intersection with the $\varepsilon$-disc of the previous point only.

We first show that the two consecutive bisectors $s_n$ and $s_{n+1}$ intersect at the distance from $T_n$ which is bounded from below by a positive constant, as $n\to\infty$. \\
\noindent The bisector $s_n$ can obviously be parametrized as follows
\begin{equation*}
T_n+t\cdot i(z_n-z_{n+1})=\frac{z_n+z_{n+1}}{2}+t\cdot i(z_n-z_{n+1}),\ t\in\mathbb{R}.
\end{equation*}
We denote by $t_n\in\mathbb{R}$ the parameter of the intersection $S_{n,1}$ of $s_n$ and $s_{n+1}$. The complex number $$\frac{z_n+z_{n+1}}{2}+t_n \cdot i(z_n-z_{n+1})-T_{n+1}=\frac{z_n-z_{n+2}}{2}+t_n\cdot i(z_n-z_{n+1})$$ is perpendicular to $z_{n+1}-z_{n+2}$. Therefore their scalar product, denoted by $(.|.)$, is equal to $0$, and we get:
{\small $$
t_n=-\frac{1}{2}\frac{(z_{n}-z_{n+1}|z_{n+1}-z_{n+2})+|z_{n+1}-z_{n+2}|^2}{Re(z_n-z_{n+1}) Im(z_{n+1}-z_{n+2})-Im(z_n-z_{n+1}) Re(z_{n+1}-z_{n+2})}.
%=&-\frac{1}{2}\cdot\frac{(z_{n}-z_{n+1}|z_{n+1}-z_{n+2})+|z_{n+1}-z_{n+2}|^2}{Re(z_n-z_{n+1})\cdot Im(z_{n+1}-z_{n+2})-Im(z_n-z_{n+1})\cdot Re(z_{n+1}-z_{n+2})}.\nonumber
$$}

\noindent Using development \eqref{znrasy},
after some computation, we get that the denominator is $O(n^{-\frac{3k+3}{k}})$, while
the numerator is $\frac{3|A|^2}{k^2}n^{-\frac{2k+2}{k}}+o(n^{-\frac{2k+2}{k}})$.  Therefore, $t_n\geq Cn^{\frac{k+1}{k}}$, for some positive constant $C>0$ and $n>n_0$. Since $|z_n-z_{n+1}|\simeq n^{-\frac{k+1}{k}}$,
the distance $$d(T_n,S_{n,1})=|t_n|\cdot|z_n-z_{n+1}|$$ is bounded from below by some positive constant for $n\geq n_0$, say by $M>0$.

It is left to show that the same lower bound holds not only for consecutive, but for any two bisectors $s_n$ and $s_{n+k}$, $k\in\mathbb{N}$, $n\geq n_0$. We can see from the development \eqref{znrasy} that the points of the orbit approach the origin in the direction $A$.
We draw the stripe of width $M/2$ on both sides of that tangent direction. Obviously, for $n$ big enough, no two bisectors can intersect inside the stripe without two consecutive bisectors being intersected inside the stripe, which is a contradiction with the first part. Therefore, the distances from the midpoints to the intersections of corresponding bisectors when $n\to\infty$ are uniformly bounded from below by e.g. $M/4$.

Taking $\varepsilon<M/4$, we have proven the statement.
\end{proof}

%The next proposition shows how to compute the areas and vectors of the center of mass of the crescents $D_i$ from Proposition~\ref{nhood}.
\begin{proposition}\label{crescent}
Let $z,\ w\in \mathbb{C}$ $($or $\mathbb{R}^2$$)$, $\varepsilon>0$. Suppose $|z-w|<2\varepsilon$. Let $D$ denote the crescent $D=K(z,\varepsilon)\backslash K(w,\varepsilon)$. Then its area is equal to
\begin{equation*}
A(D)=2\varepsilon^2 \left(\frac{|z-w|}{2\varepsilon}\sqrt{1-\frac{|z-w|^2}{4\varepsilon^2}}+\arcsin{\frac{|z-w|}{2\varepsilon}}\right),
\end{equation*}
and its center of mass is equal to
\begin{equation*}
t(D)=z+\varepsilon^2(w-z)\frac{\frac{|z-w|}{2\varepsilon}\sqrt{1-\frac{|z-w|^2}{4\varepsilon^2}}-\arcsin\sqrt{1-\frac{|z-w|^2}{4\varepsilon^2}}}{A(D)}.
\end{equation*}
\end{proposition}

\begin{proof} The proposition is proved by integration, \begin{equation}\label{std}A(D)=\iint_{D} dx\ \!dy,\  t(D)=\frac{1}{A(D)}\left(\iint_{D} x\ \!dx\ \!dy\ + i\cdot\iint_{D} y\ \!dx\ \!dy\right).\end{equation} For simplicity, we put $d=|z-w|$.
\smallskip

$(i)$ \emph{The area.} We choose the coordinate system such that $z$ is the origin and $zw$-line is the real line, oriented so that $w$ is on the positive half-line. The area is equal to $A(D)=\iint_{D} dx dy$, which is computed as follows:
\begin{align*}
A(D)=&2\int_0^{\varepsilon} dv \int_{-\sqrt{\varepsilon^2-v^2}}^{\min\{d-\sqrt{\varepsilon^2-v^2},\sqrt{\varepsilon^2-v^2}\}}du
=2\int_0^{\sqrt{\varepsilon^2-\frac{d^2}{4}}}d\ dv+2\int_{\sqrt{\varepsilon^2-\frac{d^2}{4}}}^{\varepsilon}2\sqrt{\varepsilon^2-v^2} dv.
\end{align*}  
The second integral is computed making the change of variables $y=\varepsilon \sin t$, and the result follows.
\smallskip

$(ii)$ \emph{The center of mass.} We chose the same coordinate system as in $(i)$ and compute the centre of mass of $D$ in this new system. Obviously, due to the simetry of balls, in this new system the vector points at the real negative half-line, at the point that we denote $t<0$ (which represents the shift in the centre of mass from $z$ when we cut off part of the ball). The vector is then equal to $(t,0)$.

We can compute $t$ by the standard integral formula for the center of mass \eqref{std} in the new, simpler system. The formula can be found in e.g. \cite{pratap}.
\begin{align}\label{compu}
t=&\frac{1}{A(D)}\iint_{D} u\ du\ dv=\frac{2}{A(D)}\int_0^{\varepsilon} dv \int_{-\sqrt{\varepsilon^2-y^2}}^{\min\{d-\sqrt{\varepsilon^2-y^2},\sqrt{\varepsilon^2-y^2}\}}u\ du=\nonumber\\
&=\frac{2}{A(D)}\int_0^{\sqrt{\varepsilon^2-\frac{d^2}{4}}} dv \int_{-\sqrt{\varepsilon^2-y^2}}^{d-\sqrt{\varepsilon^2-y^2}}u\ du +\frac{2}{A(D)}\int_{\sqrt{\varepsilon^2-\frac{d^2}{4}}}^{\varepsilon} dv \int_{-\sqrt{\varepsilon^2-v^2}}^{\sqrt{\varepsilon^2-v^2}} u\ du=\nonumber\\
&=\frac{1}{A(D)}\int_0^{\sqrt{\varepsilon^2-\frac{d^2}{4}}}d(d-2\sqrt{\varepsilon^2-v^2})\ dv +0=\nonumber\\
&=\Big|y=\varepsilon\sin t\Big|=\frac{1}{A(D)}\left(\frac{d^2\varepsilon}{2}\sqrt{1-\frac{d^2}{4\varepsilon^2}}-d\varepsilon^2\arcsin\sqrt{1-\frac{d^2}{4\varepsilon^2}}\right).
\end{align}

In the original coordinate system, the vector of the centre of mass of the crescent $D$ is equal to
\begin{equation}\label{mcc}
t(D)=z+t\frac{w-z}{d}.
\end{equation} From \eqref{compu} and \eqref{mcc}, we get the result.
\end{proof}
\medskip

The following two propositions are auxiliary results in the proof of Lemma~\ref{asynucl}.
\begin{proposition}\label{auxsum}
The sum
\begin{equation}\label{ssum}
\sum_{n=n_{\varepsilon}}^{\infty}\left(\frac{d_n}{2\varepsilon}\sqrt{1-\frac{d_n^2}{4\varepsilon^2}}+\arcsin{\frac{d_n}{2\varepsilon}}\right)
\end{equation}
can, as $\varepsilon\to 0$, be represented as the integral
$$ \int_{x=n_\varepsilon}^{\infty}\left(\frac{d(x)}{2\varepsilon}\sqrt{1-\frac{d(x)^2}{4\varepsilon^2}}+\arcsin{\frac{d(x)}{2\varepsilon}}\right) dx +O(1).\
$$
Here, $d(x)$ is given by
\begin{equation}\label{dx}
d(x)=q_1x^{-1-\frac{1}{k}}+q_2x^{-1-\frac{2}{k}}+\ldots+q_kx^{-2}+q_{k+1}x^{-2-\frac{1}{k}}\log x+Dx^{-2-\frac{1}{k}}.
\end{equation}
All the coefficients $q_i$ are the same as in development \eqref{dn} of $d_n$ and $D\in\mathbb{R}$ is some constant.
\end{proposition}

\begin{proof}
The idea is to apply integral approximation of the sum. The problem is that we only have formal asymptotic development of $d_n$. The idea is to cut off the formal asymptotic development at the $(k+1)$-st term, to get a continuous and decreasing function of $n$ under the summation sign. We show here that the cut-off remainder is in some sense small and contributes to the sum with no more than $O(1)$, as $\varepsilon\to 0$.

We denote by $J_{k+1}d_n$ the first $k+1$ terms in the asymptotic expansion~\eqref{dn}.
For the sum with truncated $d_n$,
\begin{equation*}
\sum_{n=n_{\varepsilon}}^{\infty}\left(\frac{J_{k+1} d_n}{2\varepsilon}\sqrt{1-\frac{(J_{k+1}d_n)^2}{4\varepsilon^2}}+\arcsin{\frac{J_{k+1}d_n}{2\varepsilon}}\right),
\end{equation*}
to be well-defined, we have to ensure that $0<J_{k+1}d_n<2\varepsilon$ for $n\geq n_\varepsilon$. Since $d_n<2\varepsilon$ for $n\geq n_\varepsilon$ by \eqref{neps}, it is enough to achieve that $J_{k+1}d_n<d_n$, for $n\geq n_\varepsilon$, where $\varepsilon$ is sufficiently small. This is obtained by adding the term $Dn^{-2-\frac{1}{k}}$ to $J_{k+1}d_n$. Here, $D$ is chosen negative and sufficiently big by absolute value. We denote $d_n^*=J_{k+1}d_n+Dn^{-2-\frac{1}{k}}$. Obviously,
\begin{equation}\label{manip1}
d_n=d_n^*+O(n^{-2-\frac{1}{k}}).
\end{equation}
Let us denote the function under the summation sign in \eqref{ssum} by $h(x)$: $$h(x)=\frac{x}{2\varepsilon}\sqrt{1-\frac{x^2}{4\varepsilon^2}}+\arcsin{\frac{x}{2\varepsilon}}.$$ Then, $h'(x)=\frac{1}{\varepsilon}\sqrt{1-\big(\frac{x}{2\varepsilon}\big)^2}$. By \eqref{manip1} and by the mean value theorem, \begin{equation}\label{mvt}h(d_n)=h(d_n^*)+h'(\xi_n)\cdot O(n^{-2-\frac{1}{k}}),\ \xi_n\in[d_n^*,d_n].\end{equation}
Furthermore, \begin{equation}\label{h}
0<h'(\xi_n)<\frac{1}{\varepsilon}, \ n\geq n_\varepsilon.
\end{equation}
The initial sum \eqref{ssum} can, by \eqref{mvt}, be evaluated as follows:
\begin{equation}\label{ssum1}
S=\sum_{n=n_\varepsilon}^{\infty}h(d_n)=\sum_{n=n_\varepsilon}^\infty h(d_n^*)+\sum_{n=n_\varepsilon}^{\infty}h'(\xi_n)O(n^{-2-\frac{1}{k}}).
\end{equation}
By \eqref{h} and Lemma~\ref{asyneps}, using integral approximation of the sum, we get
\begin{equation}\label{st1}
|\sum_{n=n_\varepsilon}^{\infty}h'(\xi_n)
\cdot O(n^{-2-\frac{1}{k}})|<\frac{C_1}{\varepsilon}\sum_{n=n_\varepsilon}^{\infty}n^{-2-\frac{1}{k}}<\frac{C_2}{\varepsilon}n_\varepsilon^{-1-\frac{1}{k}}<C,
\end{equation}
for some constant $C>0$, as $\varepsilon\to 0$.

\noindent Furthermore, using integral approximation of the sum and the fact that the subintegral function is bounded from above, we get
\begin{equation}\label{st2}
\sum_{n=n_\varepsilon}^\infty h(d_n^*)=\int_{x=n_\varepsilon}^{\infty}\left(\frac{d(x)}{2\varepsilon}\sqrt{1-\frac{d(x)^2}{4\varepsilon^2}}+\arcsin{\frac{d(x)}{2\varepsilon}}\right) dx+O(1),\ \varepsilon\to 0.
\end{equation}
Finally, by \eqref{ssum1}, \eqref{st1} and \eqref{st2}, the result follows.
\end{proof}

\begin{proposition}\label{auxi2}
The integral
$$
\int_{0}^{\frac{d(n_\varepsilon)}{2\varepsilon}}\left(t \sqrt{1-t^2}+\arcsin{t}\right) \frac{1}{d'(x(t))} dt,
$$
where $d(x),\ x(t)$ and $n_\varepsilon$ are as in Lemma~\ref{asynucl}, can, as $\varepsilon\to 0$, be represented as the integral
$$
\int_{0}^{1}\left(t \sqrt{1-t^2}+\arcsin{t}\right) \frac{1}{d'(x(t))} dt+o(\varepsilon^{-1}).
$$
\end{proposition}

\begin{proof}
We first show that the upper boundary $\frac{d(n_\varepsilon)}{2\varepsilon}$ in the integral is equal to
\begin{equation}\label{lower}
\frac{d(n_\varepsilon)}{2\varepsilon}=1+O(\varepsilon^{1-\frac{1}{k+1}}),\ \varepsilon\to 0.
\end{equation}
By \eqref{devneps} and \eqref{manip1},
\begin{equation}\label{n1}
d(n_\varepsilon)=d_{n_\varepsilon}^*=d_{n_\varepsilon}+O(\varepsilon^{2-\frac{1}{k+1}}).
\end{equation}
From \eqref{dn}, it can easily be seen that $d_{n+1}=d_n+O(n^{-2-\frac{1}{k}})$, thus by \eqref{devneps} and \eqref{neps}, we get
\begin{equation}\label{n2}
d_{n_\varepsilon}=2\varepsilon+O(\varepsilon^{2-\frac{1}{k+1}}).
\end{equation}
Combining \eqref{n1} and \eqref{n2}, \eqref{lower} follows.

Using \eqref{lower}, the above integral $I=\int_{0}^{\frac{d(n_\varepsilon)}{2\varepsilon}}\left(t \sqrt{1-t^2}+\arcsin{t}\right) \frac{1}{d'(x(t))} dt$ can be written as the sum
\begin{equation*}
I=\int_{0}^{1}(t\sqrt{1-t^2}+\arcsin{t})\frac{1}{d'(x(t))} dt+\int_{1}^{1+O(\varepsilon^{1-\frac{1}{k+1}})}(t\sqrt{1-t^2}+\arcsin{t})\frac{1}{d'(x(t))} dt.
\end{equation*}
By \eqref{ratio}, $\frac{1}{d'(x(t))}=O((\varepsilon t)^{-2+\frac{1}{k+1}})$. It is then easy to see that the second integral equals $O(\varepsilon^{-1})$, as $\varepsilon\to 0$, due to the boundedness of the subintegral function in the neighborhood of $t=1$.
\end{proof}
\subsection{Perspectives}\label{threetwosix}

The concept of reconstructing a diffeomorphism from its one realization  is somewhat similar to the concept of the famous problem:\emph{ Can one hear a shape of a drum?}, presented by M. Kac in 1966. The question that is raised is if one can reconstruct the equation from only one solution, or, if not completely, how much can be said.

The vibrations of a drum are given by the Laplace equation with zero boundary condition on a given domain $\Omega$. The domain of the equation is the only unknown in the problem. The eigenvalues of the Laplace operator, $0<\lambda_1\leq \lambda_2\leq\ldots$, $\lambda_i\to\infty$,  present the frequencies. They are coefficients in the Fourier development of the solution. One tries to reconstruct the domain of the equation from these eigenvalues.

Let $N(\lambda)=\{\lambda_i:\lambda_i<\lambda\}$ be the eigenvalue counting function for the Laplace operator on $\Omega$. It was conjectured that from the asymptotic development of $N(\lambda)$, as $\lambda\to\infty$, one can obtain some properties of the domain:
\begin{hypothesis}[Modified Weyl-Berry conjecture, Conjecture 5.1 in \cite{lappom}]
If $\Omega\subset \mathbb{R}^N$ has a Minkowski measurable boundary $\Gamma$, with box dimension $d\in(N-1,N)$, then
$$
N(\lambda)=(2\pi)^{-N}\mathcal{B}_N\cdot A(\Omega)\lambda^{N/2}+c_{N,d}\ \cdot \mathcal{M}(\Gamma)\cdot \lambda^{\frac{d}{2}}+o(\lambda^{\frac{d}{2}}),\ \lambda\to\infty.
$$
\end{hypothesis}
Here, $\mathcal{B}_N$ is the volume of the unit ball in $\mathbb{R}^N$, $A(\Omega)$ the Lebesgue measure of the set $\Omega\in\mathbb{R}^N$ and $\mathcal{M}(\Gamma)$ the Minkowski content of the boundary. The constant $c_{N,d}$ is a real constant depending only on $N$ and $d$.
\smallskip

The conjecture was proven in the one-dimensional case, $N=1$, in Corollary 2.3 in \cite{lappom}. In other dimensions, it is still open.

Although we do not see a direct relation between this problem and the problem studied in this chapter, in many aspects they appear similar. The general idea of reconstructing the equation from one solution is common, as well as the fact that, for obtaining more information on the equation, we need to use further terms in appropriate developments.

\newpage

\section{Application of the formal classification result to formal orbital classification of complex saddles in $\mathbb{C}^2$}\label{threethree}

We consider holomorphic germs of vector fields in $\mathbb{C}^2$, with \emph{saddle point} at the origin. Their linear part is given by
$$
\omega=z\ dw+r w\ dz,\ r\in\mathbb{R}_+.
$$
Complex saddles belong to the reduced (elementary) singularities characterized by $r\in\mathbb{C}\setminus \mathbb{Q}_-^*$ (hyperbolic singularities, saddles, nodes, and saddle-nodes). For classification, see e.g.\cite[Chapter 5]{loray}. Complex saddles are complexifications of planar saddles, as well as of planar weak focus points. 

Complex saddles have two separatrices, corresponding to the complex axes. Complex planes vertical to the axes, called \emph{cross-sections}, are transversal to the 'flow'. The equivalent of the real Poincar\' e map on a transversal is the complex \emph{holonomy map} on a cross-section, defined below. The equivalents of trajectories for planar fields are leaves of a foliation. They can be regraded as trajectories, but in complex time. 

In this section, we follow the same line of thought as for planar vector fields. Elements of a normal form of a planar field around a focus point were reconstructed from the box dimension, either of orbits of the Poincar\' e map or of spiral trajectories themselves, in papers of \v Zubrini\' c, \v Zupanovi\' c \cite{buletin} and \cite{belgproc}. The elements of the formal normal form are related to the cyclicity of the focus point. In cases of planar saddles, we would like to do the same: to conclude a normal form of a saddle, using fractal analysis of trajectories. However, around a planar saddle point, the trajectory is not recurring. The Poincar\' e map cannot be defined. Therefore, we need to consider the complexified dynamics. For complex saddles, complex leaves of a foliation accumulate at the saddle and an analogon of the Poincar\' e map, called \emph{the holonomy map}, can be defined. 

In Subsection~\ref{threethreeone}, we define precisely the notions related to complex saddles. Then, in Subsection~\ref{threethreetwo}, we directly apply results from Section~\ref{threetwo} to holonomy maps, and deduce formal normal form of complex saddles from fractal properties of their holonomy maps. In the future, we would like to compute the box dimension of leaves of a foliation and relate it to the formal normal form for complex saddles. 

In planar cases, in order to obtain a monodromy around the saddle, we can connect the separatrices of the saddle in a saddle connection (saddle loop), by a regular map. The computation of the box dimension of a spiral trajectory around a planar saddle loop, depending on the saddle, as well as on the connection, is made in Subsection~\ref{threethreethree}. It can be considered as a preliminary technique for computing the box dimension of leaves in the complexified case. By analogy with planar saddle loops, a conjecture about the box dimension of leaves for complex saddles is made in short Subsection~\ref{threethreefour}. It needs to be proven in the future.

\subsection{Introduction to complex saddles}\label{threethreeone} 

The definitions and descriptions that we give here can be found in, for example, \cite[Sections 1,2,22]{ilyajak}, \cite[Chapitre 4,5]{loray}, \cite{teyssier}.

We consider germs of holomorphic vector fields in $\mathbb{C}^2$, with complex saddles at the origin. That means, with linear part of the form
\begin{equation}\label{ssh}
\begin{cases}
\dot{z}=z,\\
\dot{w}=-r\cdot w,\ \ r\in\mathbb{R}_+^*.
\end{cases}
\end{equation}
We call $r$ the \emph{hyperbolicity ratio} of the saddle. The derivatives $\dot{z}$ and $\dot{w}$ are meant in complex time, $t\in\mathbb{C}$.
We say that the saddle is \emph{resonant}, if $r\in\mathbb{Q}_+^*$. On the other hand, if $r\in\mathbb{R}_+^*\setminus \mathbb{Q}_+$, we call the saddle \emph{nonresonant}. %We treat here only the \emph{resonant formally nonlinearizable saddles}.
\medskip

First we describe the local dynamics (in complex time) at complex saddles.
%By existence/uniqueness theorem for holomorphic vector fields, see Theorem 1.1 in \cite{ilyajak}, any open domain of $U\subset \mathbb{C}^2$ is the union of disjoint and connected holomorphic phase curves. 
The phase curves are called \emph{leaves}. They form a \emph{singular foliation of the phase space}, with singular point $0$. We take the following definition of singular foliation from \cite{ilyajak}.
\begin{definition}(\cite[p.14]{ilyajak})\label{sf}
\emph{Singular foliation} with the singular point at the origin, denoted by $\mathcal{F}$, is a partition\footnote{Every point in phase space belongs to exactly one phase curve.} of the phase space into a continuum of connected phase curves, called \emph{leaves}, which locally except at the origin look like family of parallel affine subspaces.
\end{definition} 
More precisely, let $U$ be any neighborhood of the origin. We can partition $U\setminus \{0\}$ into a disjoint union of connected leaves, $U=\bigcup_{\alpha}L_\alpha$. Let $L_\alpha$ be any leaf of a foliation. For each $a\in L_\alpha\neq 0$, there exists a neighborhood $a\in V$, and a set of \emph{indices} $\mathcal{A}\subset \mathbb{C}$, such that $L_\alpha \cap V$ is biholomorphically equivalent to $\bigcup_{y\in\mathcal{A}} L_y$, where $L_y=\{(z,y)\ |\ |z|<1\}$ is a complex unit disc. 

\smallskip
Locally around the point $a$, different local leaves (biholomorphic images of unit discs) may belong to the same global leaf $L_\alpha$. The positions of local leaves belonging to the same global leaf are described by the so-called \emph{holonomy map}. We describe it in the following paragraph.

\smallskip
As in the case of planar saddles, each complex saddle has \emph{two complex separatrices}\footnote{Leaves through singular point $0$, whose closure by the singular point is a germ of an analytic curve.}. By change of variables, we can always suppose that they correspond to the complex axes ($z=0$ and $w=0$ planes).

Let $a\neq 0\in\mathbb{C}^2$ be some 'initial' point lying close to the $z$-axis separatrix and let $L$ be a global leaf of a foliation through $a$. Due to the complexity of the time, the flow along the leaf does not have a predefined direction, and we can move along any oriented path on the $z$-axis while staying on the same leaf in $\mathbb{C}^2$. By symmetry, the same can be done along the oriented paths in the $w$-axis separatrix. 

Consider the two-dimensional \emph{vertical cross-section} $\tau_2\simeq\mathbb{C}$ at $a$, transversal to the separatrix $\{w=0\}$ and the foliation. Let $(z_0,0)$ denote the intersection of $\tau_2$ with the separatrix. Obviously, $\tau_2\equiv \{z=z_0\}$ and we parametrize it simply by variable $w$. Take any closed oriented path $\gamma$ around the origin lying in the separatrix plane $\{z=0\}$, with the initial point $(z_0,0)$. Consider further the transversal cross-sections at some small distances along the path. Due to the product structure locally around each point, if the distances between the cross-sections are chosen sufficiently small, we can make a \emph{parallel transport} of $a=(z_0,w_a)$ along the leaf $L$, following the path $\gamma$. We return to the same cross-section $\tau_2$ in some (other) point $a'=(z_0,w_{a'})\in L$. We consider the function defined on $\tau$ by
$$
h_\gamma(w_a)=w_{a'},  
$$  
for each point $a$ belonging to some neighborhood $V$ of $0$ on $\tau_2$ and $a'$ obtained by its parallel transport along the corresponding leaf, following the path $\gamma$. The function $h_\gamma$ is called \emph{the holonomy map of the $z$-axis} on cross-section $\tau_2$. It can be shown that its form does not depend on the choice of closed path $\gamma$ with initial point $z_0$ going around the origin, therefore we denote it by $h$ only.  \emph{The holonomy map of $w$-axis} is defined in the same way, considering a \emph{horizontal cross-section} $\tau_1=\{w=w_0\}$. 

In general, $a'\neq a$, due to the fact that phase curves are described by multiple-valued complex functions. For example, in the simplest case of the linear saddle \eqref{ssh} with hyperbolicity ratio $r$, the phase curves are given by $z^r w=c,\ c\in\mathbb{C}$. Geometrically, this shows that each global leaf locally has infinitely many parallel levels, accumulating at the separatrices. The passage from one level to another happens while circling around the origin. The form of a holonomy map describes the density of this accumulation. Figure~\ref{sedlo} taken from \cite{loray} can help us visualize the dynamics around a complex saddle. 

\begin{figure}[htp]
\begin{center}
  \vspace{-23cm}
  % replace aims_logo.eps by your figure file name
  \includegraphics[scale=1, trim={-1.5cm 0cm 0cm 0cm}]{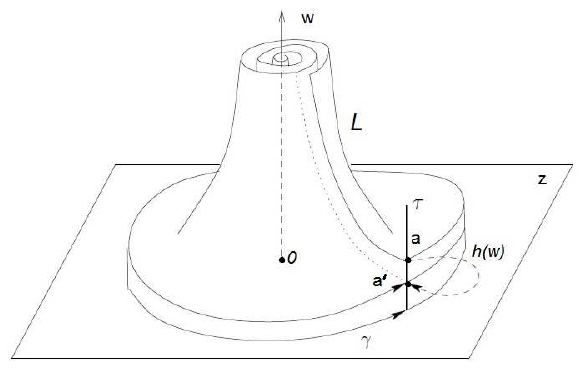}
  \vspace{-2.3cm}
  \caption{\small{A global leaf of a foliation at a complex saddle in $\mathbb{C}^2$ and its holonomy map $h$ on a vertical cross-section $\tau$, Figure 1. in \cite{loray}}.}\label{sedlo}
  \end{center}
\end{figure}
Finally, note the parallel between the complex holonomy map for complex saddles and the Poincar\' e map defined on a transversal to a planar saddle.
\bigskip

At the end of the section, we give an overview of orbital formal normal forms for complex saddles, from \cite{ilyajak} or \cite{loray}. 
Recall the definition of \emph{(formally) orbitally equivalent} germs of vector fields from e.g.\cite{teyssier}. Two germs of holomorphic vector fields are \emph{(formally) orbitally equivalent} if they are (formally) conjugated\footnote{To remind, \emph{(formal) conjugacy} of vector fields means that they can be translated one to another by a (formal) change of variables.}, up to multiplication with a holomorphic function non-vanishing at the origin.  The (formal) orbital equivalence can equivalently be defined as the (formal) conjugacy of the induced foliations. Note that (formal) conjugacy of vector fields implies their (formal) orbital equivalence, but the contrary is not true. For definitions, see e.g. \cite{teyssier}.

\begin{proposition}[Orbital formal normal forms for complex saddles, Section 22B in \cite{ilyajak}]\label{fnsaddle}\  
Let $X$ be a complex saddle germ with hyperbolicity ratio $r\in \mathbb{R}_+^*$. 
\begin{enumerate}
\item \emph{Nonresonant saddles} $(r\in \mathbb{R}_+^*\setminus \mathbb{Q}_+)$ are formally orbitally linearizable, that is, formally orbitally equivalent to the linear germ
\begin{equation}\label{fnflv}
\begin{cases}
\dot z=z,\\
\dot w=-r\cdot w.
\end{cases}
\end{equation}

\item \emph{Resonant saddles} $(r\in \mathbb{Q}_+^*)$ are either formally orbitally linearizable or formally orbitally equivalent to the germ
\begin{equation}\label{fnfv}
\begin{cases}
\dot z=z,\\
\dot w=w(-r+\frac{u^{k+1}}{1+\lambda u^k}),\ \ u=z^p w^q,
\end{cases}
\end{equation}
for some $k\in\mathbb{N},\ \lambda\in\mathbb{C}$. Here, $r=\frac{p}{q}$, for $p,\ q\in\mathbb{N}$, $(p,q)=1$. 
\end{enumerate}
\end{proposition}

The germs \eqref{fnflv}, \eqref{fnfv} respectively are called \emph{orbital formal normal forms} of the initial field X. In case $(ii)$ of nonlinearizable resonant germs, the quadruple $(p,\ q,\ k,\ \lambda)$ represents its \emph{formal invariants}.

\subsection{Application via holonomy map}\label{threethreetwo}

It has been proven by Mattei, Moussu \cite{mm} that (formal) orbital equivalence of complex saddle germs is equivalent to (formal) conjugacy of their holonomy maps of both the horizontal and the vertical axis, see \cite[Theorem 5.2.1]{loray}. On the other hand, the holonomy maps of formal normal forms \eqref{fnflv} and \eqref{fnfv} can be computed integrating the normal forms, see Example~\ref{exa}. Proposition~\ref{holo} follows.
\begin{example}[The holonomy maps of linear saddles, Section 5.1\cite{loray}]\label{exa}
The first integral\footnote{A first integral of a field $X$ is a nonconstant holomorphic function $H(z,w)$, constant along the phase curves (leaves) of the field, if such exists. The phase curves are then given by $H(w,z)=c,\ c\in\mathbb{C}$. Equivalently, the derivative of $H$ vanishes along the phase curves, that is, $\frac{d}{dt}H(z(t),w(t))=H_z\cdot z'(t)+H_w \cdot w'(t)=0.$} of a linear saddle \eqref{fnflv} is given by $H(z,w)=z^r w$. The leaves (phase curves) are given by  $$L_c\ \ldots\quad z^r w=c,\ c\in\mathbb{C}.$$ The holonomy map of the $z$-axis, denoted $h_z(w)$, on the cross-section $\{z=z_0\}$ parametrized by $w$, is computed by definition:
$$H(z_0,w)=H(e^{2\pi i}z_0,h_z(w)).$$ We get the rotation $h_z(w)=e^{-2\pi i r}w$. Similarly, for the holonomy of the $w$-axis on the cross section $\{w=w_0\}$ parametrized by variable $z$, we get $h_w(z)=e^{-(2\pi i)/r}z$.
\end{example}

In the sequel, $h_z(w)$ denotes a holonomy map of the $z$-axis, defined on a cross-section $\{z=z_0\}$ parametrized by variable $w$. Analogously, $h_w(z)$ denotes a holonomy map of the $w$-axis, defined on a cross-section $\{w=w_0\}$ parametrized by variable $z$.

\begin{proposition}[Holonomy maps of complex saddles, Lemma 22.2 in \cite{ilyajak}]\label{holo}
Let $X$ be a saddle germ with hyperbolicity ratio $r\in\mathbb{R}_+^*$. If $r\in\mathbb{Q}_+^*$, let $r=\frac{p}{q},\ p,\ q\in\mathbb{N},\ (p,q)=1$. 
\begin{enumerate}
\item For \emph{nonresonant} or \emph{formally orbitally linearizable resonant germs}, the holonomy maps $h_z(w)$ and $h_w(z)$ respectively are formally conjugated to the rotations
$$
h_z(w)\sim e^{-2\pi i r} w,\ h_w(z)\sim e^{-(2\pi i)/r} z.
$$
Moreover, for formally orbitally linearizable resonant germs it holds that $h_z^{\circ q}(w)=id$ and $h_w^{\circ p}(z)=id$.

\item For \emph{resonant, formally orbitally nonlinearizable germs} with formal invariants $(p,\ q,\ k,\ \lambda)$, the holonomy map $h_z(w)$ is a germ with multiplier $e^{-2\pi i p/q}$, whose $q$-th iterate $h^{\circ q}(z)$ is a parabolic diffeomorphism tangent to the identity, of the formal type $(kq,\ \lambda)$.

\noindent Analogously, the holonomy map $h_w(z)$ is a germ with multiplier $e^{-2\pi i q/p}$, whose $p$-th iterate $h^{\circ p}(z)$ is a parabolic diffeomorphism tangent to the identity, of the formal type $(kp,\ \lambda)$.
\end{enumerate}
\end{proposition}

By Proposition~\ref{holo}, nonhyperbolic germs of complex diffeomorphisms treated at the beginning of the chapter are related to complex saddle fields through holonomy maps. We thus directly apply our results of Section~\ref{threetwo} about classification of complex diffeomorphisms using fractal properties of orbits, to read the formal normal form of complex saddles from fractal properties of orbits of their holonomy maps. We use the holonomy map $h_z(w)$ of the $z$-axis. Similar conclusion can be drawn using the holonomy of the $w$-axis.

We suppose here that the saddle is \emph{resonant}. In nonresonant cases, its holonomy map is a complex diffeomorphism of the type (NH1), with irrational rotation in the linear part. This is the case we omitted from our analysis in Section~\ref{threetwo}.
\medskip

Let $X$ be a resonant complex saddle with hyperbolicity ratio $r=\frac{p}{q}$, $p,\ q\in\mathbb{N}$,\linebreak $(p,q)=1$. Let $h(w)$ be the holonomy map of the $z$-axis, defined on some cross-section $\tau\equiv \{z=z_0\}$. Let $S^h(w_0)$ be any orbit of $h$ on $\tau$, with $w_0$ sufficiently close to the origin.

\begin{proposition}[Orbital formal normal form of resonant complex saddles and fractal properties of orbits of holonomy maps]\label{holoholo}

Let $X$ be a resonant saddle germ and $h(w)$, $S^h(w_0)$ as above.
If $\dim_B(S^h(w_0))=0$, the saddle is formally orbitally linearizable. Otherwise, if $\dim_B(S^h(w_0))>0$, the orbital formal invariants $k,\ \lambda$ of the saddle are uniquely determined by fractal properties of any orbit $S^{h^{\circ q}}(w_0)\subset\tau$ of the $q$-th iterate of the holonomy map, $$\Big(\dim_B(S^{h^{\circ q}}(w_0)),\ \mathcal{M}^\mathbb{C}(S^{h^{\circ q}}(w_0)),\ \mathcal{R}^\mathbb{C}(S^{h^{\circ q}}(w_0))\Big).$$  
The following explicit formulas hold:
\begin{equation*}
\begin{split}
k\ =\ &\frac{1}{q}\cdot\frac{\dim_B(S^{h^{\circ q}}(w_0))}{1-\dim_B(S^{h^{\circ q}}(w_0))},\\
\lambda\ =\ &2(kq+1)\cdot i\cdot Re\left(\frac{\mathcal{R}^\mathbb{C}(S^{h^{\circ q}}(w_0))}{\nu_{\mathcal{M}^\mathbb{C}(S^{h^{\circ q}}(w_0))}}\right)+2\pi\cdot \phi(kq)\cdot\mathcal{M}(S^{h^{\circ q}}(w_0))\cdot Im\left(\frac{\mathcal{R}^\mathbb{C}(S^{h^{\circ q}}(w_0))}{\nu_{\mathcal{M}^\mathbb{C}(S^{h^{\circ q}}(w_0))}}\right).
\end{split}
\end{equation*}
Here, $\nu_{\mathcal{M}^\mathbb{C}(S^{h^{\circ q}}(w_0))}$ denotes the normalized directed Minkowski content and $\phi(n)$ is a function of $n\in\mathbb{N}$, explicitly given by
$$
\phi(n)=\frac{n(n+1)}{n-1}\cdot \frac{1}{\sqrt\pi}\cdot\frac{\frac{\Gamma(\frac{1}{n+1})}{\Gamma(\frac{3}{2}+\frac{1}{n+1})}+\sqrt{\pi}}{\frac{\Gamma(\frac{1}{2}+\frac{1}{2n+2})}{\Gamma(2+\frac{1}{2n+2})}-\sqrt{\pi}}\cdot\frac{\Gamma(1 + \frac{1}{2n+2})}{\Gamma(\frac{3}{2}+\frac{1}{2n+2})}.
$$
\end{proposition}
Note that the hyperbolicity ratio of the saddle can be read beforehand from the geometry of orbit $S^h(w_0)$, as was commented at the beginning of Section~\ref{threetwo}.  Note also that the second formula holds in the case when $kq>1$. The case $kq=1$ requests a slightly different definition of fractal properties for orbits of $h^{\circ q}(w)$, see comment after Definition~\ref{asy1}.

\begin{proof}
By Proposition~\ref{holo}, the holonomy map $h(w)$ of a resonant saddle is a nonhyperbolic complex diffeomorphism with $0$ as a fixed point, with multiplier $e^{-2\pi i p/q}$. Here, $r=p/q$ is the hyperbolicity ratio of the saddle. First, the elements $p$ and $q$ of the formal normal form of the saddle can be read from geometry of any orbit $S^h(w_0)$, as described at the beginning of Section~\ref{threetwo}. 

Furthermore, we need to read whether the saddle is linarizable and, if not, the elements $(k,\lambda)$ of its orbital formal normal form. It holds that $h^{\circ q}(w)$ is either the identity map or a parabolic diffeomorphism tangent to the identity. By Proposition~\ref{holo}, in the latter case it is of the formal type $(kq,\lambda)$. 

The formally orbitally linearizable case of resonant saddle corresponds to a special case of formally linearizable holonomy map. In this case, $h^{\circ q}$ is formally linearizable germ tangent to the identity, therefore it is the identity map by Proposition~\ref{snff}. The orbit $S^h(w_0)$ consists of $q$ points only ($w_0$ rotated $q$ times). In this case, $\dim_B(S^h(w_0))=0.$ Note however that this is the only case of trivial box dimension, since in all other cases we have that $\dim_B(S^h(w_0))=\dim_B(S^{h^{\circ q}}(w_0))=1-\frac{1}{kq+1}>0$, by finite stability property of the box dimension. The linearizability result follows.
 
The nonlinearizable cases ($f^{\circ q}\equiv\!\!\!\!\!/\ id$) follow directly applying Theorem~\ref{fnf} to $h^{\circ q}(w)$. 
\end{proof}

\subsection{A preliminary result: box dimension of a planar saddle loop}\label{threethreethree}

We consider an analytic planar vector field with a saddle loop. After necessary changes of variables, we can suppose that the saddle lies at the origin and that its separatrices correspond to the coordinate axes. This section is dedicated to computing the box dimension of a spiral trajectory accumulating at a saddle loop. This is a preliminary result for computing the box dimension of leaves of a foliation for complex saddles, which is left for the future work. We expect to be able to apply similar techniques as here.

In our computations, we use the notation and the results from the book of Roussarie \cite[Chapter 5]{roussarie}.
\smallskip

Let $X$ be a planar analytic vector field with saddle point at the origin:
\begin{align*}
\begin{cases}
\dot x&= x+P(x,y),\\
\dot y&=-r\cdot y+Q(x,y),
\end{cases}
\end{align*}
where $r\in\mathbb{R_+^*}$ and $P$,\ $Q$ are analytic functions of higher order. If $r\in \mathbb{Q}_+^*$, we say that the saddle is \emph{resonant}. We put $r=p/q$, where $p,\ q\in\mathbb{N}$, $(p,q)=1$. In the case $r\in\mathbb{R_+^*}\setminus \mathbb{Q}$, the saddle is called \emph{nonresonant}. Moreover, by exchanging the roles of $x$ and $y$ and dividing the field by $-1/r$, if necessary, we can always suppose that $r\geq 1$. 

\smallskip
In computing the box dimension of a spiral trajectory, we will use the known asymptotics of its Poincar\' e map on a transversal to the loop. Let $\tau_1\equiv \{y=1\}$ and $\tau_2\equiv \{x=1\}$ represent the horizontal and the vertical transversal, parametrized so that the origin lies on the loop. Let $P(s)$ denote the Poincar\' e map on any transversal $\tau$ not passing through the origin.

\begin{proposition}[Poincar\' e map on a transversal to the loop, Sections 5.1.3,\ 5.2.2 \cite{roussarie}]\label{rouss}
Let $X$ be an analytic vector field with saddle loop at the origin, with ratio of hyperbolicity $r\in\mathbb{R}_+^*$. The Poincar\' e map $P(s)$ has the following asymptotic expansion on $\tau$, as $s\to 0$:
\begin{enumerate}
\item[(i)] $r=1$:
\begin{align*}
P(s)=\beta_1 s +&\alpha_2 s^2 (-\log s)+\beta_2 s^2+\ldots+\beta_{l-1}s^{l-1}+\alpha_l s^l(-\log s)+O(s^l),\  l\geq 1.
\end{align*}
Here, the coefficients $\alpha_i$ and $\beta_i$ are obtained as follows: 
\begin{align*}D(s)=s+\alpha_2 s^2(-\log s)+&\alpha_3 s^3(-\log s)+\ldots+\alpha_l s^l(-\log s)+O(s^l)\end{align*} is the transition (Dulac) map at the saddle, and $$R^{-1}(s)=\beta_1 s+\beta_2 s^2+\beta_3 s^3+o(s^3)$$ is the inverse of the analytic transition map closing the connection.
\item[(ii)] $r\neq 1$:
$$
P(s)=C s^r+O(s^r),\ C>0.
$$
\end{enumerate}
\end{proposition}
Note that the case $r<1$ corresponds to the repelling saddle loop. It can easily be transformed to attracting saddle loop case by the transformation mentioned above.
\medskip

Before stating the dimension result, we introduce the notion of \emph{codimension} of the saddle loop. There are many different interpretations of codimension in the literature. This definition is taken from \cite[Definition 27]{roussarie}. It can be understood as the minimal number of parameters that we have to add in $X$ to get a generic unfolding of the saddle loop, see p. 335\cite{perko}, or equivalently, as the number of conditions imposed on the loop. Codimension is at least one, since one parameter is always needed to close the loop (it breaks in unfoldings).
\begin{definition}[Codimension of the saddle loop, Definition 27 in \cite{roussarie}]
Let $r=1$ and $k\geq 1$. We say that the saddle loop is \emph{of codimension $2k$} if $s-P(s)\sim s^k$, and \emph{of codimension $2k+1$} if $s-P(s)\sim s^{k+1}(-\log s)$. If $r\neq 1$, we say that the saddle loop is \emph{of codimension $1$}.
\end{definition}

We state now our main dimension result. Let $x_0$ be the initial point lying in the neighborhood of the loop and let $S(x_0)$ denote the spiral trajectory with the initial point $x_0$, accumulating at the loop. 

\begin{theorem}[Box dimension of the spiral trajectory around a saddle loop]\label{codim}
Let $k\geq 1$ be the codimension of the saddle loop and $S(x_0)$ any spiral trajectory as above. Then
$$
\dim_B(S(x_0))=\begin{cases}
2-\frac{2}{k},& \text{$k$ even},\\
2-\frac{2}{k+1},& \text{$k$ odd}.\end{cases}
$$
\end{theorem}

\noindent \emph{Sketch of the proof.} We compute the box dimension dividing the trajectory in two parts, $S(x_0)=S_1(x_0)\cup S_2(x_0)$, and using the finite stability property of box dimension. $S_1(x_0)$ is the non-regular part of the spiral trajectory near the saddle, between $\tau_1$ and $\tau_2$. $S_2(x_0)$ is the remaining regular part of the trajectory. We first compute the box dimension of $S_1(x_0)$. The main tools are Lemma~\ref{fh} and Lemma~\ref{dim} stated below. First, in Lemma~\ref{fh}, we show that $S_1(x_0)$ can by bilipschitz mapping be transformed to a family of \emph{parallel}\footnote{Let $r>0$. We call the family $\mathcal{H}=\{x^r y=c,\ c\in S\}$, where $S\subset\mathbb{R}$, the family of \emph{parallel} hyperbolas parametrized by $S\subset\mathbb{R}$. The set $S$ determines the points where hyperbolas from the family intersect the transversal.} hyperbolas, intersecting the transversals at the points with the same asymptotics. Then, in Lemma~\ref{dim}, we compute the box dimension of the family of parallel hyperbolas with the known asymptotics on the transversal. To compute the box dimension of the regular part $S_2(x_0)$, we can directly apply the well-known \emph{flow-box theorem}, stated in Lemma~\ref{kuzn}.

We first state the mentioned lemmas. 

\begin{lemma}[Flow-box theorem, p.75\cite{kuzne}]\label{kuzn}
Let us consider a planar vector field of class $C^1$. Assume that $U\subset\mathbb{R}^2$ is a closed set the boundary of which is the union of two trajectories and two curves transversal to trajectories. If U is free of singularities and periodic orbits, then the vector field restricted to U is diffeomorphically equivalent to the field
\begin{align*}
\begin{cases}
\dot{x}=1,\\
\dot{y}=0,
\end{cases}
\end{align*}
on the unit square $\{(x, y) |\ 0\leq x,y\leq 1 \}$.
\end{lemma}
\noindent That is, the flow on $U$ can be represented as a paralel flow.

\begin{lemma}\label{fh}
Let $X$ be an analytic vector field with saddle at the origin, $r\in\mathbb{R},\ r>0$:
\begin{align}\label{field}
\begin{cases}
\dot{x}&=x+P(x,y),\\
\dot{y}&=-r\cdot y +Q(x,y).
\end{cases}
\end{align}
There exists a neighborhood of the origin such that \eqref{field} is orbitally diffeomorphically equivalent to its linear part.
Moreover, the diffeomorphism acts quadrant-wise.
\end{lemma}
  
\begin{proof}
Let $X$ have a \emph{nonresonant} saddle at the origin, $r\notin \mathbb{Q}$. By \cite[Theorem 13]{roussarie}, such a vector field is linearizable in a neighborhood of the origin: it is diffeomorphically equivalent to its linear part. This proves the statement in the nonresonant case. 

We prove here the more complicated, \emph{resonant} case, which is nonlinearizable. We show that it is orbitally $C^{1}$-linearizable. Let $X$ have a resonant saddle at the origin, with $r\in\mathbb{Q}_+^*$. Let $r=\frac{p}{q}$ with $p,\ q\in\mathbb{N},\ (p,q)=1$. By Theorem 13 in \cite{roussarie}, there exists an integer $N\in \mathbb{N}$, the coefficients $a_2,\ldots,a_{N+1}\in\mathbb{R}$, and a neighborhood of the origin, such that $X$ is diffeomorphically equivalent to the polynomial vector field
\begin{align}\label{polin}
\begin{cases}
\dot{x}&=x,\\
\dot{y}&=-r\cdot y+\frac{1}{q}\sum_{i=1}^{N}a_{i+1}\cdot (x^p y^q)^i\cdot y.
\end{cases}
\end{align}

We construct a diffeomorphism $F(x,y)$, acting quadrant-wise in a neighborhood of the origin, which sends phase curves of field \eqref{polin} to phase curves of its linear part. We show here the construction of $F^{I}(x,y)$ in the first quadrant. Afterwards we glue functions $F^{I,II,III,IV}$ constructed in each quadrant to a global diffeomorphism $F$ at the origin.

We proceed as in \cite[5.1.2]{roussarie} (a similar technique was used there for obtaining the Dulac map at the resonant saddle). We solve the system \eqref{polin} by substitution $u=x^p y^q$, whereas we get the system
\begin{align}\label{pomo}
\begin{cases}
\dot{x}&=x,\\
\dot{u}&=\sum_{i=2}^{N} a_i u^i.
\end{cases}
\end{align} 

Solving \eqref{pomo} by expanding $u(t,u)$ in series with respect to the initial condition $u_0$, we get that
\begin{equation}\label{ju}
u(t,u_0)=u_0+\sum_{i=2}^{N} g_i(t) u_0^{i}.
\end{equation}
The form of $g_i(t)$, $i\geq 2$, is described in Proposition 10 in \cite{roussarie}: $g_i(t)$ are polynomials in $t$, of degree $\leq i-1$. Therefore, we easily obtain the bounds
\begin{equation}\label{bd}
|g_i(t)|\leq C_i t^{i-1},\ |g_i'(t)|\leq D_i t^{i-2},\ \ C_i,\ D_i>0,\quad i=2,\ldots,N,
\end{equation}
for $t$ sufficiently big.

Let $\tau_1\equiv \{y=1\}$ be a horizontal transversal to the saddle. We should in fact consider the transversal at some small height $\delta>0$ instead at $1$, but the computations are the same. 
Using \eqref{ju} and $u=x^p y^q$, we can now derive the formula for the phase curve of the field \eqref{polin}, passing through the initial point $(x(0),y(0))=(s,1)\in\tau_1$. We put $u_0=s^p$ in \eqref{ju} and solve $\dot x=x$. We get that $t=\log\frac{x}{s}$, and then, for the phase curve through $(s,1)$, we have the formula:
\begin{equation}\label{s1}
y=\frac{s^r}{x^r}\big(1+\sum_{i=2}^{N}g_i(\log\frac{x}{s})s^{p(i-1)}\big)^{1/q},\ s\leq x\leq 1.
\end{equation}
The phase curve of the linear part passing through $(s,1)$ is, on the other hand, given by
\begin{equation}\label{s2}
y=\frac{s^r}{x^r},\ s\leq x\leq 1.
\end{equation}

We now define mapping $F^{I}(x,y)$ sending phase curves of \eqref{polin} to phase curves of the linear part in the following manner. For any point $(x,y)\in(0,1]\times(0,1]$ close to the saddle,  there exists a unique phase curve of the linear field passing through it, and it is determined by the point $(s,1)$ on $\tau_1$. We consider the phase curve of \eqref{polin} passing through the same point $(s,1)$. $F^{I}(x,y)$ is then defined as the orthogonal projection of $(x,y)$ on this phase curve. More precisely, by \eqref{s1} and \eqref{s2}, we get the formula for $F^I(x,y)$:
\begin{align}\label{ef}
F^{I}(x,y)=&\left(\ x,\ y\cdot \Big(1+\sum_{i=2}^{N}g_i\big(\log(y^{-1/r})\big)\cdot x^{p(i-1)}\cdot y^{q(i-1)}\Big)^{1/q} \right),\\ &\hspace{6.5cm} (x,y)\in(0,1]\times(0,1].\nonumber
\end{align}
Function $F^{I}$ is obviously well-defined, continuous and differentiable in $(0,1]\times(0,1]$. We show now that we can extend the definition of $F^{I}$ to the axes, so that the function is continuously differentiable in $[0,1]\times[0,1]$. 

We define $F^I(x,0):=(x,0),\ x\geq 0,$ and $F^I(0,y):=(0,y),\ y\geq 0$. Note that $F^I(x,y)$ cannot be extended continuously simply by formula \eqref{ef} to $y=0$ due to the logarithmic term. However, using \eqref{bd}, $\lim_{y\to 0}F^I(x,y)=0$. Therefore, extended as above, $F^I(x,y)$ is continuous on $[0,\delta)\times [0,\delta)$. 

Furthermore, $F^{I}(x,y)$ given by \eqref{ef} is differentiable on $(0,\delta]\times (0,\delta]$. We can show, by direct computation of the differential and using bounds \eqref{bd}, that the differential is bounded, as $x\to 0$ and $y\to 0$. Moreover,
\begin{align}
DF^I(x,y)=&\left[\begin{array}{cc}1& 0\\\partial_x F^I_2(x,y)&\partial_y F^I_2(x,y)\end{array}\right],\ x,\ y>0,\quad \text{where}\label{dif}\\[0.3cm]
&\qquad \lim_{y\to 0} \partial_y F^I_2(x,y)=1,\ \ \lim_{x\to 0}\partial_y F^I_2(x,y)=1.\nonumber\\
&\qquad \lim_{y\to 0} \partial_x F^I_2(x,y)=0,\ \ \lim_{x\to 0}\partial_x F^I_2(x,y)=G^I(y).\label{limi}
\end{align}
Here, 
$$
G^I(y)=\begin{cases}
\frac{1}{q}\cdot y^{q+1} \cdot g_1(\log y^{-q}),& p=1,\\
0,& p>1.
\end{cases}
$$
Obviously, $G^I(y)\to 0$ as $y\to 0$. We can check directly by definition of differentiability at $(x,0)$ and $(0,y)$, $x>0$, $y>0$, after some computation and using bounds \eqref{bd} and \eqref{limi}, that $F^I(x,y)$ extended to the axes in the above manner is continuously differentiable at the axes in the first quadrant and that the differential is given by 
\begin{equation}\label{differe}
DF^I(x,y)=\begin{cases}\text{\eqref{dif}},&(x,y)\in(0,\delta]\times (0,\delta],\\[1mm]
\left[\begin{array}{cc}1& 0\\0&1\end{array}\right],&y=0,\ x\in[0,\delta],\\[5mm]
\left[\begin{array}{cc}1& 0\\G^I(y)&1\end{array}\right],&x=0,\ y\in[0,\delta].
\end{cases}
\end{equation}
Note that we can obtain similar formulas for $F_{II,III,IV}$ in other quadrants, with the same linear part in \eqref{ef}, and other parts possibly differing in signs (depending on the quadrant). In the second quadrant, in \eqref{s2}, we have
$y=s^r/x^r,\ -1\leq x\leq s$, $s\leq 0$. In the third quadrant, $y=-s^r/x^r,\ -1\leq x\leq s$, $s<0$. In the fourth quadrant, $y=s^r/x^r,\ s\leq x\leq 1$, $s\geq 0$. Therefore, in the first and in the second quadrant $F^{I,II}(x,y)$ are given by the same formula \eqref{ef}, while in the third and in the fourth quadrant we have
\begin{align*}
F^{III,IV}(x,y)=&\left(\ x,\ y\cdot \Big(1+\sum_{i=2}^{N}g_i\big(\log(-y^{-1/r})\big)\cdot x^{p(i-1)}\cdot (-y)^{q(i-1)}\Big)^{1/q} \right).
\end{align*}
It can be checked that the functions and their differentials in quadrants glue nicely to a continuously differentiable function $F(x,y)$ around the origin, with differential at the origin equal to identity operator, as in \eqref{differe}. 
Now we can apply the inverse function theorem at the origin. We conclude that $F(x,y)$ is a local diffeomorphism. Moreover, by construction, it leaves the axes invariant and maps each quadrant to itself. 

Note in the course of the proof that $C^1$ is the best class that we can obtain applying bounds \eqref{bd}.
\end{proof}

\begin{lemma}[Box dimension of the countable union of hyperbolas at given distances]\label{dim}
Let $r\in\mathbb{R}$, $r>0$. Let $\{s_l|\ l\in\mathbb{N}_0\}$ be a discrete set of points on transversal $\tau_1\equiv \{y=1\}$, accumulating at $0$, such that the points and the distances between them are eventually decreasing. Let $$\dim_B(\{s_l\})=s\in [0,1),$$ and let the $r$-power sequence $\{s_l^r|\ l\in\mathbb{N}_0\}$ have box dimension equal to $$\dim_B(\{s_l^r\})=\frac{s}{s+r(1-s)}.$$ 
Let $$\mathcal{H}_1=\{(x,y)\in(0,1]\times (0,1]\ |\ x^r y=s_l^r,\ l\in\mathbb{N}_0\}$$ be a countable family of hyperbolas passing through points $s_l$ on $\tau_1$. It holds that 
$$
\dim_B(\mathcal{H}_1)=\max\Big\{1+\dim_B(\{s_l\}),\ 1+\dim_B(\{s_l^r\})\Big\}=\left\{1+s,\ 1+\frac{s}{s+r(1-s)}\right\}.
$$ 
\end{lemma}

Let $\tau_1$ denote a horizontal and $\tau_2$ a vertical transversal. 
Note that, by Lemma~\ref{dim}, the box dimension of the union of hyperbolas $\mathcal{H}_1$ between $\tau_1$ and $\tau_2$ is in fact the box dimension of a \emph{product structure} generated by hyperbolas around transversals $\tau_1$ or $\tau_2$. The neighborhood of the transversal on which the sequence of the intersections has a bigger box dimension prevails. The proof resembles to the proof for the box dimension of the Cartesian product, where the $\varepsilon$-neighborhood of $(N+1)$-dimensional product $U\times [0,1]$ may be directly estimated using the $\varepsilon$-neighborhood of $N$-dimensional set $U$. 

\begin{proof}
The intersections of $\mathcal{H}_1$ with the transversals form a discrete sets of points, $\{ s_l\}$ on $\tau_1$, and $\{ s_l^r\}$ on $\tau_2$. The family $\mathcal{H}_1$ can be considered as a family of phase curves of the linear saddle with ratio of hyperbolicity $r$. Immediately, by \emph{flow-box} Lemma~\ref{kuzn}, we get that the box dimension of $\mathcal{H}_1$ on a small rectangle around $\tau_1$ is equal to $1+s$, and around $\tau_2$ equal to $1+\frac{s}{s+r(1-s)}$. Therefore, by the finite stability and the monotonicity property of box dimension, we get
\begin{equation}\label{prod}
\underline{\dim_B}(\mathcal{H}_1),\ \overline{\dim_B}(\mathcal{H}_1)\geq \max\Big\{1+s,\ 1+\frac{s}{s+r(1-s)}\Big\}.
\end{equation}

If $r=1$, the hyperbolas are symmetric with respect to $y=x$ and sequences generated on $\tau_1$ and on $\tau_2$ are equal. If $r\neq 1$, we first symmetrize the family by a inverse-lipschitz change of variables. For $r>1$, we apply the change $u=x,\ v=y^{1/r}$. For $r<1$, the change $u=x^r,\ v=y$. We compute the box dimension of the symmetrized family and, by the inverse lipschitz property, conclude that the box dimension of the original family $\mathcal{H}_1$ is smaller or equal.

Let $r\geq 1$. The case $r<1$ is treated analogously. In the new coordinate system, we get the symmetric family $\mathcal{H}_2$ of hyperbolas $\mathcal{H}_2=\{uv=s_l|\ l\in\mathbb{N}\}$. It generates on $\tau_1$ and on $\tau_2$ the same discrete set $\{s_l\}$, of box dimension $s$. The box dimension of family $\mathcal{H}_2$ is, by symmetry and by finite stability property of box dimension, equal to the box dimension of its subset between the diagonal $y=x$ and transversal $\tau_1$.

We now directly estimate the area of the $\varepsilon$-neighborhoods of the union of hyperbolas $\mathcal{H}_2$ in this area, establishing the almost product relation with the $\varepsilon$-neighborhoods of the discrete set on the transversal. We can easily see that the intersection points of $\mathcal{H}_2$ and the diagonal $y=x$ are given by
$\{ \sqrt{2} s_l^{1/2}|\ l\in\mathbb{N}\}$. Therefore, the distances between the points on $\tau_1$ are eventually the smallest, compared to any other transversal in the area.

We compute the area dividing the $\varepsilon$-neighborhood into tail and nucleus. By \emph{tail} of the $\varepsilon$-neighborhood, $T_\varepsilon$, we mean the disjoint neighborhoods of the first finitely many hyperbolas. The remainder of the $\varepsilon$-neighborhood, where the neighborhoods of hyperbolas start overlapping, is called the \emph{nucleus}, and denoted $N_\varepsilon$. For the idea of division in tail and nucleus, see \cite{tricot}. Since the distances are the smallest on $\tau_1$, the critical index separating the tail and the nucleus is the same as for one-dimensional set $\{s_l\}$ generated on $\tau_1$. Let $T_\varepsilon^1$, $N_\varepsilon^1$ denote the tail and the nucleus respectively of the $\varepsilon$-neighborhood of the set $\{s_l\}$ on $\tau_1$. Similarly, let $T_\varepsilon^2$, $N_\varepsilon^2$, denote the tail and the nucleus of the $\varepsilon$-neighborhood of the union of hyperbolas $\mathcal{H}_2$. We now bring them into direct relation, and use the box dimension of the one-dimensional system on $\tau_1$ to directly conclude about the dimension of $\mathcal{H}_2$.

By $H_i\in\mathcal{H}_2$, we denote the hyperbola passing through point $s_i$ on $\tau_1$, $i\in\mathbb{N}_0$. Their lengths are bounded from below and above: there exist $A,\ B>0$ such that $$A<l(H_i)<B,\ i\in\mathbb{N}_0.$$
Therefore,
\begin{equation}\label{svv}
A(T_\varepsilon^2)\simeq |T_\varepsilon^1|+\varepsilon^2\pi\cdot n_\varepsilon\simeq |T_\varepsilon^1|+\frac{\varepsilon\pi}{2}\cdot|T_\varepsilon^1|\simeq |T_\varepsilon^1|,\ \varepsilon\to 0.
\end{equation}
Here, we use that $|T_\varepsilon^1|=2\varepsilon\cdot n_\varepsilon$. 
Since $\dim_B(\{s_l\})=s$, by definition of box dimension it holds that 
\begin{equation}\label{sv}
\lim_{\varepsilon\to 0}\frac{|T_\varepsilon^1|}{\varepsilon^{1-s-\delta}}=0,\quad  \lim_{\varepsilon\to 0}\frac{|N_\varepsilon^1|}{\varepsilon^{1-s-\delta}}=0,\ \text{ for all }\delta>0. 
\end{equation}
By \eqref{svv} and \eqref{sv}, we get 
\begin{equation}\label{prv}
\lim_{\varepsilon\to 0}\frac{A(T_\varepsilon^2)}{\varepsilon^{1-s-\delta}}=\lim_{\varepsilon\to 0}\frac{A(T_\varepsilon^2)}{\varepsilon^{2-(s+1)-\delta}}=0,\ \text{ for all }\delta>0.
\end{equation}

For the nucleus, we can give an easy upper bound. The hyperbola separating the tail and the nucleus is given by $H_{n_\varepsilon}\equiv \{uv=|N_\varepsilon^1|-\varepsilon\}$. The 
area $A(N_\varepsilon^2)$ is smaller than or equal to the area between the $y$-axis, the hyperbola $H_{n_\varepsilon}$ lifted by $\varepsilon$, the diagonal $y=x$ and the transversal $\tau_1$:
\begin{align}\label{nucleu}
A(N_\varepsilon^2)&\leq \varepsilon\cdot l(H_{n_\varepsilon})+\frac{|N_\varepsilon^1|-\varepsilon}{4}+\int_{(|N_\varepsilon^1|-\varepsilon)^{1/2}}^{1}\frac{|N_\varepsilon^1|-\varepsilon}{y}\ dy.
\end{align}
Take any fixed $\delta_0>0$. There exists a small $\nu>0$, such that $\delta_0-\nu>0$. Integrating and estimating \eqref{nucleu}, it holds that there exist $C>0$ and $\varepsilon_0$, such that
\begin{align*}
\frac{A(N_\varepsilon^2)}{\varepsilon^{1-s-\delta_0}}&\leq C\cdot\frac{|N_\varepsilon^1|-\varepsilon}{\varepsilon^{1-s-\delta_0}}(-\log\varepsilon)=C\cdot\frac{|N_\varepsilon^1|-\varepsilon}{\varepsilon^{1-s-(\delta_0-\nu)}}\cdot \varepsilon^\nu(-\log\varepsilon),\ \ \varepsilon<\varepsilon_0.
\end{align*} 
Passing to limit as $\varepsilon\to 0$ in the above inequality and using \eqref{sv}, we get that
\begin{equation}\label{drug}
\lim_{\varepsilon\to 0}\frac{A(N_\varepsilon^2)}{\varepsilon^{1-s-\delta}}=\lim_{\varepsilon\to 0}\frac{A(N_\varepsilon^2)}{\varepsilon^{2-(s+1)-\delta}}=0, \text{ for all } \delta>0.
\end{equation}
By \eqref{prv} and \eqref{drug}, it follows that $\underline{\dim_B}(\mathcal{H}_2),\ \overline{\dim_B}(\mathcal{H}_2)\leq 1+s.$ 

It follows that $\underline{\dim_B}(\mathcal{H}_1)\leq \underline{\dim_B}(\mathcal{H}_2)\leq 1+s.$ The same for the upper box dimension. By \eqref{prod}, we get that, for $r\geq 1$,
$$
\dim_B(\mathcal{H}_1)=1+s.
$$
It can be proven analogously that for $r<1$ it holds $\dim_B(\mathcal{H}_1)= 1+\frac{s}{s+r(1-s)}.$
\end{proof}

We illustrate the product statement of the lemma on Figure~\ref{hipi} below. In the figure, $r>1$. The family of hyperbolas $\{x^r y=s_l^r,\ l\in\mathbb{N}\}$, intersects the transversals $\tau_1$ and $\tau_2$ in the one-dimensional sequences $\{s_l\}$ and $\{s_l^r\}$ respectively. Box dimension of the sequence $\{s_l^r\}$ is smaller than of $\{s_l\}$ ($x^r$ for $r>1$ is lipschitz). The accumulation of density is therefore around the transversal $\tau_1$, and the set of hyperbolas takes the product box dimension around $\tau_1$, 
$$
\dim_B(\mathcal{H}_1)=\dim_B(\{s_l\})+1.
$$
\vspace{-1cm}
\begin{figure}[htp]
\begin{center}
\vspace{-30cm}
  % replace aims_logo.eps by your figure file name
  \includegraphics[scale=1.2,trim={-4cm 0cm 0cm 0cm}]{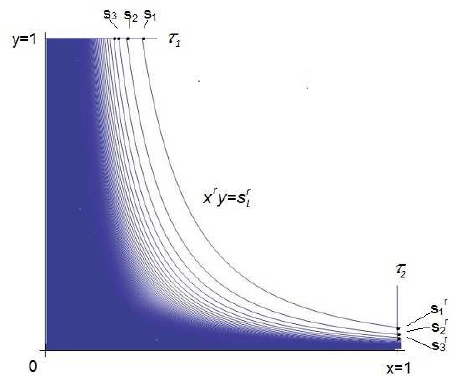}
  \caption{\small{Family $\mathcal{H}_1$ of hyperbolas from Lemma~\ref{dim}, $r>1$.}}\label{hipi}
  \vspace{0.5cm}
  \end{center}
\end{figure}

In the next remark, we list the one-dimensional discrete systems that we have considered so far and that satisfy the conditions of Lemma~\ref{dim}.

\begin{remark}\label{staro}
Let $\{s_l=g^{\circ l}(s_0)|\ l\in\mathbb{N}_0\}$ be a one-dimensional discrete system with initial point $s_0$, generated by function $g(s)$ on $(0,\delta)$.
\begin{enumerate}
\item \ (Theorem 1 in \cite{neveda})\ Let $g(s)=s-f(s)$, where 
 $$
 f(s)\simeq s^{\alpha},\ \alpha>1,\ \text{ as }s\to 0.
 $$
It holds that 
 \begin{equation}\label{seq} s_l\simeq l^{-\frac{1}{\alpha-1}},\ l\to\infty.\end{equation} The box dimension of a sequence that satisfies \eqref{seq} is equal to $\dim_B(\{s_l\})=1-\frac{1}{\alpha}$.
 
 \item \ (Theorem~\ref{gensaus} in Chapter~\ref{two})\ Let $g(s)=s-f(s)$, where 
 $$f(s)\simeq s^{\alpha}(-\log s),\ \alpha>1,\ \text{ as }s\to 0.$$
 It holds that $\dim_B(\{s_l\})=1-\frac{1}{\alpha}$.
 
\item \ (Lemma 1, Theorem 5 in \cite{neveda})\ Let $$ g(s)=k s+o(s),\ 0<k<1, \text{\ \ or\ \ \quad} g(s)=C s^\beta+o(s^\beta),\ \beta>1,\ C>0,\ \ \text{ as }s\to 0.$$
The discrete system generated by $g(s)$ accumulates at zero exponentially fast. There exist $k\in(0,1)$ and $C>0$ such that
\begin{equation}\label{expo}
0<s_l<C k^l. 
 \end{equation}
 The box dimension of a sequence satisfying \eqref{expo} is trivial, $\dim_B(\{s_l\})=0$.
 \end{enumerate}
 \end{remark} 

\smallskip

\noindent\emph{Proof of Theorem~\ref{codim}}.
Suppose that the saddle loop is of codimension $k$, with ratio of hyperbolicity $r\in\mathbb{R}_+^*$. Additionally, we can assume that $r\geq 1$, otherwise we divide the field by $-1/r$ and exchange the roles of $x$ and $y$, which leaves the box dimension intact. We consider one spiral trajectory $S(x_0)$ accumulating at the loop, with $x_0$ close to the loop. We divide the trajectory in two parts: the part $S_1(x_0)$ near the saddle, between the transversals $\tau_1$ and $\tau_2$, and the remaining regular part $S_2(x_0)$. 

We first compute $\dim_B(S_1(x_0))$. The Poincar\' e map $P(s)$ on the transversal $\tau_1$ is given in Proposition~\ref{rouss}.  By Lemma~\ref{fh} and its proof, by a diffeomorphic equivalence, the arcs of the trajectory in $(0,1]\times (0,1]$ can be simplified as the union $\mathcal{H}_1$ of countably many hyperbolas of the type $x^r y=c,\ c>0$, which intersect $\tau_1$ in the same points generated by $P(s)$. By Proposition~\ref{rouss} and Remark~\ref{staro}, Lemma~\ref{dim} can be applied to compute the box dimension of $\mathcal{H}_1$. Furthermore, diffeomorphic equivalence is a bilipschitz map,and the dimension of $S_1(x_0)$ is equal to the dimension of $\mathcal{H}_1$:
$$
\dim_B(S_1(x_0))=\begin{cases}
2-\frac{2}{k},& \text{$k$ even},\\
2-\frac{2}{k+1},& \text{$k$ odd}.\end{cases}
$$

It is left to compute the dimension of the remaining, regular part $S_2(x_0)$ of the trajectory. In this area, there are no singularities of the vector field. Therefore we can directly apply the flow-box Lemma~\ref{kuzn}. By Remark~\ref{staro}, the box dimension of an orbit of the Poincar\' e map on any transversal is equal to $1-\frac{2}{k}$, if $k$ is even, or $1-\frac{2}{k+1}$, if $k$ is odd.

The box dimension of $S_2(x_0)$ is computed as the box dimension of Cartesian product, 
$$\dim_B(S_2(x_0))=\begin{cases}
1+(1-\frac{2}{k})=2-\frac{2}{k},& \text{$k$ even},\\
1+(1-\frac{2}{k+1})=2-\frac{2}{k+1},& \text{$k$ odd}.\end{cases}
$$

Finally, by finite stability property of the box dimension, the result follows.
 \hfill $\Box$

\bigskip
\subsubsection{Application of results: the cyclicity of a saddle loop and the box dimension of a spiral trajectory around the loop}\label{cyclhom}

We have seen in Chapter~\ref{two} that, in recognizing the cyclicity of monodromic limit periodic sets for planar systems, we can use fractal analysis of orbits of the Poincar\' e map on a transversal. It was shown that in focus and limit cycle cases there exists a bijective correspondence between the box dimension of any orbit of the Poincar\' e map and the cyclicity of a set. Similary, the bijective correspondence was established between the cyclicity and the box dimension of a spiral trajectory around focus or limit cycle, see \cite{belgproc}. 

However, in Chapter~\ref{two} (see Example~\ref{defffi}), we have shown that the box dimension of orbits of the Poincar\' e map is imprecise in recognizing cyclicity in cases of saddle loops. Having computed the box dimension of a spiral trajectory around a saddle loop in the previous subsection, we show here that its box dimension exhibits the same deficiency. 

\begin{proposition}[Cyclicity of the saddle loop and the box dimension of a spiral trajectory]
Let $(X_\lambda,\Gamma)$ be a generic analytic unfolding of the loop $\Gamma$, such that the regularity condition \eqref{regu} is satisfied. Let $x_0$ be any point sufficiently close to the loop and let $S(x_0)$ denote the spiral trajectory passing through $x_0$, with box dimension $$\dim_B(S(x_0))=d\in[1,2).$$ By the box dimension, the cyclicity is not uniquely determined. More precisely, 
$$
Cycl(\Gamma,X_\lambda)\in \left \{\frac{2}{2-d}-1,\frac{2}{2-d}\right \}.
$$
\end{proposition}

\begin{proof}
If $\dim_B(S(x_0))=d$, by Theorem~\ref{codim}, the saddle loop may be of codimension $\frac{2}{2-d}-1$ or $\frac{2}{2-d}$. By \cite{roussarie}, under the regularity assumption on the unfolding, the codimension is equal to the cyclicity.
\end{proof}

\subsection{A conjecture about the box dimension of leaves of a foliation of complex resonant saddles}\label{threethreefour}

We consider \emph{resonant} complex saddles in $\mathbb{C}^2$, defined by \eqref{ssh}. By Proposition~\ref{fnsaddle}, they are either formally orbitally linearizable or of the formal type
$
(p,\ q,\ k,\ \lambda).
$ Here, $r=\frac{p}{q}\in\mathbb{Q}_+^*,\ (p,q)=1,$ is the ratio of hyperbolicity. In this section, we investigate if the box dimension of leaves of a foliation can reveal formal orbital linearizability and, otherwise, formal invariants.
\smallskip

Let $L_{a}$ denote one leaf of a foliation, passing through a point $a\in\mathbb{C}^2$ sufficiently close to the saddle. Let $\tau_1=\{w=w_0\}$ and $\tau_2=\{z=z_0\}$ denote a horizontal and a vertical cross-section. Let $h_w(z)$ denote the holonomy map induced by $L_a$ on $\tau_1$ and $h_z(w)$ on $\tau_2$. We know by Proposition~\ref{sf} that, around each cross-section $\tau$, the leaf of a foliation has a locally parallel structure (it may be exchanged by a bilipschitz map with family of unit complex discs). The positions of different levels belonging to one leaf on $\tau$ are given by an orbit of the holonomy map. The holonomy maps of a complex resonant saddle are known to be complex parabolic diffeomorphisms, see Proposition~\ref{holo}. We have computed the box dimension of their orbits in Theorem~\ref{two} in Section~\ref{threetwo}.

We distinguish two cases of resonant saddles:
\begin{enumerate}
\item If the saddle is \emph{formally orbitally linearizable}, by Proposition~\ref{holo}, orbits of holonomy maps $h_w(z)$ and $h_z(w)$ on cross-sections $\tau_1$ and $\tau_2$ are given by $p$, $q$ points respectively. The box dimension of orbits is equal to $0$. By product structure, we conclude that the box dimension of $L_a$ locally around each cross-section is equal to 2.

\item If the saddle is \emph{formally orbitally nonlinearizable}, by Proposition~\ref{holo}, $h_z^{(\circ q)}(w)$ belongs to the formal class $(kq,\ \lambda)$, and $h_w^{(\circ p)}(z)$ belongs to the formal class $(kp,\ \lambda)$. By Theorem~\ref{fnf} in section~\ref{threetwo}, box dimension of their orbits is equal to
$$
\dim_B\left(S^{h_z^{(\circ q)}}(w_0)\right)=1-\frac{1}{kq+1},\ \ \dim_B\left(S^{h_w^{(\circ p)}}(z_0)\right)=1-\frac{1}{kp+1}.
$$
Since orbits of $h_z$ ($h_w$) consist of $q$ ($p$) disjoint orbits of $h_z^{(\circ q)}$ $\big(h_w^{(\circ p)}\big)$, by finite stability of box dimension we get the same dimension result for orbits of holonomy maps. By product structure, we conclude that the box dimension of $L_a$ locally around cross-section $\tau_1$ is equal to $3-\frac{1}{kq+1}$. Locally around $\tau_2$, it is equal to $3-\frac{1}{kp+1}$. 
\end{enumerate}

Therefore, in 1. and 2., by finite stability property of box dimension, we have that 
\begin{equation}\label{onne}
\dim_B(L_a)\geq \begin{cases}
&2, \ \text{saddle resonant, orbitally formally linearizable},\\[0.2cm]
&\max\left\{3-\frac{1}{kq+1},\ 3-\frac{1}{kp+1}\right\},\\
&\quad  \text{ saddle resonant, nonlinearizable, of the formal type $(p,q,k,\lambda)$}.
 \end{cases}
\end{equation}

To verify the other side of the inequality \eqref{onne}, we need to compute the box dimension of a leaf in a small neighborhood of the origin, where the product structure is lost. This is left for future research. Driven by results of Lemma~\ref{dim} for the planar saddle case, we hope that, also in the complex case, we have equality in \eqref{onne}. However, due to connectedness of levels of each leaf in any neighborhood of the origin, which was not the case for planar saddle, we cannot directly apply results from Subsection~\ref{threethreetwo} to complex cases.

\begin{hypothesis}
Let $X$ be a resonant complex saddle and $L_a$ any leaf of a foliation, passing through $a\in\mathbb{C}^2$ close to the saddle.
In \eqref{onne}, the equality holds. 
\end{hypothesis}

If this conjecture is true, then it holds: a resonant saddle is linearizable if and only if the box dimension of leaves of a foliation is equal to 2 (trivial). If this is not the case, the first formal invariant $k$ of the orbital formal normal form can be read from the box dimension of any leaf of a foliation, assumed that the ratio of hyperbolicity $r=p/q$ of the saddle is known. 

For the other formal invariant $\lambda$ of the saddle, we expect that the further terms in the development, as $\varepsilon\to 0$, of the $\varepsilon$-neighborhoods of leaves are needed, similary as in the case of parabolic diffeomorphisms in Section~\ref{threetwo}.

\chapter{About analytic classification of complex parabolic diffeomorphisms\\ using $\varepsilon$-neighborhoods of orbits}\label{four}

We consider germs of parabolic diffeomorphisms $f:(\mathbb{C},0)\to(\mathbb{C},0)$, as in Section~\ref{threezero}:
\begin{equation}\label{diffe}
f(z)=z+a_{k+1}z^{k+1}+a_{k+2}z^{k+2}+o(z^{k+2}),\ k\in\mathbb{N},\ a_i\in\mathbb{C},\ a_{k+1}\neq 0.
\end{equation} 
We described in Section~\ref{threetwo} that the formal class of \eqref{diffe} is given by two formal invariants, $(k,\lambda)$, $\lambda\in\mathbb{C}$. Applying formal changes of variables reducing a germ of multiplicity $k+1$ to its formal normal form, we noticed that the formal class depends only on $(2k+1)$-jet of the germ. According to that, in Subsection~\ref{threetwofour}, we showed that the formal class of a diffeomorphism can be deduced only from \emph{the first $k+1$ coefficients} in the formal asymptotic development in $\varepsilon$ of the directed area of the $\varepsilon$-neighborhood of \emph{only one} orbit, as $\varepsilon\to 0$.

On the other hand, analytic class cannot be read from any finite jet of parabolic germ, see e.g.\cite[21H]{ilyajak}. It can be shown that there exist analytically non-conjugated germs with the same $l$-jet, for every $l>2k+1$ (they are formally conjugated by previous considerations). By Ecalle \cite{ecalle} and Voronin \cite{voronin}, the analytic class of a parabolic diffeomorphism is given by $2k$ diffeomorphisms, the so-called \emph{Ecalle-Voronin moduli} or \emph{horn maps}. We describe the moduli in more detail in Section~\ref{fourone}. Accordingly, in the hope to deduce the analytic class of a diffeomorphism from directed areas of $\varepsilon$-neighborhoods of orbits, we analyse the whole functions of $\varepsilon$-neighborhoods of orbits, not just finitely many terms in their asymptotic developments in $\varepsilon$.

\medskip
\noindent This chapter is motivated by the following question: 
\smallskip

\emph{Can we read the analytic class of a diffeomorphism from $\varepsilon$-neighborhoods of its orbits, regarded as functions of parameter $\varepsilon>0$ and of initial point $z\in\mathbb{C}$?}
\medskip 

For simplicity, we consider only the diffeomorphisms in the formal class $(k=1,\lambda=0)$, formally equivalent to the time-one map $f_0(z)=Exp(z^2\frac{d}{dz})=\frac{z}{1-z}$. We call $f_0(z)$ the \emph{model diffeomorphism}. Furthermore, we assume that $f$ is \emph{prenormalized}. That is, the first normalizing change of variables is already made, and further we admit only changes of variables tangent to the identity. Therefore, all such diffeomorphisms are of the form:
$$
f(z)=z+z^2+z^3+o(z^3).
$$
In this case, there exists only one attracting petal $V_+$, invariant for $f$ (around negative real axis) and one repelling petal $V_-$, invariant for $f^{-1}$ (around positive real axis) for the local dynamics, see Figure~\ref{orb1}.  

\begin{figure}[ht]
\begin{center}
\vspace{-0.3cm}
  % replace aims_logo.eps by your figure file name
  \includegraphics[scale=0.4]{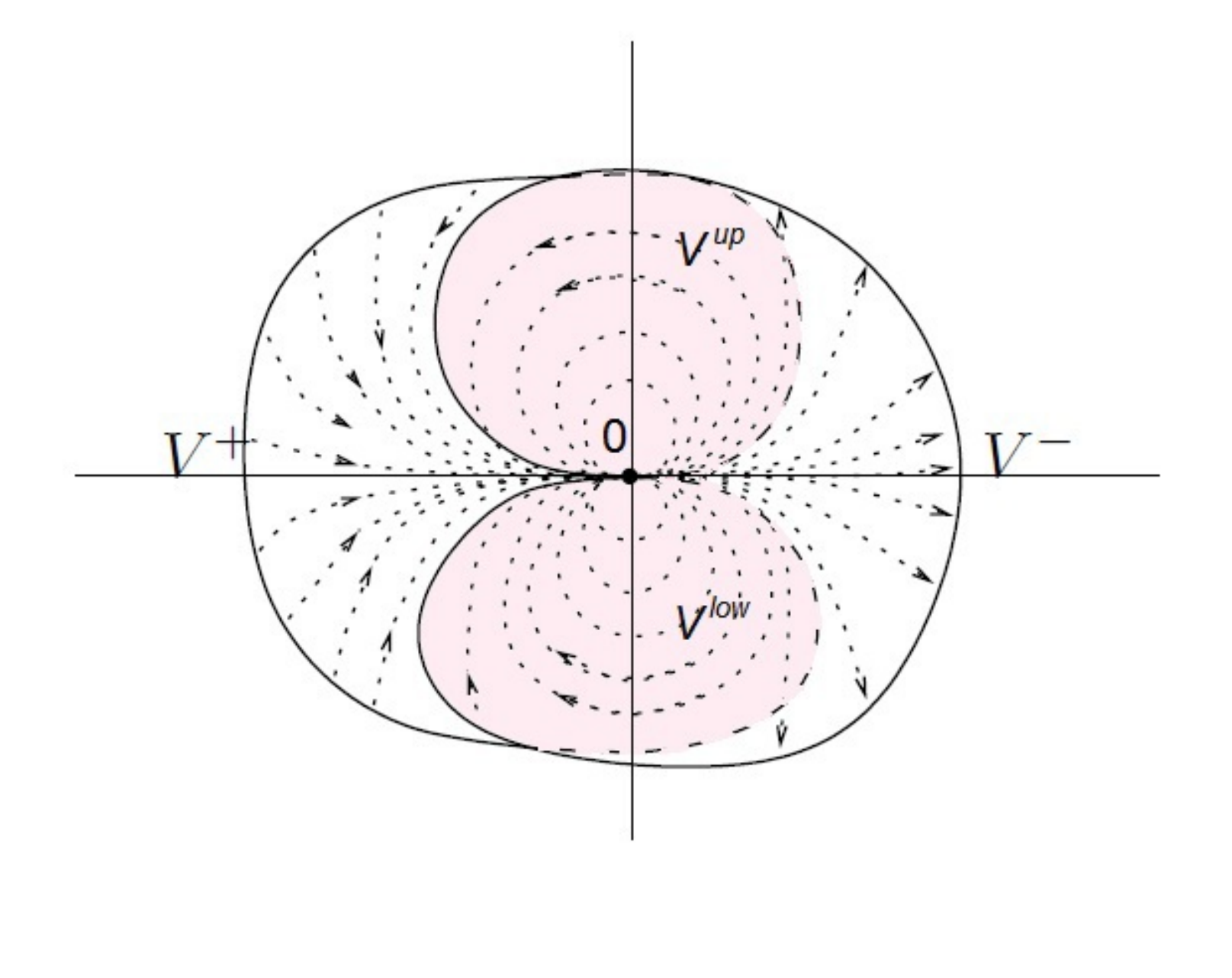}
  \vspace{-0.7cm}
  \caption{{\small Orbits of $f(z)=z+z^2+z^3+o(z^3)$ locally around the origin, Figure 10 in \cite{loray}}.}\label{orb1}
  \end{center}
\end{figure}
We denote intersections of petals above and below the real axis by $V^{up}$, $V^{low}$ respectively. Near the origin, these sets consist of \emph{closed orbits}.
\bigskip

To avoid confusion with the notion of \emph{sectors}\footnote{We call a \emph{sector} any open set between two rays emanating from the origin, of finite or infinite radius.}, we explain here the term \emph{petal}. It is an open bounded set $V$ \emph{in the form of petal}: it is contained in some sector of opening say $(\theta_1,\theta_2)$ and, for every $\varepsilon>0$, there exists $R_\varepsilon>0$, such that the sector of radius $R_\varepsilon$ and of opening $(\theta_1+\varepsilon,\theta_2-\varepsilon)$ is a subset of $V$. The boundary of the petal is tangent to the directions $\theta_1$ and $\theta_2$ at the origin.

\smallskip
Petals in the Leau-Fatou flower are obtained as \emph{invariant sets} for the dynamics: attracting petals $V_+$ are invariant sets for $f$ ($f(V_+)\subset V_+$), and repelling petals $V_-$ are invariant sets for $f^{-1}$ ($f^{-1}(V_-)\subset V_-$). 

\begin{remark}[Leau construction of petals, from Section A.6, \cite{loraypre} or Leau Theorem 2.3.1 in \cite{loray}]\label{poslije}
We explain here how we construct invariant petals $V_+$ and $V_-$ of opening angles $2\pi$, centered at positive and negative real axis respectively, in case of a diffeomorphism $f(z)=z+z^2+z^3+o(z^3)$, as shown in Figure~\ref{orb1}. 

We first derive $V_+$ such that $f(V_+)\subset V_+$. Let $R_0$ be the radius of convergence of $f$. Inside the discs of radii $0<R<R_0$, we can find $f$-invariant petals $V_+^R$ of opening angle less than, but tending to $2\pi$, as $R\to 0$. To show that, we work at infinity. Instead of considering $f(z)$ at $z=0$, we \emph{send} $f$ to infinity and get $$F(w)=-\frac{1}{f(-\frac{1}{w})}=w+1+\frac{1}{w^2}h\left(-1/w\right).$$ Here, $h(w)$ is bounded by an uniform constant $C$ on the complement of the disc $K(0,1/R_0)$. For $R>0$, we search for the set of all $w$ in the complement of disc $K(0,1/R)$, such that the whole trajectory $F^{\circ n}(w)$ remains in this set. Using $|\frac{1}{w^2}h(-1/w)|\leq C R^2$, it can be seen that positive, infinite trajectory of $F$, with initial point $w$, stays in the infinite sector $S_w$ with basepoint $w$, horizontal central line and of opening $2\arcsin (C R^2)$. It is left to find the greatest subset of the complement of $K(0,1/R)$, such that, for every point $w$, its corresponding sector $S_w$ remains inside this set. Inverting the set at zero, we get a petal $V_R^+$ with origin as the basepoint, centered at the negative real axis, and of opening smaller, but tending to $2\pi$, as $R\to 0$. By construction, it holds that $f(V_R^+)\subset V_R^+$. Finally, we can take our invariant set $V_+$ to be the union of all sets $V_R^+$, as $0<R<R_0$. It can be seen that it is again a petal, centered at the negative real axis, and of opening $2\pi$. 

\noindent The same can be repeated for $V_-$, invariant under $f^{-1}$. 

Furthermore, on the inverse image of each constructed subpetal $V_+^R$,  it holds that $|F^{\circ n}(w)|\geq 1/R+n\cdot CR^2\geq n\cdot CR^2.$ This can be seen by induction. That is, for each subpetal $V_+^R$, there exists a constant $C_R>0$, such that \begin{equation}\label{hhi}|f^{\circ n}(z)|\leq \frac{C_R}{n},\ z\in V_+^R.\end{equation}
On the other hand, we cannot find an uniform constant $C>0$ on the whole petal $V_+$, since $C_R\to\infty$, as $R\to 0$.

Note that we can extend $V_+$ and $V_-$ to \emph{maximal invariant petals}, which contain all closed orbits, and whose intersections consist exactly of closed orbits. We can divide them in the same manner into subpetals $V_R^+,\ R>0$, such that \eqref{hhi} holds. In the sequel, let $V_+$ and $V_-$ denote the maximal Fatou petals.  
\end{remark}

\bigskip
We do not provide an answer to the above question about analytic classification. We investigate analyticity properties of complex measures of $\varepsilon$-neighborhoods of orbits in Section~\ref{fourtwo}. This leads us to studying analyticity of solutions of special difference equations for $f$, which we call the \emph{$k$-Abel equations}, $k\in\mathbb{N}$, in Section~\ref{fourthree}. The analytic class of a diffeomorphism $f$ is determined by sectorial solutions of the 0-Abel equation for $f$, as we will explain in Section~\ref{fourone}. On the other hand, complex measures of $\varepsilon$-neighborhoods of orbits of $f$ are related to sectorial solutions of the 1-Abel equation for $f$. In Section~\ref{fourfour}, we show that the classes of diffeomorphisms derived using solutions of 0-Abel and of 1-Abel equation are not close to each other. Rather than that, they are in a transversal position. Considering complex measures of $\varepsilon$-neighborhoods of orbits from this viewpoint thus gives us no information about the analytic class of a diffeomorphism, but provides some interesting classifications.
\bigskip

\noindent We introduce here a few definitions and the \emph{Borel-Laplace summation technique} that will be used in the sequel.
\smallskip

Let $S^f(z)$ be the orbit of a diffeomorphism \eqref{diffe} with initial point $z$ belonging to an attracting petal $V_+$. In case $k>1$, according to Subsection \ref{threetwotwo}, we have the following asymptotic development of the complex measure of the $\varepsilon$-neighborhood of $S^f(z)$, as $\varepsilon\to 0$: 
\begin{align}\label{asyi}
\widetilde{A^\mathbb{C}}(S^f(z)_\varepsilon)=&q_1\varepsilon^{1+\frac{2}{k+1}}+q_2\varepsilon^{1+\frac{3}{k+1}}+\ldots+q_{k-1}\varepsilon^{1+\frac{k}{k+1}}+q_k\varepsilon^{2}\log\varepsilon+\\+H^{f,V_+}(z)\varepsilon^2+
&q_{k+1}\varepsilon^{2+\frac{1}{k+1}}\log\varepsilon+R(z,\varepsilon),\ R(z,\varepsilon)=O(\varepsilon^{2+\frac{1}{k+1}}),\ z\in V_+.\nonumber
\end{align}
Due to the modification in the definition of complex measure with respect to the definition of directed area, see Definitions~\ref{dir} and \ref{dir1}, the exponents from the development for the directed area \eqref{asy} are shifted by $\frac{1}{k+1}$, but the proof is essentially the same. Similarly, in the boundary case $k=1$, we have the development:
\begin{align*}
\widetilde{A^\mathbb{C}}(S^f(z)_\varepsilon)=q_1&\varepsilon^2\log\varepsilon+H^{f,V_+}(z)\varepsilon^2+\\
&+q_2\varepsilon^{\frac{5}{2}}\log\varepsilon+R(z,\varepsilon),\ R(z,\varepsilon)=O(\varepsilon^\frac{5}{2}),\ z\in V_+.\nonumber
\end{align*}
Here, $q_1,q_2,\ldots,q_{k+1}$ are functions of (finitely many) coefficients of $f$ and \emph{do not depend on the initial point}. The coefficient $H^{f,V_+}(z)$ is the first coefficient dependent on the initial point $z$. It is a well-defined function of $z$ on $V_+$. 

For orbits $S^{f^{-1}}(z)$ of the inverse diffeomorphism $f^{-1}$, with initial point $z$ belonging to a repelling petal $V_-$, a similar development is valid. In the same way, we get the function $H^{f^{-1},V_-}(z)$, $z\in V_-$, as the first coefficient dependent on the initial point.

\begin{definition}[The principal initial point dependent parts]\label{princpart}
\emph{The principal initial point dependent part of the complex measure of $\varepsilon$-neighborhoods of orbits of $f$ in $V_+$} is the first coefficient $H^{f,V_+}(z)$ in the development~\eqref{asyi} depending on the initial point $z$, regarded as a function of $z\in V_+$.
\end{definition}
By abuse, for the sake of simplicity, we will call $H^{f,V_+}(z)$ only \emph{the principal part of the complex measure for $f$ on $V_+$}. Naturally, on a repelling petal $V_-$, we define \emph{the principal part of the complex measure for $f^{-1}$ on $V_-$}, denoted $H^{f^{-1},V_-}$, as the first coefficient dependent on the initial point in the development \eqref{asyi} for the orbit $S^{f^{-1}}(z),\ z\in V_-,$ of the inverse diffeomorphism $f^{-1}$.
\medskip

We define now \emph{generalized Abel equations} for parabolic diffeomorphisms.
Let us recall from e.g. \cite{sauzin} or \cite{loray} that the difference equation
\begin{equation}\label{abel}
H(f(z))-H(z)=1
\end{equation}
is called the \emph{Abel equation} for a diffeomorphism $f$. We generalize this notion.

\begin{definition}[Generalized Abel equation for a diffeomorphism $f$]\label{propabelgen}
\emph{A generalized Abel equation for a diffeomorphism $f$ with the right-hand side $g\in\mathbb{C}\{z\},\ g\equiv\!\!\!\!\!/\ 0$,} is the difference equation
\begin{equation}\label{abelgen}
H(f(z))-H(z)=g(z),
\end{equation}
in some neighborhood of $z=0$. The function $H(z)$ that satisfies \eqref{abelgen} on some domain is called \emph{a solution of the generalized Abel equation on the given domain}. In particular, if $$g(z)=Cz^k,\ k\in\mathbb{N}_0,\ C\in\mathbb{C}^*,$$ we call the equation \eqref{abelgen} the \emph{$k$-Abel equation for $f$}. 
\end{definition}
Such equations have already been mentioned in \cite[Section A.6]{loraypre}, but were not given a name. Note that the $0$-Abel equation ($g(z)\equiv 1$) is the standard Abel equation \eqref{abel}.

\subsubsection{Borel-Laplace summation}
We describe the \emph{Borel-Laplace summation} technique that can be used for recovering sectorial summability of divergent formal series. We use it in Examples in Subsection~\ref{fourfourone}. The technique in more detail and the following definitions can be looked up in e.g. \cite{canalis}, \cite{candelpergher}, \cite{dudko}, \cite{jpr}, \cite{sauzin}. 

We state first the definitions of Borel and Laplace transform. Note that we work all the time at infinity, but we can pass to the origin simply by inverting the variable.
Suppose $\widehat \varphi(z)$ is a formal series at infinity ($z\approx \infty$), without the constant term:
$$\widehat \varphi(z)=\sum_{n=0}^{\infty}c_n z^{-n-1}\in z^{-1}\mathbb{C}[[z^{-1}]].$$ 
Its \emph{formal Borel transform} is the linear operator $\mathcal{B}:z^{-1}\mathbb{C}[[z^{-1}]]\to \mathbb{C}[[\xi]]$, attributing to the formal series $\widehat \varphi(z)$ at infinity the formal series $\mathcal{B}\widehat \varphi(\xi)$ at zero, by the following formula:
$$
\mathcal{B}\left(\sum_{n\geq 0}c_n z^{-n-1}\right)=\sum_{n\geq 0}c_n \frac{\xi^n}{n!},
$$
see Definition 1 in \cite{sauzin}.
\smallskip

It can be easily seen that if (and only if) $\varphi(z)\in z^{-1}\mathbb{C}\{z^{-1}\}$ is convergent at infinity, then $\mathcal{B}\widehat \varphi(\xi)$ is an entire function. It is moreover exponentially bounded in every direction -- for every $A>\frac{1}{R_0}$, where $R_0$ is the radius of convergence of $\varphi(1/z)$, it holds that:
$$
|\mathcal{B}\widehat \varphi(\xi)|\leq C e^{A|\xi|}, \ \xi\in\mathbb{C}.
$$
We consider the class of formal series $\widehat\varphi(z)$ whose Borel transform is a convergent germ with finite radius of convergence, that can be extended to exponentially bounded analytic function to all rays emanating from the origin, except to (at most) finitely many on which it has singularities. This singular rays are called \emph{Stokes directions}. Such formal series appear in many natural problems. 
\smallskip

We define now a Laplace transform, see Definition 2 in \cite{sauzin}. Let $f(\xi)$ be a function analytic on some ray of direction $\theta$ emanating from the origin, $\{r e^{i\theta}|\ r>0\}$, and of bounded exponential type:
$$
|f(r e^{i\theta})|\leq C e^{Ar},\ r>0,
$$ 
for some constant $A>0$. The \emph{Laplace transform of $f$ in direction $\theta$} is a linear operator $\mathcal L$ defined by
$$
\mathcal L f(z)=\int_{0}^{\infty\cdot e^{i\theta}}f(\xi) e^{-z \xi} d\xi.
$$
The Laplace transformation of $f$, $\mathcal L f(z)$, is an analytic function on the half-plane $Re(z e^{i\theta})>A$, see Figure~\ref{BL}.
\begin{figure}[ht]
\begin{center}
  % replace aims_logo.eps by your figure file name
  \includegraphics[width=2.5in, trim={1cm 0cm 0cm 0cm}]{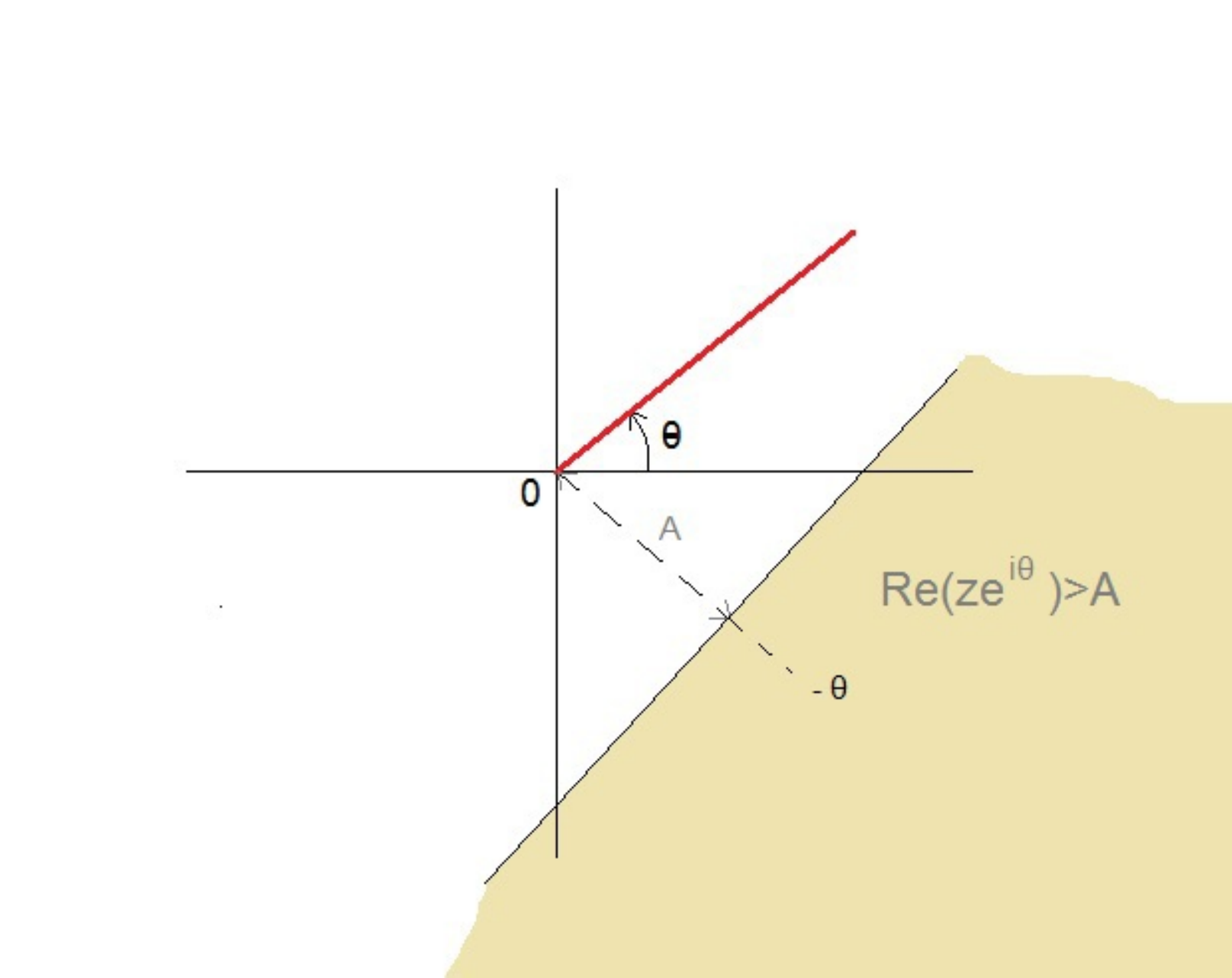}
  \caption{\small{The direction $\theta$ and the corresponding half-plane $Re(z e^{i\theta})>A$.}}\label{BL}
  \end{center}
\end{figure}
\smallskip

In a way, Laplace transform acts as the inverse of formal Borel transform. We explain now how application of Borel and then Laplace transform on some formal (divergent) series recovers analytic sums of the series on sectors. 

Let $\theta_0$ be some fixed direction emanating from the origin ($\theta_0$ denotes its angle) and $\varepsilon>0$. We say that the formal series $\widehat\varphi(z)$ is \emph{$1$-summable in arc of directions $I=(\theta_0-\varepsilon,\theta_0+\varepsilon)$}, if its Borel transform $\mathcal B\widehat\varphi(\xi)$ is a convergent germ and can be extended analytically to all directions in $I$, with continuous exponential bounds: there exist continuous, strictly positive functions $A(\theta), C(\theta)$, $\theta\in(\theta_0-\varepsilon,\theta_0+\varepsilon)$, such that it holds
$$
|\mathcal{B}\widehat\varphi(\xi)|\leq C(\theta) e^{A(\theta)|\xi|}, \text{ for all $\xi\in\mathbb{C}$ such that } Arg(\xi)=\theta.
$$
For more on $1$-summability, see e.g. \cite[2.3]{jpr} or \cite{dudko}.

In this case, using the definition above, we can apply the Laplace transform of $\mathcal{B}\widehat{\varphi}(\xi)$ in all directions in $I$, and we get that $\mathcal{L}\mathcal{B}\widehat{\varphi}(z)$ is an analytic function in the petal\footnote{\emph{Petal} at the origin, of opening $(\theta_1,\theta_2)$, is defined as an open set $U$, contained in a sector of opening $(\theta_1,\theta_2)$, such that, for every $\varepsilon>0$, there exists a subsector of finite radius of opening $(\theta_1+\varepsilon,\theta_2-\varepsilon)$ which is a subset of $U$. It is called the \emph{ouvert sectoriel} in \cite{loraypre}. The petal at infinity is the unbounded set obtained by inverting the petal at the origin.} $V$ of opening $(-\theta_0-\pi/2-\varepsilon,-\theta_0+\pi/2+\varepsilon)$ at infinity. Note that $V$ is of opening angle bigger than $\pi$, and bisected by $-\theta_0$. It can be shown that \emph{the asymptotic development of the analytic function $\mathcal{L}\mathcal{B}\widehat{\varphi}(z)$ on $V$, as $z\to\infty$, is exactly the original formal series $\widehat\varphi(z)$}. Moreover, for every subsector of $V$, there exist constants $C>0,\ M>0$, such that, for every $n\in\mathbb{N}$, it holds
\begin{equation}\label{jsum}
\left|\mathcal{L}\mathcal{B}\widehat{\varphi}(z)-\sum_{k=0}^{n-1}\frac{c_k}{z^{k+1}}\right|\leq C M^n n! |z|^n.
\end{equation}(The constants $C$ and $M$ do not depend on $n$).
\medskip

In the case that \eqref{jsum} holds, a function (here, $\mathcal{L}\mathcal{B}\widehat{\varphi}(z)$) is said to be the \emph{$1$-sum of $\widehat\varphi(z)$ on V}. Let us mention \emph{the Watson's uniqueness theorem}, see e.g. \cite[2.3]{jpr}, that states that a $1$-sum of a formal series on a sector of opening greater than $\pi$ is unique. In this sense, we consider the Borel-Laplace transformation $\mathcal L\mathcal B\widehat\varphi(z)$ as a sectorially analytic sum of the original formal series $\widehat\varphi(z)$. The number of sectors is determined by the number of Stokes directions (singular directions) of the Borel transformation $\mathcal{B}\widehat\varphi(\xi)$. The sectors overlap.
\smallskip

We illustrate the Borel-Laplace summation method on the simplest example of Euler's divergent series with only one Stokes direction.
\begin{example}[Borel-Laplace sum of Euler series, see \cite{candelpergher}]
Let $$\widehat \varphi(z)=\sum_{n=0}^\infty n! z^{-n-1}\text{\ \ (Euler series)}.$$ The series is divergent at $z=\infty$. After applying Borel transform, we get the convergent germ
$$
\mathcal{B}\widehat \varphi(\xi)=\frac{1}{1-\xi}.
$$ 
It can be extended to an analytic function, exponentially bounded with any $A>0$, in every direction except $\{\theta=0\}$. Applying Laplace transform in arc of directions $I_1=(0,\pi)$, we get an analytic function $\mathcal{L}^{I_1}\mathcal{B}\widehat \varphi(z)$ on sector $V_1$ of opening $(\frac{\pi}{2},\frac{5\pi}{2})$. Applying Laplace transform in arc of directions $I_2=(-\pi,0)$, we get an analytic function $\mathcal{L}^{I_2}\mathcal{B}\widehat \varphi(z)$ on sector $V_2$ of opening $(-\frac{\pi}{2},\frac{3\pi}{2})$. They are the sectorial $1$-sums of Euler series, and differ on the intersection of sectors $V_1\cap V_2=\{Re(z)>0\}\cup \{Re(z)<0\}$ by exponentially small differences:
\begin{align*}
&\mathcal{L}^{I_1}\mathcal{B}\widehat \varphi(z)-\mathcal{L}^{I_2}\mathcal{B}\widehat \varphi(z)=\int_{\infty\cdot e^{i\theta_2}}^{\infty\cdot e^{i\theta_1}}\frac{e^{-\xi z}}{1-\xi}d\xi=-2\pi i\cdot Res\left(\frac{e^{-\xi z}}{1-\xi},\xi=1\right)=2\pi i e^{-z},\ Re(z)>0,\\
&\mathcal{L}^{I_1}\mathcal{B}\widehat \varphi(z)-\mathcal{L}^{I_2}\mathcal{B}\widehat \varphi(z)=0,\ Re(z)<0.
\end{align*}
Here, $\theta_1>0$ and $\theta_2<0$.
\end{example}
\bigskip

At the end, we state some \emph{properties of the formal Borel transform}, see e.g. \cite{candelpergher}. For two formal series $\widehat{\varphi}(z),\ \widehat{\psi}(z)\in z^{-1}\mathbb{C}[[z^{-1}]]$, it holds that:
\begin{enumerate}
\item \emph{(translations)}\begin{equation}\label{propi}\mathcal{B}\big(\widehat{\varphi}\circ T_s\big)(\xi)=e^{-s\xi}\mathcal{B}\widehat{\varphi}(\xi),\quad \text{where } T_s(z)=z+s,\end{equation}
\item \emph{(formal products)}\quad $\mathcal{B}(\widehat{\varphi}\cdot\widehat{\psi})(\xi)=\mathcal{B}\widehat{\varphi}(\xi)*\footnote{a convolution} \mathcal{B}\widehat{\psi}(\xi),$
\item \emph{(formal derivatives (term by term))}\quad $\mathcal{B}(\frac{d}{dz}\widehat{\varphi})(\xi)=-\xi\mathcal{B}\widehat{\varphi}(\xi).$
\end{enumerate}
\section{Ecalle-Voronin moduli of analytic classification}\label{fourone}

We recall in short the well-known results on analytic classification of parabolic diffeomorphisms, using the Ecalle-Voronin moduli or the so-called horn maps. For more details, see the original papers of Ecalle and Voronin \cite{ecalle,voronin} or a good overview in e.g. Loray \cite[Chapter 2]{loray} or in the thesis of Dudko \cite[Section 1.1.2]{dudko}.

For simplicity, we consider prenormalized parabolic germs of the formal type $(k=1,\lambda=0)$ and describe their analytic classes. For dynamics, see Figure~\ref{orb1} above. They are all formally equivalent to the formal normal form $f_0(z)=\frac{z}{1-z}$. As we have explained in Section~\ref{threezero}, this means that there exists a formal series $\widehat\varphi\in z+z^2\mathbb{C}[[z]]$ such that
\begin{equation}\label{konj}
\widehat\varphi \circ f=f_0\circ \widehat\varphi.
\end{equation}
If $\varphi(z)$ converges, then $f$ belongs to the analytic class of $f_0$. However, the formal series may be divergent and only sectorially analytic, which provides other analytic classes. It can be shown by Borel-Laplace summation technique that the Stokes directions for this problem are $\theta=\frac{\pi}{2}$ and $\theta=-\frac{\pi}{2}$, with equidistant singularities at $2\pi i \mathbb{Z}^*$. The proof can be seen in \cite{ecalle} or in e.g. \cite{candelpergher}. This recovers the analytic $1$-sums of divergent $\widehat\varphi(z)$ on the Fatou petals $V_+$ and $V_-$. The differences of analytic solutions on intersection of petals are \emph{exponentially small}. Comparing the solutions on the intersections of petals in an appropriate way reveals the nature of singularities and provides the analytic classes. We explain ways of comparing them here, using their \emph{compositions} or their \emph{differences}.

The proof of sectorial analyticity by Ecalle, around the year $1980$, was made by transforming the conjugacy equation~\eqref{konj} to the trivialisation equation. Putting \begin{equation}\label{vezza}\widehat{\Psi}(z)=\Psi_0\circ \widehat\varphi(z),\end{equation} the formal conjugacy equation \eqref{konj} is transformed to the \emph{trivialisation equation} or the \emph{Abel equation} for $f$:
\begin{equation}\label{triviabel}
\widehat \Psi(f(z))-\widehat \Psi(z)=1.
\end{equation}
Here, $\Psi_0(z)$ is a solution of trivialisation equation \eqref{triviabel} for the formal normal form $f_0(z)$, which is easily computed as global and equal to, up to addition of a complex constant, $$\footnote{In more general case of germs of formal class $(k=1,\lambda)$, $\Psi_0(z)=-\frac{1}{z}+\frac{\lambda}{2\pi i}Log(z).$}\Psi_0(z)=-\frac{1}{z}.$$

It was shown already by Leau and Fatou at the end of the $19^{th}$ century that there exist unique analytic sectorial trivialisation functions for $f$, with an asymptotic development of the type $-1/z+\mathbb{C}[[z]]$, defined on the petals $V_+$ and $V_-$. Their constructive proof is described in Proposition~\ref{formal} in Section~\ref{fourthree} in more generality. Later, in the $20^{th}$ century, the same proof was made by Ecalle using the Borel-Laplace technique. We denote the sectorial solutions by $\Psi_+(z),\ z\in V_+,$ and $\Psi_-(z),\ z\in V_-.$ Sometimes they are also called the \emph{Fatou coordinates} for $f$. They are unique up to an arbitrary chosen additive constant. They transform orbits on petals simply in translations by $+1$. 

Therefore, as shown in Figure~\ref{ecalle}, the quotient spaces of orbits on $V_+$ and on $V_-$ can be represented as two Riemann spheres, simply by composition of sectorial trivializations and the exponential function. Thus whole orbits become just points on spheres. The moduli of analytic classification are obtained by relating points of both spheres which correspond to the same orbit, for orbits that lie in the intersections of petals. They are given by two diffeomorphisms, at $t=0$ and at $t=\infty$, of the Riemann sphere, $\varphi_0\in  \text{Diff\ }(\overline{\mathbb{C}},0)$ and $\varphi_\infty\in \text{Diff\ }(\overline{\mathbb{C}},\infty)$,  
\begin{align}\label{mod}
\varphi_0(t)&=e^{{-2\pi i}\Psi_-\circ (\Psi_+)^{-1}\left(-\frac{Log t}{2\pi i}\right)},\ \ t\approx 0,\\
\varphi_\infty(t)&=e^{{-2\pi i}\Psi_-\circ (\Psi_+)^{-1}\left(-\frac{Log t}{2\pi i}\right)},\ t\approx\infty.\nonumber
\end{align}

\begin{figure}[ht]
\begin{center}
	\vspace{-36cm}
  % replace aims_logo.eps by your figure file name
  \includegraphics[scale=1.5, trim={-2.5cm 0cm 0cm 0cm}]{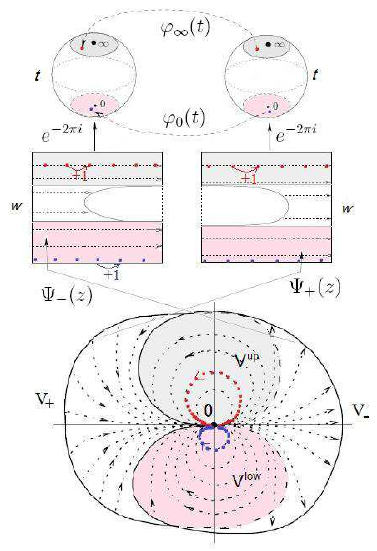}
  \caption{\small{Construction of the Ecalle-Voronin moduli from the sectorial trivialisations, as a pair of diffeomorphisms on spheres. The same colors denote the corresponding areas and the corresponding orbits (Figure 10 from \cite{loray}, adapted). }}\label{ecalle}
  \end{center}
\end{figure}

It can be checked that for a diffeomorphisms $f$ of the simplest formal type $(k=1,\lambda=0)$, it holds that\footnote{We say that $f$ is a germ of a diffeomorphism at $\infty$, with $\infty$ as a fixed point, if the inverted germ $g(z)=\frac{1}{f(1/z)}$ is a germ of a diffeomorphism at zero, with zero as a fixed point. In this notation, the multiplier at infinity means $f'(\infty)=g'(0).$}$$\varphi_0'(0)\cdot \varphi_\infty'(\infty)=1.$$ Otherwise, if the formal invariant $\lambda\neq 0$, it holds that $\varphi_0'(0)\cdot \varphi_\infty'(\infty)=e^{-2\pi i \lambda}.$

\smallskip
The diffeomorphisms constructed in \eqref{mod} are connected to the analytic class of $f(z)$ by the following theorem:
\begin{theorem}[\cite{ecalle,voronin} or Theorem 17 \cite{dudko}]\label{ev}
Two germs of diffeomorphisms $f$ and $g$ of multiplicity 2 are analytically conjugated if and only if there exist constants $a,\ b\in\mathbb{C}^*$ such that \begin{equation}\label{upto}\varphi_0^f(t)=a\varphi_0^g(b t)\text{ and }\varphi_\infty^f(t)=a\varphi_\infty^g(b t).\end{equation}
\end{theorem}
\noindent The allowed multiplications by constants come from the fact that sectorial trivialisations may be chosen up to an arbitrary chosen additive constants. It is the only freedom we have in the choice of sectorial trivialisation functions.

\smallskip
The pair of germs $(\varphi_0,\varphi_\infty)$ constructed from $f$ in this manner at the poles of Riemann sphere, up to multiplications by nonzero constant in \eqref{upto}, are called the \emph{Ecalle-Voronin moduli} or \emph{horn maps} for a diffeomorphism $f$. They were constructed independently by Ecalle and Voronin around the year $1980$.
\smallskip

The converse also holds. For any two germs $\varphi_0\in  \text{Diff\ }(\overline{\mathbb{C}},0)$ and $\varphi_\infty\in \text{Diff\ }(\overline{\mathbb{C}},\infty)$ of the Riemann sphere, such that $(\varphi_0)'(0)\cdot (\varphi_\infty)'(\infty)=1$, there exists a parabolic germ of the formal type $(k=1,\lambda=0)$ whose horn maps are given by $(\varphi_0,\varphi_\infty)$. Otherwise, for pairs of germs such that $(\varphi_0)'(0)\cdot (\varphi_\infty)'(\infty)=e^{-2\pi i\lambda}$, for some $\lambda\in\mathbb{C}$, there exists a parabolic germ of formal type $(k=1,\lambda)$, with horn maps equal to this pair.
\medskip

To conclude, \emph{there exists a bijective correspondence between all analytic classes of diffeomorphisms of formal type $(k=1,\lambda=0)$ and all possible pairs of diffeomorphisms $($up to multiplications \eqref{upto}$)$ $$\Big(\varphi_0\in  \text{Diff\ }(\overline{\mathbb{C}},0),\ \varphi_\infty\in \text{Diff\ }(\overline{\mathbb{C}},\infty)\Big),$$ such that $(\varphi_0)'(0)\cdot (\varphi_\infty)'(\infty)=1$}.

For example, consider the simplest class of diffeomorphisms analytically conjugated to the model $f_0(z)=Exp\Big(z^2\frac{d}{dz}\Big)=\frac{z}{1-z}$. That is, the class of the diffeomorphisms that are \emph{time-one maps of the flows of vector fields}. The class is described by horn maps equal to the identity (up to multiplication by some nonzero complex constant). Indeed, by \eqref{vezza}, a diffeomorphism is analytically conjugated to $f_0$ if and only if there exists a global trivialisation function $\Psi(z)$, that is, if and only if $\Psi_+\equiv \Psi_-$ on intersections of petals.
\medskip

\subsubsection{Fourier representation of the Ecalle-Voronin moduli}
The analytic class of a diffeomorphism $f$ can, instead by a pair of diffeomorphisms, be described using two infinite sequences of complex numbers. This method consists in analysing the exponentially small \emph{differences} of sectorial trivialisations on intersections of petals, \emph{instead of compositions} as above. By $V^{up}$ we denote the part of $V_+\cap V_-$ above the real axis, and by $V^{low}$ the part below the real axis, as shown in Figure~\ref{ecalle}.
If we subtract equations \eqref{triviabel} for trivialisation functions $\Psi_+(z)$ and $\Psi_-(z)$ on the intersections of petals $V^{up}\cup V^{low}$, the difference is constant along the closed orbits in $V^{up}$ and in $V^{low}$:
\begin{equation}\label{corb}
\Psi_+(z)-\Psi_-(z)=\Psi_+(f(z))-\Psi_-(f(z)),\ z\in V^{up}\cup V^{low}.
\end{equation}
Its composition with $\Psi_+^{-1}(w)$, $$(\Psi_- -\Psi_+)\circ \Psi_+^{-1}(w),$$ is therefore $1$-periodic on $\Psi_+(V^{up})=\{Im(w)>M\}$ and on $\Psi_+(V^{low})=\{Im(w)<-M\}$, for some big $M>0$. 
Therefore it can be expanded in Fourier series on both domains:
\begin{align}\label{fur}
(\Psi_+ -\Psi_-)\circ \Psi_+^{-1}(w)=&\sum_{k=0}^{\infty} A_k^{up} e^{2\pi i k w},\ \ Im(w)>M,\nonumber \\
(\Psi_+ -\Psi_-)\circ \Psi_+^{-1}(w)=&\sum_{k=0}^{\infty} A_k^{low} e^{-2\pi i k w},\ \ Im(w)<-M.
\end{align}
We have then that $$A_0^{up}-A_0^{low}=\lambda,$$ for diffeomorphisms of formal type $(k=1,\lambda)$. Specially, $A_0^{up}-A_0^{low}=0$ for diffeomorphisms of formal type $(k=1,\lambda=0)$. 
\medskip

The relation between the sequences $(A_k^{low})_{k\in\mathbb{N}_0},\ (A_k^{up})_{k\in\mathbb{N}_0}$ and the analytic class of $f$ is given in the following theorem.
\begin{theorem}[\cite{ecalle} or Theorem 19 \cite{dudko}]\label{fr}
Two germs of diffeomorphisms $f$ and $g$ of multiplicity 2 are analytically conjugated if and only if their Fourier coefficients defined by \eqref{fur} are related in the following manner:
\begin{align}\label{multkoef}
A_0^{up}&(f)-A_0^{low}(f)=A_0^{up}(g)-A_0^{low}(g),\nonumber\\
&A_k^{up}(f)=b^k A_k^{up}(g),\ A_k^{low}(f)=b^{-k} A_k^{low}(g),\ k\in\mathbb{N},\ \ \text{for some $b\in\mathbb{C}^*$.}
\end{align}
\end{theorem}
\noindent We must admit above multiplications in the coefficients, due to uniqueness of sectorial trivialisations for $f$ only up to arbitrary additive constants. For each diffeomorphism, the coefficients may be exchanged with $A_0^{up}=A_0^{up}+a,\ A_0^{low}=A_0^{low}+a,\ a\in\mathbb{C}$, $A_k^{up}=b^k A_k^{up},\ A_k^{low}=b^{-k}A_k^{low},$ $b\in\mathbb{C}^*$.
\medskip

Furthermore, if $(A_k^{up})_{k\in\mathbb{N}_0}, (A_k^{low})_{k\in\mathbb{N}_0}$ are any two sequences of coefficients, such that the corresponding Fourier series from \eqref{fur} converge and such that $A_0^{up}-A_0^{low}=\lambda$, then there exists a germ of the formal class $(k=1,\lambda)$ which realizes these sequences.
\medskip

To conclude, \emph{there exists a bijective correspondence between the analytic classes inside the formal class $(k=1,\lambda=0)$ and all possible sequences of complex coefficients $A_0^{up}=A_0^{low}=0$, $(A_k^{low,up})_{k\in\mathbb{N}}$, up to additions and multiplications from \eqref{multkoef}, for which the series in \eqref{fur} converge}. 

For example, the model analytic class is characterised by 
$A_0^{up}=A_0^{low}=a,$ for any $a\in\mathbb{C}$, and $A_k^{up}=A_k^{low}=0,\ k\in\mathbb{N}$. That is,
$$
\Psi_+(z)-\Psi_-(z)=a,\ z\in V^{up}\cup V^{low}.
$$

\medskip

\subsubsection{Representation of the Ecalle-Voronin moduli as $1$-cocycles of the trivialisation series, lifted to the space of orbits}
We describe yet another way of expressing the moduli through differences of trivialisation functions on the intersections of petals, that is equivalent to Fourier representation. For ideas and definitions, see for example \cite[Sections A.4, A.5, A.6]{loraypre}. We could not find this approach to the Ecalle-Voronin moduli explicitely stated in the literature, it is though implicit in the Fourier coefficient approach. We precise it here, since it is the most convenient approach for this work. Indeed, we will use the same line of thought in Section~\ref{fourfour} to define new classifications imposed by generalized Abel equations, in the same way as the analytic classification was imposed here by Abel equation.
\medskip

We simplify a little the following definitions from \cite[A.4, A.5]{loraypre}, restricting them to our situation, but they are otherwise the same. 

Let us consider a formal series $\widehat{H}(z)\in \mathbb{C}[[z]]$. We call it \emph{$1$-summable, with Stokes directions at imaginary axes,} if it is $1$-summable (in the sense defined before) in arcs of directions $I_1=(-\pi/2,\pi/2)$ and $I_2=(\pi/2,3\pi/2)$. Equivalently, if there exist two analytic functions $H_+(z)$ and $H_-(z)$, defined on some petals\footnote{Only the opening angle and the central direction of a petal is important, not the size and the shape. Any two petals with the same opening and the same central direction are identified.} $V_+$ and $V_-$ of opening $2\pi$, centered at $\theta=\pi$ and $\theta=0$ respectively, which are $1$-sums of series $\widehat{H}(z)$ on respective petals. See \cite[Section 2.3]{jpr}. We denote the set of all such series by $\mathbb{C}\{z\}_1$. This notation is taken from \cite{loraypre}.

Let $V^{up}$ and  $V^{low}$ denote some petals of opening $\pi$ and centered at $\theta=\pi/2$ and $\theta=-\pi/2$ respectively. We call \emph{$1$-cocycle, with Stokes directions at imaginary axes,} the pair $\big(h(z),k(z)\big)$ of analytic functions on petals\footnote{Two cocycles that are defined on the petals of the same opening and central direction, and agree on their intersections, are identified.} $V^{up}$ and $V^{low}$ respectively, with an exponential decrease:
$$
|h(z)|<C e^{-\frac{A}{|z|}},\ z\in V^{up},\quad |k(z)|< C e^{-\frac{A}{|z|}},\ z\in V^{low},\quad C,\ A>0.
$$
We denote the set of all such $1$-cocycles by $\mathcal H^1$.
\smallskip

Each $1$-summable formal series with $1$-sums $H_+$ on $V_+$ and $H_-$ on $V_-$ defines a $1$-cocycle $(h(z),k(z))$ by 
$$
h(z)=H_+(z)-H_-(z),\ z\in V^{up};\quad  k(z)=H_-(z)-H_+(z),\ z\in V^{low},
$$
where $V^{up}$ and $V^{low}$ are the intersections of petals $V_+$ and $V_-$. In the above manner, we can define the mapping $$\mathbb{C}\{z\}_1\longrightarrow \mathcal H^1,$$
which is a morphism of additive groups. The question of \emph{bijectivity} between $1$-summable series and $1$-cocycles is solved in the following theorem: 
\begin{theorem}[Ramis-Sibuya theorem \cite{ram,sib}, \emph{ Th\' eor\` eme}, p.23 in\cite{loraypre} or Theorem 2.5 in \cite{canalis}]\label{koji}
The mapping $\mathbb{C}\{z\}_1\longrightarrow \mathcal H^1$ is \emph{surjective}. Moreover, it is \emph{bijective} on the quotient space $\mathbb{C}\{z\}_1 / \mathbb{C}\{z\}$. 
\end{theorem} 
That means that to each $1$-cocycle corresponds a unique $1$-summable formal series, up to addition of a convergent series.

\bigskip
\smallskip

Preliminaries being done, we derive now Ecalle-Voronin moduli from $1$-cocycles of the formal trivialisation series $\widehat\Psi(z)$. After subtracting the first term $-1/z$, $\widehat\Psi(z)$ is $1$-summable, and defines the cocycle $\big(h(z), k(z)\big)$ as described above: 
\begin{align*}
h(z)=\Psi_+(z)-\Psi_-(z),\ z\in V^{up};\quad
k(z)=\Psi_-(z)-\Psi_+(z),\ z\in V^{low}.
\end{align*}

We exploit now the additional fact that $\widehat\Psi(z)$ satisfies Abel equation \eqref{triviabel}. Therefore, $h(z)$ and $k(z)$ are \emph{constant along the closed orbits} in $V^{up}$ and $V^{low}$, see \eqref{corb}. The cocycle $(h(z),k(z))$ can thus be \emph{lifted to the space of orbits} to a well-defined function. For representation of the space of orbits, we fix one of the two trivialisation functions, say $\Psi_+(z)$, up to an additive constant. The space of orbits is then a Riemann sphere in the variable $t=e^{-2\pi i \Psi_+(z)}$. As before, see Figure~\ref{ecalle}, the closed orbits in $V^{up}$ lift to the punctured neighborhood of the pole $t=\infty$ and the closed orbits in $V^{low}$ to the punctured neighborhood of the pole $t=0$. We thus lift $(h(z),k(z))$ to a space of orbits represented by $\Psi_+$ through a pair of germs $(g_\infty(t),g_0(t))$ around $t=\infty$ and $t=0$ of Riemann sphere:
$$
h(z)=g_\infty(e^{-2\pi i \Psi_+(z)}),\ z\in V^{up};\quad k(z)=g_0(e^{-2\pi i \Psi_+(z)}),\ z\in V^{low}.
$$

The relation with the Fourier representation is the following. We can rewrite \eqref{fur} as
\begin{align*}
h(z)=\sum_{k=0}^{\infty} A_k^{up} e^{2\pi i k \Psi_+(z)},\ \ z\in V^{up}; \quad
k(z)=\sum_{k=0}^{\infty} -A_k^{low} e^{-2\pi i k \Psi_+(z)},\ \ z\in V^{low}.
\end{align*}
Therefore,
\begin{align*}
g_\infty(t)=\sum_{k=0}^{\infty} A_k^{up} t^{-k},\ \ t\approx \infty;\quad
g_0(t)=\sum_{k=0}^{\infty} -A_k^{low} t^k,\ \ t\approx 0.
\end{align*}
Additionally, inverting $g_\infty$ as $g_\infty(t)=g_\infty(1/t)$, it becomes also a germ at $t=0$:
\begin{align*}
g_\infty(t)=\sum_{k=0}^{\infty} A_k^{up} t^{k},\ \ t\approx 0;\quad
g_0(t)=\sum_{k=0}^{\infty} -A_k^{low} t^k,\ \ t\approx 0.
\end{align*}
The germs are analytic at punctured neighborhoods of 0, since they are obtained simply by composing differences of two holomorphic functions at $V^{up}$ and $V^{low}$ with the logarithmic function. Furthermore, they can be extended continuously to $0$ by $g_\infty(0)= A_0^{up}$ and $g_0(0)=-A_0^{low}$. This extension is analytic at $t=0$ by Riemann's characterization of removable singularities. Therefore we get a \emph{pair of analytic germs} $\big(g_\infty(t),g_0(t)\big)$ at $t=0$ of Riemann sphere (equivalently, at the origin). Note that $A_0^{up}=A_0^{low}$ correspond exactly to the difference of the constant terms chosen in sectorial trivialisation functions $\Psi_+$ and $\Psi_-$, which can be chosen freely. Furthermore, note that $g_\infty(t)$ and $g_0(t)$ are not necessarily diffeomorphisms.

\medskip

We now reformulate Theorem~\ref{fr} considering the pair of analytic germs $(g_\infty(t),g_0(t))$ at zero obtained in the above manner as the Ecalle-Voronin modulus of $f$.  

Before, we identify two pairs of germs, $(g_\infty^1(t),g_0^1(t))$ and $(g_\infty^2(t),g_0^2(t))$, if it holds that:
\begin{align}\label{idii}
g_\infty^1(0)=g_\infty^2(0)+a,&\quad g_0^1(0)=g_0^2(0)-a,\\
g_\infty^1(t)=g_\infty^2(bt),&\quad g_0^1(t)=g_0^2(t/b),\nonumber
\end{align}
for $a\in\mathbb{C}$ and $b\in\mathbb{C}^*$.
This corresponds to choosing trivialisation functions up to an additive constant.

\begin{theorem}[Ecalle-Voronin moduli as 1-cocycles of trivialisation series lifted to orbit space]\label{fr1}
Two germs of the formal class $(k=1,\lambda=0)$ are analytically conjugated if and only if $1$-cocycles generated by their trivialisations and lifted to Riemann spheres of their attracting sectors, give the same pair of analytic germs $\big(g_\infty(t),g_0(t)\big)$ at zero, up to identifications \eqref{idii}. It holds that $g_\infty(0)+g_0(0)=0$. The same holds using Riemann spheres for repelling sectors.
\end{theorem}
\begin{proof} This is only a reformulation of Theorem~\ref{fr}.
\end{proof}

\noindent  On the contrary, for any pair of analytic germs  $\big(g_\infty(t),g_0(t)\big)$ at zero, up to identifications \eqref{idii}, such that $g_\infty(0)+g_0(0)=0$, there exists a germ of the formal type $(k=1,\lambda=0)$ which realizes this pair in the above described manner. The same can be concluded using trivialisations of repelling sectors. 
\medskip

We can conclude as before that \emph{there exists a bijective correspondence between all analytic classes of diffeomorphisms from the model formal class and all pairs $(g_\infty(t),g_0(t))$ of analytic germs at $t=0$ such that $g_\infty(0)+g_0(0)=0$, after identifications \eqref{idii}}. 

For example, the \emph{trivial analytic class} of diffeomorphisms (analytically conjugated to the model) is described by the trivial pair of germs, $(0,0)$, up to identifications \eqref{idii}. That is, by a pair of constant germs of the type $(-a,a)$, $a\in\mathbb{C}$. 
\section{Analyticity of solutions of generalized Abel equations}\label{fourthree}
In this section, we analyse formal series solutions and sectorial analyticity of solutions of generalized Abel equations for a diffeomorphism $f$, see Definition~\ref{propabelgen}:
\begin{equation}\label{hha}
H(f(z))-H(z)=g(z),\quad g(z)\in\mathbb{C}\{z\},\ g\equiv\!\!\!\!\! / \ 0.
\end{equation} 
The results we obtain here will be applied to the $\varepsilon$-neighborhoods of orbits of parabolic diffeomorphisms in the following sections. We suppose in the sequel that the diffeomorphism $f$ is of formal type $(k=1,\lambda=0)$ and prenormalized.
\smallskip

To understand equation \eqref{hha}, in the following Proposition~\ref{formal} we state results mostly taken and adapted from \cite[Section A.6]{loraypre}. The proof in \cite{loraypre} follows the idea from \cite{fatou}, where Fatou constructed sectorial solutions of the Abel (trivialisation) equation. In \cite{loraypre}, the case when $g(z)=O(z^2)$ was treated. Here we adapt it for all $g(z)\in\mathbb{C}\{z\}$.  
\begin{proposition}[Formal and analytic solutions of generalized Abel equations, \cite{loraypre}]\label{formal}
Let $g(z)\in\mathbb{C}\{z\},\ g(z)=\alpha_0+\alpha_1 z+\alpha_2 z^2+o(z^2)$, $\alpha_i\in\mathbb{C}$, $i\in\mathbb{N}_0$. There exists a unique formal series solution $\widehat{H}(z)$ of equation~\eqref{hha} without the constant term of the form
\begin{equation}\label{formi}
\widehat{H}(z)\in -\frac{\alpha_0}{z}+\alpha_1 Log(z)+z\mathbb{C}[[z]].
\end{equation}
All other formal series solutions in the given scale are obtained by adding an arbitrary constant term.

Furthermore, there exist unique sectorially analytic solutions $H_+$ and $H_-$ without the constant term defined on Fatou petals $V_+$ and $V_-$ of $f$ respectively, which  admit $\widehat{H}(z)$ as their asymptotic development on petals\footnote{We say that $\widehat{H}(z)$ is an asymptotic development of $H_+(z)$ on \emph{petal} $V_+$ if it is an asymptotic development of $H_+(z)$ on \emph{every subsector} $V\subset V_+$.}, as $z\to 0$. Moreover, $H_+(z)$ and $H_-(z)$ are $1$-sums of formal series \eqref{formi}, as $z\to 0$. 
\end{proposition}

The proof, mainly taken from \cite[Section A.6]{loraypre}, is constructive, since it gives explicit (though unoperable) formulas for sectorial solutions, which we will exploit in the proof of Theorem~\ref{ppart} and Corollary~\ref{conn} in Section~\ref{fourtwo}. Therefore we give here the main lines.

\begin{proof}
The proof of existence and uniqueness of the formal solution is straightforward, solving the difference equation \eqref{hha} term by term. 
To prove the \emph{existence} of sectorially analytic solutions, instead of $H(z)$, we consider the function 
$$
R(z)=H(z)+\frac{\alpha_0}{z}-\alpha_1 Log(z).
$$
This is done to eliminate terms in $g(z)$ of order less than $2$. By \eqref{hha}, $R(z)$ now satisfies the difference equation
\begin{equation}\label{r}
R(f(z))-R(z)=\delta(z),
\end{equation}
where $\delta(z)=z^2\mathbb{C}\{z\}$. Now we directly apply results from \cite[A.6]{loraypre}. We construct two sectorially analytic functions $R_+(z)$ and $R_-(z)$, defined on invariant Leau-Fatou petals $V_+$ and $V_-$ for $f$, which satisfy equation \eqref{r} and which are $1$-sums of $\widehat{R}(z)=\widehat{H}(z)+\frac{\alpha_0}{z}-\alpha_1 Log(z)\in z\mathbb{C}[[z]]$.
We consider the following series on $V_+$ and $V_-$ respectively:
\begin{equation}\label{for1}
-\sum_{n\geq 0}\delta\big(f^{\circ n}(z)\big),\ z\in V_+,
\end{equation}
and
\begin{equation}\label{for2}
\sum_{n\geq 1}\delta\big(f^{\circ (-n)}(z)\big),\ z\in V_-.
\end{equation}
The idea behind the construction of these series is simple. Suppose that $R_+(z)$ and $R_-(z)$ are solutions of \eqref{r} on petals, with asymptotic development $\widehat{R}(z)$, as $z\to 0$. Then \eqref{r} must be satisfied for all positive iterates $f^{\circ n}(z),\ n\in\mathbb{N}_0,$ in $V_+$, and for all negative iterates $f^{\circ(-n)}(z),\ n\in\mathbb{N}_0,$ in $V_-$. Summing the equations for positive and negative iterates separately, and passing to the limit as $n\to \infty$, we get formulas \eqref{for1} and \eqref{for2}. It is left to prove that the series converge uniformly on all compact subsets of $V_+$, $V_-$ respectively, see below. Then, by Weierstrass theorem\footnote{See e.g. Theorem 1 in \cite[Ch. 5]{ahlfors}: Let the sequence $(f_n(z))_{n\in\mathbb{N}}$ of analytic functions on a domain $\Omega$ converge to $f(z)$, uniformly on every compact subset of $\Omega$. Then $f(z)$ is also analytic in $\Omega$.}, they converge to analytic functions on petals, which we denote $R_+(z)$ on $V_+$ and $R_-(z)$ on $V_-$:
\begin{align}\label{for}
R_+(z)&=-\sum_{n\geq 0}\delta\big(f^{\circ n}(z)\big),\ z\in V_+,\nonumber\\
R_-(z)&=\sum_{n\geq 1}\delta\big(f^{\circ (-n)}(z)\big),\ z\in V_-.
\end{align}
It can be shown furthermore that both $R_+(z)$ and $R_-(z)$ admit $\widehat{R}(z)$ as their $1$-sum, as $z\to 0$. For the definition of $1$-sum, see the introductory part at the beginning of Chapter~\ref{four}.

The \emph{uniqueness} of the sectorial analytic solutions $R_+(z)$ and $R_-(z)$ on $V_+$ and $V_-$ respectively, with the asymptotic development $\widehat{R}(z)$, is easy to prove. Any such solution $R_+(z)$ on $V_+$ is, by the above discussion, necessarily given by the same convergent series \eqref{for1}, and is thus unique. The same can be concluded for $V_-$ and formula \eqref{for2}. 

Finally, the solutions of initial equation \eqref{hha} are given by
\begin{align*}
H_\pm(z)=R_\pm(z)-\frac{\alpha_0}{z}+\alpha_1 Log(z) \text{ on $V_\pm$},
\end{align*}
where $R_{\pm}(z)$ are as in \eqref{for}.\ On each petal we choose the appropriate branch of logarithm. Using results for $R_\pm$, the analyticity and uniqueness results for $H_\pm$ on $V_\pm$ respectively easily follow. 
\medskip

\emph{Proof of uniform convergence of \eqref{for1} and \eqref{for2} on compacts, from \cite{loraypre}.} The proof is done considering the germ $f$ at infinity. By Remark~\ref{poslije}, it holds that, for every compact subset $K$ of Leau petal $V_+$, there exists $C>0$, such that it holds $|f^{\circ n}(z)|\leq \frac{C}{n}$. Indeed, every compact subset $K$ of $V_+$ can be covered by finitely many subpetals $V_R^+$ from Remark~\ref{poslije}. $C$ is taken to be the maximum of $C_R$ from the estimates \eqref{hhi}. Now, using the fact that $\delta(z)=O(z^2)$, we conclude that the series \eqref{for1} converges \emph{uniformly} on $K$.    
\end{proof}

\begin{remark}[About complex logarithms] Let us remark that in above computations $($namely, in deriving equation \eqref{r}$)$ we use the formula
$$
Log(f(z))-Log(z)=Log\frac{f(z)}{z}, 
$$
which is in general not true for complex logarithms. However, for orbits inside each petal the formula holds. Since all orbits converge to the origin in a tangential direction, it holds that
$\frac{f(z)}{z}=1+O(z)$ is arbitrarily close to $1$, for $z$ close enough to the origin. Thus, the logarithms on the left-hand side are appropriate branches for a given petal and the logarithm on the right-hand side always denotes the main branch\footnote{The main branch of the complex logarithm: $Log(z)=\log |z|+i\cdot Arg(z), \ -\pi<Arg(z)<\pi$.}.  
\end{remark}

\medskip
Having proven that generalized Abel equations posess two sectorially analytic solutions, we pose the question about the necessary and sufficient conditions on a diffeomorphism $f$ for the existence of a \emph{globally analytic solution} of its generalized Abel equation. That means that the sectorial analytic solutions glue to a global analytic solution $H(z)$ on some neighborhood of $0$.
\smallskip

Let the right-hand side $g(z)$ of \eqref{hha} be \emph{of multiplicity $k$}. That is, 
\begin{equation}\label{multip}
g(z)=\alpha_k z^k+o(z^k)\in z^k\mathbb{C}\{z\},\ \alpha_k\neq 0,\ k\in\mathbb{N}_0.
\end{equation} 
If $\alpha_0\neq 0$ or $\alpha_1\neq 0$, let us define $$h_{\alpha_0,\alpha_1}(z)=-\frac{\alpha_0}{z}+\alpha_1 Log(z).$$

\begin{theorem}[Existence and uniqueness of a globally analytic solution of a generalized Abel equation]\label{glo}
Let $g(z)\in \mathbb{C}\{z\},\ g(z)\equiv\!\!\!\!\!\!/\ \ 0,$ be of multiplicity $k\in\mathbb{N}_0$, as in \eqref{multip}. The generalized Abel equation 
$$
H(f(z))-H(z)=g(z)
$$
has a  global analytic solution on some neighborhood of $z=0$ if and only if the diffeomorphism $f(z)$ is of the form
\begin{equation}\label{f2}
f(z)=\left\{ \begin{array}{ll}
\varphi^{-1}\bigg(h_{\alpha_0,\alpha_1}^{-1}\Big(h_{\alpha_0,\alpha_1}\big(\varphi(z)\big)+g(z)\Big)\bigg)& ,\ k=0,\ 1,\\[0.2cm]
\varphi^{-1}\left(\varphi(z)\cdot\Big(1+\frac{k-1}{\alpha_k}\frac{g(z)}{\varphi(z)^{k-1}}\Big)^{\frac{1}{k-1}}\right)& ,\ k\in\mathbb{N},\ k\geq 2,
\end{array}\right.
\end{equation}
for some analytic germ $\varphi(z)\in z+z^2\mathbb{C}\{z\}$. The global analytic solution $H(z)$ is then given by 
\begin{equation}\label{given}
H(z)=\left\{ \begin{array}{ll}
h_{\alpha_0,\alpha_1}\circ \varphi(z)& ,\ k=0,\ 1,\\[0.2cm]
\frac{\alpha_k}{k-1}\varphi(z)^{k-1}& ,\ k\in\mathbb{N},\ k\geq 2.
\end{array}\right.
\end{equation}
It is unique up to an arbitrary chosen additive constant.
\end{theorem}

Here and in the sequel, we will use the term \emph{globally analytic} in a slightly incorrect manner. In the case where the linear term of $g(z)$ is non-zero, $H(z)$ contains a logarithmic term in the asymptotic development, as $z\to 0$. Also, when $g(z)$ contains the constant term, a term $-1/z$ appears in the development. Therefore, by globally analytic, we actually mean that $H(z)$ is globally analytic on some neighborhood of $0$, after possibly subtracting a logarithmic term $Log(z)$ and a term $-1/z$. The \emph{global analyticity} of the solution $H(z)$ of \eqref{hha} in these cases in fact means the global analyticity of the solution $R(z)$, $H(z)=-\frac{\alpha_0}{z}+\alpha_1 Log(z)+R(z)$, of the modified equation $$R(f(z))-R(z)=g(z)+\alpha_0\left(\frac{1}{f(z)}-\frac{1}{z}\right)-\alpha_1 Log \Big(\frac{f(z)}{z}\Big).$$

\smallskip
In the proof, we need the following technical Lemma~\ref{seriesana}. The proof of the lemma is in Section~\ref{fourseven}.
\begin{lemma}\label{seriesana}
Let $\widehat{g}(z)\in\mathbb{C}[[z]]$ be a formal series, and let $h(z)\in\mathbb{C}\{z\}$ be a non-constant analytic germ. Let $\widehat T\in\mathbb{C}[[z]]$, such that
\begin{equation}\label{tet}\widehat{T}=h\circ\widehat{g}.\end{equation} Then $\widehat{T}$ is analytic if and only if $\widehat{g}$ is analytic.
\end{lemma}
Note that the assumption about the existence of a formal Taylor development of $g(z)$ is essential for one direction of Lemma~\ref{seriesana} to hold. Let us suppose, for example, that  $T(z)=g(z)^2$, where $T$ is analytic. Without any assumptions on $g$,  $g(z)$ may jump from one complex root to another, and thus be discontinuous and non-analytic. Such situations are excluded by imposing the formal development condition on $g$.

\medskip

\noindent \emph{Proof of Theorem~\ref{glo}}. We consider two cases separately.

$i)$ $\mathbf{k\geq 2}.$ It is easy to check that the formal solution $\widehat{H}(z)\in z\mathbb{C}[[z]]$ is of the form
$$
\widehat{H}(z)=\frac{\alpha_k}{k-1}z^{k-1}+o(z^{k-1}).
$$
Equivalently, we can write
$$
\widehat{H}(z)=\frac{\alpha_k}{k-1}\widehat{\varphi}(z)^{k-1},
$$
where $\widehat{\varphi}(z)$ is a formal series of the form $z+z^2\mathbb{C}[[z]]$. By Lemma~\ref{seriesana}, $H(z)$ is globally analytic if and only if $\varphi(z)$ is globally analytic. 

Suppose now that $H(z)$ is globally analytic. Putting $H(z)=\frac{\alpha_k}{k-1}\varphi(z)^{k-1}$ in equation \eqref{hha}, we can uniquely express $f(z)$: 
\begin{equation}\label{f1}
f(z)=\varphi^{-1}\left(\Big(\varphi(z)^{k-1}+\frac{k-1}{\alpha_k}g(z)\Big)^{\frac{1}{k-1}}\right).
\end{equation}
Here, $\varphi(z)^{k-1}\sim z^{k-1}$ and $g(z)\sim \alpha_k z^k$, as $z\to 0$. The $(k-1)$-th root we take is uniquely determined, since $f(z)$ and $\varphi(z)$ are tangent to the identity. Formula \eqref{f1} easily transforms to \eqref{f2}.

Conversely, if $f(z)$ is of the form \eqref{f2} for $\varphi(z)\in z+z^2\mathbb{C}\{z\}$, it is easy to see that $H(z)=\frac{\alpha_k}{k-1}\varphi(z)^{k-1}$ satisfies equation \eqref{hha} for $f(z)$ and that the formal development is of the form \eqref{form}. By uniqueness in Proposition~\ref{formal}, $H(z)$ is the unique analytic solution of \eqref{hha}.
\smallskip 

$ii)$ $\mathbf{k=0,\ 1}$. It can easily be computed that the formal solution in is of the form $$\widehat{H}(z)=h_{\alpha_0,\alpha_1}(z)+z\mathbb{C}[[z]]=h_{\alpha_0,\alpha_1}\circ \widehat{\varphi}(z),$$ where $\widehat{\varphi}(z)\in z+z^2\mathbb{C}[[z]]$. It is easy to see that $\widehat{H}(z)$ can be written as 
$$
\widehat{H}(z)=-\frac{\alpha_0}{z}+\alpha_1 Log(z)+g\left(\frac{\widehat\varphi (z)-z}{z}\right),
$$ 
where $g$ is a nonconstant analytic germ. Now, by Lemma~\ref{seriesana}, $\widehat{H}(z)$ is globally analytic $($in the sense of $\widehat{H}(z)+\frac{\alpha_0}{z}-\alpha_1 Log(z)$ being globally analytic$)$ if and only if $\widehat\varphi(z)$ is. We can proceed as in $i)$. 
The function $h_{\alpha_0,\alpha_1}(z)=-\frac{\alpha_0}{z}+\alpha_1 Log(z)$ in the expression \eqref{f2} is \emph{invertible} since, in the case $\alpha_0\neq 0$, it can be regarded as global Fatou coordinate for the flow of the vector field $X_{1,\lambda}$, $\lambda=2\pi i \frac{\alpha_1}{\alpha_0}$, see e.g.\cite{loray}. In the case $\alpha_0=0$, it is merely the logarithmic function, therefore invertible on sectors. 
$\hfill\Box$

\medskip

\begin{example}[Application of Theorem~\ref{glo} to the Abel equation]\

The trivialization (Abel) equation for a parabolic germ $f(z)$ was a central object for obtaining the moduli of analytic classification in Section~\ref{fourone}:
\begin{equation}\label{aabel}
\Psi(f(z))-\Psi(z)=1.
\end{equation}
We use here Theorem~\ref{glo} to derive a well-known result by Ecalle and Voronin that the analytic class of the model diffeomorphism $f_0(z)=\frac{z}{1-z}$ is described by the existence of a global solution $\Psi(z)$ to the trivialisation equation \eqref{aabel}. Indeed, it is related to the analytic conjugacy $\varphi(z)$ by $\Psi(z)=-\frac{1}{z}\circ \varphi(z)$.
Of course this is not a new result, and we put it here only as an example.
\smallskip

\noindent \emph{Proof by Theorem~\ref{glo}}.
The Abel equation \eqref{aabel} is a special case of generalized Abel equations, with the right-hand side $g(z)\equiv 1$. Therefore, $h_{1,0}(z)=-1/z$. By \eqref{f2}, we get that there exists a global analytic solution of \eqref{aabel} if and only if $f(z)$ is given by
$$
f(z)=\varphi^{-1}\left(-\frac{1}{-\frac{1}{\varphi(z)}+1}\right)=\varphi^{-1}\circ \frac{z}{1-z}\circ \varphi (z),
$$
for some analytic diffeomorphism $\varphi(z)$. It is unique up to an additive constant and, by \eqref{given}, of the form $\Psi(z)=-\frac{1}{z}\circ \varphi(z)$. $\hfill\Box$
\end{example}

\section{Analyticity properties of complex measures of $\varepsilon$-neighborhoods of orbits of parabolic germs}\label{fourtwo}
Let $f:(\mathbb{C},0)\to (\mathbb{C},0)$ be a parabolic diffeomorphism tangent to the identity, of any multiplicity $k\in\mathbb{N}$. Let $V_+$ and $V_-$ denote any attracting and repelling petal respectively. Let $S^f(z)$, $z\in V_+$, denote orbits of $f$ on attracting petals and $S^{f^{-1}}(z)$, $z\in V_-$, orbits of the inverse diffeomorphism $f^{-1}$ on repelling petals. 
\smallskip

The asymptotic development in $\varepsilon$ of the complex measure of $\varepsilon$-neighborhoods of orbits was given by \eqref{asyi} at the beginning of the chapter. We saw in Chapter~\ref{threetwo} that the formal class of $f$ can be read from the first $k+1$ coefficients independent of the initial point in this development. To get some insight about the analytic class, we analyse here the analytic properties of the function of the complex measure $\widetilde{A^\mathbb{C}}(S^f(z)_\varepsilon)$, in both parameter $\varepsilon>0$ and variable $z\in V_+$. Similarly, for the function $\widetilde{A^\mathbb{C}}(S^{f^{-1}}(z)_\varepsilon)$, in parameter $\varepsilon>0$ and variable $z\in V_-$. We analyse in more detial the remainder term $R(z,\varepsilon)$ from \eqref{asy}. 
\smallskip

In Subsections~\ref{fourtwoone} and \ref{fourtwotwo}, we state some \emph{bad properties} of these functions -- nonexistence of the full asymptotic development and accumulation of singularities in $\varepsilon$ for a fixed $z$, nonanalyticity in $z$ for a fixed $\varepsilon$. Then, in Subsection~\ref{fourtwothree}, we derive a sectorial analyticity property of \emph{principal parts of complex measures}, defined in Definition~\ref{princpart}. This is the main reason why in the following sections we concentrate only on principal parts, as the only parts of  complex measures of $\varepsilon$-neighborhoods of orbits with sectorial analyticity property.

\subsection{Analyticity of complex measures as functions of the parameter $\varepsilon$}\label{fourtwoone}

In this subsection, let $z\in V_+$ be fixed. Let $S^f(z)=\{z_n\ |\ n\in\mathbb{N}_0\}$, where $z_0=z$, denote the orbit with the initial point $z$. Let the sequence $(\varepsilon_n)_{n\in\mathbb{N}_0}$ denote the sequence of half-distances between consecutive points of the orbit:
$$
\varepsilon_n=\frac{|z_n-z_{n+1}|}{2},\ n\in\mathbb{N}_0.
$$
Then $\varepsilon_n\to 0$ decreasingly, as $n\to\infty$.
\medskip

Let $\varepsilon\mapsto \widetilde{A^\mathbb{C}}(S^f(z)_\varepsilon)$ denote the complex measure of the $\varepsilon$-neighborhood of the orbit $S^f(z)$, as a function of $\varepsilon\in (0,\varepsilon_0)$. We had initially hoped to be able to extend the function in $\varepsilon$ to the complex plane in a way that it exhibits some sectorial analyticity properties. Two propositions that follow show the difficulties in this approach.

\medskip
Proposition~\ref{nonasy} states that the remainder term $R(z,\varepsilon)$ in the development \eqref{asyi} does not have a development in $\varepsilon$ in a power-logarithm scale any more after a certain number of terms. This presents an obstacle for extending the function from the positive real line to complex $\varepsilon$, by means of formal series.

\begin{proposition}[Nonexistence of a full power-logarithmic asymptotic development in $\varepsilon$, as $\varepsilon\to 0$]\label{nonasy}
Let $z\in V_+$ be fixed. A full asymptotic development of $\widetilde{A^\mathbb{C}}(S^f(z)_\varepsilon)$ in a power-logarithmic scale, as $\varepsilon\to 0$, does not exist. That is, there exists $l\in\mathbb{N}$, such that the remainder term $R(z,\varepsilon)$ in \eqref{asyi} is of the form:
$$
R(z,\varepsilon)=h_1(z)g_1(\varepsilon)+\ldots+h_{l-1}(\varepsilon) g_{l-1}(\varepsilon)+h(z,\varepsilon),\ \ h(z,\varepsilon)=O\big(g_l(\varepsilon)\big),\ \varepsilon\to 0.$$
The monomials $g_i(\varepsilon)$ are of power-logarithmic type in $\varepsilon$, of increasing flatness at zero, but the limit
\begin{equation*}
\lim_{\varepsilon\to 0}\frac{h(z,\varepsilon)}{g_l(\varepsilon)}
\end{equation*}
does not exist.
\end{proposition}

\begin{proof}
We show the obstacle for the existence of a full asymptotic development: the index $n_\varepsilon$ separating the tail and the nucleus of the $\varepsilon$-neighborhood of the orbit does not have asymptotic development in $\varepsilon$ after the first $k+1$ terms. 

By Lemma~\ref{asyneps} in Subsection~\ref{threetwotwo}, $n_\varepsilon$ has the following development, as $\varepsilon\to 0$:
\begin{equation}\label{razne}
n_\varepsilon=p_1\varepsilon^{-1+\frac{1}{k+1}}+\ldots+p_k\varepsilon^{-1+\frac{k}{k+1}}+p_{k+1}\log\varepsilon+r(z,\varepsilon),
\end{equation}
where $r(z,\varepsilon)=O(1)$ in $\varepsilon$, for $z$ fixed. We put $z$ here only to denote the dependence of the function on the initial point. Here, $z$ is only a fixed complex number. 

Suppose that the limit $\lim_{\varepsilon\to 0}r(z,\varepsilon)$ exists. Then, 
\begin{equation}\label{ff}
r(z,\varepsilon)=C(z)+o(1),\ \varepsilon\to 0\ \ \text{\ $(C$ can be 0$)$}.
\end{equation}
In the points $\varepsilon_n$ as above, it holds
$$
n(\varepsilon_n+)=n,\ n(\varepsilon_n-)=n+1.
$$
The $(k+1)$-jet of the development \eqref{razne} is continuous on $(0,\varepsilon_0)$. By \eqref{ff}, $r(\varepsilon_n)=C+o(1)$, as $n\to\infty$. Therefore we get that
$$
1=n(\varepsilon_n+)-n(\varepsilon_n-)=o(1),\ n\to\infty,
$$
which is a contradiction. The limit $\lim_{\varepsilon\to 0} r(z,\varepsilon)$ does not exist. 

Furthermore, we return to the proofs of Lemmas~\ref{massnucl} and \ref{masstail} in Subsection~\ref{threetwotwo}, to analyse the remainder term $R(z,\varepsilon)$ in \eqref{asyi}. First, $z_n$ has full asymptotic development, as $n\to\infty$, of the type $\mathbb{C}[[n^{-\frac{1}{k}},n^{-1}\log n]]$, with coefficients depending on $z$. Indeed, $z_n$ can be expressed using sectorial trivialisation function $\Psi_+(z)$, as $z_n=\Psi_+^{-1}(n+\Psi_+(z))$. By \cite[Tome 3, Ch. 5]{ecalle}, we have that $\widehat\Psi^{-1}(z)\in\mathbb{C}[[n^{-\frac{1}{k}},n^{-1}\log n]]$. On the other hand, the development of $n_\varepsilon$ \eqref{razne} is \emph{finite}, since $r(z,\varepsilon)$ has no limit. By computations in proofs of Lemmas~\ref{massnucl},\ \ref{masstail}, we conclude that $\widetilde{A^{\mathbb C}}(N_\varepsilon)$ and $\widetilde{A^{\mathbb C}}(T_\varepsilon)$ develop in power-logarithmic scale in $\varepsilon$, with coefficients depending on initial point $z$, but only up to a first term in which $r(z,\varepsilon)$ from $n_\varepsilon$ interferes and the development in $\varepsilon$ no longer exists. The problem is that this does not yet guarantee that their sum does not have the full development, that is, that critical terms of the tail and the nucleus do not cancel. This cannot happen in general, due to different kind of dependence of the tail and of the nucleus on $n_\varepsilon$. For simplicity, we illustrate it on an example of a germ on the real line, and considering the length of $\varepsilon$-neighborhood of orbits instead of complex measure. 

Suppose that $f(x)=\frac{x}{1+x}$. It can be computed that, for initial point $x$, the points of the orbit $S^f(x)$ are given by $x_n=\frac{x}{1+nx}=n^{-1}-x^{-1}  n^{-2}+o(n^{-2})$. We compute $n_\varepsilon=2^{-1/2}\ \varepsilon^{-1/2}+r(x,\varepsilon)$, where $\lim_{\varepsilon\to 0} r(x,\varepsilon)$ does not exist. For the length of the whole $\varepsilon$-neighborhood, we have $$|S^f(x)_\varepsilon|=x_{n_\varepsilon}+\varepsilon+2\varepsilon\cdot n_\varepsilon=2\sqrt 2\varepsilon^{1/2}-\frac{2}{x}\varepsilon +\left[\frac{4\sqrt 2}x r(x,\varepsilon)+2\sqrt 2 r^2(x,\varepsilon)\right]\varepsilon^{3/2}+o(\varepsilon^{3/2}),$$ as $\varepsilon\to 0$. The asymptotic development after the second term does not exist. The third term does not have an asymptotic behavior in $\varepsilon$: it is $O(\varepsilon^{3/2})$, but, when divided by $\varepsilon^{3/2}$, the limit in general does not exist. 
\end{proof}

\bigskip
The next Proposition~\ref{accu} expresses an obstacle for the analytic continuation of $\widetilde{A^\mathbb{C}}(S^f(z)_\varepsilon)$ on the neighborhood of the positive real line. On the positive real line, function $\varepsilon\mapsto \widetilde{A^\mathbb{C}}(S^f(z)_\varepsilon)$ has accumulation of singularities at $\varepsilon=0$.

\begin{proposition}[Accumulation of singularities at $\varepsilon=0$]\label{accu}
Let $\varepsilon_0>0$. The function $\varepsilon\mapsto A^\mathbb{C}(S^f(z)_\varepsilon)$ is of class $C^1$ on $(0,\varepsilon_0)$ and $C^\infty$ on open subintervals $(\varepsilon_{n+1},\varepsilon_n)$, $n\in\mathbb{N}_0$. However, in all $\varepsilon_n$, $n\in\mathbb{N}_0$, the second derivative is unbounded from the right:
$$
\lim_{\varepsilon\to\varepsilon_n-}\frac{d^2}{d\varepsilon^2}\widetilde{A^\mathbb{C}}(S^f(z)_\varepsilon)\in\mathbb{C},\ \lim_{\varepsilon\to\varepsilon_n+}\left|\frac{d^2}{d\varepsilon^2}\widetilde{A^\mathbb{C}}(S^f(z)_\varepsilon)\right|=+\infty.
$$
\end{proposition}

\begin{proof}
We analyse the complex measure of the tail and of the nucleus separately. Without any change in the class in $(0,\varepsilon_0)$, we can consider the complex measure divided by $\varepsilon^2\pi$. We show that the points where class $C^2$ is lost are the points $\varepsilon_n$ in which, when $\varepsilon$ decreases to zero, one disc detaches from the nucleus to the tail. We have
\begin{equation*}
\frac{\widetilde{A^\mathbb{C}}(S^f(z)_\varepsilon)}{\varepsilon^2\pi}=\frac{\widetilde{A^\mathbb{C}}(T_\varepsilon)}{\varepsilon^2\pi}+\frac{\widetilde{A^\mathbb{C}}(N_\varepsilon)}{\varepsilon^2\pi}.
\end{equation*}

The function $\varepsilon\mapsto \frac{\widetilde{A^\mathbb{C}}(T_\varepsilon)}{\varepsilon^2\pi}$ is easy to analyse: it is a piecewise constant function on the intervals $[\varepsilon_{n+1},\varepsilon_n)$, with jumps at $\varepsilon=\varepsilon_n$ of value $+z_n$.

The complex measure of the nucleus is computed adding the contribution of each crescent. By Proposition~\ref{crescent} in Subsection~\ref{threetwofive}, it holds:
\begin{align}\label{nuclis}
\frac{\widetilde{A^\mathbb{C}}(N_\varepsilon)}{\varepsilon^2\pi}=\left\{\begin{array}{l}z_{n+1}+G_{n+1}(\varepsilon),
\hfill\varepsilon\in[\varepsilon_{n+1},\varepsilon_{n}),\\[0.3cm] z_n+G_{n+1}(\varepsilon)+\\
\quad +\frac{1}{\pi}\left(\frac{\varepsilon_{n}}{\varepsilon}\sqrt{1-\frac{\varepsilon_{n}^2}{\varepsilon^2}}+\arcsin{\frac{\varepsilon_{n}}{\varepsilon}}\right)(z_{n}+z_{n+1})+\frac{z_{n+1}-z_{n}}{2},\\\hfill \varepsilon\in[\varepsilon_n,\varepsilon_{n-1}).\end{array}\right.
\end{align} 
Here, by $G_{n+1}(\varepsilon),\ n\in\mathbb{N}$, we denote the complex functions \begin{align*}G_{n+1}(\varepsilon)=&\frac{1}{\pi}\sum_{k=n+1}^{\infty}\left(\frac{\varepsilon_{k}}{\varepsilon}\sqrt{1-\frac{\varepsilon_{k}^2}{\varepsilon^2}}+\arcsin{\frac{\varepsilon_{k}}{\varepsilon}}\right)(z_{k}+z_{k+1})+\frac{z_{k+1}-z_{k}}{2}.
\end{align*}
$G_{n+1}(\varepsilon)$ presents the sum of contributions from the crescents corresponding to the points $z_{n+2},\ z_{n+3},$ etc.

Let $\delta>0$ such that $\varepsilon_{n+1}+\delta<\varepsilon_n$.  By Proposition~\ref{auxi} in Section~\ref{fourseven}, function $G_ {n+1}(\varepsilon)$ is of class $C^2$ on each interval $(\varepsilon_{n+1}+\delta,\varepsilon_{n-1})$, $\delta>0$. Therefore, by \eqref{nuclis}, the point of nondifferentiability of $\widetilde{A^\mathbb{C}}(N_\varepsilon)$ on $(\varepsilon_{n+1}+\delta,\varepsilon_{n-1})$ can only be $\varepsilon=\varepsilon_n$, where two parts defined by different formulae glue together. In the sequel, we show that at the point $\varepsilon=\varepsilon_n$, $\widetilde{A^\mathbb{C}}(N_\varepsilon)$ is of class $C^1$, but not $C^2$. Differentiating \eqref{nuclis} in $\varepsilon$ on some interval around $\varepsilon_n$, we get
\begin{align*}
\frac{d}{d\varepsilon}\frac{\widetilde{A^\mathbb{C}}(N_\varepsilon)}{\varepsilon^2\pi}\Big|_{\varepsilon=\varepsilon_n-}=G_{n+1}'(\varepsilon_n-)\Big.,\
\frac{d}{d\varepsilon}\frac{\widetilde{A^\mathbb{C}}(N_\varepsilon)}{\varepsilon^2\pi}\Big|_{\varepsilon=\varepsilon_n+}=G_{n+1}'(\varepsilon_n+)\Big.,
\end{align*}
the two being finite and equal since $G_{n+1}$ is of the class $C^2$ around $\varepsilon_n$. Therefore, $\widetilde{A^\mathbb{C}}(N_\varepsilon)$ is of class $C^1$ at $\varepsilon=\varepsilon_n$,\ $n\in\mathbb{N}$. 

Differentiating once again, we get
\begin{align}\label{druga}
\frac{d^2}{d\varepsilon^2}&\frac{\widetilde{A^\mathbb{C}}(N_\varepsilon)}{\varepsilon^2\pi}\Big|_{\varepsilon=\varepsilon_n-}={(G_{n+1})}^{\prime\prime}(\varepsilon_n-),\nonumber\\
\frac{d^2}{d\varepsilon^2}&\frac{\widetilde{A^\mathbb{C}}(N_\varepsilon)}{\varepsilon^2\pi}\Big|_{\varepsilon=\varepsilon_n+}={(G_{n+1})}^{\prime\prime}(\varepsilon_n+)\ +\nonumber\\
& \qquad +\frac{1}{\pi}\bigg(\frac{4\varepsilon_n}{\varepsilon^3}\sqrt{1-\frac{\varepsilon_n^2}{\varepsilon^2}}-\frac{2\varepsilon_n^3}{\varepsilon^5}\frac{1}{\sqrt{1-\frac{\varepsilon_n^2}{\varepsilon^2}}}\bigg)\Bigg|_{\varepsilon=\varepsilon_n+}\cdot\big(z_{n+1}+z_n).
\end{align}
Although ${(G_{n+1})}^{\prime\prime}(\varepsilon_n-)={(G_{n+1})}^{\prime\prime}(\varepsilon_n+)\in\mathbb{C}$, the other term is unbounded when $\varepsilon\to\varepsilon_n+$. Therefore, the second derivative of $\widetilde{A^\mathbb{C}}(N_\varepsilon)$ at $\varepsilon=\varepsilon_n$, $n\in\mathbb{N}$, does not exist.

Finally, glueing overlapping intervals $(\varepsilon_{n-1}+\delta,\varepsilon_{n+1}),\ n\in\mathbb{N},$ and adding the tail and the nucleus, we get the desired result.
\end{proof}

We saw, in the course of the proof, that the loss of analyticity at points $\varepsilon_n$ at which separation of the tail and the nucleus occurs is related to the different rate of growth of the tail and of the nucleus of $\varepsilon$-neighborhoods in $\varepsilon$, due to their different geometry (overlapping discs in nucleus, disjoint discs in tail). The culprit for overlapping is the way of forming the $\varepsilon$-neighborhoods using discs of the same radius $\varepsilon$ at all points. It remains for the future investigation to see if some other way of defining $\varepsilon$-neighborhoods, perhaps using discs of varying radii, would give us better properties.

\subsection{Analyticity of complex measures as functions of the initial point}\label{fourtwotwo}\
In this subsection, let $\varepsilon>0$ be fixed. The following proposition states that sectorial analyticity property cannot be obtained directly considering the function $z\mapsto A^\mathbb{C}(S^f(z)_\varepsilon)$, for a fixed $\varepsilon>0$.

Let $S^{\pm}(\varphi,r),\ \varphi\in(0,\pi),\ r>0,$ denote the (symmetric) sectors of opening $2\varphi$ and radius $r>0$ around any attracting, respectively repelling, direction. 
\begin{proposition}[Non-analyticity in the variable $z$]\label{no}
Let $\varepsilon>0$. The function $z\mapsto \widetilde{A^\mathbb{C}}(S^f(z)_\varepsilon)$ is not analytic on any attracting petal $V_+$. The function $z\mapsto \widetilde{A^\mathbb{C}}(S^{f^{-1}}(z)_\varepsilon)$ is not analytic on any repelling petal $V_-$.
\end{proposition}
Moreover, we show in the proof that $z\mapsto \widetilde{A^\mathbb{C}}(S^f(z)_\varepsilon)$ is not analytic on any open sector $S^+(\varphi,r)\subset V_+,\ r>0,\ \varphi\in(0,\pi)$. Similarly, $z\mapsto \widetilde{A^\mathbb{C}}(S^{f^{-1}}(z)_\varepsilon)$ is not analytic on any open sector $S^-(\varphi,r)\subset V_-,\ r>0,\ \varphi\in(0,\pi)$. 

\begin{proof}
Let $\varepsilon>0$ be fixed. By $U_\varepsilon$ we denote the open set $U_\varepsilon=\{z\in V_+:\ |z-f(z)|<2\varepsilon\}$. For $z\in U_\varepsilon$, the $\varepsilon$-discs centered at points $z$ and $f(z)$ in $S^f(z)_\varepsilon$ overlap. Therefore, the $\varepsilon$-neighborhoods of orbits $S^f(z)$ and $S^f(f(z))$ differ by a crescent. By Proposition~\ref{crescent} in Subsection~\ref{threetwofive}, we get
\begin{align}\label{analy}
\widetilde{A^\mathbb{C}}&(S^f(z)_\varepsilon)=\widetilde{A^\mathbb{C}}(S^f(f(z))_\varepsilon)-\frac{\pi}{2}\varepsilon^2(f(z)-z)+\\
&\hspace{3cm}+\varepsilon^2(z+f(z))\cdot G\Big(\frac{|z-f(z)|}{2\varepsilon}\Big),\ z\in U_\varepsilon.\nonumber
\end{align}
Here, $G(t)=t\sqrt{1-t^2}+\arcsin t$, $t\in(0,1)$. We define the function $T(z)$:
\begin{equation}\label{te}
T(z)=\widetilde{A^\mathbb{C}}(S^f(z)_\varepsilon)-\widetilde{A^\mathbb{C}}(S^f(f(z))_\varepsilon,\ z\in V_+.
\end{equation}
By \eqref{analy}, it holds
$$
T(z)=-\frac{\pi}{2}\varepsilon^2(f(z)-z)+\varepsilon^2(z+f(z))\cdot G\bigg(\frac{|z-f(z)|}{2\varepsilon}\bigg),\ z\in U_\varepsilon.
$$
It holds that there exists some punctured neighborhood of $0$ such that $f'(z)\neq 1$, for all $z$ in that neighborhood. Otherwise, by analyticity\footnote{\emph{Uniqueness theorem for analytic functions}:\ Let $f$ be analytic in domain $\Omega$ and let $f(z_n)=0$, where $z_n$ is a sequence
of distinct points, such that $z_n\to z_0$, $z_0\in\Omega$. Then, $f\equiv 0$ in $\Omega$. In other words, two functions analytic in $\Omega$ which coincide on a set with an accumulation point in $\Omega$ are equal on $\Omega$.} of $f$ at $z=0$, it would hold that $f'(z)\equiv 1$ on some neighborhood of 0. By inverse function theorem applied locally to $G(t)$ and $(id-f)(z)$, and since absolute value is nowhere analytic, we see that $T(z)$ is nowhere analytic on $U_\varepsilon$. 

We now take any small sector $S^+(\varphi,r)\subset V_+,\ r>0,\ \varphi\in(0,\pi)$. Suppose that $z\mapsto \widetilde{A^\mathbb{C}}(S^f(z))$ is analytic on $S_+$. Since $f$ is analytic, and $f(z)\in S_+$ for $z\in S_+$, the function $z\mapsto T(z)$ defined by \eqref{te} is also analytic on $S_+$. 
The intersection $S_+\cap U_\varepsilon$ is nonempty and therefore we obtain a contradiction.
\end{proof}

\subsection{Analyticity of principal parts of complex measures}\label{fourtwothree}
Having described \emph{bad properties} of complex measures of orbits, we concentrate now on their principal parts $H^{f}(z)$ and $H^{f^{-1}}(z)$, as functions of $z\in V_+$, $z\in V_-$ respectively. The principal parts are defined at the beginning of the chapter, in Definition~\ref{princpart}. 
\smallskip

For simplicity, from now on we restrict ourselves to diffeomorphisms of the type $f(z)=z+z^2+z^3+o(z^3)$.

\begin{theorem}[Principal parts of complex measures for orbits of $f$ and of $f^{-1}$]\label{ppart}The principal part of complex measures $H^f(z)$ for $f$ is analytic on the attracting petal $V_+$, and satisfies the following difference equation:
\begin{equation}\label{ha}
H^f(z)-H^f(f(z))=\pi z, \ z\in V_+.
\end{equation}
Analogously, the principal part $H^{f^{-1}}(z)$ of complex measures for $f^{-1}$ is analytic on the repelling petal $V_-$, and satisfies
\begin{equation}\label{hai}
H^{f^{-1}}(z)-H^{f^{-1}}(f^{-1}(z))=\pi z,\ z\in V_-.
\end{equation} 
\end{theorem}

\noindent We split the proof in two parts: 
\begin{enumerate}
\item[1.] finding equations \eqref{ha}, \eqref{hai} for the principal parts on petals in Proposition~\ref{eqprin} below,
\item[2.] proof of analyticity of the principal parts on petals. 
\end{enumerate}
\smallskip

\begin{proposition}[The equations]\label{eqprin}
The principal parts of complex measures $H^f(z)$ and $H^{f^{-1}}(z)$ satisfy the following difference equations:
\begin{align}
H^f(f(z))-H^f(z)=-\pi z,\ \ &z\in V_+,\label{h1}\\
H^{f^{-1}}(f^{-1}(z))-H^{f^{-1}}(z)=-\pi z,\ \ &z\in V_-.\label{h2}
\end{align}
\end{proposition}
\begin{proof}
Let us first derive equation \eqref{h1} for $H_f(z)$. Let $z\in V_+$.
By the definition of complex measure, it holds that
\begin{equation}\label{ac}
\widetilde{A^\mathbb{C}}(S^f(z)_\varepsilon)=\widetilde{A^\mathbb{C}}(S^f(f(z))_\varepsilon)+z\cdot \varepsilon^2\pi, \ z\in V_+,
\end{equation}
for $0<\varepsilon<\varepsilon_z$ small enough with respect to $z$.
Putting the development \eqref{asyi} in \eqref{ac}, we get that
$$
\big[H^f(z)-H^f(f(z))\big]\varepsilon^2+\big(R(z,\varepsilon)-R(f(z),\varepsilon)\big)=\varepsilon^2\pi.
$$
By \eqref{asyi}, $R(z,\varepsilon)-R(f(z),\varepsilon)=O(\varepsilon^{2+\frac{1}{k+1}})$. Dividing by $\varepsilon^2$ and passing to the limit as $\varepsilon\to 0$, \eqref{h1} follows. 

Equation \eqref{h2} is derived in the same manner, but considering complex measures of orbits of $f^{-1}$ on the repelling petal $V_-$. 
\end{proof}

\noindent \emph{Proof of Theorem~\ref{ppart}}. 

The first part of the proof, the equations for the principal parts, are obtained in Proposition~\ref{eqprin}. It is left to show analyticity. We first show analyticity of $H^f(z)$ on $V_+$. We analyse the form of the coefficient $H^f(z)$ in front of $\varepsilon^2$ in development \eqref{asy1}, as a function of $z\in V_+$. We follow the steps for obtaining the developments of the tail and of the nucleus from Subsection~\ref{threetwotwo}. Let us remind, the tail of the $\varepsilon$-neighborhood is the part of the $\varepsilon$-neighborhood which is the union of the disjoint $\varepsilon$-discs, while the nucleus is the remaining part with the overlapping discs. We denote by $H^{f}_N(z),\ H^{f}_T(z),\ z\in V_+,$ the principal parts in the developments of the complex measures of the nucleus and the tail respectively. It holds that
\begin{align}\label{sum}
H^f(z)&=H^{f}_N(z)+H^f_T(z),\ z\in V_+.
\end{align}
In the proof of Proposition~\ref{asy1} in Subsection~\ref{threetwothree}, we see that the principal part for the nucleus is constant and equal to
\begin{equation}\label{nuclei}
H^f_N(z)=-\frac{\pi}{4}(1+\log 4),\ z\in V_+.
\end{equation}
The dependence on $z$ of the principal part comes from the tail. The complex measure of the $\varepsilon$-neighborhood of the tail is equal to:
\begin{align*}
\widetilde{A^\mathbb{C}}(T_\varepsilon)(z)=\varepsilon^2\pi\cdot\sum_{k=0}^{n_\varepsilon(z)}f^{\circ k}(z).
\end{align*}
Here, $n_\varepsilon(z)$ is the index where separation of the tail and the nucleus occurs, that is, the index of the point of the orbit where $\varepsilon$-discs start overlapping. Obviously, $n_\varepsilon(z)\to\infty$, as $\varepsilon\to 0$. Developing the sum $\sum_{k=0}^{n}f^{\circ k}(z)$, as $n\to\infty$, as described in the proof of Lemma~\ref{masstail} in Subsection~\ref{threetwothree}, we get that:
\begin{align}\label{r2}
\widetilde{A^\mathbb{C}}(T_\varepsilon)(z)=\varepsilon^2\pi\cdot(-\log n_\varepsilon(z)+C(z)+o(1)),\ \varepsilon\to 0.
\end{align}
Here, $C(z)=c_0\Big(\sum_{k=0}^{n}f^{\circ k}(z)\Big)$, denotes the constant term in the asymptotic development of $\sum_{k=0}^{n}f^{\circ k}(z)$, as $n\to\infty$.
We conclude, using \eqref{r2} and the development for $n_\varepsilon(z)$ in $\varepsilon$ from Lemma~\ref{asyneps} in Subsection~\ref{threetwothree}, that the coefficient in front of $\varepsilon^2$ in development \eqref{r2}, as $\varepsilon\to 0$, can be expressed as
\begin{equation}\label{free}
H^f_T(z)=\frac{\pi}{2}\log 2+\pi \cdot c_0\Big(\sum_{k=0}^{n}f^{\circ k}(z)\Big),\ z\in V_+.
\end{equation}
By \eqref{sum}, \eqref{nuclei} and \eqref{free}, we get the expression for the principal part:
\begin{equation}\label{usp}
H^f(z)=-\frac{\pi}{4}+\pi\cdot c_0\Big(\sum_{k=0}^{n}f^{\circ k}(z)\Big),\ z\in V_+.
\end{equation}
Our next step is to prove analyticity of the function $H^f(z)$ given by \eqref{usp} on $V_+$. To this end, we consider the unique analytic solution on $V_+$ without the constant term of the following $1$-Abel equation
$$
G(f(z))-G(z)=-\pi z,
$$
see Proposition~\ref{formal} in Section~\ref{fourthree}. By the proof of Proposition~\ref{formal}, it is given by the limit
\begin{equation}\label{lim2}
G_+(z)=\pi \lim_{n\to\infty}\left(\sum_{k=0}^{n}f^{\circ k}(z)-Log f^{\circ(n+1)}(z)\right),
\end{equation}
which was proven to converge pointwise to an analytic function on $V_+$.

To prove analyticity of $H^f(z)$ on $V_+$, it suffices to show that the expression \eqref{usp} for $H^f(z)$ coincides pointwise with $G_+(z)$ in \eqref{lim2}, up to a constant. For a fixed $z$, as in \eqref{r2}, we estimate the first terms in the development of $\sum_{k=0}^{n}f^{\circ k}(z)-Log f^{\circ(n+1)}(z)$, as $n\to\infty$:
\begin{align}\label{cc}
\sum_{k=0}^{n}f^{\circ k}&(z)-Log f^{\circ(n+1)}(z)=\\
&=-\log n +c_0\Big(\sum_{k=0}^n f^{\circ k}(z)\Big)+o(1)\nonumber -Log f^{\circ(n+1)}(z)=\\
&=c_0\Big(\sum_{k=0}^n f^{\circ k}(z)\Big)-i\pi+o(1).\nonumber  
\end{align}
The last equality is obtained using the development:
\begin{align*}
-&Log^+(f^{\circ (n+1)}(z))-\log n=\\
&\ =-Log^+\left(\varphi_+^{-1}\big(\frac{\varphi_+(z)}{1-(n+1)\varphi_+(z)}\big)\right)-\log n=\\
&\ =-Log^+\left[\left(\frac{\varphi_+(z)\cdot n}{1-(n+1)\varphi_+(z)}\right)\left(1+O\big(\frac{\varphi_+(z)}{1-(n+1)\varphi_+(z)}\big)\right)\right]=\\
&\ =-Log^+\left(\frac{1}{\frac{1-\varphi_+(z)}{n\varphi_+(z)}-1}\right) -Log^-\left(1+O\big(\frac{\varphi_+(z)}{1-(n+1)\varphi_+(z)}\big)\right)=\\
&\hspace{8cm}=\ -i\pi+o(1),\ n\to\infty. 
\end{align*}
Here, $Log^-(z)$ denotes the main branch of logarithm for $Arg(z)\in(-\pi,\pi)$ and $Log^+(z)$ the branch for $Arg(z)\in(0,2\pi)$. The function $\varphi_+(z)=z+a_1z^2+o(z^2)$ denotes the analytic change of variables on $V_+$ that reduces $f(z)$ to its formal normal form $f_0(z)=\frac{z}{1-z}$.

Passing to the limit in \eqref{lim2}, by \eqref{cc}, we get the pointwise equality:
$$
G_+(z)=\pi\cdot c_0\Big(\sum_{k=0}^{n}f^{\circ k}(z)\Big)-i\pi^2,\ z\in V_+.
$$
By \eqref{usp}, we conclude
\begin{equation}\label{fhf}
H^f(z)+\frac{\pi}{4}-i\pi^2=G_+(z).
\end{equation}
Therefore, since $G_+(z)$ is analytic on $V_+$, $H^f(z)$ is also analytic on $V_+$.

\smallskip

Analyticity of $H^{f^{-1}}$ on $V_-$ can be proven in the same manner, considering the inverse diffeomorphism $f^{-1}$. 
\hfill $\Box$
\bigskip
\smallskip

Note that equations \eqref{ha} and \eqref{hai} for principal parts of complex measures from Theoerem~\ref{ppart} resemble to the trivialization (Abel) equation 
\begin{equation*}
\Psi(f(z))-\Psi(z)=1
\end{equation*}
for a parabolic diffeomorphism. Recall that the Abel equation was used in Section~\ref{fourone} for obtaining the Ecalle-Voronin moduli of analytic classification. Sectorially analytic solutions on petals correspond to the Fatou trivialisation coordinates $\Psi_+$ on $V_+$ and $\Psi_-$ on $V_-$, whose comparison reveals the analytic class of $f$. 
\smallskip

We show next, in Corollary~\ref{conn}, that the geometrically obtained principal parts, $H^f(z)$ on $V_+$ and $H^{f^{-1}}(z)$ on $V_-$, can be considered as sectorial solutions of \emph{only one difference equation for the diffeomorphism $f$}, instead of both \eqref{ha} and \eqref{hai} for $f$ and $f^{-1}$ respectively. We consider the following \emph{1-Abel equation}\footnote{The $k$-Abel equations are introduced in Definition~\ref{propabelgen} at the beginning of the chapter.} for diffeomorphism $f$:
\begin{equation}\label{abelz}
H(f(z))-H(z)=- z.
\end{equation}

Let $H_+(z)$,\ $z\in V_+$, and $H_-(z)$,\ $z\in V_-$, denote the unique sectorially analytic solutions of $1$-Abel equation \eqref{abelz}
without the constant term, which exist by Proposition~\ref{formal}.  

\begin{corollary}[Up to an explicit constant, the sectorial solutions of $1$-Abel equation without the constant term give both principal parts of complex measures]\label{conn}
With the notations as above, the following relations hold:
\begin{align*}
&\pi H_+(z)-\frac{\pi}{4}+i\pi^2=H^f(z),\quad z\in V_+,\\
&\pi H_-(z)-\frac{\pi}{4}=z-H^{f^{-1}}(z),\quad z\in V_-.\nonumber
\end{align*}
\end{corollary}

\begin{remark}
In general case, when $f(z)=z+a_1z^{k+1}+o(z^{k+1})$, similar relations hold on all petals. The constants that are added are more complicated: they depend explicitely on the first $k$ coefficients of the diffeomorphism $f$ and on the petal we consider, by means of the choice of $(\pm a_1)^{-\frac{1}{k}}$. We are not going into details here. 
\end{remark}

\begin{proof}
By formulas \eqref{for} from the proof of Proposition~\ref{formal}, we get the following convergent series for the unique solutions $H_+(z)$ and $H_-(z)$ without the constant term of the equation \eqref{abelz}:
\begin{align}\label{ss}
H_+(z)&=\lim_{n\to\infty}\left(\sum_{k=0}^{n}f^{\circ k}(z)-Log f^{\circ(n+1)}(z)\right),\ z\in V_+,\\
H_-(z)&=\lim_{n\to\infty}\left(-\sum_{k=1}^{n+1}(f^{-1})^{\circ k}(z)-Log (f^{-1})^{\circ(n+1)}(z)\right),\ z\in V_-.\nonumber
\end{align}
We now compare the first formula \eqref{ss} with formulas \eqref{lim2}, \eqref{fhf} for the principal part $H^f(z)$ on $V_+$ in the previous proof. The same can be done for $H^{f^{-1}}(z)$ on $V_-$.
\end{proof}

\bigskip
The Fatou coordinates $\Psi_+$ and $\Psi_-$ glue to a global Fatou coordinate, analytic in some punctured neighborhood of the origin, if and only if $f$ belongs to the model analytic class of $f_0(z)=\frac{z}{1-z}$. Analogously, we give here the necessary and sufficient conditions on a diffeomorphism $f$ for global analyticity of its principal parts of complex measures. 
\begin{theorem}[\emph{Global} principal parts of complex measures]\label{pringlo}
The principal parts $H^f-i\pi^2$ on $V_+$ and $z-H^{f^{-1}}$ on $V_-$ glue to a global analytic function on a neighborhood of $z=0$ if and only if the diffeomorphism $f(z)$ is of the form
$$
f(z)=\varphi^{-1}\left(e^z\cdot \varphi(z)\right),
$$ 
for some analytic germ $\varphi(z)\in z+z^2\mathbb{C}\{z\}$. The principal parts are then given by 
\begin{align*}
&H^f(z)=-\pi Log(\varphi(z))+i\pi^2-\frac{\pi}{4},\ \ z\in V_+,\\
&H^{f^{-1}}(z)=z+\pi Log(\varphi(z))+\frac{\pi}{4},\ \ z\in V_-.
\end{align*}
Here, the branches of the complex logarithm are determined by the petals.
\end{theorem}

\begin{proof} The theorem is a consequence of Theorem~\ref{glo} and Corollary~\ref{conn}. By Corollary~\ref{conn}, the principal parts of complex measures are explicitely related to the sectorial solutions of the generalized Abel equation $H(f(z))-H(z)=- z$, with right-hand side $g(z)=- z$. By \eqref{f2} in Theorem~\ref{glo}, this equation has a global analytic solution if and only if $f(z)=\varphi^{-1}(\varphi(z)\cdot e^z)$, for some $\varphi(z)\in z+z^2\mathbb{C}\{z\}$. 
\end{proof}
\medskip

\begin{example}[Examples of germs with \emph{global} principal parts in the sense of Theorem~\ref{pringlo}]\ 
\begin{enumerate}
\item[$(1)$] $f(z)=z\cdot e^z$,\ for $\varphi(z)=id$, 
\item[$(2)$] $f(z)=-Log(2-e^z)$,\ for $\varphi(z)=1-e^{-z}$.
\end{enumerate}
\end{example}

\section{Principal parts of complex measures of\\ $\varepsilon$-neighborhoods of orbits and analytic classification}\label{fourfour}

\subsection{Counterexamples for reading the analytic class from principal parts}\label{fourfourone}
We saw in Section~\ref{fourone} that the Abel equation for a diffeomorphism $f$ has two sectorially analytic solutions $\Psi_+(z)$ and $\Psi_-(z)$ on Fatou petals $V_+$ and $V_-$ respectively, unique up to an additive constant. Furthermore, we saw that the analytic class of $f$ was readable from the difference of sectorial solutions $\Psi_+(z)-\Psi_-(z)$ on the intersection of petals, $z\in V^{up}\cap V^{low}$. The trivial difference means that the diffeomorphism belongs to the analytic class of $f_0$.

On the other hand, by Corollary~\ref{conn} in Subsection~\ref{fourtwothree}, the principal parts are directly related to the $1$-Abel equation:
\begin{equation}\label{haj}
H(f(z))-H(z)=-z.
\end{equation}
By Proposition~\ref{formal} in Section~\ref{fourthree}, this equation also has two sectorially analytic solutions $H_+(z)$ and $H_-(z)$, unique up to an additive constant, on petals $V_+$ and $V_-$. It makes sense to subtract the solutions on the intersection of petals as above, and to see if we can tell the analytic class by considering only this difference in an appropriate way.

We show here some examples that the answer is \emph{negative}. That is, when considering only the differences of the sectorial solutions $H_+(z)-H_-(z)$ on $V^{up}\cap V^{low}$, the information on the analytic class is lost.
We show it providing examples of diffeomorphisms both analytically conjugated and not conjugated to the model $f_0$, with trivial differences $H_+(z)-H_-(z)$.
Obviously, the trivial difference cannot be used in any way to distinguish between analytically conjugated and not conjugated cases.

Triviality of differences of sectorial solutions of \eqref{haj} means:
\begin{equation}\label{differences}
H_+(z)-H_-(z)\equiv 0,\ z\in V^{up},\ \ H_-(z)-H_+(z)\equiv -2\pi i,\ z\in V^{low}.
\end{equation}
The nontrivial term $-2\pi i$ on $V^{low}$ stems from different branches of logarithms on petals and cannot be eliminated. The condition \eqref{differences} is equivalent to the existence of a global analytic solution of \eqref{haj} and is, by Theorem~\ref{glo} in Section~\ref{fourthree}, realized on the set of diffeomorphisms:
$$
\mathcal{S}=\Big\{f(z)=z+z^2+z^3+o(z^3)\Big| f=\varphi^{-1}(e^z\cdot\varphi(z)),\Big.\ \varphi(z)\in z+z^2\mathbb{C}\{z\}\Big\}.
$$
Let, on the other hand, $\mathcal{C}_0$ be the class of diffeomorphisms analytically conjugated to $f_0(z)$. The following example shows that the intersection $\mathcal{S}\cap\mathcal{C}_0$ is nonempty. Furthermore, neither of the sets is a subset of another. 
\begin{example}[Position of $\mathcal{S}$ versus $\mathcal{C}_0$]\label{noincl}
\begin{align*}
&f(z)=-Log(2-e^z)\in \mathcal{S}\cap \mathcal{C}_0,\\
&g(z)=z e^z,\ g(z)\in\mathcal{S},\ g(z)\notin\mathcal{C}_0,\\
&f_0(z)\in\mathcal{C}_0,\ f_0(z)\notin\mathcal{S}. 
\end{align*}
\begin{proof}
In the first example, we take $\varphi^{-1}(z)=-Log(1-z)$, for both classes. The second example follows from the fact that no entire function is analytically conjugated to $f_0$, stated in paper of Ahern and Rosay \cite{ahern}. The third example follows from Example~\ref{ex1} below, which shows non-triviality of differences $H_+-H_-$ for the model diffeomorphism $f_0(z)$.\end{proof}
\end{example}

\medskip
In the next example, we compute explicitely the differences $H_+(z)-H_-(z),\ z\in V^{up}\cup V^{low},$ for the simplest model diffeomorphism $f_0$. The difference of sectorial trivialisations $\Psi_+(z) -\Psi_-(z)$ was in this case trivial. We apply the method of Borel-Laplace summation, described at the beginning of the chapter, directly to the difference equation. The procedure is standard and a similar one can be found in e.g. \cite[Example 2]{dudko}.
\begin{example}[The differences for the model germ $f_0(z)=\frac{z}{1-z}$]\label{ex1}
We substitute $\widehat{H}(z)=-Log(z)+\widehat{R}(z),\ \widehat{R}(z)\in z\mathbb{C}[[z]],$ in the equation \eqref{haj} for $f_0$ and thus obtain the equation for $\widehat{R}(z)$:
\begin{equation*}
\widehat{R}(f_0(z))-\widehat{R}(z)=-z+Log\frac{f_0(z)}{z}.
\end{equation*}
By the change of variables $w=-\frac{1}{z}$, denoting $\widehat{\widetilde R}(w)=\widehat{R}\circ \big(-\frac{1}{w}\big)$, we get
\begin{align}\label{er}
\widehat{\widetilde{R}}(w+1)-\widehat{\widetilde{R}}(w)&= w^{-1}- Log(1+w^{-1})=\sum_{k=2}^{\infty}(-1)^{k}\frac{w^{-k}}{k}.
\end{align}
Here, $\widehat{\widetilde R}(w)=\widehat{R}\circ \big(-\frac{1}{w}\big)$. The right-hand side of this equation is of the type $w^{-2}\mathbb{C}\{w^{-1}\}$. We denote it by $$b(w)=\sum_{k=2}^{\infty}\frac{(-1)^k}{k}w^{-k}.$$  

Applying the Borel transform to \eqref{er} and using property \eqref{propi} of Borel transform, we get
\begin{equation*}
\mathcal{B}\widehat{\widetilde{R}}(\xi)=\frac{\mathcal{B}b(\xi)}{e^{-\xi}-1}, \ \ \mathcal{B}b(\xi)=\frac{e^{-\xi}+\xi-1}{\xi}.
\end{equation*}
The function $\mathcal{B}\widehat{\widetilde{R}}(\xi)$ has $1$-poles at $2i\pi\mathbb{Z}^*$ in directions $\theta=\pm\frac{\pi}{2}$, and it is exponentially bounded and analytic in every other direction. Indeed, since $b(w)$ is analytic $\big(b(1/z)$ has radius of convergence $|z|<R_0\big)$, it holds that there exists a constant $C_0>0$ such that
$$
|\mathcal B b(\xi)|\leq C_0 e^{\beta_0|\xi|},\ \text{ for every }\beta_0>\frac{1}{R_0}.
$$
Furthermore, there exists a continuous function $C(\theta)$ on $\theta\in(-\pi/2,\pi/2)$, $\theta\in(\pi/2,3\pi/2)$, such that
$$
\left|\frac{\mathcal B b(\xi)}{e^{-\xi}-1}\right|\leq C(\theta) e^{\beta_0|\xi|},\ Arg(\xi)=\theta.
$$
Therefore, $\widehat{\widetilde R}(w)$ is 1-summable in the arcs of directions $I_+=(-\pi/2,\pi/2)$ and $I_-=(\pi/2,3\pi/2)$. The Laplace transform yields two analytic solutions as $1$-sums, $\widetilde R^+(w)$ on $W_+=\{w\ |\ Re(w e^{i\theta})>\beta_0,\ \theta\in I_+\}$, and $\widetilde R^{-}(w)$ on $W_-=\{w\ |\ Re(w e^{i\theta})>\beta_0,\ \theta\in I_-\}$. By the residue theorem applied to the difference of Laplace integrals, on intersections of $W_+$ and $W_-$ they differ by $1$-periodic functions. For $w\in W^{up}=\{w|\ Im(w)>\beta_0\}$, we have:
\begin{align*}
\widetilde{R}^+(w)-\widetilde{R}^-(w)&=\int_{0}^{\infty\cdot e^{i\theta_1}} \frac{e^{-\xi w}\mathcal{B}b(\xi)}{e^{-\xi}-1}d\xi  -\int_{0}^{\infty\cdot e^{i\theta_2}}\frac{e^{-\xi w}\mathcal{B}b(\xi)}{e^{-\xi}-1}d\xi=\\
&=\int_{\infty\cdot e^{i\theta_2}}^{\infty\cdot e^{i\theta_1}} \frac{e^{-\xi w}\mathcal{B}b(\xi)}{e^{-\xi}-1}d\xi =\\
&=2\pi i\cdot\sum_{k=1}^{\infty} Res(\frac{e^{-\xi w}\mathcal{B}b(\xi)}{e^{-\xi}-1},\xi=-2\pi i k)=\\
&=-2\pi i\sum_{k\in\mathbb{N}}e^{2\pi i k\cdot w}=-2\pi i\frac{e^{2\pi i \cdot w}}{1-e^{2\pi i \cdot w}}.
\end{align*}
Here, $\theta_2\in(-\pi/2,\pi/2)$ and $\theta_1\in(\pi/2,3\pi/2)$ are close to $-\pi/2$. To be precise, some limit argument is needed to justify the application of the residue theorem along the curve with endpoints at $\infty$. We consider sectors $S_r$ of finite radii $r$, $r\to\infty$, bouded by directions $\theta_1$ and $\theta_2$. We close the sectors not necessarily by circles, but by arcs $I_r$, which avoid all $2\pi i \mathbb{N}^-$ by some fixed positive distance. The radii of points on the arcs are bounded by $r$ from below and from above, as $r\to \infty$, with the same constants independent of $r$. The indefinite integral (along the infinite curve) can be written as limit of integrals along outer lines of sectors $S_r$, as $r\to\infty$. Applying the residue theorem for each sector $S_r$, we get that the original integral is the sum of all residues at $2\pi i\mathbb{N}^-$, minus the limit of integrals along the arcs $I_r$, as $r\to\infty$. The latter limit is zero. Indeed, since $I_r$ avoid all poles by some fixed positive distance, the subintegral function on $I_r$ can be bounded by $Ce^{-Ar}$, with $C,\ A>0$, independent of $r$. It holds that
$$
\left|\frac{\mathcal{B}b(\xi)}{e^{-\xi}-1}e^{-w \xi}\right|\leq \left|\frac{\mathcal{B}b(\xi)}{e^{-\xi}-1}\right|\cdot|e^{-w \xi}|\leq C_1 e^{\beta_0 r}\cdot e^{-r Re(we^{i\theta(\xi)})},\ \xi\in I_r=\{re^{i\theta(\xi)},\ \theta(\xi)\in(\theta_1,\theta_2)\},
$$
where $C_1>0$, independent of $r$.
Since, for fixed $w\in W^{up}$, it holds that there exists $\delta>0$ such that $Re(we^{i\theta})>\beta_0+\delta$, $\theta\in(\theta_1,\theta_2)$, the exponential bound follows. The lengths of arcs $I_r$, on the other hand, grow no faster than linearly in $r$, as $r\to\infty$.

Similarly, for $w\in W^{low}=\{w|\ Im(w)<-\beta_0\}$, we get
\begin{align*}
\widetilde{R}^+(w)-\widetilde{R}^-(w)&=2\pi i\frac{e^{-2\pi i \cdot w}}{1-e^{-2\pi i \cdot w}}.
\end{align*}
\noindent Replacing $\widehat{\widetilde R}(w)$ with $\widehat{\widetilde H}(w)$ and returning to the variable $z=-\frac{1}{w}$, we get
\begin{align*}
H^+(z)-H^-(z)&=-2\pi i \frac{e^{-2\pi i \frac{1}{z}}}{1-e^{-2\pi i \frac{1}{z}}}=-2\pi i f_0(e^{-2\pi i \frac{1}{z}}),\ z\in V^{up},\\[0.1cm]
H^-(z)-H^+(z)&=-2\pi i-2\pi i \frac{e^{2\pi i \cdot \frac{1}{z}}}{1-e^{2\pi i \cdot \frac{1}{z}}}=-2\pi i-2\pi i f_0(e^{2\pi i \cdot \frac{1}{z}}),\ z\in V^{low}.\nonumber
\end{align*}
Here, $V_+$ and $V_-$ are petals in the $z$-plane, obtained by inverting $W_+$ and $W_-$ by $z=-1/w$, and $V^{up}$ and $V^{low}$ are their intersections, that is, the inverse images of $W^{up}$ and $W^{low}$.
\end{example}

We see that for the model germ $f_0(z)$, the \emph{1-cocycle} generated by $\widehat{H}(z)$ is represented on the space of orbits exactly by the germ $-2\pi i f_0(t)$ itself, in both components. For theory of 1-cocycles, see the end of Section~\ref{fourone}. This is certainly not a coincidence. It would be interesting to have some geometrical explanation.

In the above manner, the differences can be computed by Borel-Laplace transform for any diffeomorphism $f$ analytically conjugated to $f_0$, and it can be seen in general that the cocycles are not trivial. 
\smallskip

\begin{example}[Explicit formulas for the sectorial solutions $H^{f_0}_{\pm}(z)$ for the model diffeomorphism $f_0$]\label{ex2}
Using \eqref{usp} in the proof of Theorem~\ref{ppart} and Corollary~\ref{conn}, we get that
\begin{align*}
H^{f_0}_+(z)&=\pi\cdot c_0\Big(\sum_{k=0}^{n}\frac{z}{1-kz}\Big)-i\pi^2=\pi\cdot \frac{d}{dw}Log (\Gamma(w))\Big|_{-\frac{1}{z}}\Big.-i\pi^2,\ z\in V_+,\\
H^{f_0}_-(z)&=\pi z-\pi\cdot c_0\Big(\sum_{k=0}^{n}\frac{z}{1+kz}\Big)=\pi z+\pi\cdot \frac{d}{dw}Log (\Gamma(w))\Big|_{\frac{1}{z}}\Big.,\ z\in V_-.
\end{align*}
Here, $\Gamma(z)$ is the standard Gamma function, holomorphic on $\mathbb{C}\setminus \{-\mathbb{N}_0\}$. Therefore, $H^{f_0}_{\pm}(z)$ are well-defined and analytic on $V_\pm$.
\end{example}

\subsection{Higher-order moments and higher conjugacy classes.}\label{fourfourtwo}

So far, we have seen that Abel and $1$-Abel equations for germs of diffeomorphisms appear naturally in applications, in the context of the trivialisation functions or of the principal parts of complex measures of $\varepsilon$-neighborhoods of orbits. In this section, we put Abel equation and $1$-Abel equation in a more general context of \emph{$k$-Abel equations}, $k\in\mathbb{N}_0$. 

We furthermore define \emph{new classifications of diffeomorphisms with respect to their $k$-Abel equations}, mimicking the way in which Abel equation was exploited for defining analytic classes. We will call the new classes the $k$-conjugacy classes. 

\bigskip

Let
\begin{equation}\label{korder}
H(f(z))-H(z)=-z^k,\ k\in\mathbb{N}_0,
\end{equation}
be the \emph{$k$-Abel equation} for a diffeomorphism $f$, as defined in Definition~\ref{propabelgen} at the beginning of the chapter. 

It was shown in Proposition~\ref{formal} in Section~\ref{fourthree} that there exist two analytic sectorial solutions $H_+^k(z)$ and $H_-^k(z)$ of \eqref{korder} on petals $V^+$ and $V^{-}$, which define the $1$-cocycle 
\begin{equation}\label{fuj}
\Big(H_+^k(z)-H_-^k(z),\ z\in V^{up};\ \ H_-^k(z)-H_+^k(z),\ z\in V^{low}\Big).
\end{equation}
We now repeat the procedure from the very end of Section~\ref{fourone}. There, the analytic moduli of $f$ were recovered using differences of sectorial solutions of the Abel equation for $f$ on intersections of petals. Here, we use the differences of sectorial solutions of $k$-Abel equations, to define new classifications of diffeomorphisms. 

\medskip
Subtracting \eqref{korder} for both solutions, $H_+^k(z)$ and $H_-^k(z)$, on intersections of petals, we get that
$H_+^k(z)-H_-^k(z)$  is constant along the closed orbits in $V^{up}$ and $V^{low}$.
  
Therefore, the cocycle is a well-defined function on the space of closed orbits. As already explained in Section~\ref{fourone}, the space of orbits on each petal can be represented as a Riemann sphere, by composition of trivialization functions of the petal and the exponential function $e^{-2\pi i w}$. To avoid confusion, for representation of functions defined on the space of orbits in the intersection of petals, in the sequel we choose always \emph{trivialisations of attractive petals}. Then, $V^{up}$ and $V^{low}$ correspond to the punctured neighborhoods of $t=\infty$ and $t=0$ of the sphere.  The $1$-cocycle \eqref{fuj} on $V^{up}$ and $V^{low}$ is thus lifted to a pair of germs $\big(g_{\infty}^k(t),\ g_0^k(t)\big )$ on neighborhoods of poles $t=\infty$ and $t=0$ of the Riemann sphere:
\begin{align}\label{baz} 
&H_+^k(z)-H_-^k(z)=g_\infty^k(e^{-2\pi i \Psi_+(z)}),\ z\in V^{up},\\[0.1cm]\nonumber
&H_-^k(z)-H_+^k(z)=(-2\pi i)+ g_0^k(e^{-2\pi i \Psi_+(z)}),\ z\in V^{low}.
\end{align}
The term $-2\pi i$ is put in brackets, since it appears only in the case when $k=1$, due to different branches of the complex logarithm on petals. We further transform $g_\infty^k(t)=g_\infty^k(1/t)$, to obtain two analytic germs at zero. In the same way as at the end of Section~\ref{fourone}, it follows that the germs can be extended analytically to $t=0$.
Therefore, we conclude that the cocycle \eqref{fuj} of differences of sectorial solutions on intersections of petals can be represented by a pair of analytic germs at zero, $\big(g_\infty^k(t),g_0^k(t))$.
\smallskip

We note that the trivialisation function $\Psi_+(z)$ is uniquely determined only up to an arbitrary constant. Also, if we add any complex number to $H_+^k(z)$ or $H_-^k(z)$, they remain the solutions of the $1$-Abel equation \eqref{haj}. As before, due to this freedom of choice, we identify two pairs of analytic germs if \eqref{idii} from the end of Section~\ref{fourone} holds. Note that it always holds that $g_0^k(0)+g_\infty^k(0)=0$, since the constant term is the difference of constant terms chosen in solutions $H_+^k(z)$ and $H_-^k(z)$.

\begin{definition}[$k$-moments for diffeomorphisms]\label{momentdefinition}
The \emph{$k$-moment of a diffeomorphism $f$ with respect to a trivialization function of the attracting petal} or, shortly, \emph{the $k$-moment of $f$}, is the pair $$\big(g_\infty^k(t),\ g_0^k(t)\big)$$ of analytic germs\footnote{The notion \emph{germ} refers to a function defined on a small neighborhood of the origin, not addressing the size of its domain. That is, two germs are identified if they are equal on any neighborhood of the origin.} at $t=0$ from \eqref{baz}, up to the identifications \eqref{idii}.
\end{definition}

\begin{remark}\label{restr}
In the case of $1$-Abel equations, the $1$-moments are in fact defined subtracting the sectorial solutions $R_+(z)-R_-(z)$, $z\in V^{up}\cup V^{low}$, of the modified equation 
\begin{equation*}
R(f(z))-R(z)=-z+Log\Big(\frac{f(z)}{z}\Big),
\end{equation*}
instead of sectorial solutions $H_+(z)-H_-(z)$ of the original $1$-Abel equation \eqref{haj}. Here, $H(z)=-Log(z)+R(z)$. In this way we remove the constant term $-2\pi i$ in \eqref{baz}, coming from different branches of the logarithm. 
\end{remark}
\medskip

We now classify diffeomorphisms of the formal type $(k=1,\lambda=0)$ into equivalence classes, putting those which share the same $k$-moment (up to the identifications \eqref{idii}) inside the same class.
\begin{definition}[The $k$-conjugacy relation on a set of diffeomorphisms]
Let $k\in\mathbb{N}_0$. The \emph{$k$-conjugacy} is the equivalence relation on the set of all diffeomorphisms formally equivalent to $f_0$, given by:
$$
f_1\stackrel{k}\sim f_2 \text{, if and only if $f$ and $g$ have the same $k$-moments up to the identifications \eqref{idii}}.
$$
By $$[f]_k=\{g\ |\ g\stackrel{k}\sim f\}$$
we denote the equivalence class of $f$ with respect to the $k$-conjugacy.
\end{definition}

\medskip
\noindent We illustrate the definition on the two most important examples for this work.
\begin{example}[$0$- and $1$-conjugacy classes]\ 
\begin{enumerate}
\item The $0$-Abel equation is in fact the Abel equation. The $0$-conjugacy classes correspond to the standard analytic classes. The $0$-moments correspond to the Ecalle-Voronin moduli, as they are described in Theorem~\ref{fr1}. The diffeomorphisms analytically conjugated to the model $f_0$ have the trivial $0$-moment, the pair $(0,0)$ (the Abel equation has globally analytic solution).
\item The $1$-conjugacy classes are obtained using $1$-Abel equations \eqref{haj}. By Theorem~\ref{glo}, the trivial $1$-conjugacy class (the set diffeomorphisms with the $1$-moments equal to $(0,0)$, that is, the set of all diffeomorphisms with globally analytic solutions of equation \eqref{haj}) is the set $$\mathcal{S}=\Big\{f\ |\ f=\varphi^{-1}(e^z\cdot\varphi(z)),\Big.\ \varphi(z)\in z+z^2\mathbb{C}\{z\}\Big\}.$$
\end{enumerate}
\end{example}
\medskip

We finish the section with converse question of \emph{realization} of $0$-moments and $1$-moments. 
 
\begin{proposition}[Realization of $0$-moments]\label{remi} For every pair $(g_1(t),g_2(t))$ of analytic germs at zero, such that $g_1(0)+g_2(0)=0$, there exists a diffeomorphism $f$, such that the pair $\big(g_1(t),g_2(t)\big)$ is realised as is its $0$-moment. 
\end{proposition}  
\begin{proof}
The $0$-moments are in fact the Ecalle-Voronin moduli, as they are defined in Theorem~\ref{fr1}. The statement follows directly from Theorem~\ref{fr1}.  
\end{proof}

\begin{proposition}[Realization of $1$-moments]\label{realis}
For every pair $(g_1(t),g_2(t))$ of analytic germs at zero, such that $g_1(0)+g_2(0)=0$, there exists a diffeomorphism $f$ from the model formal class, such that the pair $\big(g_1(t),g_2(t)\big)$ is realised as its $1$-moment. 
\end{proposition}
Note that by varying the constant term chosen in sectorial trivialisation function $\Psi_+(z)$ and constants chosen in solutions $H_+(z)$ and $H_-(z)$, we can realise all other $1$-moments identified by \eqref{idii} using the same diffeomorphism $f$. 
\begin{proof}
This proposition follows from Theorem~\ref{surjecti} stated and proven in Section~\ref{fourfourthree}.
\end{proof}
The question is important since it states that all possible $0$- or $1$-conjugacy classes may be represented by all possible pairs of analytic germs $(g_1(t),g_2(t))$ at zero such that $g_1(0)+g_2(0)=0$, up to identifications \eqref{idii}. 

We do not address the question of realisation for the higher conjugacy classes. In this thesis, they are introduced only because they present the natural context in which $0$- and $1$-conjugacy classes appear. They remain a subject of further research.

\subsection{Relative position of the $1$-conjugacy classes and the analytic classes.}\label{fourfourthree}

It was noted in Example~\ref{noincl} in Subsection~\ref{fourfourone} that there exists no inclusion relation between the trivial analytic class and the trivial $1$-conjugacy class. We investigate here more precisely the relative position of analytic classes and $1$-conjugacy classes. We prove that they lie in a transversal position, that is, they are \emph{far away} and not related to each other. In this way, we explain and support theoretically the counterexamples from Subsection~\ref{fourfourone}, that claimed that the analytic classes cannot be read only from the differences of sectorial solutions of the $1$-Abel equation. On the other hand, they can be read from the differences of sectorial solutions of the Abel equation. 

The relative positions of higher conjugacy classes to each other are not discussed and remain the subject for further research.
\medskip

Let $\Phi$ denote the mapping
$$
\Phi(f)=[f]_1,
$$
defined on the set of all diffeomorphisms. It attributes to each diffeomorphism its $1$-conjugacy class, that is, the appropriate pair of analytic germs up to the identifications \eqref{idii}.
\bigskip

The next theorem states that not only every pair of analytic germs $\big(g_1(t),g_2(t)\big)$ such that $g_1(0)+g_2(0)=0$ can be realized as the $1$-moment of some diffeomorphism, but it is moreover realized inside each analytic class.
\begin{theorem}[Surjectivity from each analytic class onto the set of all $1$-conjugacy classes]\label{surjecti}
Let $[f]_0$ denote any analytic class. The restriction $\Phi\Big/[f]_0\Big.$ maps surjectively from $[f]_0$ onto the set of all $1$-conjugacy classes. 
\end{theorem}

We first give the outline of the proof. Take any analytic class $[f]_0$ and any representative $f$. Let $(g_1(t),g_2(t))$ be any pair of analytic germs, satisfying $g_1(0)+g_2(0)=0$. We will show that there exists a diffeomorphism $g\in[f]_0$ whose $1$-moment is equal to $(g_1(t),g_2(t))$. We first show that there exists an analytic, tangent to the identity right-hand side $\delta(z)$ of the generalized Abel equation for $f(z)$, such that the pair $(g_1(t),g_2(t))$ represents the moment of $f$ with respect to this equation. The idea for first part is \emph{borrowed} from \cite[A.6]{loraypre}. Then, simply by a change of variables, we transform the equation to $1$-Abel equation, but for a different diffeomorphism. This new diffeomorphism is analytically conjugated to $f$ by $\delta(z)$.
\smallskip

\emph{Proof of Theorem~\ref{surjecti}.} Let $[f]_0$ be any analytic class and $f\in [f]_0$ any representative. Moreover, let $\Psi_+^f(z)$ be any trivialisation of the attracting sector $V_+$ for $f$.

On some petals $V^{up}$ and $V^{low}$ of opening $\pi$ and centered at $\pi/2$ and $-\pi/2$ respectively,
we define the pair $\big(T_\infty(z),T_0(z)\big)$ by:
\begin{align}\label{jen}
T_\infty(z)&=g_1(e^{2\pi i \Psi_+^{f}(z)}),\ z\in V^{up},\nonumber\\
T_0(z)&=g_2(e^{-2\pi i \Psi_+^{f}(z)}),\ z\in V^{low}.
\end{align}
If $g_1(0),\ g_2(0)\neq 0$, we first subtract the constant term. This can be done without loss of generality, since a constant term can be added to any sectorial solution afterwards.
The functions $T_0(z)$ and $T_\infty(z)$ are thus analytic, exponentially decreasing of order one\footnote{$|T_{0,\infty}(z)|\leq C e^{-A/|z|}$, for some positive constants $A,\ C$.} on $V^{up}$ and $V^{low}$. Therefore, the pair \eqref{jen} defines a $1$-cocycle in the sense of definition from Section~\ref{fourone}. By surjectivity in Theorem~\ref{koji}, there exists a 1-summable formal series $\widehat{H}(z)\in z\mathbb{C}\{z\}_1$, whose differences of $1$-sums $H_+(z)$ on $V_+$ and $H_-(z)$ on $V_-$ realize the cocycle $(T_\infty(z),T_0(z))$. That is,
\begin{equation}\label{dva}
T_0(z)=H_+(z)-H_-(z) \text { on $V^{up}$},\quad T_\infty(z)=H_-(z)-H_+(z) \text{ on $V^{low}$}.
\end{equation}

We adapt now slightly functions $H_+(z)$ and $H_-(z)$ by adding the appropriate branch of logarithm,
\begin{equation}\label{adapt}
\widetilde{H}_+(z)=-H_+(z)+Log(z),\ z\in V_+;\ \ \widetilde{H}_-(z)=-H_-(z)+Log(z),\ z\in V_-.
\end{equation}
We define functions $\delta_\pm(z)$ on $V_{\pm}$ respectively by:
\begin{align*}
\delta_+(z)&=\widetilde{H}_+(f(z))-\widetilde{H}_+(z),\ z\in V_+,\nonumber\\ 
\delta_-(z)&=\widetilde{H}_-(f(z))-\widetilde{H}_-(z),\ z\in V_-. 
\end{align*}
From \eqref{jen} and \eqref{dva}, computing the difference $\delta_+(z)-\delta_-(z)$, we see that $\delta_+$ and $\delta_-$ glue to an analytic germ $\delta(z)$. By \eqref{adapt}, $\delta(z)\in z+z^2\mathbb{C}\{z\}$, tangent to the identity. 

To conclude, $\widetilde{H}_+(z)$ and $\widetilde{H}_-(z)$ are sectorial solutions of the generalized Abel equation for the diffeomorphism $f(z)$, with the right-hand side $\delta(z)$. That is,
$$
\widetilde{H}(f(z))-\widetilde{H}(z)=\delta(z).
$$
By the analytic change of variables $w=\delta(z)$ and multiplying by $(-1)$, we get
$$
-\widetilde H\circ\delta^{-1}(\delta\circ f\circ \delta^{-1}(w))-(-\widetilde H\circ\delta^{-1})(w)=-w.
$$
Therefore, $-(\widetilde{H}\circ\delta^{-1})_{\pm}(z)=-(\widetilde{H}_{\pm}\circ\delta^{-1})(z)$, \footnote{We are a little bit imprecise here. The new petals $V_{\pm}$ are in fact images $\delta(V_{\pm})$ of original petals. They can be identified, since $\delta$ is a conformal map that preserves shapes and angles, moreover tangent to the identity. The petals are again of opening $2\pi$, centered at the same directions.}$z\in V_\pm$, are solutions of $1$-Abel equation for diffeomorphism $g=\delta\circ f\circ \delta^{-1}$, analytically conjugated to $f$. The former equality on petals holds\footnote{In fact, the statement holds in more generality, without using the fact that formal series $\widehat{H}(z)$ is a solution of some generalized Abel equation. By \cite[Theorem 13.3]{sauzinetu}, if $\widehat{H}$ is $1$-summable in arcs of directions $I_1=(-\pi/2,\pi/2)$ and $I_2=(\pi/2,3\pi/2)$, with $1$-sums $H_+$ on $V_+$ and $H_-$ on $V_-$, and if $\varphi(z)$ is an analytic diffeomorphism tangent to the identity, then $\widehat{H}\circ\varphi$ is again $1$-summable in the same arcs of directions, with $1$-sums $(H\circ \varphi)_\pm(z)=H_\pm\circ\varphi(z)$, $z\in V_\pm$.} by formulas from Proposition~\ref{formal} applied to both generalized Abel equations, since $\delta(z)$ is an analytic change of variables. Furthermore, by \eqref{jen} and \eqref{dva},
\begin{align*}
-(\widetilde H_+&\circ\delta^{-1})(z)+(\widetilde H_-\circ\delta^{-1})(z)=T_\infty(\delta^{-1}(z))=\\
&=g_1(e^{2\pi i \Psi_+^{f}\circ \delta^{-1}(z)})=g_1(e^{2\pi i \Psi_+^{g}}(z)),\ z\in V^{up},\\
-(\widetilde H_-&\circ\delta^{-1})(z)+(\widetilde H_+\circ\delta^{-1})(z)=-2\pi i+T_0(\delta^{-1}(z))=\\
&=-2\pi i+g_2(e^{-2\pi i \Psi_+^{f}\circ \delta^{-1}(z)})=-2\pi i+g_2(e^{-2\pi i \Psi_+^{g}}(z)),\ z\in V^{low}.
\end{align*}
Here, $\Psi_+^{g}(z)=\Psi_+^f(z)\circ\delta^{-1}(z)$ is a trivialisation function for $g$, for an appropriate choice of constant term, see Lemma~\ref{idiot}.
\smallskip

Thus, the cocycle $(g_1(t),g_2(t))$ is realized as $1$-moment of the diffeomorphism $g(z)$, analytically conjugated to $f(z)$. 
$\hfill\Box$

\bigskip
We pose the question of \emph{injectivity} in Theorem~\ref{surjecti}. That is, if inside each analytic class there exist different diffeomorphisms with the same $1$-moments. We show in the next Proposition~\ref{formclas} that the injectivity is not true. Inside the trivial analytic class, we even characterize the diffeomorphisms that have the same $1$-moments in Proposition~\ref{chartriv}. 

\begin{proposition}[Non-injectivity]\label{formclas}
Let $[f]_0$ be any analytic class. Let $f,\ g\in [f]_0$. If there exists an analytic change of variables $\varphi(z)\in z+z^2\mathbb{C}\{z\}$ conjugating $f$ to $g$, $g=\varphi^{-1}\circ f\circ \varphi$, of the form
\begin{equation}\label{classi}
\varphi^{-1}(z)=id+r(f(z))-r(z),
\end{equation}
where $r(z)$ is an analytic germ, $r(z)\in\mathbb{C}\{z\}$, then $f$ and $g$ have the same $1$-moment. 
\end{proposition}
Note that although $f$ and $g$ belong to the same analytic class, \emph{not every formal change of variables conjugating $g$ to $f$ is necessarily analytic}. See proof of Lemma~\ref{idiot} for description of all formal changes conjugating $f$ and $g$. We only know that at least one formal change is analytic. Therefore, the request on change of variables being analytic in Proposition~\ref{formclas} is not superfluous.

\begin{proposition}[Characterization of diffeomorphisms analytically conjugated to $f_0$ with the same $1$-moment]\label{chartriv}
Let $f,\ g\in [f_0]_0$ be analytically conjugated to $f_0$. $f$ and $g$ have the same $1$-moment if and only if there exists a change of variables $\varphi(z)\in z+z^2\mathbb{C}\{z\}$ conjugating $f$ to $g$, $g=\varphi^{-1}\circ f\circ \varphi$, of the form
\begin{equation*}
\varphi^{-1}(z)=id+r(f(z))-r(z),
\end{equation*}
where $r(z)$ is an analytic germ, $r(z)\in\mathbb{C}\{z\}$.
\end{proposition}

\begin{remark}[About the statement of Propositions~\ref{formclas} and \ref{chartriv}]\
\begin{enumerate}
\item Note that in the propositions it suffices that only one conjugacy between $f$ and $g$ satisfies \eqref{classi}. We have seen the relation between all conjugacies expressed in Lemma~\ref{idiot}. The following question remains: if \eqref{classi} is satisfied for one conjugacy, does this imply that \eqref{classi} is in fact satisfied for all other conjugacies between $f$ and $g$? 
\item The accent in the propositions is on $r(z)$ being \emph{globally analytic}. Indeed, for any change of variables $\varphi^{-1}(z)$, there exists a sectorially analytic function $r(z)$ such that \eqref{classi} holds. Equation \eqref{classi} can be rewritten as the generalized Abel equation \begin{equation}\label{ti}r(f(z))-r(z)=\varphi^{-1}(z)-z,\end{equation} and the existence of sectorial solutions $r_+(z)$ and $r_-(z)$ is given in Proposition~\ref{formal}. However, \emph{good} changes of variables are only those $\varphi(z)$, for which equation \eqref{ti} with right-hand side $\varphi^{-1}(z)-z$ has a globally analytic solution.
\item The propositions are \emph{constructive}. Using \eqref{classi}, for any diffeomorphism $f$ we can construct infinitely many diffeomorphisms inside its analytic class, such that they all belong to the same $1$-conjugacy class.
\item The question remains if Proposition~\ref{chartriv} is true for all analytic classes, not only for the trivial analytic class. There seems to be a technical obstacle in the proof, which we do not know how to bypass.
\end{enumerate}
\end{remark}

The idea of the proof is simple. The proof follows from the proof of Theorem~\ref{surjecti}. Going through the proof, we note the ambiguity in the definition of the diffeomorphism $g(z)$: the formal series $\widehat{H}(z)$ may be chosen up to adding a convergent series (see the bijectivity statement in Theorem~\ref{koji}). In other words, the formal series $\widehat{H}(z)$ that realizes the same cocycle $(T_\infty(z),T_0(z))$ in Theorem~\ref{surjecti} is unique up to addition of a convergent series. The formal proof follows.

\smallskip
For the proof, we need the following known result. 
\begin{lemma}[Non-uniqueness of formal conjugation, reformulation of Theorem 21.12 from \cite{ilyajak}]\label{idiot}
Let $f$ be formally conjugated to $f_0$. The formal conjugation $\widehat\varphi(z)$ is unique up to a precomposition by analytic germs of the form $$f_c(z)=\frac{z}{1-cz}\in z+z^2\mathbb{C}\{z\},\quad c\in\mathbb{C}.$$ This only freedom of choice is related to addition of the constant term $+c$ in the trivialisation series $\widehat\Psi^f(z)$ of $f$. 

Furthermore, for any two germs $f$ and $g$ formally conjugated to $f_0$, and for any choice of formal trivialisations $\widehat{\Psi}^f(z)$ and $\widehat{\Psi}^g(z)$ (meaning, for any choice of constant term in them), there exists a formal conjugation $\widehat\varphi(z)\in z+\mathbb{C}[[z]]$, such that it holds 
\begin{equation}\label{sto}
g=\widehat\varphi^{-1}\circ f\circ \widehat\varphi,\text{ \ and \ \ }\widehat\Psi^g=\widehat \Psi^f\circ \widehat \varphi.
\end{equation}
Also, for any formal conjugation $\widehat\varphi(z)\in z+z^2\mathbb{C}[[z]]$, there exist trivialisations  $\widehat{\Psi}^f(z)$ and $\widehat{\Psi}^g(z)$ such that \eqref{sto} holds. All possible choices of constants in trivialisation series result in all possible conjugacies $\widehat \varphi(z)$ conjugating $f$ and $g$.
\end{lemma}

\begin{proof}
Consider a germ $f$ formally conjugated to $f_0(z)=\frac{z}{1-z}$. Then \begin{equation}\label{co}f=\widehat\varphi^{-1}\circ f_0\circ\widehat\varphi,\ \varphi(z)\in z+z^2\mathbb{C}[[z]].\end{equation} If we put $\widehat\Psi(z)=\Psi_0\circ \widehat \varphi(z)$, for $\Psi_0(z)=-1/z$, from \eqref{co} we get the trivialisation equation for $f$, $\widehat\Psi(f(z))-\widehat\Psi(z)=1$. This equation has a unique formal solution up to a constant term, and therefore $\widehat\varphi(z)=-1/z\circ \widehat\Psi$ is also unique up to a controled transformation that we derive here. Indeed, if we change the formal transformation by a constant term, 
\begin{equation}\label{need}
\widehat{\Psi}_1(z)=\widehat{\Psi}(z)+c,\ c\in\mathbb{C},
\end{equation}
and search for a formal change $\widehat\varphi_1(z)\in z+z^2\mathbb{C}[[z]]$, such that it holds 
\begin{equation}\label{rac}
\widehat{\Psi}_1(z)=-1/z\circ \widehat{\varphi}_1(z).
\end{equation} 
If such change $\widehat\varphi_1(z)$ exists, then, putting $\widehat\varphi_1(z)=-1/z\circ \widehat\Psi_1(z)$ in the trivialisation equation for $\widehat\Psi_1(z)$, we get that $\widehat\varphi_1(z)$ also conjugates $f$ to $f_0$. Putting \eqref{need} and $\widehat\varphi(z)=-1/z\circ \widehat\Psi(z)$ in \eqref{rac}, we get the formula for the new conjugation:
$$
\widehat\varphi_1(z)=\frac{z}{1-cz}\circ \widehat\varphi(z).
$$
\smallskip 

Now, let $f$ and $g$ be formally conjugated by $\widehat\varphi(z)$, \begin{equation}\label{fandg}g=\widehat\varphi^{-1}\circ f\circ\widehat\varphi.\end{equation}
Let $\widehat \Psi^g$ and $\widehat \Psi^f$ be any two trivialisations (any choice of constant term). Let $\widehat\varphi_g$ and $\widehat\varphi_f$ be the corresponding conjugacies such that $\widehat\Psi^g=(-1/z)\circ \widehat\varphi_g$ and $\widehat\Psi^f=(-1/z)\circ \widehat\varphi_f$. It holds, simply by transformations of trivialisation equations, that
\begin{align}\label{jaoj}
f=\widehat\varphi_f^{-1}\circ f_0\circ\widehat\varphi_f,\\
g=\widehat\varphi_g^{-1}\circ f_0\circ\widehat\varphi_g.\nonumber
\end{align}
Expressing $f_0$ by $f$ from the first equation and putting in the second one for $g$, we see that $\widehat\varphi_f^{-1}\widehat\varphi_g\in z+z^2\mathbb{C}[[z]]$ is again a formal conjugation conjugating $g$ to $f$, as was $\widehat\varphi$. On the other hand, by above, it holds that $\widehat\Psi^g=\widehat\Psi^f\circ (\widehat\varphi_f^{-1}\widehat\varphi_g)$.
\smallskip

Just to mention, the new conjugation $\widehat\varphi_f^{-1}\widehat\varphi_g$ is related to $\widehat\varphi(z)$. If we put $f$ and $g$ from \eqref{jaoj} in \eqref{fandg}, we get:
$$
\widehat\varphi_g^{-1}\circ f_0\circ\widehat\varphi_g=(\widehat\varphi_f\circ\varphi)^{-1}\circ f_0\circ(\widehat\varphi_f\circ \varphi).
$$
Using the first part of the lemma, there exists $f_c(z)=\frac{z}{1-cz}$, $c\in\mathbb{C}$, such that
$\widehat\varphi_f\circ\widehat\varphi=f_c\circ\widehat\varphi_g$. That is, $$\widehat\varphi=(\widehat\varphi_f)^{-1}\circ f_c\circ \widehat\varphi_g.$$
\end{proof}

\noindent \emph{Proof of Propositions~\ref{formclas} and \ref{chartriv}}.\

We first prove the additional implication of Proposition~\ref{chartriv} that holds only for germs in the trivial analytic class. Let $f$ and $g$ be two germs analytically conjugated to $f_0(z)=\frac{z}{1-z}$. In the course of the proof of Lemma~\ref{idiot}, we see that any formal conjugacy between $f$ and $g$  is necessarily analytic. This is an important property of the trivial analytic class that is not satisfied for other analytic classes. Due to this, we cannot carry out the same proof for other analytic classes.
Suppose that $f$ and $g$ have the same $1$-moments, $(g_1(t),g_2(t))$. Since $1$-moments are determined only up to the identifications \eqref{idii}, this actually means that we can choose trivialisations $\Psi^f(z)$ and $\Psi^g(z)$ (with appropriate constant terms) such that the moments are exactly the same. Here, we neglect the possible constant term in $1$-moments, simply by choosing the same constant term in $H_+^f$ and $H_-^f$ and $H_+^g$ and $H_-^g$. Let $R_\pm^{f,g}=H_\pm^{f,g}+Log(z)$, as in Remark~\ref{restr}.
\begin{align*}
&R^f_+(z)-R^f_-(z)=g_1(e^{2\pi i\Psi^f(z)}),\\
&R^g_+(z)-R^g_-(z)=g_1(e^{2\pi i\Psi^g(z)}),\ z\in V^{up}.
\end{align*}
The same holds with $g_2(t)$ for $V^{low}$. By Lemma~\ref{idiot}, for no matter what choice of trivialisations $\Psi^f(z)$ and $\Psi^g(z)$, there exists an analytic change of variables tangent to the identity $\varphi(z)$, such that it holds $\Psi^g=\Psi^f\circ \varphi$ and $g=\varphi^{-1}\circ f\circ \varphi$. We therefore get
\begin{align*}
&R^f_+(z)-R^f_-(z)=g_1(e^{2\pi i\Psi^f(z)}),\\
&R^g_+\circ \varphi^{-1}(z)-R^g_-\circ \varphi^{-1}(z)=g_1(e^{2\pi i\Psi^f(z)}),\ z\in V^{up}.
\end{align*}
Similarly for $T_\infty$ on $V^{low}$. We see now that the two formal series $\widehat{R}^f(z)$ and $\widehat{R}^g\circ \varphi^{-1}$ realize the same cocycle on $V^{up},\ V^{low}$ and can thus differ only by a converging series $r_1(z)\in \mathbb{C}\{z\}$,
\begin{equation*}
\widehat{R}^g\circ \varphi^{-1}(z)=\widehat{R}^f(z)+r_1(z).
\end{equation*}
We then have 
\begin{equation}\label{srr}
\widehat{H}^g\circ \varphi^{-1}(z)=\widehat{H}^f(z)+r(z),
\end{equation}
for $r(z)=r_1(z)-Log(\varphi^{-1}(z)/z),\ r(z)\in\mathbb{C}\{z\}$.
\medskip

\noindent Putting \eqref{srr} in the equation $\widehat H^g\circ\varphi^{-1}(f(z))-\widehat H^g\circ\varphi^{-1}(z)=-\varphi^{-1}(z)$, obtained from \eqref{haj} for $g$ after change of variables, we get
$$
-z=\widehat H^f(f(z))-\widehat H^f(z)=-\varphi^{-1}(z)-r(f(z))+r(z).
$$

We now prove the converse for diffeomorphisms in any analytic class. Let $f$ and $g$ belong to any analytic class, $f,\ g\in[f]_0$. Suppose $g=\varphi^{-1}\circ f\circ \varphi$, for some $\varphi(z)\in z+z^2\mathbb{C}\{z\}$, and suppose that there exists $r(z)\in\mathbb{C}\{z\}$ such that $\varphi^{-1}(z)=z+r(f(z))-r(z)$. Let $(g_1^{f}(t),g_2^{f}(t))$ and $(g_1^{g}(t),g_2^{g}(t))$ denote the $1$-moments for $f$ and $g$ respectively. We will prove that they coincide, up to the identifications \eqref{idii}.
From equation \eqref{haj} for $g$, after the change of variables and then using \eqref{classi}, we get
$$
\big(H^g\circ\varphi^{-1}+r\big)(f(z))-\big(H^g\circ\varphi^{-1}+r\big)(z)=-z.
$$
By the uniqueness of the formal solutions of equation \eqref{haj} for $f$ up to a constant term $C\in\mathbb{C}$, we get
\begin{equation}\label{jaja}
\widehat{H}^f(z)=\widehat H^g\circ\varphi^{-1}(z)+r(z)+C.
\end{equation}
Since $r(z)+C$ is analytic, from \eqref{jaja}, we have that (up to a constant term from the choice of sectorial solutions)
\begin{equation}\label{t1}
H^f_+(z)-H^f_-(z)=H^g_+\circ\varphi^{-1}(z)-H^g_-\circ\varphi^{-1}(z),\ \ z\in V^{up}\cup V^{low}.
\end{equation}
By Lemma~\ref{idiot}, for the conjugation $\varphi$ above, there exists a choice of trivialisations (appropriate choice of constant terms) $\Psi_+^f$ and $\Psi_+^g$, such that $\Psi_+^g=\Psi_+^f\circ\varphi$. Then, for $1$-moments with respect to these trivialisations, it holds that:
\begin{align}\label{t2}
&H^f_+(z)-H^f_-(z)=g_1^f(e^{2\pi i \Psi_+^f(z)}),\\
&H^g_+\circ\varphi^{-1}(z)-H^g_-\circ\varphi^{-1}(z)=g_1^g(e^{2\pi i \Psi_+^f(z)}),  \ z\in V^{up}.\nonumber
\end{align}
By \eqref{t1} and \eqref{t2}, and repeating the same for the other component $g_2^{f,g}$ for $V^{low}$, we get that the $1$-moments coincide (defined up to the identifications \eqref{idii}).
$\hfill\Box$
\medskip

Proposition~\ref{chartriv} enables us to define a relation on the model analytic class that identifies all diffeomorphisms with the same $1$-moment. Let $[f_0]_0$ denote the model analytic class, containing all diffeomorphisms analytically conjugated to $f_0$. 
\begin{definition}[Equivalence relation on the model analytic class]\label{eqclass}
Let $f,\ g\in [f_0]_0$. We say that \emph{$f$ and $g$ belong to the same equivalence class} in $[f_0]_0$, and write
$$f \equiv g,$$
if there exists a change of variables $\varphi(z)$ conjugating $f$ and $g$, $g=\varphi^{-1}\circ f\circ\varphi$, and an \emph{analytic function} $r(z)\in\mathbb{C}\{z\}$, such that
\eqref{classi} holds.
\end{definition}
\smallskip

It can be checked that this relation is an equivalence relation on $[f_0]_0$.
Let $[f_0]_0\big /_\equiv$ denote the quotient space of the trivial analytic class with respect to relation $\equiv$.

\begin{theorem}[Bijectivity from the quotiented model analytic class to the set of all $1$-conjugacy classes]\label{bije}
The restriction $\Phi\Big|_{[f_0]_0\big /_\equiv}\Big.$ is a \emph{bijective} map from $[f_0]_0\big /_\equiv$ onto the set of all $1$-conjugacy classes. 
\end{theorem}  
\begin{proof} The statement follows from Theorem~\ref{surjecti} and Proposition~\ref{formclas}.
\end{proof}

\medskip
We finish the section with a comment about the relative positions of the analytic and the $1$-conjugacy classes.
Theorem~\ref{bije} states that the quotiented trivial analytic class in fact \emph{parametrizes} the set of all $1$-conjugacy classes. By Theorem~\ref{surjecti}, we see that the analytic classes and the $1$-conjugacy classes lie in \emph{transversal position}. Each analytic class spreads through all $1$-conjugacy classes. Each $1$-conjugacy class spreads through all analytic classes. Moreover, by Proposition~\ref{formclas}, each $1$-conjugacy class has infinitely many representatives in each analytic class. This makes more precise our observation from Section~\ref{fourfourone} that analytic classification cannot be read only from the differences $H_+(z)-H_-(z)$. 

In particular, in the model analytic class $\mathcal{C}_0$ there exist diffeomorphisms from all $1$-conjugacy classes, and in the trivial $1$-conjugacy class $\mathcal{S}$ there exist diffeomorphisms from all analytic classes. This confirms Example~\ref{noincl}.

\medskip
The same classification analysis could have been done considering the moments with respect to trivializations $\Psi_-$, for repelling instead for attracting petals.

\smallskip
For further research, we hope to determine the relative position of all $k$-conjugacy classes to each other.

\subsection{Reconstruction of the analytic classes from the $1$-conjugacy classes}\label{fourfourfour}

We have seen in Subsection~\ref{fourfourthree}, that we cannot read the analytic classes from the $1$-conjugacy classes with respect to positive trivialisations (or with respect to negative trivialisations). Nevertheless, by comparing the $1$-conjugacy classes of a diffeomorphism \emph{with respect to both sectorial trivializations}, in cases where the $1$-moments are invertible, we can read the analytic class. This is nothing unexpected, since comparing the sectorial trivializations themselves reveals the analytic class.

Let $f$ be formally conjugated to the model $f_0$. We denote by $\big(g_\infty^+(t),g_0^+(t)\big)$ its $1$-moment with respect to trivialisations of the attracting sector, and by  $\big(g_\infty^-(t),g_0^-(t)\big)$ its $1$-moment with respect to trivialisations of the repelling sector. 
\begin{proposition}[Ecalle-Voronin moduli expressed using $1$-moments with respect to both trivialisations]
If all $1$-moment components $g_{\infty,0}^{\pm}(t)$ above are invertible at $t=0$, the Ecalle-Voronin moduli \eqref{mod} defined in Section~\ref{fourone} correspond to compositions
\begin{align}\label{moduli}
\varphi_0(t)&=(g_{0}^-)^{-1}\circ g_0^+(t),\ t\approx 0,\\
\varphi_\infty(t)&=(g_{\infty}^-\circ \tau)^{-1}\circ (g_\infty^+\circ\tau)(t),\ \ t\approx \infty.\nonumber
\end{align} Here $\tau(t)=1/t$ denotes the inversion.
\end{proposition}

\begin{proof} By definition of $1$-moments from Subsection~\ref{fourfourone}, it holds that
\begin{align*}
H_+(z)-H_-(z)&=g_\infty^+(e^{2\pi i \Psi_+(z)})=g_\infty^-(e^{2\pi i \Psi_-(z)}),\ z\in V^{up},\\
H_-(z)-H_+(z)&=-2\pi i+g_0^+(e^{-2\pi i\Psi_+(z)})=-2\pi i+g_0^-(e^{-2\pi i\Psi_-(z)}),\ z\in V^{low}.
\end{align*}
The statement now follows directly from the definition of the Ecalle-Voronin moduli \eqref{mod} in Section~\ref{fourone}.
\end{proof}

\noindent We discuss the invertibility in the next proposition.
\begin{proposition}[Invertibility of $g_{\infty,0}^\pm$]\label{inve}
Suppose 
\begin{equation}\label{suppo}
H_+(z)-H_-(z)\equiv\!\!\!\!\!/\  0,\ z\in V^{up},\quad H_+(z)-H_-(z)\equiv\!\!\!\!\!/\ 2\pi i,\ z\in V^{low}.
\end{equation}
The germs $g_{\infty,0}^\pm$ are either analytic diffeomorphisms or have finitely many analytic (except at zero) inverses.
\end{proposition}
On the other hand, if the difference $H_+(z)-H_-(z)$ is trivial on $V^{up}$ or on $V^{low}$, the moduli cannot be reconstructed using $H_+(z)-H_-(z)$ on $V^{up}\cup V^{low}$ in the above manner, and the analytic class cannot be reconstructed.

\begin{proof}
As we have noted when defining $1$-moments, $g_{0}^{\pm}(t)$ and $g_{\infty}^{\pm}(t)$ extend to $t=0$ to analytic germs. We can suppose without loss of generality that the constant term is $0$, simply by taking the same constant term in sectorial solutions $H_+(z)$ and $H_-(z)$. The germs are nonzero by \eqref{suppo}. 
Depending on the first nonzero term in their Taylor series, we distinguish between two cases of invertibility of $g_0^{\pm}(t)$ at $t=0$. 

From \eqref{baz}, it holds that $g_0^+(t)=g_0^{-}\circ\varphi_0(t)$. Since modulus $\varphi_0(t)$ is an analytic diffeomorphism, the developments for $g_0^+(t)$ and $g_0^{-}(t)$ begin with the same monomial, say $t^k$. Therefore their invertibility is discussed in the same manner. 

\noindent Let $g_0^-(t)=a_k t^k+o(t^k),$ $a_k\neq 0$, $k\in\mathbb{N}$.

\begin{enumerate}
\item[$(i)$]\ $\mathbf{a_1\neq 0}$. The functions $g_0^\pm$ are analytic diffeomorphisms and they have unique inverses $(g_0^\pm)^{-1}(t)$. 

\item[$(ii)$]\ $\mathbf{a_1,\ldots,a_{k-1}=0,\ a_k\neq 0,\ k>1}$. It holds $g_0^-(t)=a_k t^k+o(t^k),\ k>1$. There exists a unique analytic diffeomorphism $h(t)$, tangent to the identity, such that $g_0^-(t)=a_k \big(h(t)\big)^k$. There exist $k$ different analytic (except at $t=0$) inverses of $g_0^-(t)$, given by the formulas
$$
(g_0^-)^{-1}(t)=h^{-1}\Big({a_k}^{-1/k}\cdot t^{1/k}\Big),
$$
where $h^{-1}(t)$ is unique inverse of $h(t)$. Here, $k$ inverses are determined by the choice of $k$ different roots $a_k^{-1/k}$.
\end{enumerate}
\end{proof}
In the second case, one of the analytic inverses gives moduli by \eqref{moduli}. For example, if we choose both trivialisations $\Psi_+(z)$ and $\Psi_-(z)$ with the same constant term, then the moduli $\varphi_0(t)$ and $\varphi_\infty(t)$ are \emph{tangent to the identity}, and we know exactly which inverse to choose. Indeed, it can be seen directly from formulas \eqref{mod} for the moduli that $$\varphi_0(t)=e^{-2\pi i(D-C)}t+o(t),\ \ (1/t)\circ \varphi_\infty(1/t)=e^{2\pi i(D-C)}t+o(t),\quad t\approx 0,$$ where $C,\ D\in\mathbb{C}$ are constant terms of $\Psi_+(z)$ and $\Psi_-(z)$ respectively.

\section{Perspectives}\label{foursix}
\subsection{Can we read the analytic class from $\varepsilon$-neighborhoods of \emph{only one} orbit?}\label{foursixone}
We exploit ideas from the proof of Proposition~\ref{accu} to prove that a parabolic diffeomorphism is uniquely determined by the complex measures of $\varepsilon$-neighborhoods of \emph{only one orbit}, given on some small interval in $\varepsilon$. That is, by the function $\varepsilon\mapsto \widetilde{A^{\mathbb{C}}}(S^f(z_0)_\varepsilon)$, $\varepsilon\in(0,\varepsilon_0)$, where $\varepsilon_0>0$ is arbitrary small. Note that $z_0\in V_+$ is fixed and this function is realized \emph{using only one orbit}. This result suggests that the $\varepsilon$-neighborhoods of only one orbit should be enough to read the analytic class, as was the case for formal classification discussed in Chapter~\ref{three}. Therefore, it should suffice to fix $z$ and regard $\widetilde{A^\mathbb{C}}(S^f(z)_\varepsilon)$ as function of $\varepsilon$ only. However, how this can be done remains open and subject to further research. Note that this is a different approach to the problem. In Chapter~\ref{four}, we have been considering and comparing sectorial functions, derived from $\widetilde{A^\mathbb{C}}(S^f(z)_\varepsilon)$, with respect to the variable $z$. 

\begin{proposition}\label{undet}
Let $z_0\in V_+$ be fixed. Let $\varepsilon_0>0$. The mapping
$$
f \longmapsto \big(\varepsilon\mapsto \widetilde{A^\mathbb{C}}(S^f(z_0)_\varepsilon),\ \varepsilon\in(0,\varepsilon_0)\big)
$$
is injective on the set of all parabolic diffeomorphisms $f:(\mathbb{C},0)\to(\mathbb{C},0)$.  
\end{proposition}

\begin{proof}
Suppose that 
$
\widetilde{A^\mathbb{C}}(S^f(z_0)_\varepsilon)=\widetilde{A^\mathbb{C}}(S^g(z_0)_\varepsilon),\ \varepsilon\in (0,\varepsilon_0).
$ We show that the germs $f(z)$ and $g(z)$ must be equal.

Separating the tails and the nuclei, and dividing by $\varepsilon^2\pi$, we get
\begin{equation}\label{sep}
\frac{\widetilde{A^\mathbb{C}}(T^f_\varepsilon)-\widetilde{A^\mathbb{C}}(T^g_\varepsilon)}{\varepsilon^2\pi}=\frac{\widetilde{A^\mathbb{C}}(N^g_\varepsilon)-\widetilde{A^\mathbb{C}}(N^f_\varepsilon)}{\varepsilon^2\pi},\ \varepsilon\in(0,\varepsilon_0).
\end{equation}
The proof relies on the presence of singularities at points $(\varepsilon_n)^{f,g}$, which was stated in Proposition~\ref{accu}. Let $(z_n)$ denote the orbit $S^f(z_0)$ and $(w_n)$ the orbit $S^g(w_0)$. Recall from Proposition~\ref{accu} that
$$
\varepsilon_n^f=\frac{|z_{n}-z_{n+1}|}{2},\ \varepsilon_n^g=\frac{|w_{n}-w_{n+1}|}{2},\ \ n\in\mathbb{N}.
$$
Suppose that the sequences of singularities for $f$ and $g$, $(\varepsilon_n^f)$ and $(\varepsilon_n^g)$, do not eventually coincide. Then, there exists $n$ arbitrary big and an interval $(\varepsilon_n^f-\delta,\varepsilon_n^f+\delta),\ \delta>0$, such that $\varepsilon_{m-1}^g<\varepsilon_n^f-\delta$ and $\varepsilon_n^f+\delta<\varepsilon_m^g$. 
Consider the second derivative $\frac{d^2}{d\varepsilon^2}$ of \eqref{sep} from the right. With all the notations and conclusions as in the proof of Proposition~\ref{accu}, from \eqref{druga}, we have
\begin{align}\label{contra}
0={(G_{n+1}^f)}^{\prime\prime}&(\varepsilon_n^f+)+{(G_{m-1}^g)}^{\prime\prime}(\varepsilon_n^f+)+\nonumber\\  &+\frac{1}{\pi}\bigg(\frac{4\varepsilon_n^f}{\varepsilon^3}\sqrt{1-\frac{(\varepsilon_n^f)^2}{\varepsilon^2}}-\frac{2(\varepsilon_n^f)^3}{\varepsilon^5}\frac{1}{\sqrt{1-\frac{(\varepsilon_n^f)^2}{\varepsilon^2}}}\bigg)\Bigg|_{\varepsilon=\varepsilon_n^f+}\cdot\big(z_{n+1}+z_n).
\end{align}
Since all terms are bounded except for the term in brackets and $z_n+z_{n+1}\neq 0$, \eqref{contra} leads to a contradiction. Therefore, the sequences of singularities $(\varepsilon_n^f)$ and $(\varepsilon_n^g)$ eventually coincide,
$$
\varepsilon_n^f=\varepsilon_{n+k_0}^g,\ n\geq n_0,\ k_0\in\mathbb{N}.
$$

Now, considering the second derivative \eqref{sep} at the singularity $\varepsilon_n=\varepsilon_n^f=\varepsilon_{n+k_0}^g$ from the right, instead of \eqref{contra}, we have:
\begin{align*}
&0={(G_{n+1}^f)}^{\prime\prime}(\varepsilon_n+)+{(G_{n+k_0+1}^g)}^{\prime\prime}(\varepsilon_n+)+\\  &+\frac{1}{\pi}\bigg(\frac{4\varepsilon_n}{\varepsilon^3}\sqrt{1-\frac{\varepsilon_n^2}{\varepsilon^2}}-\frac{2\varepsilon_n^3}{\varepsilon^5}\frac{1}{\sqrt{1-\frac{\varepsilon_n^2}{\varepsilon^2}}}\bigg)\Bigg|_{\varepsilon=\varepsilon_n+}\cdot\hspace{-1.2cm}\cdot \Big(z_{n+1}+z_n-(w_{n+k_0+1}+w_{n+k_0})\Big).\nonumber
\end{align*}
The term in brackets is the only unbounded term, therefore 
$
(z_{n+1}+z_n)-(w_{n+k_0+1}+w_{n+k_0})=0.
$ The middle points of the orbits $S^f(z_0)$ and $S^g(z_0)$ eventually coincide. Since the distances $d_n^f=2\varepsilon_n^f$ and $d_n^g=2\varepsilon_n^g$ coincide, and since both orbits converge to some tangential direction, it is easy to see that the orbits themselves eventually coincide.

The diffeomorphisms $f$ and $g$, both analytic at $z=0$, coincide on a set accumulating at the origin. Therefore, they must be equal. 
\end{proof}

\subsection{$1$-Abel equation in analytic classification of two-dimensional diffeomorphisms}\label{fivetwo} 

This result is due to David Sauzin, in personal communication, and we put it here only as a possible application. It gives an application of the $1$-Abel equation for a diffeomorphism $f$, along with its standard Abel equation, for obtaining the analytic classification of two-dimensional germs derived from $f$, of the form $F(z,w)=(f(z),z+w)$.

Let $f$ belong to the formal class of $f_0$. We consider two-dimensional germs of diffeomorphisms $F:\mathbb{C}\times \mathbb{C}\to \mathbb{C}\times \mathbb{C}$ of the type
$$
F(z,w)=(f(z),z+w).
$$
Each two-dimensional diffeomorphism of the above type can by a unique formal change of variables $\Phi(z,w)\in \mathbb{C}[[z,w]]^2$ be reduced to a formal normal form of the type
$$
F_0(z,w)=(f_0(z),z+w).
$$
Here, $f_0(z)=\frac{z}{1-z}$ is the formal normal form for $f(z)$ and $\mathbb{C}[[z,w]]$ denotes a formal series of the form $\sum_{n=1}^\infty \sum_{\{k,l\in\mathbb{N}_0:\ k+l=n\}}z^l w^k$, without the constant term.

The formal conjugation $\widehat{\Phi}(z,w)$, $F=\widehat{\Phi}^{-1}\circ F_0\circ \widehat{\Phi}$, is given by
\begin{equation}\label{conj2}
\widehat{\Phi}(z,w)=\Big(\widehat{\varphi}(z),\ \widehat{H}(z)-\widehat{H}^{f_0}\circ\widehat{\varphi}(z)+w\Big).
\end{equation}
Here, $\widehat{\varphi}(z)$ is the formal conjugation that conjugates $f(z)$ to $f_0(z)$. $\widehat{H}(z)$ is the formal solution (without the constant term) of $1$-Abel equation \eqref{haj} for the diffeomorphism $f$ and $\widehat{H}^{f_0}(z)$ is the formal solution for $f_0$.

\noindent To conclude, \emph{$F$ is analytically conjugated to its formal normal form $F_0$ if and only if $f$ is analytically conjugated to $f_0$ by $\varphi$, and $$H_+(z)-H_-(z)\equiv H_+^{f_0}(\varphi(z))-H_-^{f_0}(\varphi(z)),\ z\in V^{up}\cap V^{low},$$
the latter difference for $f_0$ being known by Example~\ref{ex1}.}

\smallskip
The problem of formal conjugacy can be formulated equivalently using trivialization equation that conjugates $F$ with translation by $(1,0)$. We search for formal solutions $\widehat{T}(z,w)$ of the trivialization equation for $F$:
\begin{equation}\label{trivic}
\widehat{T}(F(z,w))=\widehat{T}(z,w)+(1,0).
\end{equation}
It can be checked that formal solutions of the trivialization equation \eqref{trivic} for the formal normal form $F_0(z,w)$ is given by
\begin{equation}\label{triv0}
\widehat{T}_0(z,w)=(\Psi^{f_0}(z),\ \widehat{H}^{f_0}(z)+w).
\end{equation}
Here, $\Psi^{f_0}(z)=-1/z$ is (global) trivialization function for $f_0(z)$, and $\widehat{H}^{f_0}(z)$ is the (fixed) formal solution of 1-Abel equation \eqref{haj} for $f_0$, which is only sectorially analytic, see Examples~\ref{ex1} and \ref{ex2} in Subsection~\ref{fourfourone}.

As in 1-dimensional case, by \eqref{conj2}, \eqref{trivic} and \eqref{triv0}, we get that the formal trivialization $\widehat{T}(z,w)$ for diffeomorphism $F=\widehat{\Phi}^{-1}\circ F_0\circ \widehat{\Phi}$ is given by
\begin{equation}\label{trivef}
\widehat{T}(z,w)=\widehat{T}_0\circ \widehat{\Phi}(z,w)=\Big(\widehat{\Psi}(z),\widehat{H}(z)+w\Big), 
\end{equation}
where $\widehat{\Psi}(z)$ is the formal solution of the Abel equation for $f(z)$.

Obviously, by \eqref{trivef}, Abel equation for $f(z)$ appears as first coordinate, and 1-Abel equation for $f(z)$ as second coordinate in the trivialization equation \eqref{trivic} for $F(z)=(f(z),z+w)$.

\section{Proofs of auxiliary results}\label{fourseven}
\noindent We prove here Lemma~\ref{seriesana} from Section~\ref{fourthree}. We also state an auxiliary proposition from the proof of Proposition~\ref{accu} in Section~\ref{fourtwo}.
\medskip

\noindent \emph{Proof of Lemma~\ref{seriesana}}.
One implication is obvious. 

We prove the other implication, i.e., that the analyticity of $\widehat{T}(z)$ implies the analyticity of $\widehat{g}(z)$. Let $T(z)\in \mathbb{C}\{z\}$. In the case where $h(z)$ is an analytic diffeomorphism, that is, begins with the linear term, the statement is obvious simply by inverting. In other cases, the proof relies on the existence of formal series of $\widehat{g}(z)$. 

Suppose that $h(z)$ is of order $l\geq 0$, that is, $h(z)=\alpha_l z^l+o(z^l)$, $\alpha_l\in\mathbb{C}$, $\alpha_l\neq 0$. Suppose $\widehat{g}(z)=\beta_k z^k+o(z^k)$, $\beta_k\neq 0$, $k\geq 0$. If $l=0$, since $T(z)$ and $h(z)$ are not constant functions, we can cancel the constant term on both sides and proceed as in the case $l\geq 1$. Suppose therefore $l\geq 1$. By \eqref{tet}, $T(z)=\alpha_l\cdot \beta_k^l\cdot z^{l+k}+o(z^{l+k})$, and we can divide by $\alpha_l\cdot\beta_k^l$.  We can therefore suppose without loss of generality that the leading coefficient in $T(z), h(z)$, and $\widehat{g}(z)$ is equal to 1. Since $T$ and $h$ are analytic, there exists an analytic function $p(z)\in z^k+z^{k+1}\mathbb{C}\{z\}$ of order $k$ and an analytic diffeomorphisms $q(z)\in z+z^2\mathbb{C}\{z\}$, such that $T(z)=(p(z))^l$ and $h(z)=(q(z))^l$. The equation \eqref{tet} can therefore be rewritten as
\begin{equation}\label{eqi}
(p(z))^l=(q\circ \widehat{g}(z))^l.
\end{equation}
Considering $p(z)$ and $q\circ \widehat{g}(z)$ as two formal series in the appropriate subsets of $\mathbb{C}[[z]]$, putting them in \eqref{eqi} and simply comparing the coefficients on both sides, we conclude that the formal series for $p(z)$ and $q\circ \widehat{g}(z)$ must be equal. Since $p(z)$ is analytic, the formal series $q\circ \widehat{g}(z)$ also converges to an analytic function. Now inverting the analytic diffeomorphism $q(z)$, we conclude that $\widehat{g}(z)$ is analytic.
$\hfill\Box$

\bigskip

We state and prove an auxiliary proposition that we use in the proof of Proposition~\ref{accu} in Subsection~\ref{fourtwoone}.

\begin{proposition}\label{auxi}
Let the sequence $(\varepsilon_n)$ be as defined in Subsection~\ref{fourtwoone}. Let $\delta>0$ such that $\varepsilon_{n+1}+\delta<\varepsilon_n$. For each $n\in\mathbb{N}$, the function
$$
H_{n+1}(\varepsilon)=\frac{1}{\pi}\sum_{l=n+1}^{\infty}\left[\Big(\frac{\varepsilon_{l}}{\varepsilon}\sqrt{1-\frac{\varepsilon_{l}^2}{\varepsilon^2}}+\arcsin{\frac{\varepsilon_{l}}{\varepsilon}}\Big)(z_{l}+z_{l+1})\right]
$$
is a well-defined $C^\infty$-function in $\varepsilon$ on the interval $\varepsilon\in(\varepsilon_{n+1}+\delta,\varepsilon_{n-1})$. Moreover, the differentiation of the sum is performed term by term.
\end{proposition}

\begin{proof}
We consider separately real and imaginary part of the function. We show that the real part,
$$
Re\big(H_{n+1}(\varepsilon)\big)=\frac{1}{\pi}\sum_{l=n+1}^{\infty}\left[\Big(\frac{\varepsilon_{l}}{\varepsilon}\sqrt{1-\frac{\varepsilon_{l}^2}{\varepsilon^2}}+\arcsin{\frac{\varepsilon_{l}}{\varepsilon}}\Big)\big(Re(z_{l})+Re(z_{l+1})\big)\right],
$$
has all the properties from the statement. For the imaginary part, the proof is the same. 

Let us denote the functions under the summation sign by $$g_l(\varepsilon)=\Big(\frac{\varepsilon_{l}}{\varepsilon}\sqrt{1-\frac{\varepsilon_{l}^2}{\varepsilon^2}}+\arcsin{\frac{\varepsilon_{l}}{\varepsilon}}\Big)\big(Re(z_{l})+Re(z_{l+1})\big),\ l\in\mathbb{N}.$$ Then, $$Re(H_{n+1}(\varepsilon))=\frac{1}{\pi}\sum_{l=n+1}^\infty g_l(\varepsilon).$$ Since $z_n\to 0$, as $n\to\infty$, and $\lim_{x\to 0}\frac{x}{\arcsin{x}}=1$, the following upper bound holds: there exists a constant $M_n>0$ such that
$$
|g_l(\varepsilon)|\leq M_n\cdot \varepsilon_l,\ \ l\geq n+1,\ \text{for all }\varepsilon\in [\varepsilon_{n+1}+\delta,\varepsilon_{n-1}].
$$
Since $\varepsilon_l\simeq l^{-1-\frac{1}{k}}$, $l\to\infty$, see \eqref{dn} in Subsection~\ref{threetwotwo}, the series $\sum_{l=n+1}^{\infty}\varepsilon_l$ converges. Therefore the series for $Re(H_{n+1}(\varepsilon))$ converges uniformly on $[\varepsilon_{n+1}+\delta,\varepsilon_{n-1}]$. As a uniform limit of continuous functions on a segment, $Re(H_{n+1}(\varepsilon))$ is a continuous function on $[\varepsilon_{n+1}+\delta,\varepsilon_{n-1}]$.

We now show that $Re(H_{n+1}(\varepsilon))$ is differentiable on $(\varepsilon_{n+1}+\delta,\varepsilon_{n-1})$ and that it can be differentiated term by term. The functions $$g_l'(\varepsilon)=\left(-\frac{2\varepsilon_l}{\varepsilon^2}\sqrt{1-\frac{\varepsilon_l^2}{\varepsilon^2}}\right)\big(Re(z_l)+Re(z_{l+1})\big), \ l\geq n+1,$$ are continuous functions on $[\varepsilon_{n+1}+\delta,\varepsilon_{n-1}]$, and it can be shown similarly as above that the series $\sum_{l=n+1}^{\infty}g_l'(\varepsilon)$ converges uniformly on $[\varepsilon_{n+1}+\delta,\varepsilon_{n-1}]$. Therefore, it can be differentiated termwise,
$$
(Re\ H_{n+1})'(\varepsilon)=\frac{1}{\pi}\sum_{l=n+1}^{\infty}g_l'(\varepsilon),\ \varepsilon\in[\varepsilon_{n+1}+\delta,\varepsilon_{n-1}].
$$
Moreover, $(Re\ H_{n+1})'(\varepsilon)$ is a continuous function on $[\varepsilon_{n+1}+\delta,\varepsilon_{n-1}]$, as uniform limit of continuous functions. Therefore $Re(H_{n+1}(\varepsilon))$ is of class $C^1$ on $[\varepsilon_{n+1}+\delta,\varepsilon_{n-1}]$. The same procedure can now be repeated for higher-order derivatives.
\end{proof}
\chapter{Conclusions and perspectives}

In this thesis, we considered discrete dynamical systems generated by germs of diffeomorphisms of the real line and by complex germs of diffeomorphisms, locally around a fixed point.

In the \textbf{first part of the thesis}, Chapters~\ref{two} and \ref{three}, we used fractal analysis in recognizing diffeomorphisms of the real line and complex diffeomorphisms and saddles, from the viewpoint of formal classification. In case of diffeomorphisms of the real line, we considered those differentiable at the origin, as well as those which are not, but which additionally decompose in so-called \emph{Chebyshev scales}, well-ordered by flatness. In case of complex diffeomorphisms, we considered all except those with irrational rotation in the linear part. In all mentioned cases, it was shown that fractal analysis of only one trajectory of the system tending to a fixed point is enough to classify the generating diffeomorphisms. That is, to read the multiplicity of the fixed point or the formal class of the diffeomorphism. The method of fractal analysis is applicable since fractal properties of only one orbit can be computed numerically.

Furthermore, we applied the obtained results to continuous dynamical systems. This was done using Poincar\' e maps in simple cases of planar limit periodic sets, or using holonomy maps for germs of complex saddle fields. We showed that the box dimension (or its appropriate generalization) of only one trajectory of the Poincar\' e map shows the cyclicity of the set in a generic analytic unfolding. Similarly, we proved that generalized fractal properties of only one trajectory of the holonomy map show the orbital formal normal form of a complex saddle itself. We expect that the normal form can equivalently be read using fractal properties of one leaf of a foliation around the saddle. We made some preliminary results. This was motivated by an analogy with planar cases, where the box dimension of one spiral trajectory could have been used instead of the box dimension of one orbit of its Poincar\' e map, see e.g. \cite{buletin, belgproc}.
\medskip

However, many questions related to this research are still open, and we list here some of them. We expect to consider them in the future.

To start with, we need to compute the box dimension of leaves of a foliation at complex saddles and prove the conjecture from Subsection~\ref{threethreethree} that the box dimension of any leaf of a foliation around the formally orbitally nonlinearizable complex saddle reveals its first orbital formal invariant. For another formal invariant, we are obliged to analyse further terms in the asymptotic development of the $\varepsilon$-neighborhoods of leaves, as was the case with parabolic diffeomorphisms before.

Secondly, we mentioned in Section~\ref{threeone} that we omitted the very complicated case of complex diffeomorphisms with \emph{irrational rotations} in the linear part. They are formally linearizable. On the other hand, the necessary and sufficient conditions imposed on them for analytic linearizability at the origin are very complicated. They were discovered by Bryuno and Yoccoz. It would be of interest to see if fractal analysis of their orbits can give us some insight to analytic linearizability. We have seen that if a diffeomorphism of the above form is analytically linearizable, then the box dimension of its any trajectory is equal to 1. The question is the converse. 

\bigskip
The \textbf{second part of the thesis}, Chapter~\ref{four}, was dedicated to the problem of reading the analytic classes of parabolic diffeomorphisms from $\varepsilon$-neighborhoods of their orbits. We provided results describing analyticity properties of the complex measures of $\varepsilon$-neighborhoods of orbits, but we did not solve the original question of analytic classification. Many interesting questions for the future work have appeared.

The first one concerns the analytic classification of parabolic diffeomorphisms using $\varepsilon$-neighborhoods of \emph{only one orbit}, that is, for fixed initial point $z$. We proved in Section~\ref{foursix}, that the complex measures of $\varepsilon$-neighborhoods of only one orbit on some small interval $(0,\varepsilon_0)$ determine the diffeomorphism uniquely. In particular, its analytic class is determined. It is nevertheless unclear how to read it, and this question remains for the future research.

Secondly, investigating analyticity properties of measures of $\varepsilon$-neighborhoods of orbits of parabolic diffeomorphisms, special type of difference equation has appeared, which we have named the \emph{generalized Abel equation of order one}. These equations resemble the trivialisation (Abel) equation, which is standardly used in the analytic classification problem. This suggested considering the \emph{generalized Abel equations of higher orders}, into which both equations fit as special cases. It motivated us to introduce new classifications of diffeomorphisms with respect to the equations of higher orders, in a similar manner as the analytic classification was deduced from the trivialisation equation. We were mostly interested in the first order equation, that appeared naturally when considering the $\varepsilon$-neighborhoods of orbits. For the future, many questions are open. For example, if these new classes can be interpreted geometrically and what geometric properties of a generating function they reveal. For the first and the zero order equations, we analysed to some extent the relative position of their classes, to see if they are in any relation to each other, and found out that they are \emph{transversal}. In the future, this should be also done for the generalized Abel equations of higher orders.

\end{document}